\documentclass{article}

\usepackage{fullpage}

\usepackage{amsmath,amsfonts,amsthm,amssymb}
\usepackage{algorithm,algorithmic}
\usepackage{mathrsfs}
\usepackage{color,xcolor}
\usepackage{ifpdf}
\usepackage{psfrag}
\usepackage{graphicx,graphics,epsfig,subcaption}
\usepackage{url}
\usepackage{hyperref}
\usepackage{xspace}
\usepackage{xargs}
\usepackage[colorinlistoftodos,prependcaption,textsize=small]{todonotes}
 \usepackage{ booktabs,array}
 \usepackage{nicefrac,bm,mathtools}

\usepackage{stmaryrd}

\usetikzlibrary{patterns}
\newtheorem{theorem}{Theorem}[section]

\newtheorem{lemma}[theorem]{Lemma}

\newtheorem{proposition}[theorem]{Proposition}
\newtheorem{rem}{Remark}
\numberwithin{rem}{section}
\numberwithin{equation}{section}
\newtheorem{definition}[theorem]{Definition}
\newtheorem{assumption}[theorem]{Assumption}
\definecolor{OliveGreen}{rgb}{0.33, 0.42, 0.18}
\definecolor{Plum}{rgb}{0.8, 0.6, 0.8}
\newcommandx{\unsure}[2][1=]{\todo[linecolor=red,backgroundcolor=red!25,bordercolor=red,#1]{#2}}
\newcommandx{\change}[2][1=]{\todo[linecolor=blue,backgroundcolor=blue!25,bordercolor=blue,#1]{#2}}
\newcommandx{\info}[2][1=]{\todo[linecolor=OliveGreen,backgroundcolor=OliveGreen!25,bordercolor=OliveGreen,#1]{#2}}
\newcommandx{\improvement}[2][1=]{\todo[linecolor=Plum,backgroundcolor=Plum!25,bordercolor=Plum,#1]{#2}}
\newcommandx{\thiswillnotshow}[2][1=]{\todo[disable,#1]{#2}}

\definecolor{mygreen}{rgb}{0,0.6,0}

\newcommand{\dpar}[2]{\dfrac{\partial #1}{\partial #2}}

 \newcommand{\R}{\mathbb R}
 
 \newcommand{\N}{\mathbb N}

\renewcommand{\P}{\mathbb P}

\newcommand{\Id}{\mathbb I}

\newcommand{\BB}{\mathcal{B}}
\newcommand{\CC}{\mathcal{C}}
\newcommand{\DD}{\mathcal{D}}
\newcommand{\EE}{\mathcal{E}}
\newcommand{\FF}{\mathcal{F}}

\newcommand{\II}{\mathcal{I}}

\newcommand{\KK}{\mathcal{K}}

\newcommand{\PP}{\mathcal{P}}
\newcommand{\QQ}{\mathcal{Q}}

\renewcommand{\SS}{\mathcal{S}}
\newcommand{\TT}{\mathcal{T}}

\newcommand{\VV}{\mathcal{V}}

\newcommand{\bba}{{\mathbf{{a}}}}

\newcommand{\bbf}{{\mathbf{{f}}}}
\newcommand{\bbg}{{\mathbf{{g}}}}

\newcommand{\bbn}{{\mathbf{{n}}}}

\newcommand{\bbs}{{\mathbf{{s}}}}

\newcommand{\bbu}{{\mathbf{{u}}}}
\newcommand{\bbv}{{\mathbf{{v}}}}
\newcommand{\bbw}{{\mathbf{{w}}}}
\newcommand{\bbx}{{\mathbf{{x}}}}
\newcommand{\bby}{{\mathbf{{y}}}}
\newcommand{\bbz}{{\mathbf{{z}}}}

\newcommand{\bbA}{{\mathbf{{A}}}}
\newcommand{\bbB}{{\mathbf{{B}}}}

\newcommand{\bbD}{{\mathbf{{D}}}}

\newcommand{\bbJ}{{\mathbf{{J}}}}

\newcommand{\bbL}{{\mathbf{{L}}}}
\newcommand{\bbM}{{\mathbf{{M}}}}
\newcommand{\bbN}{{\mathbf{{N}}}}

\newcommand{\bbP}{{\mathbf{{P}}}}

\newcommand{\bbR}{{\mathbf{{R}}}}

\newcommand{\bbT}{{\mathbf{{T}}}}
\newcommand{\bbU}{{\mathbf{{U}}}}

\newcommand{\bbW}{{\mathbf{{W}}}}
\newcommand{\bbX}{{\mathbf{{X}}}}

\newcommand{\btheta}{\bm{\theta}}
\newcommand{\btau}{\bm{\tau}}
\newcommand{\bbeta}{\bm{\beta}}
\newcommand{\bsigma}{\bm{\sigma}}

\newcommand{\hbbf}{\hat{\bbf}}

\newcommand{\bmw}{\bm{w}}

\renewcommand{\ker}{\text{Ker }}
\begin{document}
\title{
 Embedding General Conservation Constraints in Discretizations of  Hyperbolic Systems on Arbitrary Meshes: A
Multidimensional Framework}

\author{R\'emi Abgrall$^{(1)}$,    Pierre-Henri Maire$^{(2)}$ and Mario Ricchiuto$^{(3)}$ \\
(1) Institute of Mathematics, University of Z\"urich, Switzerland\\
(2) CEA Cesta, 15 avenue des sabli\`{e}res, Le Barp, France\\
(3) Inria centre de l'Universit\'e de Bordeaux, Talence, France
}
\maketitle


\begin{abstract}
The purpose of this review is to discuss the notion of conservation in hyperbolic systems and how one can formulate it at the discrete level 
depending on the  representation  of the solution on the mesh.
 Since it is impossible to have a fully general theory, we 
 discuss several alternatives possibilities: cases where the solution is represented by average in volumes; cases where the mesh is staggerred (i.e. the components of the solution are not localised at the same places);  cases where the solution is solely represented by point values;  
 and an  example where all the previous options are mixed.
  We show how each configuration can provide, or not, enough flexibility. Though the discussion could be adapted to any hyperbolic system endowed with an entropy, we focus on compressible fluid mechanics, it its Eulerian and Lagrangian formulations.

On a given mesh, the unifying element is that we systematically express
 the update of conserved variables as  
 $u^{n+1}=u^n- \Delta t\; \delta u$, where the functional $u\mapsto \delta u$ depends on the value of $u$ at the current degree of freedom and its values from a set of  degrees of freedom. This set of define the stencil of the scheme. From the stencil, one can naturally define a graph connecting the states  that appears in $\delta u$. The notion of local conservation can be defined from this graph. We are aware of only two possible situations: either the graph is constructed from the faces of the mesh elements (or the dual mesh), or it  is defined from the mesh itself. Two notions of local conservation then emerge: either we define a numerical flux, or we define a "residual" attached to elements and the degrees of freedom within the element. We show that this two notions are in a way equivalent, but the one with residual allows much more flexibility, especially if additional algebraic constraints must be satisfied.
 Examples of specific additional conservation constraints are provided to illustrate this flexibility. We also show that this notion of conservation gives a very clear framework for the design of schemes in the Lagrangian setting.
 
 In the ending section we will provide a number of ongoing research avenues strongly related to the formulation discussed,
 and we highlight some   open questions which will be explored in the future.

\end{abstract}

\section{Introduction}
We are interested in the numerical approximation of hyperbolic problems.  
The   number of problems in Physics and Mechanics that can be formulated by a set of hyperbolic (or perturbation of hyperbolic) Partial Differential Equations (PDE) is  enormous: 
compressible fluid dynamics, in plasma physics, meteorological and oceanic flows, astrophysics, materials, multiphase and multi-component flows, etc.  All these problems can be formulated by a partial differential equation that has the following structure:
$$ 
\dpar{\bbu}{t}+\text{ div }\bbf(\bbu)+\mathcal{B}(\bbu) \nabla \bbu=\mathbf{s}(\bbx, t, \bbu),
$$ where $\bbu$ is a vector of $\R^p$, $\bbf=(f_1, \ldots, f_d)$ is a flux with $f_i$ being $C^1$ functions of $\bbu$, $\mathcal{B}(\bbu)=(\bbB_1, \ldots, \bbB_d)$ is a tensor with $\bbB_j\in M_{p\times p}(\R)$. We have set for $\bbW=(\bbw_1, \ldots \bbw_d)$ with 
for $\bbw_j\in \R^p$ and $1\leq j\leq d$,
$$\mathcal{B} \bbW=\sum_{j=1}^d \bbB_j \bbw_j.$$ Last,  $\mathbf{s}$ is a source term.  

In this paper, we will focus our attention to the case where $\bbB=\mathbf{0}$ and $\bbs=0$, {\it i.e.}, a conservative system. The case where $\bbs$ is not identically $0$ has its special difficulties such as well balancing, {\it i.e.}, the reproduction of steady solutions. The case where $\bbB\neq \mathbf{0}$ is more challenging because the notion of solution is not completely clear despite the work of Dal Maso et al.  \cite{dalmaso} because it depends on 'paths' which are not obvious to define. Numerically speaking, it is also very challenging to define numerical scheme that converge, even empirically, to given solutions defined by path, see e.g. \cite{Karni,XU2023112297}.

The case of fully conservative problem is better understood, even if most of the problems are open for the system case, even the notion of solution where uniqueness may not be guaranteed, see  Section~\ref{intro} below for some references. We thus focus on the numerical approximation of
%
\begin{equation}
\label{1_1}
\dpar{\bbu}{t}+\text{ div }\bbf(\bbu)=0.
\end{equation}
with initial and boundary conditions. Considering fluid mechanics, and if one believes that the Euler equations mimic the relevant physics, one must admit solutions that are discontinuous. For example, in a tube closed on the left by a piston, if the piston starts moving to the right, a shock wave will appear. A shock wave is a moving surface across which all variables are discontinuous. On the same line, looking at a wing, it is clear that there is a surface, or something more complex, across which the tangential velocity is discontinuous.  In any case, this means that the formulation \eqref{1_1} is not sufficient because it contains derivatives. P.D. Lax, in the 50's, has introduced a notion of weak solution. This notion avoids the use of derivatives of $\bbu$ and mimics mathematically the concept of conservation: "what comes in,  comes out". Looking at \eqref{1_1}, we can also observe that the operator is a space time divergence, {\it i.e}, \eqref{1_1} writes
\begin{equation}
\label{1_1st}
\text{div}_{\bbx,t}\mathbf{F}=0,\end{equation}
where $\mathbf{F}=(\bbu, \bbf(\bbu))$.  The same philosophy of "what comes in,  comes out" thus also applies to time dependent problems. 

 The notion of conservation plays a crucial 
 role in the numerical approximation of \eqref{1_1}, and is the main focus of 
 this paper. This has been first identified by Lax and Wendroff in \cite{LxW}  where they propose 
 a generic form of finite volume schemes 
guaranteeing that if the discrete solution is converging to some measurable function, then 
this function is a weak solution of the problem. Since then, this result 
has been extended to multiple dimensions, with all the kind of meshes one can imagine, see for example \cite{Kroner:96}. The discussion is still going on to define what is the "best" form, see \cite{shuConservation} for example.  The main purpose of this paper is to discuss an alternative definition of what is a locally conservative scheme, and use this notion to provide practical improvements to existing schemes. 
Considering fluid mechanics, we discuss this in the Eulerian and Lagrangian setting. 

In the Eulerian context, the solution can be represented in several manners. Let us consider a tesselation of the computational domain by polygons or polyhedrons. We call them "elements" by abuse of language.   One can logically attach the degrees of freedom to these elements, as in the cell-centered finite volume methods, or the discontinuous Galerkin methods. If the elements are simplex, one can approximate the solution by polynomials in each simplex. If   discontinuities in the polynomials are admitted across the faces of the elements, then we are back to the discontinuous Galerkin method. If the solution is assumed to be globally continuous, we are in the finite element universe, where again several variants are possible. One intermediate solution,  in the piecewise linear case, is to construct for each vertex a control volume that allows to reformulate a finite volume method as a finite element one.  The variety of algorithms is very large.

The same variety also exist in the Lagrangian context, with additional particularities due to the fact that the mesh movement is part of the solution process. Indeed, 
the Lagrangian representation is less commonly used than the Eulerian one. It is particularly well suited to describe the time evolution of fluid flows within regions undergoing large deformations due to strong compressions and expansions. Since no mass flux crosses the boundary of a control volume moving with the material velocity, the Lagrangian formulation naturally tracks material interfaces in multimaterial compressible flows. These features motivate its use in the simulation of physical phenomena encountered in fields such as astrophysics and inertial confinement fusion. Unlike Eulerian numerical methods, Lagrangian approaches rely on moving computational grids. Consequently, in the Lagrangian framework one must discretize not only the governing conservation laws but also determine the nodal velocities that drive mesh motion. The volume of a zone, computed directly from its vertex coordinates, must exactly match the volume obtained from the discrete volume equation, also known as the Geometric Conservation Law (GCL). Ensuring this compatibility between grid motion and the GCL is a cornerstone of any robust multidimensional Lagrangian method.

The most straightforward way to satisfy this requirement is through a staggered discretization, in which  the velocity is defined at the nodes while thermodynamic variables (density, pressure, specific energy) are cell-centered. This approach was first introduced in the seminal work of von Neumann and Richtmyer \cite{Neumann1950} at Los Alamos during the Manhattan Project, and has since undergone continuous development (see, e.g., \cite{Morgan2021} for a historical overview).

An alternative is to adopt a cell-centered discretization, where all physical variables, including momentum, are stored at cell centers. In this case, interface fluxes are evaluated using approximate Riemann solvers defined along face normals. The nodal velocity required for mesh motion is then obtained by averaging the normal velocities of the faces adjacent to each node. This strategy, introduced in the mid-1980s by Dukowicz and collaborators \cite{Adessio1988}, suffers from incompatibility with the GCL. This issue was later resolved by Despr\'{e}s and Mazeran \cite{Despres2005}, who proposed a cell-centered, moving-mesh finite volume scheme for Lagrangian hydrodynamics that enforces GCL compatibility. Their method relies on a node-based approximation of the volume flux, defined as the dot product between the nodal velocity and the corner normal (i.e., the gradient of the cell volume with respect to nodal coordinates). Extending this idea, momentum and total energy are also discretized using nodal fluxes. These fluxes are consistently approximated with Riemann solvers defined in the directions of the corner normals and parametrized by projections of the nodal velocity onto these directions. Because the nodal fluxes are not uniquely defined, classical face-based conservation is lost; it is instead restored by imposing a node-based conservation condition, which yields a linear system for computing nodal velocities and fluxes. While elegant, this approach lacks robustness due to insufficient entropy production. To address this limitation, in \cite{Maire2009} was introduced a more robust node-based finite volume scheme that retains the same philosophy but employs Riemann solvers defined along the normals of the faces incident to each node, still parametrized by the projection of nodal velocities onto these directions.

These unconventional finite volume schemes differ from classical face-based methods in that conservation is not enforced face by face but instead recovered through node-based conservation conditions. In this work, we demonstrate that these node-based conservation conditions are rigorously equivalent to the consistency condition underlying Residual Distribution (RD) schemes: the sum of flux fluctuations around a node equals the approximation of the flux integral across the boundary of the dual cell surrounding that node. Furthermore, we show that these unconventional schemes are locally conservative. In two dimensions, by splitting each face into two subfaces defined by the vertices and the face midpoint, we derive an explicit and unique expression for the numerical flux associated with each subface. This half-face flux depends on all states surrounding the node, which motivates the designation multipoint finite volume method.

\bigskip

 The structure of this paper is as follows. We first recall  some elements on hyperbolic problems, introduce the notion of entropy and symmetrization. Then focusing on fluid mechanics, we recall several formalisms: Eulerian, Lagrangian and others, and discuss briefly 
 the Riemann problem. In a second part, we provide several examples of numerical schemes, all very classical. These examples are meant to represent the various cases one may encounter. We believe that all the generic situations are covered\footnote{with, possibly, one exception that will be briefly discussed in the conclusion}. We  show that they all follow the same structure: they are (or can be cast) in distribution form. In a third part, we formalise what we call the distribution form, and provide a local conservation principle. Using this we show a Lax Wendroff theorem, where the assumptions on the solution are very standard. We also show that any scheme that is in this distribution form can be written in flux form, and we provide a way to compute the flux.

 Using this, we show a systematic way of construction locally conservative schemes for a non conservative formulation of the Euler equations. We show how this notion of local conservation can be used for mesh refinement, for the control of entropy production, and some other features. We end by some remarks.

 In this paper, we do not provide any numerical illustrations of the methods and discussions. They are all contained in the references quoted in the text.

\paragraph{First notations.} In addition to $\BB \bbW$ defined above, we introduce additional operator notations that will be used throughout the paper.
If $\bbA\in M_{k\times d}(\R)$ is a matrix with columns 
$\bba_j$, $1\leq j\leq d$, and $\bbn\in \R^d$, 
the  product of $\bbA$ by $\bbn$ is $\bbA\bbn$.
The Euclidian scalar product between the vectors $\bbx$ and $\bby$ of $\R^d$ will be denoted by $\bbx\cdot\bby$ or $\bbx^\mathtt{T}\bby$.

\section{Mathematical and physical setting}
\subsection{Hyperbolic problems.}
\label{intro}

As already said above, we are interested in the approximation of the following hyperbolic problem,
\begin{equation}\label{hyper}
\dpar{\bbu}{t}+\text{ div }\bbf(\bbu)=0.
\end{equation} Here, $\bbu:(\bbx,t)\rightarrow \bbu(\bbx,t)\in \mathcal{R}\subset  \R^p$ for $\bbx\in \R^d$ and $t>0$. The set $\mathcal{R}$ will be called the invariant domain. The flux $\bbf=(f_1, \ldots , f_d)\in M_{p\times d}(\R)$ is supposed to be a $C^1$ mapping from  the invariant domain $\mathcal{R}$, which is supposed to be open and convex in $\R^p$, to $\R^{dp}$. The problem \eqref{hyper} is assumed hyperbolic: considering the tensor
$\bbA=\nabla_\bbu\bbf=(\nabla_\bbu f_1, \ldots , \nabla_\bbu f_d)$ where $\nabla_\bbu f_j\in M_{p\times p}(\R)$, for any $\bbn=(n_1, \ldots , n_d)\in \R^d$, the matrix
$$\bbA \bbn:=\sum_{j=1}^d \dpar{f_j}{\bbu} n_i$$ 
 is assumed to be diagonalisable in $\R$. 

\medskip
To make sense, the problem \eqref{hyper} is complemented with an initial condition,
$$\bbu(\bbx, 0)=\bbu_0(\bbx),$$ and boundary conditions on $\partial \Omega$. 
The boundary conditions we may consider in this paper are:
\begin{equation}
\label{hyper:bc}
\big ( \bbA\cdot \bbn\big )^-\big (\bbu(\bbx,t)-\bbu_b(\bbx, t)\big )=0, \quad t>0,
\end{equation}
where $\bbn$ is the local outward unit normal at $\bbx\in \partial \Omega$ and $\bbu_b\in \R^p$.
Doing this we assume that $\partial \Omega$ is smooth. We refer to \cite{GuermondErn} for more details.

\bigskip

We will leave the mathematical technicalities out of the scope of this paper, referring to\cite{Raviart1996,Kroner,Feireisl} for this, nor consider the problem of existence of  a solution, nor its uniqueness (see e.g. \cite{Szekelyhidi} among many others) We will also often assume that $\Omega=\R^d$ (i.e. we have no boundary conditions).  The only thing we will really need is the notion of weak solution: $\bbu\in L^1(\Omega\times[0,T[)\cap L^\infty(\Omega\times [0,T[)$, for some $T>0$ is said to be a weak solution of \eqref{hyper} if for all $\varphi\in C^1_0(\Omega\times [0,T[)$, the set of $C^1$ functions with compact support in $\Omega\times [0,T]$, we have
\begin{equation}
    \label{weak}
        \int_{\Omega\times[0,T[} \bigg ( \dpar{\varphi}{t}(\bbx,t)\bbu(\bbx,t)+\nabla_\bbx\varphi \bbf(\bbu(\bbx,t))\bigg ) \; d\bbx \; dt+\int_{\Omega}\varphi(\bbx,0)\bbu_0(\bbx)\; d\bbx \; dt=0.
\end{equation}
We will assume that $T$ can be chosen as large as wished, i.e. also  $T=+\infty$, though this assumption is debatable see \cite{merle1,merle2}.

It is very easy to find examples, for instance in the scalar case,  where, given an initial condition $\bbu_0$, there are infinitely  many weak solutions to the problem \eqref{hyper}.    In order to rule out some of them, and in the scalar case all but one, see \cite{GodlewskiRaviartTome1}, we  introduce the notion of entropy and entropy solutions. This is inspired by fluid mechanics.

\bigskip

Any "interesting" hyperbolic system is equipped with an entropy, denoted by $\eta$. This is a strictly concave\footnote{Usually, the mathematical entropy is strictly convex, here we have chosen to define the entropy with a more physical approach. This is a cosmetic choice.} real valued function defined on the invariant domain $\mathcal{R}$. It has to be combined with an "entropy flux" $\bbg=(g_1, g_2, \ldots, g_d)$ such that for all $i=1, \ldots , d$
\begin{equation}
    \label{entropy}
   \big ( \nabla_\bbu \eta\big )^{\mathtt{T}} \nabla_\bbu f_i=\nabla_\bbu g_i
\end{equation}
Using \eqref{entropy}, and assuming that the solution of \eqref{hyper} is smooth, we see that we have an additional conservation relation:
$$\dpar{\eta(\bbu)}{t}+\text{ div }\bbg(\bbu)=0.$$
It is well known that the solution of the Cauchy problem is not smooth in general. Hence,  this equality has to be understood in the sense of distribution and, imitating the Navier-Stokes equations, we should look for solution satisfying
$$\dpar{\eta(\bbu)}{t}+\text{ div }\bbg(\bbu)\geq 0,$$
i.e., for any positive $C^1_0(\R^d\times [0, +\infty[)$ solution, we look for solutions such that
\begin{equation}\label{entropyinequality}
\int_{\R^d\times[0,+\infty[)} \bigg ( \dpar{\varphi}{t}(\bbx,t)\eta(\bbx,t)+\nabla_\bbx\varphi \bbg(\bbu(\bbx,t))\bigg ) \; d\bbx \; dt+\int_{\Omega}\varphi(\bbx,0)\eta(\bbu_0(\bbx))\; d\bbx \; dt\geq 0.
\end{equation}

An important property of systems admitting an entropy is that of symmetrisation. If $\bbA_0$ represents the Hessian of the entropy $\eta$ with respect to the conserved variables, then one can show that for all $j=1, \ldots , d$, the matrices
$$\bbA_0\nabla_\bbu f_i$$ are  symmetric matrices. 
This is the essence of Godunov-Mock's theorem \cite{Mock,Godunov} (see also \cite{GodKsenia} for an English translation)
\begin{theorem}
Let $\eta(\bbu)$ be a strictly concave entropy function for \eqref{hyper}; then the entropy variable $\bbu\mapsto \bmw=\nabla_\bbu \eta(\bbu)$ defines a one-to-one mapping  and allows to symmetrize \eqref{hyper}.
\end{theorem}
This can be seen by considering the two potentials $q(\bmw)=\bmw^{\mathtt{T}}\bbu(\bmw)-\eta(\bbu(\bmw))$ and $\psi_j(\bmw)=\bmw^{\mathtt{T}}f_j(\bbu(\bmw))-g_j(\bbu(\bmw)),$ where we use the fact that the change of variable $\bbu\mapsto \bmw$ is one-to-one if the entropy is strictly concave. Then a simple calculation shows that $\dpar{\psi_j}{\bmw}=f_j(\bbu(\bmw))$, so that 
$$\nabla_{\bmw} f_j\big (\bbu(\bmw)\big )=\nabla^2_\bbu f_j(\bbu) \nabla_{\bmw}\bbu$$
is symmetric and finally using \eqref{entropy}, and some algebra, we get
$$\bbA_0\nabla_\bbu f_j (\bbu  )=\bbA_0\dfrac{\partial^2\psi_j}{\partial \bmw^2}\bbA_0.$$
This shows that $\bbA_0\nabla_\bbu f_j (\bbu  )$ is symmetric.

\subsection{The Euler equations of compressible fluid dynamics.}
The canonical example for hyperbolic systems is given by the
 Euler equations for compressible fluid dynamics. There are several variants of their formulation, depending if we are in an Eulerian setting of a Lagrangian one. The most natural formulation is the Lagrangian formulation because one clearly sees how are translated the principles of mass conservation, the Newton's law, the first and second principles of thermodynamics, refer for instance to \cite{Gurtin2009}.  In many practical applications, the Eulerian one is the simplest to use.

\subsubsection{The Lagrangian formalism}
Let $\omega(t)$ be a moving region of the $d$-dimensional space $\mathbb{R}^{\text{d}}$ filled by an inviscid, non heat conducting compressible fluid characterized by its mass density $\rho >0$ and its velocity $\mathbf{v}$. The Lagrangian control volume form of the gas dynamics equations is nothing but the integral form of the conservation laws written over the control volume $\omega(t)$, moving with the material velocity $\mathbf{v}$. In this framework the conservation laws of volume, momentum and total energy write \cite{Gurtin2009}
\begin{equation}
  \label{eq:Lcv}
  \frac{\mathrm{d}}{\mathrm{d} t} \int_{\omega(t)} \rho \mathbf{u}\,\mathrm{d}\bbx+\int_{\partial \omega(t)} \mathbf{f}\mathbf(\mathbf{u}) \mathbf{n}\,\mathrm{d}\gamma=\mathbf{0},
\end{equation}
where $\mathrm{d}\bbx$ (resp. $\mathrm{d}\gamma$) is the volume (resp. surface) element and $\mathbf{n}$ is the unit outward normal to the boundary surface $\partial \omega(t)$ of the control volume. The vector of conservative variables, $\mathbf{u}=\mathbf{u}(\mathbf{x},t)$ for $\mathbf{x} \in \mathbb{R}^{\text{d}}$, writes 
$$\mathbf{u}=\begin{pmatrix}
    \tau \\
    \mathbf{v}\\
    e
\end{pmatrix} \in \mathbb{R}^{\text{d}+2},$$
where 
$\tau=1/\rho$ is the specific volume and $e$ the specific total energy. The physical flux tensor $\mathbf{f}$, and its normal projection, 
$\mathbf{f}_{\mathbf{n}}=\mathbf{f}\mathbf{n}$
write respectively
$$
\mathbf{f}(\mathbf{u})=\begin{pmatrix}
-\mathbf{v}^{t} \\
p \Id_{\text{d}} \\
p \mathbf{v}^{\mathtt{T}}
\end{pmatrix}
\text{ and }\,
\mathbf{f}_{\mathbf{n}}(\mathbf{u})=\begin{pmatrix}
-\mathbf{v}\cdot \mathbf{n} \\
p \mathbf{n}\\
p \mathbf{v}\cdot \mathbf{n}
\end{pmatrix}
,
$$
where $p$ denotes the thermodynamic pressure and $\mathbb{I}_{\text{d}}$ the d-dimensional identity matrix. The specific internal energy reads 
$\varepsilon=e- \mathbf{v}^{2}/2$.

We point out that the Lagrangian representation is characterized by the conservation of the material mass contained within the moving control volume. Namely, the mass conservation law writes
\begin{equation}
  \label{eq:mass}
  \frac{\mathrm{d}}{\mathrm{d} t} \int_{\omega(t)} \rho\,\mathrm{d}\bbx=0,
\end{equation}
which means that the mass flux at the boundary of the moving control volume is equal to zero.

For a point located on the boundary of the control volume $\omega(t)$, {\it i.e.}, $\mathbf{x} \in \partial \omega(t)$, its path is determined by the trajectory equation
\begin{equation}
  \label{eq:traj}
  \frac{\mathrm{d} \mathbf{x}}{\mathrm{d} t}=\mathbf{v}(\mathbf{x},t),\;\mathbf{x}(0)=\mathbf{X},
\end{equation}
where $\mathbf{X}$ denotes its initial position.

For this system of conservation laws, the specific volume and the internal energy are positive, namely $\mathbf{u}$ should stay in the region
$$\mathcal{R}=\left \{ \mathbf{u}=
\big ( \tau, \bbv, e\big )^{\mathtt{T}}
\in \mathbb{R}^{\text{d}+2} : \;\tau >0, \; \varepsilon=e-\frac{1}{2}\mathbf{v}^{2} >0 \right \},$$
which is a convex set.

It is well known that the gas dynamics equations admit discontinuous solutions, refer to \cite{Raviart1996}. To select the physically admissible solution and to define the thermodynamic closure of the foregoing system of conservation laws, we introduce the specific physical entropy $\eta$. We make the fundamental assumption  that  $(\tau,\eta) \mapsto \varepsilon(\tau,\eta)$ is strictly convex which is equivalent to assume that $(\tau,\varepsilon) \mapsto \eta(\tau,\varepsilon)$ is strictly concave, refer to \cite{Raviart1996}. We work with the physical entropy, entropy flux pair $(\eta,\mathbf{0})$ and the Lagrangian gas dynamics system of conservation laws \eqref{eq:Lcv} is equipped with the entropy inequality
\begin{equation}
\label{eq:entropineqgd}
 \frac{\mathrm{d}}{\mathrm{d} t} \int_{\omega(t)} \rho \eta\,\mathrm{d}\bbx \geq 0,
\end{equation}
which is the mathematical expression of the second law of thermodynamics. Namely, for smooth flows, the entropy is conserved, whereas it is increasing across discontinuities such as shock waves. In addition, the entropy inequality \eqref{eq:entropineqgd} ensures a selection of the physical solution. 
The thermodynamic closure is ensured by means of the complete equation of state (EOS) \cite{Menikoff1989}
\begin{equation}
  \label{eq:ceosmultiD}
  p(\tau,\eta)=-\frac{\partial \varepsilon}{\partial \tau},\quad \theta (\tau,\eta)=\frac{\partial \varepsilon}{\partial \eta}.
  \end{equation}
In addition, we make the classical assumption  that  the absolute temperature is strictly positive: $\theta >0$. The convexity of the specific internal energy with respect to the specific volume allows us to define the isentropic sound speed
\begin{equation}
  \label{eq:iss}
  \frac{a^{2}}{\tau^{2}}=-\frac{\partial p}{\partial \tau}=\frac{\partial^{2}\varepsilon}{\partial \tau^{2}}.
\end{equation}

\begin{rem}[Gamma gas law]
  \label{rem:gamma}
A relation widely used when testing numerical methods is
  the perfect gas EOS which expresses the pressure in terms of the specific volume and the specific internal energy as follows
  \begin{equation}
    \label{eq:ggl}
    p = (\gamma-1) \rho\varepsilon,
  \end{equation}
where $\gamma$ is the polytropic index. The corresponding isentropic sound speed writes
\begin{equation}
  \label{eq:aggl}
  a^{2}=\frac{\gamma p}{\rho}=\gamma p \tau.
\end{equation}
\end{rem}

\begin{rem}[Gibbs relation]
  \label{rem:Gibbs}
  By virtue of the complete EOS, the differential of specific internal energy writes $\mathrm{d}\varepsilon=-p\mathrm{d}\tau+\theta \mathrm{d}\eta$ which leads to the famous Gibbs thermodynamic relation
  \begin{equation}
    \label{eq:Gibbsrelation}
    \theta \mathrm{d}\eta=p \mathrm{d}\tau+ \mathrm{d}\varepsilon.
  \end{equation}
  Noting that $\mathrm{d}\varepsilon=\mathrm{d}e-\mathbf{v} \cdot \mathrm{d} \mathbf{v}$ we obtain
  $$\mathrm{d}\eta=\frac{1}{\theta} \left (p\mathrm{d}\tau-\mathbf{v} \cdot \mathrm{d} \mathbf{v}+\mathrm{d}e \right)=\bmw\cdot \mathrm{d} \mathbf{u}.$$
  This last relation allows us to write the vector of entropic variables
  \begin{equation}
    \label{eq:entropicvar}
    \bmw=\frac{\partial \eta (\mathbf{u})}{\partial \mathbf{u}}=\frac{1}{\theta} \begin{pmatrix}
      p \\
      -\mathbf{v} \\
      1
    \end{pmatrix}
  \end{equation}.

\end{rem}
\begin{rem}[Calorically perfect gas]
In a calorically perfect gas,  the specific internal energy is linear in term of the temperature:
$$\varepsilon=c_v \theta.$$
The specific heat at constant volume, $c_v$,  is a constant. Assuming that the gas is also a perfect gas, the Gibbs identity can be integrated to compute the specific entropy.
One can easily find:
$\eta=(\gamma-1)c_v\log \tau+c_v\log \theta+\eta_0$
i.e.
$$\eta=c_v\big ( \log p-\gamma \log \rho\big )+\eta'_0,$$
where $\eta_0$ and $\eta'_0$ are constants.
\end{rem}

\begin{rem}[Geometrical Conservation Law]
The Reynolds transport theorem \cite{Gurtin2009} that states, for any function $f$, that
  $$\dfrac{\mathrm{d}}{\mathrm{d}t}\int_{\omega(t)} f(\bbx,t)\; d\bbx=\int_{\omega(t)}\dpar{f}{t}\; d\bbx+\int_{\partial \omega(t)}f(\bbx,t) \bbn \; d\gamma.$$
  Applying this relation  to the characteristic function of the set $\omega(t)$, we get
  \begin{equation}
    \label{eq:GCL}
    \frac{\mathrm{d} |\omega(t)|}{\mathrm{d} t}-\int_{\partial \omega(t)} \mathbf{v}\cdot \mathbf{n}\,\mathrm{d}\gamma=0.
  \end{equation}
  This is the Geometrical Conservation Law (GCL). It can be rewritten as 
  $$\frac{\mathrm{d} }{\mathrm{d} t} \int_{\omega(t)} \mathrm{d}v=\int_{\partial \omega(t)} \frac{\mathrm{d} \mathbf{x}}{\mathrm{d}t} \cdot \mathbf{n}\,\mathrm{d}\gamma.$$
The GCL is intrinsically connected to the trajectory equation \eqref{eq:traj}. When solving the Lagrangian gas dynamics equations, it is therefore necessary not only to satisfy the physical conservation laws of momentum and total energy corresponding to the second and third equations of \eqref{eq:Lcv} but also to enforce the GCL, ensuring that its geometrical property, namely its connection to the trajectory equation, is preserved. The interested reader might read the following references for more details about this very important topic \cite{Maire2009,MaireHDR2011,LMR2016,Despres2017}.
\end{rem}

\subsubsection{The Eulerian formalism}
To go from the Lagrangian formulation to the Eulerian one, i.e. in fixed coordinates, the idea is to make the change of variables induced by \eqref{eq:traj}: the position at time $t$ is $\bbx=\bbx(t; \bbX)$, and this transformation is one-to-one under very mild assumptions on the velocity field.  Using the Reynolds transport theorem,
and thanks to  mass conservation, we obtain the following result, valid for any function $f$:
$$\dfrac{\mathrm{d}}{\mathrm{d}t}\int_{\omega(t)} \rho (\bbx,t)f(\bbx,t)\; d\bbx=\int_{\omega(t)}\rho(\bbx,t)\dfrac{\mathrm{d}}{\mathrm{d}t}f(\bbx,t)\; d\bbx,$$
we obtain the Eulerian formulation.

 In the Eulerian setting, the vector of conserved variables is
$$\mathbf{u}=
\begin{pmatrix}
    \rho \\
    \rho \mathbf{v}\\
    E
\end{pmatrix} \in \mathbb{R}^{\text{d}+2},
$$
where $\rho>0$ is the mass density, $\bbv\in \R^d$ is the fluid velocity, and $E=\rho e=\rho\varepsilon+\tfrac{\rho}{2}\bbv^2$ is the total energy. The flux is defined by
$$\bbf=\begin{pmatrix}\rho\bbv \\ \rho\bbv\otimes \bbv+p\mathbb{Id}_d\\ (E+p)\bbv\end{pmatrix}.$$
The invariant domain is
$$\mathcal{R}=\left \{(\rho, \rho \mathbf{v}, E) \in \mathbb{R}^{\text{d}+2}:\rho>0,\; \frac{E}{\rho}-\frac{1}{2} \frac{(\rho \mathbf{v})^{2}}{\rho^{2}}>0 \right\}.$$

\subsubsection{Non conservative forms of the Euler equations}
As long as the solution is smooth, one may write the Euler equations in a non conservative form. Two are particularly interesting.

The first uses the internal energy and the velocity as main variables, and reads:
\begin{equation}
\label{nc:system}
\begin{split}
\dpar{\rho}{t}&+\text{ div }(\rho \bbv)=0,\\
\rho \dpar{\bbv}{t}&+ \rho\big ( \bbv  \cdot\nabla\big )\bbv+\nabla p=0,\\
\rho \dpar{\varepsilon}{t}&+ \rho \big ( \bbv  \cdot\nabla\big )\varepsilon+p\;\text{div }\bbv=0.
\end{split}
\end{equation}
The system \eqref{nc:system} can be rewritten
in a second interesting form involving directly the pressure as
\begin{equation*}
\begin{split}
\dpar{\rho}{t}&+\text{ div }(\rho \bbv)=0\\
\rho \dpar{\bbv}{t}&+ \rho \big ( \bbv  \cdot\nabla\big )\bbv+\nabla p=0\\
\dpar{\rho\varepsilon }{t}&+ \big ( \bbv  \cdot \nabla\big )(\rho\varepsilon)+\rho h\;\text{div }\bbv=0,
\end{split}
\end{equation*}
where the specific enthalpy $h$ is
$$h=\varepsilon+\frac{p}{\rho}.$$
The interest of this formulation is to separate the thermodynamical quantities, $\rho$, $p$ and $\rho\varepsilon$ from the kinematic ones. It is at the basis of the staggered mesh technique that we briefly discuss later 
in section \ref{sec:stagerred}.
\subsubsection{The Riemann problem.}
A very important case of a Cauchy problem is the Riemann problem, which in 1 dimension writes:
\begin{equation}
\label{eq:RP}
\frac{\partial \mathbf{u}}{\partial t}+\frac{\partial \bm{f}(\mathbf{u})}{\partial x}=0
\end{equation}
with the initial conditions
$$\mathbf{u}(x,0)=
\begin{dcases}
  \mathbf{u}_l \quad \text{if} \quad x <0, \\
  \mathbf{u}_r \quad \text{if} \quad x \geq 0.
\end{dcases}
$$

We denote by $\mathbf{r}(\mathbf{u}_l,\mathbf{u}_r,\tfrac{x}{t})$ the solution of the Riemann problem. For the Euler equations, an exact solution can be constructed when the equation of state is analytical and the left and right states are sufficiently close, see \cite{Raviart1996}. The Riemann problem plays a central role in the Godunov scheme \cite{Godunov1979}, where the numerical flux at each cell interface is obtained directly from its solution. In practice, however, obtaining an exact solution is often difficult, and even when available, its use is computationally expensive. Therefore, it is preferable to approximate the interface flux using an approximate Riemann solver, an approach that was formalized in the seminal works \cite{Harten1981,Harten1983}.

\section{Several motivating  examples}
\label{sec:motivating}
In this section, we review  several numerical formulations for solving \eqref{hyper}, 
some very classical and others less well known, and reformulate them in our framework. 
\subsection{The one dimensional case}
\label{sec:motivating:1D}
We consider the Riemann problem \eqref{eq:RP} and its solution $\mathbf{r}(\mathbf{u}_l,\mathbf{u}_r,\tfrac{x}{t})$. The one-dimensional computational domain is paved with cells $[x_{i-\frac{1}{2}}, x_{i+\frac{1}{2}}]$, the cell volume is $\Delta x_i=x_{i+\frac{1}{2}}-x_{i-\frac{1}{2}}$ and we denote the cell midpoint by $x_i=\frac{1}{2}(x_{i+\frac{1}{2}}+x_{i-\frac{1}{2}})$. The cell averaged value at time $t^{n}$ is
$$\mathbf{u}_i^n=\frac{1}{\Delta x_i}\int_{x_{i-\frac{1}{2}}}^{x_{i+\frac{1}{2}}} \mathbf{u}(x,t^n)\,\mathrm{d}x.$$
We define
\begin{subequations}
  \label{eq:FVS}
    \begin{align}
      \mathbf{u}_{i-\frac{1}{2}}^{+}=&\dfrac{2}{\Delta x_i} \int_{x_{i-\frac{1}{2}}}^{x_{i}} \mathbf{r}(\mathbf{u}_{i-1}^{n}, \mathbf{u}_{i}^{n},\frac{x-x_{i-\frac{1}{2}}}{\Delta t})\,\mathrm{d}x,\label{eq:FVSa}\\
      \mathbf{u}_{i+\frac{1}{2}}^{-}=&\dfrac{2}{\Delta x_i} \int_{x_{i}}^{x_{i+\frac{1}{2}}} \mathbf{r}(\mathbf{u}_{i}^{n}, \mathbf{u}_{i+1}^{n},\frac{x-x_{i+\frac{1}{2}}}{\Delta t})\,\mathrm{d}x,\label{eq:FVSb}
    \end{align}
\end{subequations}
and the Godunov-type scheme \cite{Harten1981,Harten1983}
\begin{equation}
  \label{eq:God}
  \Delta x_i \mathbf{u}_i^{n+1}=\frac{\Delta x_i}{2}\mathbf{u}_{i-\frac{1}{2}}^++\frac{\Delta x_i}{2}\mathbf{u}_{i+\frac{1}{2}}^-,
\end{equation}
where $\Delta t>0$ is the time step and $t^{n+1}=t^n+\Delta t$. Introducing $\mathbf{u}_i^n$, \eqref{eq:God} might be rewritten as
\begin{equation}
\label{eq:Godm}
\Delta x_i\left (\mathbf{u}_i^{n+1}-\mathbf{u}_i^{n}\right ) +\frac{\Delta x_i}{2}\left [ \left ( \mathbf{u}_i^{n}-\mathbf{u}_{i-\frac{1}{2}}^{+}\right)+ \left ( \mathbf{u}_i^{n}-\mathbf{u}_{i+\frac{1}{2}}^{-}\right )\right ]=0.
\end{equation}
We identify the fluctuations in cell $i$ attached to the nodes $i-\frac{1}{2}$ and $i+\frac{1}{2}$:
\begin{equation}
\label{eq:fluc}
\mathbf{\Phi}_i^{i-1/2}=\frac{1}{2}\frac{\Delta x_i}{\Delta t} \left (\mathbf{u}_i^n-\mathbf{u}_{i-\frac{1}{2}}^{+}\right )\;\text{and}\; \mathbf{\Phi}_i^{i+1/2}=\frac{1}{2}\frac{\Delta x_i}{\Delta t} \left( \mathbf{u}_i^n-\mathbf{u}_{i+\frac{1}{2}}^{-}\right).
\end{equation}
It is also interesting to identify the flux form of the Godunov-type scheme \eqref{eq:Godm}. We get it simply introducing
\begin{equation}
\label{eq:fluxm}
\mathbf{f}_{i+\frac{1}{2}}^{-}-\mathbf{f}_i^{n}=\frac{1}{2}\dfrac{\Delta x_i}{\Delta t} \left ( \mathbf{u}_i^n-\mathbf{u}_{i+\frac{1}{2}}^{-}\right ),
\end{equation}
where $\mathbf{f}_i^n=\mathbf{f}(\mathbf{u}_i^n)$ and
\begin{equation}
\label{eq:fluxp}
\mathbf{f}_{i-\frac{1}{2}}^+-\mathbf{f}_i^n=-\frac{1}{2}\dfrac{\Delta x_i}{\Delta t} \left ( \mathbf{u}_i^n-\mathbf{u}_{i-\frac{1}{2}}^+\right ).
\end{equation}
Gathering the previous notations, we write the expression of the left-sided flux at $x_{i+\frac{1}{2}}$ from \eqref{eq:FVSb} and \eqref{eq:fluxm} 
$$\mathbf{f}_{i+\frac{1}{2}}^{-}=\mathbf{f}_i^n+\frac{1}{2}\frac{\Delta x_i}{\Delta t}\mathbf{u}_i^n-\frac{1}{\Delta t}\int_{x_{i}}^{x_{i+\frac{1}{2}}} \mathbf{r} \left (\mathbf{u}_i^n, \mathbf{u}_{i+1}^n, \frac{x-x_{i+\frac{1}{2}}}{\Delta t}\right )\,\mathrm{d}x,$$
and the expression of the right-sided flux at $x_{i+\frac{1}{2}}$ from \eqref{eq:FVSb} and \eqref{eq:fluxp}
$$\mathbf{f}_{i+\frac{1}{2}}^{+}=\mathbf{f}_{i+1}^n-\frac{1}{2}\frac{\Delta x_{i+1}}{\Delta t}\mathbf{u}_{i+1}^n+\frac{1}{\Delta t}\int_{x_{i+\frac{1}{2}}}^{x_{i+1}} \mathbf{r}\left (\mathbf{u}_i^n, \mathbf{u}_{i+1}^n, \frac{x-x_{i+\frac{1}{2}}}{\Delta t}\right)\, \mathrm{d}x.$$
We point out that these expressions of the left and right-sided flux at $x_{i+\frac{1}{2}}$ could also have been obtained invoking the integral form of the conservation law \eqref{eq:RP}.

With these definitions, \eqref{eq:Godm} rewrites under the flux form
\begin{equation}
  \label{eq:FVff}
\Delta x_i \left ( \mathbf{u}_i^{n+1}-\mathbf{u}_i^n\right ) +\Delta t \left [ \left ( \mathbf{f}_{i+\frac{1}{2}}^{-} - \mathbf{f}_i^n\right ) +\left ( \mathbf{f}_i^n-\mathbf{f}_{i-\frac{1}{2}}^{+}\right ) \right ]=0
\end{equation}

The left and right-sided fluxes at $x_{i+\frac{1}{2}}$ might be also expressed in terms of the fluctuations at node $i+\frac{1}{2}$ as follows
\begin{subequations}
  \label{eq:fluxfluc}
  \begin{align}
    \mathbf{f}_{i+\frac{1}{2}}^{-}=&\mathbf{f}_i^n+\frac{1}{2}\frac{\Delta x_i}{\Delta t} \left ( \mathbf{u}_i^n-\mathbf{u}_{i+\frac{1}{2}}^{-}\right ) =\mathbf{f}_i^n+\mathbf{\Phi}_i^{i+1/2},\label{eq:fluxflucm}\\
    \mathbf{f}_{i+\frac{1}{2}}^{+}=&\mathbf{f}_{i+1}^n-\frac{1}{2}\frac{\Delta x_{i+1}}{\Delta t} \left ( \mathbf{u}_{i+1}^n-\mathbf{u}_{i+\frac{1}{2}}^+\right ) =\mathbf{f}_{i+1}^n-\Phi_{i+1}^{i+1/2}.\label{eq:fluxflucp}
    \end{align}
\end{subequations}
By subtracting \eqref{eq:fluxflucm} and \eqref{eq:fluxflucp}, we get
\begin{align*}
\mathbf{f}_{i+\frac{1}{2}}^{+}-&\mathbf{f}_{i+\frac{1}{2}}^{-}=\mathbf{f}_{i+1}^n-\mathbf{f}_i^n-\left( \mathbf{\Phi}_i^{i+1/2}+\mathbf{\Phi}_{i+1}^{i+1/2}\right )\\
& =\mathbf{f}_{i+1}^n-\mathbf{f}_i^n-\frac{1}{2}\frac{\Delta x_{i+1}}{\Delta t} \mathbf{u}_{i+1}^n-\frac{1}{2}\frac{\Delta x_i}{ \Delta t}\mathbf{u}_i^n+\frac{1}{\Delta t}
\int_{-\frac{1}{2}\Delta x_i}^{\frac{1}{2}\Delta x_{i+1}} \mathbf{r}\left (\mathbf{u}_i^n, \mathbf{u}_{i+1}^n, \frac{x}{\Delta t}\right )\,\mathrm{d}x.
\end{align*}
From the above equation  we can deduce that  the Godunov-type scheme is conservative in the finite volume sense, {\it i.e.}, $\mathbf{f}_{i+\frac{1}{2}}^{+}-\mathbf{f}_{i+\frac{1}{2}}^{-}=\mathbf{0}$, if and only if the following condition holds
\begin{equation}
\label{eq:conscondrd}
\mathbf{\Phi}_i^{i+1/2}+\mathbf{\Phi}_{i+1}^{i+1/2}=\mathbf{f}_{i+1}^n-\mathbf{f}_i^n.
\end{equation}
The above condition can be interpreted as follows: \emph{the  sum of the fluctuations attached to the node $i+\frac{1}{2}$ is equal to the integral of $\frac{\partial \mathbf{f}}{\partial x}$ over the dual cell $[x_i, x_{i+1}]$}. This is the consistency condition employed in the residual distribution approach \cite{Abgrall2018,RD-ency}. We also note that \eqref{eq:conscondrd} is strictly equivalent to
\begin{equation}
\label{eq:conscondfv}
\frac{1}{\Delta t}\int_{-\frac{1}{2}\Delta x_i}^{\frac{1}{2}\Delta x_{i+1}} \mathbf{r}\left (\mathbf{u}_i^n, \mathbf{u}_{i+1}^n, \frac{x}{\Delta t} \right)\,\mathrm{d}x=\frac{\Delta x_i}{2\Delta t} \mathbf{u}_i^n+\frac{\Delta x_{i+1}}{2\Delta t}\mathbf{u}_{i+1}^n+\mathbf{f}_{i+1}^n-\mathbf{f}_i^n.
\end{equation}
This is nothing more than the classical consistency condition \cite{Harten1981,Harten1983} with the integral form of the conservation law \eqref{eq:RP}. 

This set of remarks can easily be generalised if one have a scheme of the form
\begin{equation}
\label{eq:fv-1d}
\mathbf{u}_i^{n+1}=\mathbf{u}_i^n-\frac{\Delta t}{\Delta x_i}\left( \hat{\mathbf{f}}_{i+\frac{1}{2}}-\hat{\mathbf{f}}_{i-\frac{1}{2}}\right ),
\end{equation}
where 
$$
\hat{\mathbf{f}}_{i+\frac{1}{2}}=\hat{\mathbf{f}}_{i+\frac{1}{2}}(\widetilde{\mathbf{u}}_{i},\widetilde{\mathbf{u}}_{i+1})
$$
is a numerical flux verifying the consistency property
\begin{equation}
\label{fv1d:cons}
\mathbf{f}_{i+\frac{1}{2}}(\widetilde{\mathbf{u}}, \widetilde{\mathbf{u}})= \mathbf{f}(\widetilde{\mathbf{u}}).
\end{equation}
A classical way to represent the numerical flux is the dissipative form
\begin{equation}
\label{eq:1dflux-dissipative}
\hat{\mathbf{f}}_{i+\frac{1}{2}}=\frac{1}{2}\left [ \mathbf{f}(\widetilde{\mathbf{u}}_{i+1})+\mathbf{f}(\widetilde{\mathbf{u}}_{i})- \mathbf{Q}_{i+\frac{1}{2}}\left (\widetilde{\mathbf{u}}_{i+1}-\widetilde{\mathbf{u}}_{i}\right )\right] ,
\end{equation}
where $\mathbf{Q}_{i+\frac{1}{2}}$ is some viscosity matrix and $\widetilde{\mathbf{u}}_{j}$ is some consistent approximation of $\mathbf{u}_j$. Here, we have in mind for example the result of 
the MUSCL extrapolation or of a  WENO type reconstruction. Following the same procedure as before, it is possible to introduce the fluctuations 
$$\mathbf{\Phi}_i^{i+1/2}=\frac{1}{2}\left [ \mathbf{f}(\widetilde{\mathbf{u}}_{i+1})-\mathbf{f}(\widetilde{\mathbf{u}}_{i})- \mathbf{Q}_{i+\frac{1}{2}}\left (\widetilde{\mathbf{u}}_{i+1}-\widetilde{\mathbf{u}}_{i}\right)\right ]$$
and
$$\mathbf{\Phi}_{i+1}^{i+1/2}=\frac{1}{2}\left [ \mathbf{f}(\widetilde{\mathbf{u}}_{i+1})-\mathbf{f}(\widetilde{\mathbf{u}}_{i})+ \mathbf{Q}_{i+\frac{1}{2}}\left (\widetilde{\mathbf{u}}_{i+1}-\widetilde{\mathbf{u}}_{i}\right)\right]$$
and we again get the relation \eqref{eq:conscondrd}, whatever the type of reconstruction that has been used.
\subsubsection{Simple solver}
\label{sssec:ss}
Let us consider the case for which $\mathbf{r}(\mathbf{u}_l, \mathbf{u}_r, \frac{x}{t})$ is the simple solver \cite{GalliceHDR2002,Gallice2003} characterized by the intermediate states $\mathbf{u}_k$, for $k=1, \cdots, m+1$ with $\mathbf{u}_1=\mathbf{u}_l$ and $\mathbf{u}_{m+1}=\mathbf{u}_r$. The wave speeds are $\lambda_k$ for $k=1,\ldots,m$ and they satisfy $\lambda_1\leq \ldots \leq \lambda_k\leq \ldots \leq \lambda_m$.

For any $x\in \mathbb{R}$, we set $x^{-}=\frac{1}{2}(x-|x|)$ and $x^{+}=\frac{1}{2}(x+|x|)$. 
We compute the left and right-sided fluxes
$$\mathbf{f}_{i+\frac{1}{2}}^-=\mathbf{f}_i^n+\sum_{k=1}^m \lambda^-_k \left (\mathbf{u}_{k+1}-\mathbf{u}_k\right ),$$
$$\mathbf{f}_{i+\frac{1}{2}}^+=\mathbf{f}_{i+1}^n-\sum_{k=1}^m \lambda^+_k \left (\mathbf{u}_{k+1}-\mathbf{u}_k\right )$$
we have
$$\mathbf{\Phi}_i^{i+1/2}=\sum_{k=1}^m \lambda^{-}_{k} \left (\mathbf{u}_{k+1}-\mathbf{u}_k\right ), \quad 
\mathbf{\Phi}_{i+1}^{i+1/2}=\sum_{k=1}^m \lambda^{+}_{k} \left (\mathbf{u}_{k+1}-\mathbf{u}_k\right ).$$
The conservative flux at node $i+1/2$ takes the classical form
$$\hat{\mathbf{f}}_{i+\frac{1}{2}}=\frac{1}{2}\left ( \mathbf{f}_i^n+\mathbf{f}_{i+1}^n\right ) -\frac{1}{2}\sum_{k=1}^m \vert \lambda_k\vert  \left (\mathbf{u}_{k+1}-\mathbf{u}_k\right )=
\frac{1}{2}\left ( \mathbf{f}_i^n+\mathbf{f}_{i+1}^n\right )-\frac{1}{2}\left (\mathbf{\Phi}_{i+1}^{i+1/2}-\mathbf{\Phi}_i^{i+1/2}\right).$$

\subsubsection{Example of the Roe scheme}
The Roe scheme \cite{Roe81} is often presented with the flux
\begin{equation}
  \label{eq:RoeFlux}
\hat{\mathbf{f}}(\mathbf{u}_{L},\mathbf{u}_{R})=\frac{1}{2}\left [\mathbf{f}(\mathbf{u}_{L})+\mathbf{f}(u_R)-\vert \bar A(\mathbf{u}_{L},\mathbf{u}_{R})\vert (\mathbf{u}_{L}-\mathbf{u}_{R})\right ].
\end{equation}
However, this formula does not explains the choice of the Roe average,
\begin{equation}
\label{eq:RoeLinearisation}
\bar A(\mathbf{u}_{L},\mathbf{u}_{R})(\mathbf{u}_{L}-\mathbf{u}_{R})=\mathbf{f}(\mathbf{u}_{L})-\mathbf{f}(\mathbf{u}_{R}).
\end{equation}
In the original paper, the solution of \eqref{eq:RP} is approximated by the solution of a problem of the type
$$\frac{\partial \mathbf{u}}{\partial t}+\bar{A}\frac{\partial \mathbf{u}}{\partial x}=0,$$
and the author is more particularly interested in upwind discretisation,
\begin{equation}
\label{eq:RoeFluct}
\mathbf{u}^{n+1}_j=\mathbf{u}_j^n-\frac{\Delta t}{\Delta x_i}\left [\bar{A}(\mathbf{u}_j^n,\mathbf{u}_{j+1}^n)^{-}(\mathbf{u}_{j+1}^n-\mathbf{u}_j^n)+\bar{A}(\mathbf{u}_{j-1},\mathbf{u}_j)^{+}(\mathbf{u}_{j}^n-\mathbf{u}_{j-1}^n)\right]. 
\end{equation}
This is naturally put in fluctuation form with
$$\Phi_{j+1}^{j+1/2}=\bar{A}(\mathbf{u}_j^n,\mathbf{u}_{j+1}^n)^{+}\left (\mathbf{u}_{j+1}^n-\mathbf{u}_j^n\right ), \quad
\Phi_{j}^{j+1/2}=\bar{A}(\mathbf{u}_j^n,\mathbf{u}_{j+1}^n)^{-}\left (\mathbf{u}_{j+1}^n-\mathbf{u}_j^n\right ).$$
This scheme will satisfy the conservation relation 
$$\Phi_{j+1}^{j+1/2}+\Phi_{j}^{j+1/2}=\mathbf{f}(\mathbf{u}_{j+1}^n)-\mathbf{f}(\mathbf{u}_j^n)$$
if and only if
\begin{equation}
  \label{eq:linearisation}
  \bar{A}(\mathbf{u}_j^n, \mathbf{u}_{j+1}^n)\left ( \mathbf{u}_{j+1}^n-\mathbf{u}_j^n\right)= \mathbf{f}(\mathbf{u}_{j+1}^n)-\mathbf{f}(\mathbf{u}_j^n)
\end{equation}
and in this case we can write a flux which is precisely \eqref{eq:RoeFlux}.  Let us notice that \eqref{eq:RoeFlux}, as well as \eqref{eq:RoeFluct}, assumes that $\bar A(\mathbf{u}_{L},\mathbf{u}_{R})$ is diagonalisable in $\mathbb{R}$. This is not guaranteed {\it a priori} by the linearisation relation \eqref{eq:linearisation}.
\subsection{{Finite volume schemes I: formulation with two-states  Riemann fluxes}}\label{sec:FV:1}
We now consider the approximation of \eqref{hyper} in two space dimensions, on a mesh composed of triangular elements.
Note that the  restriction to the  two dimensional case is only done in order  to work with simpler notations.
The discussion can be immediately generalized to more dimensions, and also to other types of element shapes.
As we will see, the key point for our discussion is that for a given mesh  we can 
construct a "dual" mesh. The latter is  the mesh of the control volumes, each of which can be  associated to 
 by the vertices of the primal mesh.

\medskip

Let us denote by $K_l$, $l=1, \ldots, n_e$ the set of  triangles  composing the grid, and note
their vertices by $\sigma_j$ with $j=1, \ldots, n_v$. The triangles cover $\R^2$, and the resulting mesh is conformal. To each vertex $\sigma$, one associates a control volume obtained  for example by joining the mid-points of the edges coming out of $\sigma$, and the centroids of the element containing $\sigma$, see figure \ref{fig:fv}. This type of control volumes have been considered by many authors, see e.g. \cite{dervieux1,dervieux2,dervieux3,stoufflet}, because it allows a natural connection between finite volume methods and finite element methods for linear convection diffusion problem. Denoting by $\vert C_\sigma\vert$ the area of the control volume, a natural discretisation is:
\begin{subequations}\label{fv:dervieux}
\begin{equation}\label{fv:dervieux:1}
\vert C_\sigma\vert\dfrac{d\bbu_\sigma}{dt}+\sum_{\sigma'\in \mathcal{V}_\sigma}
\hbbf_{\bbn_{\sigma,\sigma'}}(\widetilde{\bbu}_\sigma, \widetilde{\bbu}_{\sigma'})=0,
\end{equation}
{where $\hbbf_{\bbn_{\sigma,\sigma'}}(\widetilde{\bbu}_\sigma, \widetilde{\bbu}_{\sigma'})$  is a numerical flux
consistent in the sense of \eqref{fv1d:cons}, which in multiple space dimensions becomes
\begin{equation}\label{fvcons1d:2}
\hbbf_{\bbn_{\sigma,\sigma'}}(\widetilde{\bbu}, \widetilde{\bbu})
=\bbf(\widetilde{\bbu}) \bbn_{\sigma,\sigma'}.
\end{equation}
 The states $\widetilde{\bbu}_\sigma$ and $\widetilde{\bbu}_{\sigma'}$ are,
 in a MUSCL-like manner, } consistent high order approximation of the true state at the mid-point of the edge $[\sigma, \sigma']$, while $\bbn_{\sigma,\sigma'}$ is the normal to $C_\sigma$. We refer to e.g. \cite{stoufflet,abgrall_eno,Oliver_Friedrich,ADERCWENO} and references therein for possible choices on general meshes. 
 This formulation is actually suited for steady simulation.  
We consider in this paper a form better suited for  unsteady simulations in which the term 
$$\hbbf_{\bbn_{\sigma,\sigma'}}(\widetilde{\bbu}_\sigma, \widetilde{\bbu}_{\sigma'})$$
is split as
\begin{equation}\label{fv:dervieux:2}\hbbf_{\bbn_{\sigma,\sigma'}}(\widetilde{\bbu}_\sigma, \widetilde{\bbu}_{\sigma'}):=
\hbbf_{\bbn_{\sigma,\sigma'}^{K}}(\widetilde{\bbu}_\sigma, \widetilde{\bbu}_{\sigma'})+\hbbf_{\bbn_{\sigma,\sigma'}^{K'}}(\widetilde{\bbu}_\sigma, \widetilde{\bbu}_{\sigma'}),
\end{equation}
\end{subequations}
where the edge $[\sigma,\sigma']=K\cap K'$.

\begin{figure}[ht]
\begin{subfigure}{0.45\textwidth}
\centering
\label{fig:fv:1}
\includegraphics[width=\textwidth]{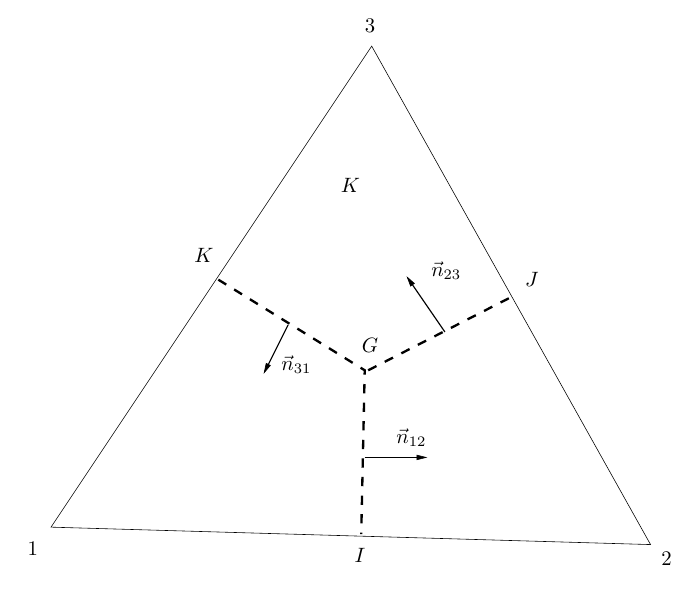}
\end{subfigure}
\begin{subfigure}{0.45\textwidth}
\centering
\label{fig:fv:2}
\includegraphics[width=\textwidth]{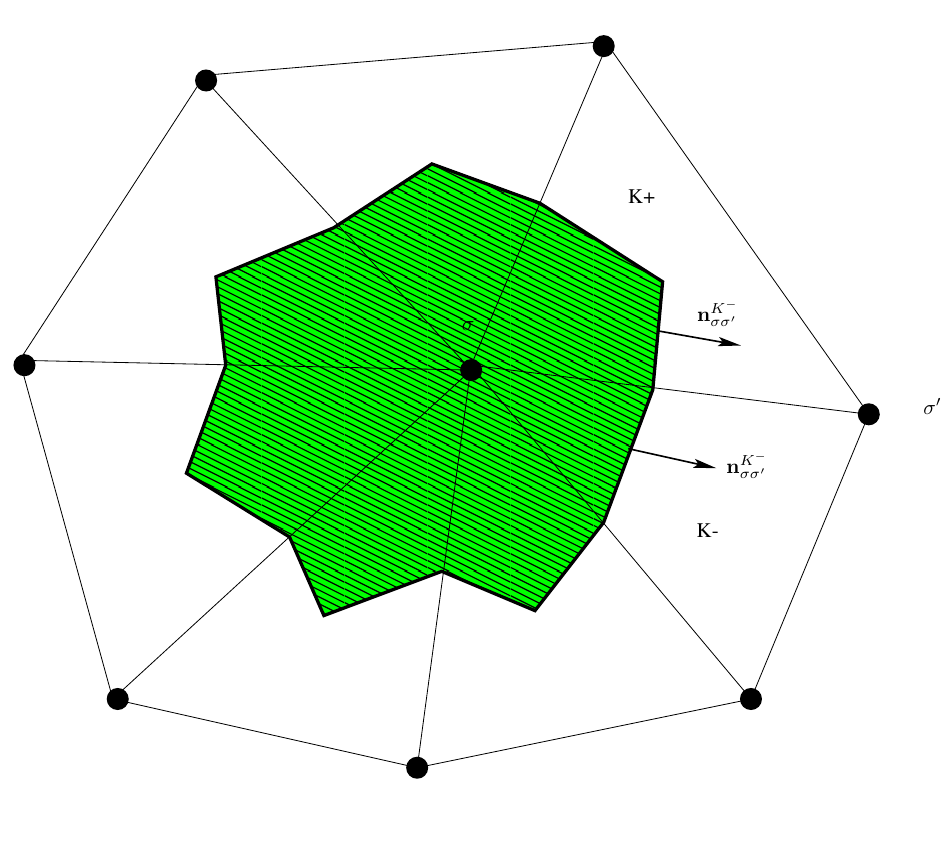}

\end{subfigure}
\caption{\label{fig:fv} Notations for the finite volume schemes. On the left: definition of the control volume for the degree of freedom $\sigma$.
 The vertex $\sigma$ plays the role of the vertex $1$ on the left picture, etc for the triangle K. The control volume $C_\sigma$ associated to $\sigma=1$ is green on the right and corresponds to $1IGK$ on the left. The vectors $\bbn_{ij}$ are normal to the internal edges scaled by the corresponding edge length}
\end{figure}

We now specialize the discussion to the case of triangular elements. 
Note however that  the  \emph{exactly the same arguments} can be given for more general elements,
 provided a conformal approximation space can be constructed. 
 We refer to section \ref{sec:Lagrangian} for another very different example. In the following, we denote  with a superscript $+$ and $-$ the exterior (resp. the interior) value of the solution and to simplify the notations, we omit the $\widetilde{~}$. 
 
   Since the boundary of $C_\sigma$ is a closed polygon, we have
$$
\sum_{\gamma \subset \partial C_\sigma}\bbn_\gamma=0,$$
where $\gamma$ is any of the segment included in $\partial C_\sigma$, such as $IG$ on Figure \ref{fig:fv}.

Hence
\begin{equation*}
\begin{split}
\sum_{\gamma \subset \partial C_\sigma} \hbbf_{\bbn_\gamma }(\bbu^+_\sigma , \bbu^-& )= \sum_{\gamma \subset \partial C_\sigma} \hbbf_{\bbn_\gamma }(\bbu^+_\sigma, \bbu^- )-  \bbf (\bbu_\sigma) \bigg (\sum_{\gamma \subset \partial C_\sigma}\bbn_\gamma\bigg )\\
&=\sum\limits_{K, \sigma\in K} \sum\limits_{\gamma \subset \partial C_\sigma\cap K} \big ( \hbbf_{\bbn_\gamma }(\bbu^+_\sigma, \bbu^- )-\bbf (\bbu_\sigma) \bbn_\gamma \big )
\end{split}
\end{equation*}
To make things explicit, in $K$, the internal boundaries are $IG$, $JG$ and $KG$, and those around $\sigma$ (which is labeled by $1$) are $IG$ and $KG$.
We set
\begin{equation}
\begin{split}
\Phi_\sigma^K(\bbu^h)&=\sum\limits_{\gamma\subset \partial C_\sigma\cap K} \big ( \hbbf_{\bbn_\gamma }(\bbu^+_\sigma, \bbu^- )-\bbf (\bbu_\sigma) \bbn_\gamma \big )\\
&=\sum\limits_{\gamma\subset \partial ( C_\sigma\cap K )}  \hbbf_{\bbn_\gamma }(\bbu^+_\sigma, \bbu^- ).
\end{split}
\label{fv:res:sigma}
\end{equation}
The last relation uses the consistency of the flux and the fact that $C_\sigma\cap K$ is a closed polygon. The quantity $\Phi_\sigma^K(\bbu^h)$ can also be interpreted as  the normal flux on $C_\sigma\cap K$.

We notice that if we sum up these residuals, we get
\begin{equation*}
\begin{split}
\sum_{\sigma\in K} \Phi_\sigma^K(\bbu_h)&= \bigg ( \hbbf_{\bbn_{12}}(\bbu_1^+,\bbu_2^+)-\hbbf_{\bbn_{13}}(\bbu_1^+,\bbu_3^+)-\bbf(\bbu_1)\bbn_{12}+\bbf(\bbu_1)\bbn_{31}\bigg )\\
&+\bigg ( \hbbf_{\bbn_{23}}(\bbu_2^+,\bbu_3^+)-\hbbf_{\bbn_{12}}(\bbu_2^+,\bbu_1^+)+\bbf(\bbu_2)\bbn_{12}-\bbf(\bbu_2)\bbn_{23}\bigg )\\
&+\bigg ( -\hbbf_{\bbn_{23}}(\bbu_3^+,\bbu_2^+)+\hbbf_{\bbn_{31}}(\bbu_3^+,\bbu_1^+)-\bbf(\bbu_3)\bbn_{23}+\bbf(\bbu_3)\bbn_{31}\bigg )\\
&= \bbf(\bbu_1) \big ( \bbn_{12}-\bbn_{31}\big ) +\bbf(\bbu_2) \big ( -\bbn_{23}+\bbn_{31}\big )
+\bbf(\bbu_3) \big ( \bbn_{31}-\bbn_{23}\big )\\
&=\bbf(\bbu_1)\frac{\bbn_1}{2}+\bbf(\bbu_2)\frac{\bbn_2}{2}+\bbf(\bbu_3)\frac{\bbn_3}{2},
\end{split}
\end{equation*}
where $\bbn_j$ is the scaled inward normal of the edge opposite to vertex $\sigma_j$, i.e. twice the gradient of the $\P^1$ basis function
 $\varphi_{\sigma_j}$ associated to this degree of freedom.
Thus, we can reinterpret the sum as the boundary integral of the Lagrange interpolant of the flux.

In summary, the finite volume scheme \eqref{fv:dervieux} can be rewritten as 
\begin{equation}
\label{fv:dervieux:3}
\vert C_\sigma\vert \dfrac{d\bbu_\sigma}{dt}+\sum_{K, \sigma\in K}\Phi_\sigma^K(\bbu^h)=0.
\end{equation}
Denoting the characteristic function of a domain $\omega$ by $1_\omega$, the residuals $\Phi_\sigma^K$ satisfy the relation
\begin{equation}\label{fv:dervieux:4}
\Phi^K(\bbu^h):=\sum_{\sigma \in K}\Phi_\sigma^K(\bbu^h)=\int_{\partial K} \bbf^h \bbn ,  \text{ with }\bbf^h=\sum_{\sigma\in K} \bbf(\bbu_\sigma)1_{C_\sigma\cap K}.
\end{equation}
What we have done here is a simple generalisation of the one dimensional setting.  This writing is independant of the formal accuracy of the spatial discretisation.


\subsection{{Finite volume methods II: multi-dimensional Riemann fluxes}}\label{sec:motivating:FV:II}
In the previous paragraph, we have shown that classical finite volume methods based on two-states numerical fluxes can be formulated in residual form as in relation \eqref{fv:dervieux:3}. The sum of the residuals $\Phi_\sigma^K(\bbu^h)$  gives  a cell residual coupling all the nodes of  cell $K$. Despite of this multi-dimensional aspect, the  finite volume 
fluxes involved only use  two states. To be more precise, these numerical 
fluxes are  obtained from some approximate  one-dimensional Riemann solver, connecting two states. We now consider a genuinely multidimensional formulation, which follows the works by~\cite{Gallice2022,PHRaph2}, and more recently \cite{grd25,bcrt25}.
These are  non-conventional   finite volume methods which do not involve  face-based fluxes, but corner multidimensional fluxes.
\begin{figure}
  \begin{center}
    \begin{subfigure}{0.4\textwidth}
      \centering
      \includegraphics[width=\textwidth]{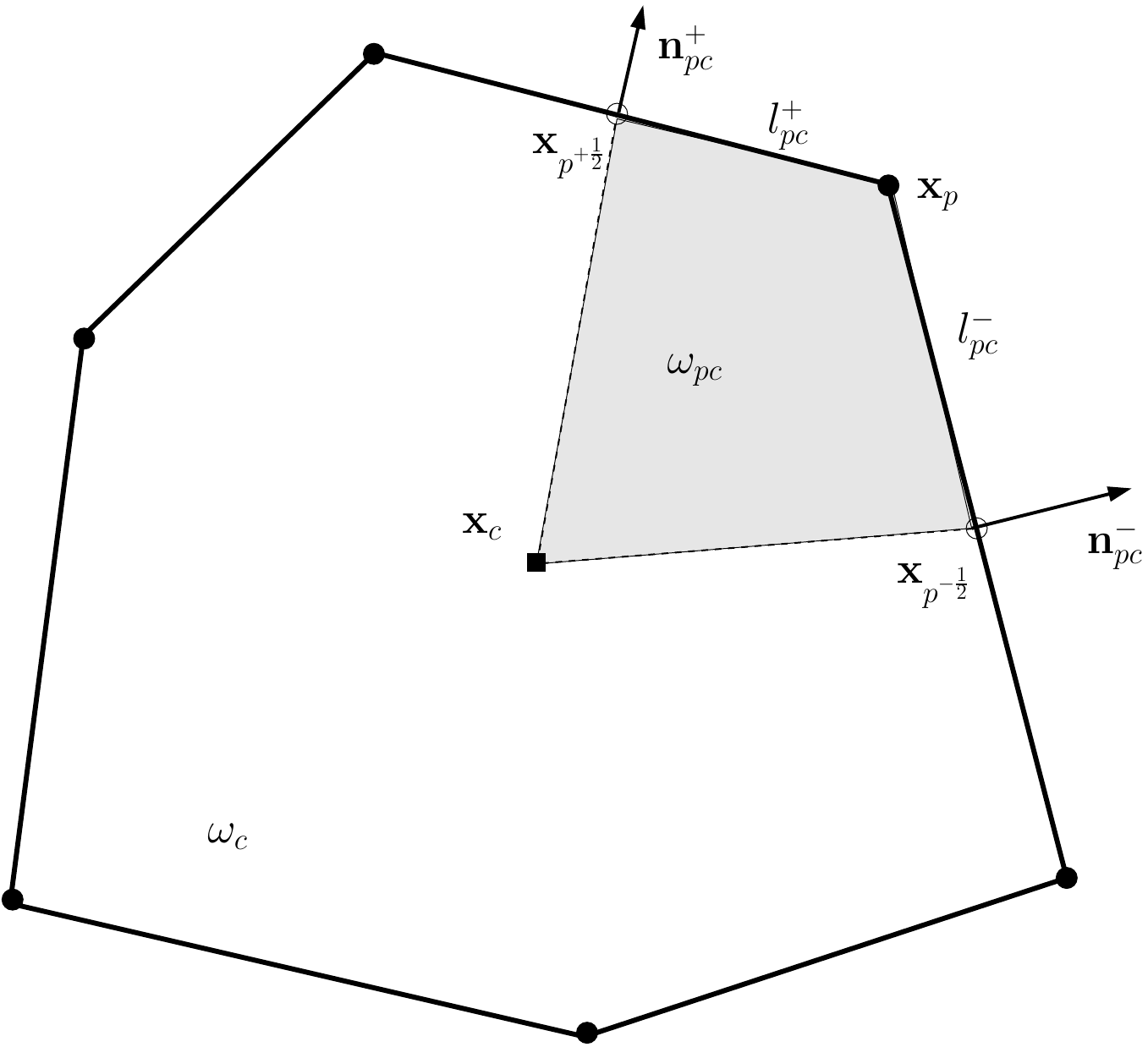}
      \caption{Geometry of polygonal cell $\omega_c$.}
      \label{fig:polygrida-mr}
    \end{subfigure}
    \begin{subfigure}{0.5\textwidth}
      \centering
      \includegraphics[width=\textwidth]{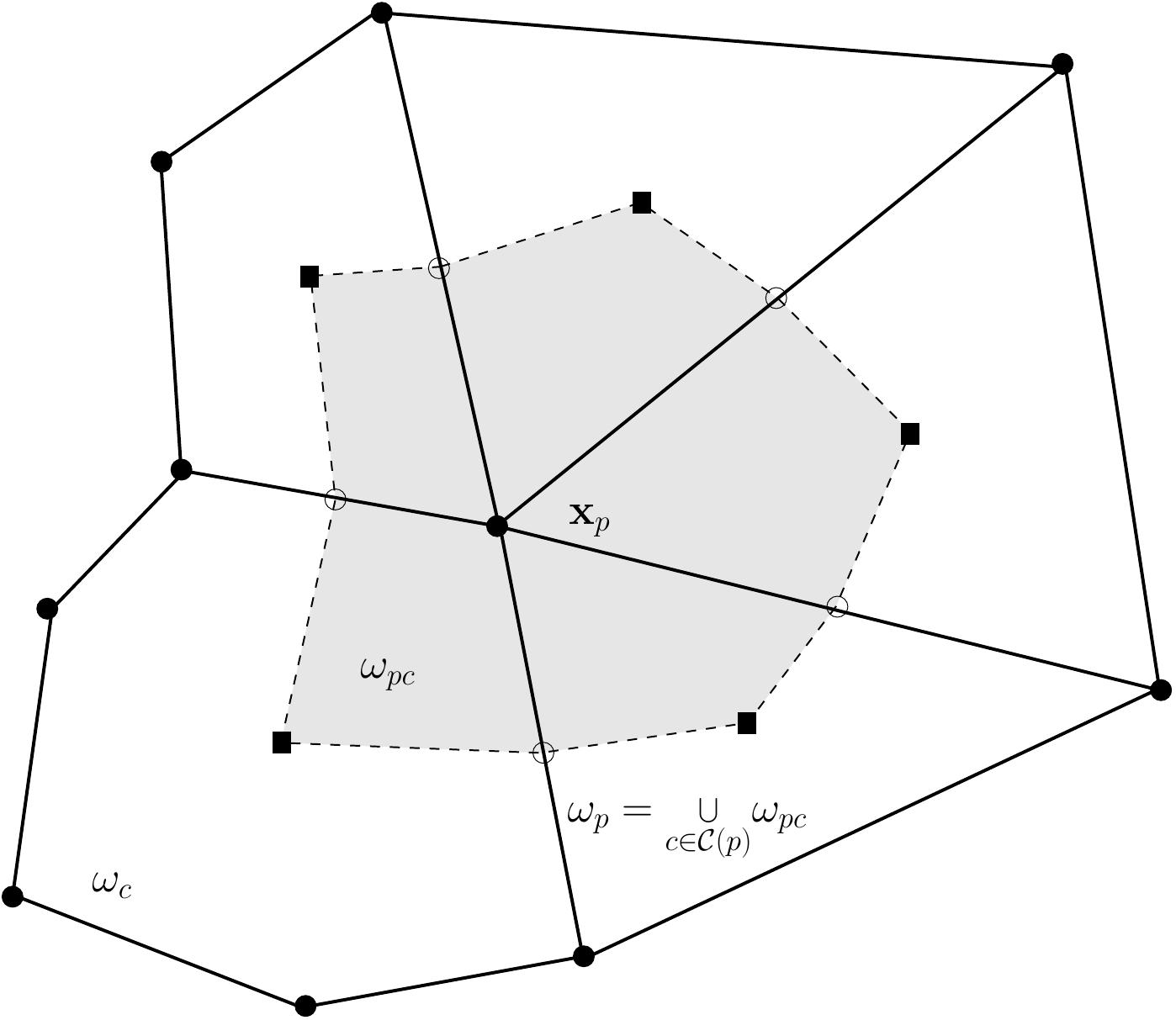}
      \caption{Dual cell $\omega_p$.}
      \label{fig:polygridb-mr}
       \end{subfigure}
    \caption{Geometry of a polygonal grid: primal and dual cells.}
    \label{fig:polygrid-mr}
  \end{center}
\end{figure}
We start with some geometrical definitions.
Let us   consider a tesselation of the computational given by non overlapping conformal polygonal cells denoted $\omega_c$: $\mathcal{D}=\cup_c \omega_c$. Each cell is  defined by the set of its vertices $\mathcal{P}(c)$.
The generic vertex also named point is denoted using the label $p$ and its vector position is $\mathbf{x}_p$. Vertices are denoted in a counterclockwise fashion, and $p^{+}$ and $p^{-}$  denote the next, and previous  points in this order. We also introduce the  midpoints of $[p^{-},p]$ and $[p,p^{+}]$: $p^{-\frac{1}{2}}$ and $p^{+\frac{1}{2}}$. The quadrangle obtained joining successively the cell centroid to $p^{-\frac{1}{2}}$, $p$, $p^{+\frac{1}{2}}$ and the centroid defines the sub-cell $\omega_{pc}$, refer to Fig.~\ref{fig:polygrida-mr}. The set of sub-cells of a given cell constitutes a partition of this cell, that is $\omega_c=\cup_{p \in \mathcal{P}(c)} \omega_{pc}$. Gathering the sub-cells surrounding the generic vertex $p$ we define the dual cell $\omega_p$. We denote by $\mathcal{C}(p)$ the set of all the cells containing $p$ as a node, and contributing to the assembly of the dual cell, {\it cf.} Figure~\ref{fig:polygridb-mr}). Let also $N_p$ denote the cardinality of  $\mathcal{C}(p)$:
$$
N_p := \sum_{c\in \mathcal{C}(p)} 1\,.
$$
The length of the half edges emanating from $p$ are denoted by $l_{pc}^+$ and $l_{pc}^-$, and the corresponding unit normals by $\bbn_{pc}^+$ and $\bbn_{pc}^-$, respectively. Finally,
we introduce the scaled corner normal
\begin{equation}\label{eq:corner_normal0}
l_{pc}\bbn_{pc} :=  l_{pc}^+\bbn_{pc}^+ + l_{pc}^-\bbn_{pc}^-.
\end{equation}
Finally, we denote by  $\mathcal{F}(c)$
 the set of faces $f$ of cell $c$. Each face $f \in \mathcal{F}(c)$ might be split in two sub-faces obtained  by connecting the midpoint with the vertices. For a point $p \in \mathcal{P}(c)$ we consider the set $\mathcal{SF}(pc)$ of sub-faces attached to $p$ and  $c$, that is the set of sub-faces of cell $c$ impinging at point $p$. In the present two-dimensional case, the sub-faces are nothing but the ``half-faces'' introduced previously and thus $\mathcal{SF}(pc)=[p^{-\frac{1}{2}},p] \cup [p,p^{+\frac{1}{2}}]$. 
For $f \in \mathcal{SF}(pc)$ we denote by $l_{pcf}$ the measure of the face and $\mathbf{n}_{pcf}$ its unit outward normal. 
Clearly we have
 \begin{equation}
  \label{eq:partition:mr}
  \mathcal{F}(c)=\bigcup_{p \in \mathcal{P}(c)} \mathcal{SF}(pc).
\end{equation}
and
\begin{equation}
  \label{eq:corner_normal}
  l_{pc}\bbn_{pc} = \sum\limits_{f\in \mathcal{SF}(pc)} l_{pcf}\mathbf{n}_{pcf},
\end{equation}
which is a more general form of  \eqref{eq:corner_normal0}.

We now integrate \eqref{hyper} on $\omega_c$, apply Gauss'  theorem, 
and evaluate the resulting face integrals with a second order quadrature formula 
involving the face nodes (trapezium rule in 2 dimensions). This leads to
\begin{equation}
\label{eq:point-fv1}
|\omega_c| \dfrac{d\overline{\bbu}_c}{dt}
+ \sum\limits_{p\in\mathcal{P}(c)}\sum\limits_{f\in\mathcal{SF}(pc)}
l_{pcf}\mathbf{f}_{pcf} =0,
\end{equation}
where $\mathbf{f}_{pcf}$ is a consistent numerical flux in the direction $\bbn_{pcf}$, in other words:
$$
l_{pcf}\mathbf{f}_{pcf}\approx\int_{f}\bbf \, \bbn \,\mathrm{d}s\,,\quad\forall f\in\mathcal{SF}(pc).
$$
Following \cite{grd25} we formally replace the  second sum in
\eqref{eq:point-fv1} by a single point flux
\begin{equation}
\label{eq:point-fv2}
l_{pc}\hbbf_{\bbn_{pc}} \equiv  \sum\limits_{f\in\mathcal{SF}(pc)}
l_{pcf}\mathbf{f}_{pcf}
\end{equation}
representing the numerical flux  along $\bbn_{pc}$. 
In the absence of any polynomial  reconstruction, the above flux  function depends  in general on  all the average states $\{\overline{u}_q\}_{q\in\mathcal{C}_p}$.  This requires the definition of a more general notion of consistency.

\begin{definition}
\label{sec:rd:MD:consistency-mr}
 A multidimensional flux 
$$\hbbf_\bbn:=\hbbf_\bbn(\bbu_1, \ldots , \bbu_N)$$
is consistent if, when $\bbu_1= \bbu_2= \ldots = \bbu_N=\bbu$ then
$$\hbbf_\bbn(\bbu, \ldots , \bbu)=\bbf(\bbu) \bbn.$$
\end{definition}

The finite volume prototype can now be cast as 
\begin{equation}\label{eq:point-fv3}
|\omega_c| \dfrac{d\overline{\bbu}_c}{dt}
+ \sum\limits_{p\in\mathcal{P}(c)}l_{pc} \hbbf_{\bbn_{pc}} =0,
\end{equation}
for which the conservation condition is now
\begin{equation}\label{eq:point-fv4}
\sum\limits_{c\in\mathcal{C}(p)} l_{pc}\hbbf_{\bbn_{pc}} =0.
\end{equation}
Note that the sum in \eqref{eq:point-fv3} is over $p\in \PP(c)$ while the sum in \eqref{eq:point-fv4} is over $c\in \CC(p)$.

\medskip

We can now proceed in several ways. A possible approach is to follow \cite{Gallice2022,PHRaph2}  and define the sub-face fluxes $\mathbf{f}_{pcf}$, with a free parameter that allows imposing the nodal conservation constraint \eqref{eq:point-fv4}. This approach is related to numerical methods for Lagrange hydrodynamics and is discussed in more details in Section~\ref{sec:Lagrangian}.

A second approach is to construct some generalization of the dissipative form \eqref{eq:1dflux-dissipative}. To this end, following \cite{grd25,bcrt25} we set
\begin{equation}
\label{eq:mdflux-dissipative1}
\hbbf_{\bbn_{pc}} = \dfrac{1}{N_p}\sum\limits_{c\in\mathcal{C}(p)}\bbf(\overline{\bbu}_c ) \bbn_{pc} + \boldsymbol{\mathcal{D}}_{pc},
\end{equation}
where $\boldsymbol{\mathcal{D}}_{pc}$ denotes dissipation operators such that
\begin{equation}
\label{eq:mdflux-dissipative2}
\sum\limits_{c\in\mathcal{C}(p)}\boldsymbol{\mathcal{D}}_{pc} =0\,.
\end{equation}
Several strategies can be used to define these operators, using for example
the geometry of the dual cell $\omega_p$.  Examples can be found in  \cite{grd25,bcrt25}, 
and in  Section~\ref{sec:motivating:fem} below.

\begin{figure}
  \begin{center}
      \includegraphics[width=0.4\textwidth]{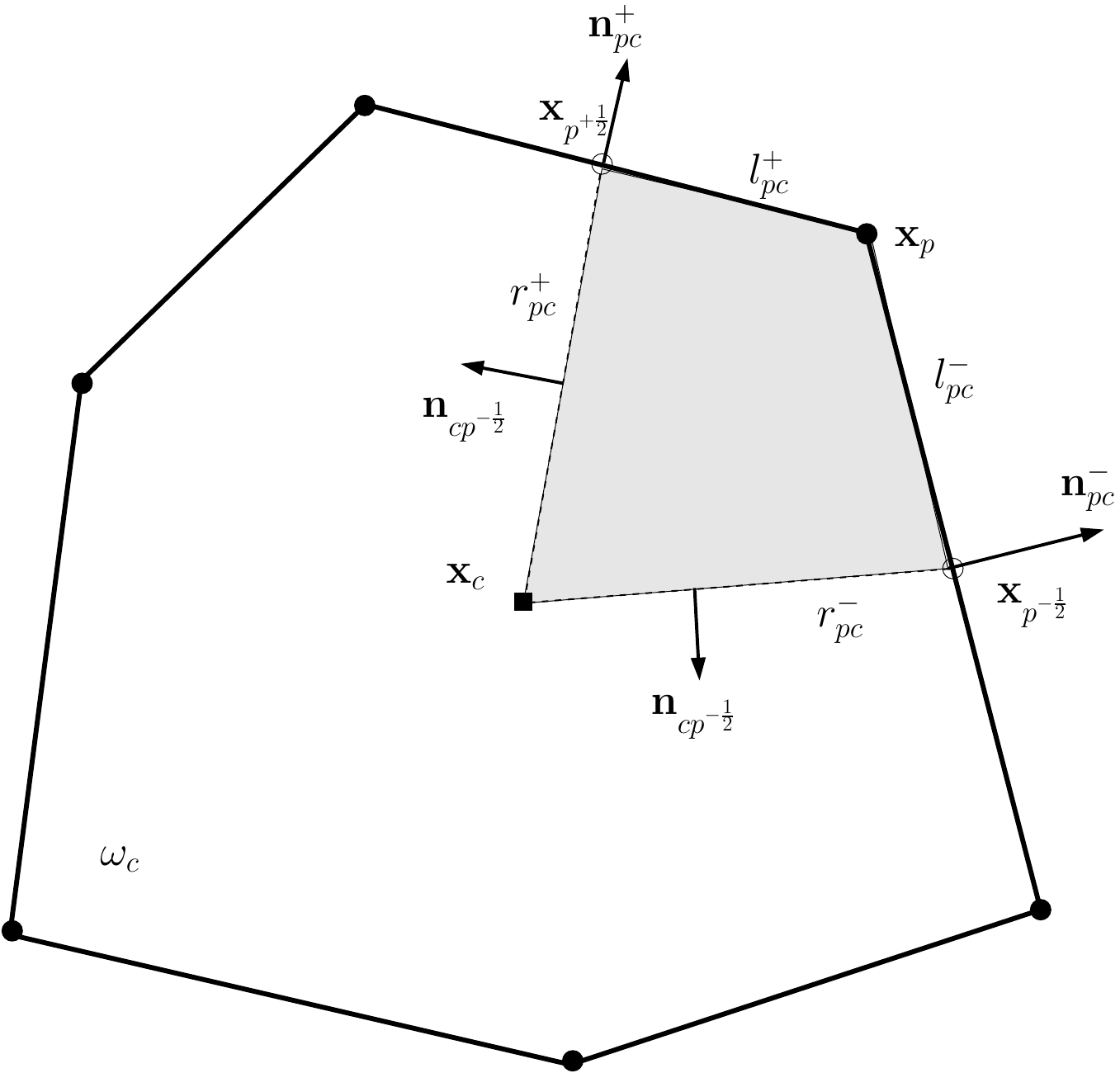}
    \caption{Sub-cell and reversed normals.}
    \label{fig:sub-cell-normals}
  \end{center}
\end{figure}

A third strategy to obtain multidimensional fluxes is the following. First one formally sets
\begin{equation}\label{eq:point-fv5}
l_{pc}\hbbf_{\bbn_{pc}} = l_{pc} \bbf(\overline{\bbu}_c) \bbn_{pc}
+ \Phi_c^{p} \;.
\end{equation}
The conservation condition \eqref{eq:point-fv4} is now equivalent to
\begin{equation}\label{eq:point-fv6}
\sum\limits_{c\in\mathcal{C}(p)}\Phi_c^{p}
=  \sum\limits_{c\in\mathcal{C}(p)} \bbf(\overline{\bbu}_c) ( - l_{pc} \bbn_{pc} ) := \Phi^{\omega_p}\;.
\end{equation}
We note that the point normal can be also decomposed as (cf. also equation \eqref{eq:corner_normal})
$$
l_{pc}\bbn_{pc} = 
l_{pc}^+\bbn_{pc}^++
l_{pc}^-\bbn_{pc}^-
=  - r_{pc}^+\bbn_{cp^{+\frac{1}{2}}}
- r_{pc}^-\bbn_{cp^{-\frac{1}{2}}},
$$
where $r_{pc}^\pm$ are the lengths of the cell radii  joining $\bbx_c$
with the mid-points  $\bbx_{p^{\pm\frac{1}{2}}}$, and $\bbn_{cp^{\pm\frac{1}{2}}}$ the normals relative to these segments (see Fig.~\ref{fig:sub-cell-normals}).
Using this notation we can now see that, if the $\overline{\bbu}_c$ values are interpreted as point values, we have the estimate (within first order quadrature)
$$
\begin{aligned}
\Phi^{p} 
= & \sum_{c\in\mathcal{C}(p)}  \left( 
r_{pc}^+ \bbf(\overline{\bbu}_c)  \bbn_{cp^{+\frac{1}{2}}}
+ r_{pc}^- \bbf(\overline{\bbu}_c) \bbn_{cp^{-\frac{1}{2}}}
\right)=\int_{\partial \omega_p}\bbf^h  \bbn \;d\gamma,
\end{aligned}
$$
where 
$\bbf^h$ is constant in subdomain $[\bbx_c,\bbx_{p^{-1/2}}, \bbx_p,\bbx_{p^{+1/2}},\bbx_c]$ with value $\bbf(\bbu_c)$.

\bigskip

We can now use  techniques issued from upwind residual distribution methods
to define the cell residuals $\Phi_c^{\omega_p}$ verifying the conservation condition
\eqref{eq:point-fv6}. Multidimensional upwind fluxes have been   proposed  in this way in \cite{grd25}.  Examples of multidimensional upwind residual distribution are recalled later.

Note that if the multidimensional flux is defined using the subface Riemann fluxes
as in \cite{Gallice2022,PHRaph2}, or the dissipative form \eqref{eq:mdflux-dissipative1},
the residual formulation still holds.  In particular, starting from \eqref{eq:point-fv3}, and using the ansatz \eqref{eq:point-fv5},
we can immediately write
\begin{equation}\label{eq:point-fvrd}
|\omega_c| \dfrac{d\overline{\bbu}_c}{dt} =
- \sum\limits_{p\in\mathcal{P}(c)}l_{pc} \hbbf_{\bbn_{pc}} 
= - \bbf(\overline{\bbu}_c) \sum\limits_{p\in\mathcal{P}(c)}l_{pc} \bbn_{pc} -  \sum\limits_{p\in\mathcal{P}(c)} \Phi_c^{p}
= -  \sum\limits_{p\in\mathcal{P}(c)}   \Phi_c^{p}
\end{equation}
since $\sum_p l_{pc} \bbn_{pc}=0$. This  shows the full equivalence between finite volumes with multidimensional corner fluxes and residual distribution.  
High order extensions can be obtained by replacing the arguments of the of the multidimensional  numerical flux  with consistent  point-wise approximations
at $\bbx_p$, exactly as discussed in the previous section. We refer to \cite{grd25} for some examples and more details.


\subsection{Finite element-like examples}\label{sec:motivating:fem}

\subsubsection{Notations}

 \label{notations_fem}In this section we  fix additional notations  needed to describe the 
 functional spaces used.

From now on, we assume that $\Omega$ has a polyhedric boundary\footnote{i.e. we assume that $\Omega$ may be different of $\R^d$ here.}.  This simplification is by no mean essential. We denote by $\mathcal{E}_h$ the set of internal edges/faces of $\mathcal{T}_h$, and by $\mathcal{F}_h$ those contained in $\partial \Omega$.  $\KK$ stands either for an element $K$ or a face/edge $e\in \mathcal{E}_h\cup \mathcal{F}_h$. The boundary faces/edges are denoted by $\Gamma$.  The mesh is assumed to be shape regular, $h_K$ represents the diameter of the element $K$. Similarly, if $e\in \mathcal{E}_h\cup \mathcal{F}_h$, $h_e$ represents its diameter.

 
 Throughout this paper, we follow Ciarlet's definition \cite{ciarlet,ErnGuermond} of a finite element approximation: we have a set of degrees of freedom $\Sigma_K$ of linear forms acting on the set $\P^k$ of polynomials of degree $k$ such that the linear mapping
 $$q\in \P^k\mapsto \big (\sigma_1(q), \ldots, \sigma_{|\Sigma_K|}(q)\big )$$
 is one-to-one. The space $\P^k$ is spanned by the basis function $\{\varphi_{\sigma}\}_{\sigma\in \Sigma_K}$  defined by
 $$\forall \sigma,\, \sigma',  \sigma(\varphi_{\sigma'})=\delta_\sigma^{\sigma'}.$$
  We have in mind either Lagrange interpolations where the degrees of freedom are associated to points in $K$, or other type of polynomials approximation such as B\'ezier polynomials where we will also do the same geometrical identification.
 Considering all the elements covering $\Omega$, the set of degrees of freedom is denoted by $\mathcal{S}$ and a generic degree of freedom  by $\sigma$. We note that for any $K$, 
 $$\forall \bbx\in K, \quad \sum\limits_{\sigma\in K}\varphi_\sigma(\bbx)=1.$$
 For any element $K$, $\#K$ is the number of degrees of freedom in $K$. If $\Gamma$ is a face or a boundary element, $\#\Gamma$ is also the number of degrees of freedom in $\Gamma$.

 The integer $k$ is assumed to be the same for any element.  We define 
$$\mathcal{V}^h=\bigoplus_K\{ \bbv\in L^2(K), \bbv_{|K}\in \P^k\}.$$
The solution will be sought for in a  space $V^h$ that is:
\begin{itemize}
\item Either $V^h=\mathcal{V}^h$. In that case, the elements of $V^h$ can be discontinuous across internal faces/edges of $\mathcal{T}_h$. There is no conformity requirement on the mesh.
\item Or  $V^h=\mathcal{V}_h\cap C^0(\Omega)$ in which case the mesh needs to be conformal.
\end{itemize}

Throughout the text, we need to integrate functions. This is done via quadrature formulas, and the symbol $\oint$ used in volume integrals
and for volume integrals
indicates 
that these integrals are not evaluated in a  classical sense, but discretely  via some
user defined numerical quadratures.

 If $e\in \mathcal{E}_h$, represents any  \emph{internal} edge, i.e. $e\subset K\cap K^+$ for two elements $K$ and $K^+$,  we define for any function $\psi$ the jump  $[\nabla \psi ]=\nabla \psi_{|K}-\nabla \psi_{| K^+}$. Similarly, $\{\bbv\}=\tfrac{1}{2}\big (\bbv_{|K}+\bbv_{|K^+}\big )$ denotes the average.
 

%

\subsubsection{Examples}\label{sec:fem:examples}
 Consider a variational formulation of the steady version of \eqref{hyper}:
$$\text{ find } \bbu^h\in \big ( V^h \big )^p \text{ such that for any } {v}^h\in V^h,  a(\bbu^h,\bbv^h)=0.$$
It is obtained by considering the weak form of \eqref{hyper} where we have specialised the test functions to  $V^h$ and the solution to  $\big (V^h\big )^p$, and where we have added stabilisation terms. 
 Examples are
\begin{itemize}
\item The SUPG \cite{Hughes1} variational formulation,  with $\bbu^h, \bbv^h\in V^h=\mathcal{V}^h\cap C^0(\Omega)$:
\begin{equation}
\label{SUPG:var}
\begin{split}
a(\bbu^h,\bbv^h)&:=-\int_\Omega  \bbf(\bbu^h) \nabla\bbv^h \; d\bbx\\&+\sum\limits_{K\subset \Omega}h_K\int_K\big [ \nabla \bbf(\bbu^h) \nabla \bbv^h\big ] \; \tau_K\; \big [\nabla\bbf(\bbu^h) \nabla \bbu^h\big ] d\bbx\\
&\qquad +\int_{\partial \Omega} \bbv^h \big (\bbf_\bbn(\bbu^h,\bbu_b)-\bbf(\bbu^h)\bbn\big ) \;d\gamma .
\end{split}
\end{equation}
Here $\bbf_\bbn(\bbu^h,\bbu_b)$ is a numerical flux that is consistent with the 'inflow' boundary condition of \eqref{hyper:bc}, and 
 $\tau_K$ is a positive parameter (or of the form $\bbA_0(\bbu^h)\btheta_K \bbA_0(\bbu^h)$ with $\btheta_K$ symmetric to get a dissipative term in the case of  the system case, assuming it is symetrisable).
\item The Galerkin scheme with jump stabilization,  see \cite{burman} for details.  We have
\begin{equation}
\label{burman:var}
\begin{split}
a(\bbu^h,\bbv^h)&:=-\int_\Omega  \bbf(\bbu^h) \nabla \bbv^h\; d\bbx+\sum\limits_{e \subset \Omega}\theta_e h_e^2\int_e  \big [ \nabla \bbu^h\big ]  \big [ \nabla \bbv^h \big ]\; d\gamma \\
&\qquad +\int_{\partial \Omega} \bbv^h \big (\bbf_\bbn(\bbu^h,\bbu_b)-\bbf(\bbu^h) \bbn\big ) \; d\gamma .
\end{split}
\end{equation}
 Here,  $\bbu^h, \bbv^h\in V^h=\mathcal{V}^h\cap C^0(\Omega)$, and $\theta_e$ is a positive parameter.
\item The discontinuous Galerkin formulation: we look for $\bbu^h, \bbv^h\in V^h=\mathcal{V}^h$ such that
\begin{equation}\label{DG:var}
a(\bbu^h,\bbv^h):=\sum\limits_{K\subset \Omega}\bigg ( -\int_K\bbf(\bbu^h) \nabla\bbv^h\; d\bbx+\int_{\partial K}\bbv^h\hat{\bbf}_\bbn(\bbu^{h},\bbu^{h,-}) \;d\gamma \bigg ).
\end{equation}
In \eqref{DG:var}, the boundary integral is a sum of integrals on the faces of $K$, and here for any face of $K$
 $\bbu^{h,-}$ represents the approximation of $\bbu$ on the other side of that face in the case of internal elements, and $\bbu_b$ when that face is on $\partial \Omega$. 
 Note that  we should have defined for boundary faces $\bbu^{h,-}=\bbu^h$, and then \eqref{DG:var} is rewritten as 
 \begin{equation}\label{DG:var:2}
 \begin{split}
a(\bbu^h,\bbv^h)&:=\sum\limits_{K\subset \Omega}\bigg ( -\int_K\bbf(\bbu^h)  \nabla\bbv^hd\bbx+\int_{\partial K}\bbv^h\, \hat{\bbf}_\bbn(\bbu^{h},\bbu^{h,-}) \; d\gamma \bigg )\\
&\quad+\sum\limits_{\Gamma\subset\partial\Omega}\int_{\Gamma}\bbv^h\bigg ( \bbf_\bbn(\bbu^h,\bbu_b)-\bbf(\bbu^h) \bbn \bigg )\; d\gamma.
\end{split}
\end{equation}
In \eqref{DG:var}, we have implicitly assumed $\hbbf_\bbn=\bbf_\bbn$ on the boundary edges.
\end{itemize}
Using the fact that the basis functions that span $V_h$ have a \emph{compact} support, then each scheme can be rewritten in the form: for all degrees of freedom $\sigma$, 
\begin{equation}
\label{eq:residual}
\sum_{K, \sigma\in K}\phi_\sigma^K(\bbu^h)+\sum_{\Gamma, \sigma\in \Gamma, \Gamma\subset \partial \Omega} \Phi_\sigma^\Gamma (\bbu^h)=0
\end{equation} with the following expression for the residuals:
    \begin{itemize}
    \item For the  SUPG scheme \eqref{SUPG:var}, the  residual are defined by
    \begin{equation}\label{SUPG}
    \begin{split}\Phi_\sigma^K(\bbu^h)&=\int_{\partial K}\varphi_\sigma \bbf(\bbu^h) \bbn \; d\gamma -\int_K  \bbf(\bbu^h) \nabla \varphi_\sigma \; d\bbx\\
   & \qquad+h_K
    \int_K \bigg (\nabla_\bbu\bbf(\bbu^h) \nabla \varphi_\sigma \bigg )\btau_K \bigg (\nabla_\bbu\bbf(\bbu^h) \nabla \bbu^h \bigg )\;d\bbx .
    \end{split}\end{equation}
    \item For the Galerkin scheme with jump stabilization \eqref{burman:var}, the residuals are defined  by:  
        \begin{equation}\label{burman}\Phi_\sigma^K(\bbu^h)=\int_{\partial K}\varphi_\sigma \bbf(\bbu^h) \bbn\; d\gamma -\int_K  \bbf(\bbu^h) \nabla \varphi_\sigma \; d\bbx +
    \sum\limits_{e \text{ faces of }K} \frac{\theta_e}{2} h_e^2 \int_{\partial K} [\nabla \bbu^h] [\nabla \varphi_\sigma]\; d\gamma\end{equation}
    with $\theta_e>0$.
    Here, since the mesh is conformal, any internal edge $e$ (or face in 3D) is the intersection of the element $K$ and another element denoted by $K^+$.
    \item For the discontinuous Galerkin scheme,
    \begin{equation}\label{DG}
    \Phi_\sigma^K(\bbu^h)=-\int_K\bbf(\bbu^h) \nabla\varphi_\sigma\;  d\bbx+\int_{\partial K}\varphi_\sigma \hat{\bbf}_\bbn(\bbu^{h},\bbu^{h,-}) \; d\gamma
    \end{equation}
    using the second definition of $\bbu^{h,-}$.
    \item The boundary residuals are 
  \begin{equation}
  \label{boundary}
  \Phi_\sigma^\Gamma(\bbu^h)=  \int_{\Gamma}\varphi_\sigma\big ( \bbf_\bbn(\bbu^h, \bbu_b)-\bbf(\bbu^h) \bbn\big )\; d\gamma
  \end{equation}
  \end{itemize}
All these residuals satisfy the relevant conservation relations, namely 
\begin{equation}
\sum_{\sigma\in K}\Phi_\sigma^K(\bbu^h)=\oint_{\partial K}\hbbf_\bbn(\bbu^h,\bbu^{h,-})\; d\gamma
\label{conservation:K} \end{equation}
and 
\begin{equation}
\label{conservation:Gamma}
\sum_{\sigma\in \Gamma}\Phi_\sigma^\Gamma(\bbu^h)=\oint_{\partial K}\big( \hbbf_\bbn(\bbu^h,\bbu^b)-\bbf(\bbu^h)\bbn\big )\; d\gamma, 
\end{equation}
depending if we are dealing with element residuals or boundary residuals.

\bigskip

For now, we are just rephrasing classical finite element schemes into a purely numerical framework. However, considering  the pure numerical point of view and forgetting the variational framework, we can go further and define schemes that have no clear variational formulation. 
These are Residual Distribution schemes with bounded splitting coefficients 
see \cite{abgrallLarat,abgralldeSantisSISC,RD-ency2,RD-ency}. The solution is sought in $V^h=\mathcal{V}^h\cap C^0(\Omega)$, i.e. globally continuous, and the residuals have the form
    \begin{equation}
    \label{schema RDS}\Phi_\sigma^K(\bbu^h)=\bbeta_\sigma \int_{\partial K}\bbf(\bbu^h) \bbn\; d\gamma
    \end{equation}
    or,  with $\theta_K\geq 0$
    \begin{equation}
    \label{schema RDS SUPG}\Phi_\sigma^K(\bbu^h)=\bbeta_\sigma \int_{\partial K}\bbf(\bbu^h) \bbn\; d\gamma+\theta_Kh_K
    \int_K \bigg (\nabla_\bbu\bbf(\bbu^h) \nabla \varphi_\sigma \bigg )\btau_K \bigg (\nabla_\bbu\bbf(\bbu^h) \nabla \bbu^h \bigg )\;d\bbx, 
    \end{equation}
or
\begin{equation}
\label{schema RDS jump}\Phi_\sigma^K(\bbu^h)=\bbeta_\sigma \int_{\partial K}\bbf(\bbu^h) \bbn\; d\gamma+
    \theta_e \;h_e^2 \int_{\partial K} [\nabla \varphi_\sigma] [\nabla \bbu^h]\; d\gamma \qquad \theta_e\geq 0,\end{equation}
   where the parameters $\bbeta_\sigma$ are defined to guarantee conservation,
   $$\sum\limits_{\sigma\in K} \bbeta_\sigma=\Id_{p\times p}$$
   and such that \eqref{schema RDS SUPG},  without the streamline term,  and \eqref{schema RDS jump},  without the jump term,  satisfy a discrete maximum principle in the scalar case. The system case is done formally. 
   
   The streamline term and jump term are introduced because one can easily see that spurious modes may exist, but their role is very different compared to \eqref{SUPG} and \eqref{burman} where they are introduced to stabilize the Galerkin scheme: if formally the maximum principle is violated, experimentally the violation is extremely small if existent at all. See \cite{energie,abgrallLarat,ABGRALL20091314} for more details. 
   A similar construction can be done starting from a discontinuous Galerkin scheme, see \cite{abgrall:shu,abgrall:dgrds}.

   The non-linear stability is provided by the coefficient $\bbeta_\sigma$ which is a non-linear function of $\bbu^h$.  Possible values of $\bbeta_\sigma$ are described in remark \ref{beta} below.
   \begin{rem}
   \label{beta} In the scalar case, 
   the coefficients $\bbeta_\sigma$ introduced in the relations \eqref{schema RDS SUPG} and \eqref{schema RDS jump} are defined by:
\begin{equation}\label{eqbeta}
\bbeta_\sigma=\dfrac{\max(0,\frac{\Phi^L_\sigma}{\Phi})}
{
\sum\limits_{\sigma'\in K} \max(0,\frac{\Phi^L_{\sigma'}}{\Phi})},
\end{equation}
where the residuals $\{\Phi_\sigma^L\}$ define a scheme that are invariant domain preserving, ans such that
$$\sum_{\sigma\in K}\Phi_\sigma^L=\Phi.$$
A typical example of such a scheme is the scalar N scheme defined in section \eqref{sec:fem:examples:upwind} below. Another one is defined as follow:
$$\Phi_\sigma^L=\frac{\Phi}{N_\sigma}+\alpha^K\big ( \bbu_\sigma-\bar \bbu^K),$$ where $\bar \bbu^K$ is the arithmetic average of the $N_K$ degrees of freedom in $K$ and $\alpha>0$ larger than some evaluation of the maximum spectral radius of the Jacobian of the flux in the element $K$, see \cite{energie,abgrallLarat}. The term
$$\alpha^K\big ( \bbu_\sigma-\bar \bbu^K)$$ is sometimes called "graph viscosity".

 These coefficients are always defined and   guarantee a local maximum principle for \eqref{schema RDS SUPG} and \eqref{schema RDS jump}: this is again a consequence of the conservation properties,  see e.g. \cite{abgrallLarat}. Note this is true for any order of interpolation.
 Extensions to the system case can be found in \cite{abgrallLarat}
 \end{rem}
\begin{rem}[Petrov Galerkin formulation of the limited RD scheme]
\label{remark:petrovGalerkin}
Again we forget the boundary conditions to simplify. If one multiply
\eqref{eq:residual} by $v(\sigma)$ with $v\in \VV^h$, and sum, we get
$$\sum_K\sum_{\sigma\in K} v(\sigma)\Phi_\sigma^K=0$$ that, using 
\eqref{schema RDS}, we can rewrite as
$$\sum_K\int_K \II(v)\text{ div }\bbf(\bbu^h)\; d\bbx=0$$
with $\II(v)$ defined as:
$$\II_{\vert K}(v)=\sum_{\sigma\in K}\beta_\sigma v(\sigma).$$
This function is piecewise constant in each $K$, so that one can interpret \eqref{schema RDS} as a Petrov Galerkin formulation of the scheme. The main difference with a classical setting is that the test function will depend on the solution via the coefficients $\beta$. This interpretation plays a role for developing unsteady versions of these schemes, see section \ref{sec:unsteady}
\end{rem}

For the sake of completeness, one should also mention the  finite element FCT schemes developed by Kuzmin and his collaborators, see \cite{KuzminHennes} or \cite{KuzminFCT0,KuzminFCT1} for initial contributions. We also  mention the work by Guermond and co-workers, see \cite{IDP0,guermond2018second,guermond2019invariant} for example. Each of these schemes can also  be put in the form \eqref{eq:residual}.
   \subsubsection{Multidimensional upwind examples}\label{sec:fem:examples:upwind}

Multidimensional upwind schemes  in the residual setting follow ideas put forward in the 80s and 90s by P.L. Roe, and the work done at that time mostly by groups   of Roe at U. Michigan, and of H. Deconinck  at the von Karman Institute \cite{Roe:87,Roe:90,Roe92,Deconinck1993} (see also \cite{RD-ency} for a review).

For scalar advection 
$$
\dfrac{\partial u}{\partial t} + \bba \cdot \nabla u=0,
$$
multidimensional upwinding in a cell $K$ boils down to the simple idea that
$$
k_\sigma^K:=\bba\cdot \bbn_{\sigma}^K \le 0 \Rightarrow \Phi_{\sigma}^K = 0, 
$$
where $\bbn_{\sigma}^K$ denote  the inward pointing vector 
normal to the edge facing node $\sigma$, scaled by the length of the edge, see Fig.~\ref{fig:inflo-outflow}. The quantities $k_\sigma^K$ are called "inflow parameters".  The fact that $K$ is a triangle (or a tetrahedron in 3D) is essential here because the velocity can be written as a linear combination of two of the vectors $\bbn_\sigma^K$ and they sum up to zero.
Using this, as shown in figure \ref{fig:inflo-outflow}, there are two types of triangles: the one for which only one of the $\bba\cdot\bbn_\sigma^K$ is negative and the other  inflow parameters are positive, and those for which only one inflow parameter is positive and the two others are negative. This does not occur for quadrangles nor hexahedrons, while it is true also for tetrahedrons.
\begin{figure}
\centering
\includegraphics[width=0.6\textwidth]{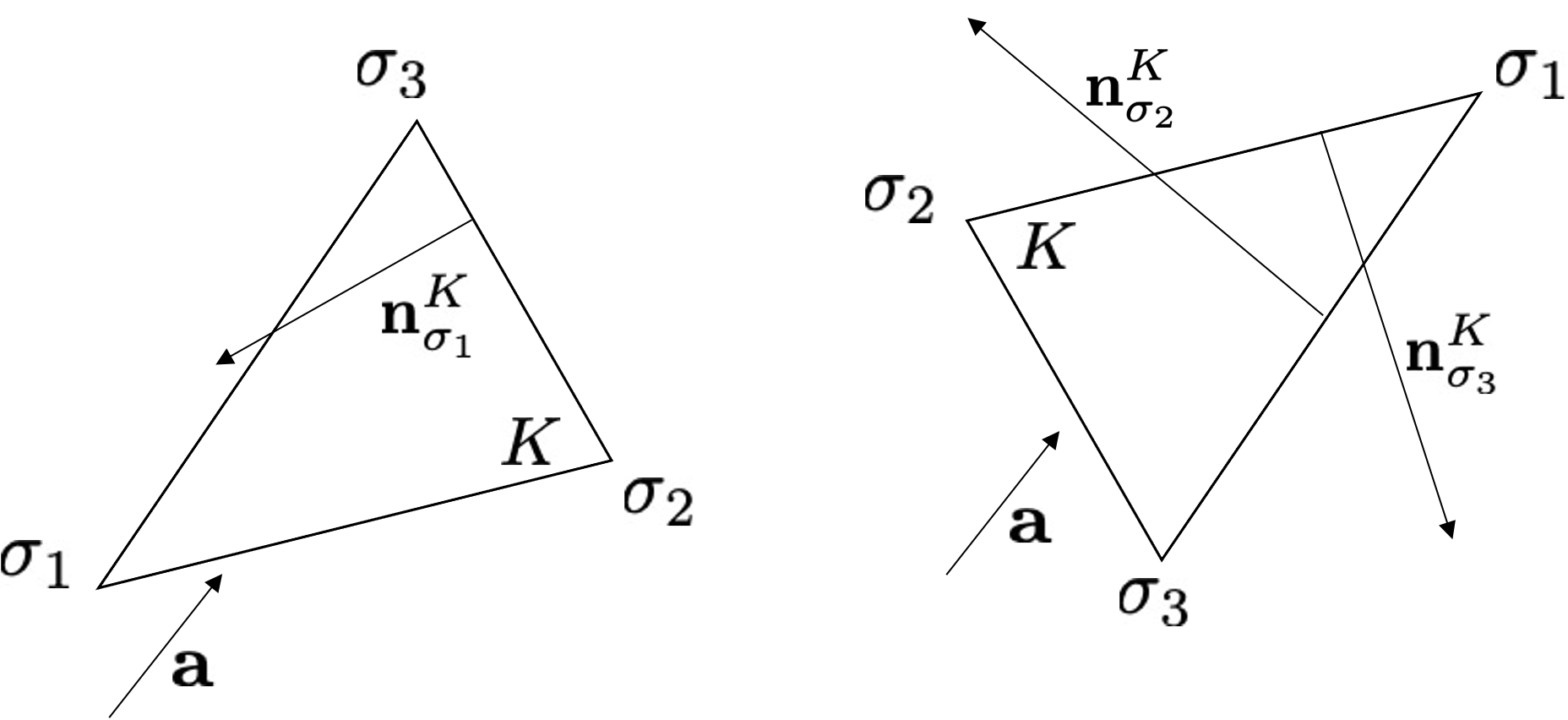}
\caption{Multidimensional upwinding.}
\label{fig:inflo-outflow}
\end{figure}

  Several examples exist of schemes following this philosophy. An example is 
$$
\Phi_{\sigma}^K = \dfrac{  \max(0, \bba\cdot\bbn_{\sigma}^K) )}{\sum\limits_{\mu\in T}    \max(0, \bba\cdot \bbn_{\mu}^K) )}\Phi\text{ with }\Phi=\int_{\partial K}\bba\cdot\bbn u\; d\gamma=\frac{1}{d}\sum_{\sigma\in K}k_\sigma^K u_\sigma.  
$$
Another example is the well-known N scheme which is defined by \cite{Roe:87,Roe:90}
\begin{subequations}\label{eq:scalar:N}
\begin{equation}\label{eq:scalar:N1}
\Phi_{\sigma}^K =  \max(0, \bba\cdot \bbn_{\sigma}^K) ) ( u_{\sigma} - u_{\text{in}} ),
\end{equation}
where, for linear advection, $u_{\text{in}}$ is the inflow state over cell $K$ defined as
\begin{equation}\label{eq:scala:N:2}
u_{\text{in}}=N\sum_{\sigma\in K} \min(0, \bba\cdot \bbn_{\sigma}^K)u_\sigma
\quad \text{ with }\quad N^{-1}=\sum_{\sigma\in K} \min(0, \bba\cdot \bbn_{\sigma}^K).
\end{equation}
\end{subequations}
The success of the above scheme is related to many important structural properties: monotonicity preservation; energy stability; very low dissipation.
These properties rely on the use of the upwinding, which is strongly based on the non-conservative form of the hyperbolic equation.
The conservation condition for the N scheme is guaranteed by the fact that 
in $d$ space dimensions
$$
\int_K\bba\cdot \nabla u \,\mathrm{d}\bbx  =  
\int_{\partial K}(\bba \, u) \cdot \bbn  \,\mathrm{d}s =
\sum_{\sigma\in K} \dfrac{1}{d} \bba\cdot \bbn_{\sigma}^K u_{\sigma}.
$$
 
 A version of the N scheme for systems has been proposed in \cite{vanderWeideDeconinck}.
 For a linear system
 $$\dpar{\bbu}{t}+\bbA \nabla \bbu=0,$$
 it writes
 \begin{subequations}\label{eq:system:N}
 \begin{equation}
 \bm\Phi_\sigma^K= \bbA_\sigma^+(\bbu_\sigma-\bbu_{in}\big )
 \text{ with }\bbA_\sigma=\bbA\bbn_\sigma.
 \end{equation}
 The "inflow" state is defined by conservation, assuming a linear approximation of $\bbu$ in $K$: from 
 \begin{equation}\label{eq:system:N:2}\sum_{\bsigma\in K}{\bm \Phi}_{\bsigma}=\int_K\bbA\nabla \bbu\; f\; d\bbx=\frac{1}{d}\sum_{\bsigma\in K}\bbA_\sigma\bbu_{\bsigma},
 \end{equation}
we get
\begin{equation}\label{eq:system:3}\bbN^{-1} \bbu_{in}=\sum_{\bsigma\in K} \bbA_{\bsigma}^-\bbu_{\bsigma} \text{ with }
\bbN^{-1}=\sum_{\bsigma\in K} \bbA_{\bsigma}^+.
\end{equation}
\end{subequations}
In some cases, for  non-linear systems,  $\bbN^{-1}$ may be locally non-invertible. However, for   symmetrizable systems  
$\bbu_{in}$ is always well defined even in these cases,  see \cite{Abgrall99}.

 Extensions to non-linear problems have required the use of complex multidimensional generalizations of the linearization \eqref{eq:RoeLinearisation}, see e.g.    \cite{Deconinck1993,Abgrall99}.
 In the multidimensional case, these linearizations are not always available,
 and may introduce spurious numerical artifact \cite{crd02}.
To cope with  this   general variants which allow the validity of 
  \eqref{conservation:K}  independently of the linearization of the flux Jacobians 
 have been proposed  in \cite{crd02,rcd05}.
 In this case the scheme reads
 \begin{equation}\label{eq:system:CRDN}
 \bm\Phi_\sigma^K= \bbA_\sigma^+(\bbu_\sigma-\bbu_{c}\big )
 \end{equation}
 and the conservative state $\bbu_{c}$ is computed by imposing conservation:

 \begin{equation}\label{eq:system:CRDN1}
 \sum_{\sigma\in K}\bbA_\sigma^+(\bbu_\sigma-\bbu_{c}\big ) = \bm\Phi
 \Rightarrow  \bbN^{-1}\bbu_{c} = \sum_{\sigma\in K} \bbA_\sigma^+\bbu_\sigma - \bm\Phi
 \end{equation}
For non-linear systems, this formulation gives freedom in the  evaluation of   $\bm\Phi$ (type of quadrature formula, variables being interpolated, etc) while allowing the use of arbitrary averages to linearize the flux Jacobians. 
  The interested reader may refer to last references and \cite{RD-ency,HDR-MR}   for more details. 
  Multidimensional upwind finite volume fluxes based on this method are discussed in \cite{grd25}.
 

\begin{rem}[About the invariance domain property of the N scheme]
\label{re:nscheme}
It is very easy to show that the scalar N scheme is monotone under a CFL condition if a forward Euler time stepping is used. This comes from the fact that the residual write as:
$$\sum_{\sigma'\in K} c_{\sigma\sigma'}(\bbu_\sigma-\bbu_{\sigma'})$$ with $$c_{\sigma\sigma'}=\max(0,\bba\cdot \bbn_\sigma)N\min(0,\bba\cdot\bbn_{\sigma'})>0.$$

For the system $N$ scheme, any formal result is not known. The only known result is an energy inequality.
\end{rem}
\subsection{Time dependent problems: space-time conservation}
\label{sec:unsteady}
The time dependant problems need special care especially if one wants to extend in a simple manner the methods of section \ref{sec:fem:examples} and \ref{sec:fem:examples:upwind}. The examples of \ref{sec:fem:examples} are or can be rewritten in with a Galerkin or Petrov-Galerkin formulation.  To get consistency in time and space, the most natural way to proceed is to see the problem \eqref{eq:hyper} as a steady problem in space and time. Of course we need to enforce a causality principle: the future depends only on the past. This means that we will consider time slabs, $[t_n, t_{n+1}$ and the space time domain $\Omega \times [t_n, t_{n+1}]$ where $\Omega\subset \R^d$ open in $\R^d$ (and possibly equal to $\R^d$). The domain $\Omega$ is covered by polygons or simplex. Then the idea is to test in $\Omega\times [t_n, t_{n+1}]$ using test functions that are tensor product of test functions in time and test functions in space.  By letting $t_{n+1}-t_n\rightarrow 0$ this leads to the methods of lines that, to be easily implemented, need an easy to inverse mass matrix.  This is clearly the case for the discontinuous Galerkin methods because the mass matrix is block diagonal. This is less easy, or at least more demanding, when the mass matrix is only sparse. This is the case for the stabilized finite element methods, for example when the  jump stabilisation \`a la Hansbo-Burman is used. The SUPG method is more complex, because the mass matrix, via the stabilisation parameter $\tau$, may need to be recomputed at each time step, and then inverted again. Concerning the non linear RD methods briefly described at the end of section \ref{sec:fem:examples} this may even look impossible because it is unclear that the "mass" matrix is invertible.

For second order in time and space approximation with triangle element, one can find a solution to this apparent paradox in \cite{Mario2015}: how to invert a non invertible mass matrix. The idea is to start from the Petrov-Galerkin formulation that is sketched in remark \ref{remark:petrovGalerkin}, then to integrate consistently in time and space, and then to modify  the mass matrix via some form of mass lumping.  This amounts to develop an iterative method with only two steps. Note that if time-space consistency is lost, the scheme is first order only, at most.

On the problem $$\dpar{\bbu}{t}+\bba\cdot \nabla \bbu=0,$$ with $\bba$ constant, we do the following:
\begin{itemize}
    \item Petrov-Galerkin formulation:
    \begin{equation*}
    \begin{split}\Phi_\sigma^K(\bbu^n,\bbu^{n+1})=\beta_\sigma(\bbu^n,\bbu^{n+1})&\bigg ( \int_K\dfrac{\bbu^{n+1}-\bbu^n}{\Delta t}\; d\bbx\\&+\int_{\partial K} \frac{ \bba\cdot\bbn\; \bbu^{n+1}+\bba\cdot\bbn\; \bbu^{n}}{2}\; d\bbx\bigg ).\end{split}\end{equation*}
    Here $\bbu$ is approximated by a $\P^1$ polynomial, and the time discretisation is done by Crank-Nicholson. Last $\beta_\sigma$ is evaluated at time $t_{n+1/2}$. A jump stabilisation {\it \`a la } Burman-Hansbo can be added or not. Clearly at this level the method is fully implicit, and requires inverting the matrix whose entries are the $\beta_\sigma$. Note that this is the case, 
    even if the time integral is done with some backward, fully explicit approach.  
    \item For the scheme above we  introduce the functionals 
    $$L^2_{\sigma,K}\bbv\mapsto \Phi_\sigma^K(\bbu^n,\bbv)$$
    and
    $$L^1_{\sigma,K}(\bbu^n,\bbv)\mapsto \frac{\vert K\vert}{3}\frac{\bbv_\sigma-\bbu_\sigma^n}{\Delta t}+\beta_\sigma^K(\bbu^n, \bbu^n)\int_{\partial K}\bba\cdot\bbn\; \bbu^n\; d\gamma$$
    These functionals satisfy the conservation relations 
    $$\sum_{\sigma\in K}L^2_{\sigma,K}(\bbv)=\int_K\frac{\bbv-\bbu^n}{\Delta t}+\int_{\partial K} \frac{\bba\cdot \bbn\; \bbu^n+\bba\cdot\bbn\; \bbv}{2}\; d\gamma.$$ 
 We now define    the order 1 functionals
    $$\sum_{\sigma\in K}L^1_{\sigma,K}(\bbv)=\int_K\frac{\bbv-\bbu^n}{\Delta t}+\int_{\partial K} \bba\cdot \bbn\; \bbu^n\; d\gamma$$
    in $K\times [t_n,t_{n+1}]$.
    We also introduce  the operators
    $$L_\sigma^2(\bbu,\bbv)=\sum_{K, \sigma\in K}L^2_{\sigma,K}(\bbu^n,\bbv)$$
    and 
    $$L_\sigma^1(\bbu,\bbv)=\sum_{K, \sigma\in K}L^2_{\sigma,K}(\bbu^n,\bbv)$$
    \item Then we set $\bbv^{(0)}=\bbu^n$, and construct $\bbv^{(p)}$ as
    $$L^1(\bbu^n,\bbv^{(p+1)})=L^1(\bbu^{n}, \bbv^{(p)})-L_2(\bbu^n,\bbv^{(p)}).$$
    Written explicitly, this amounts to write:
    \begin{equation*}
    \begin{split}\vert C_\sigma\vert \big ( \bbv_\sigma^{(p+1)}-\bbv_\sigma^{(p)}\big )=\sum_{K, \sigma\in K}&
    \beta_\sigma^K(\bbu^n,\bbv^{(p)})\bigg (\int_K(\bbv^{(p)}-\bbu^n )\; d\bbx\\
    &\qquad +\Delta t\int_{\partial K}\frac{\bba\cdot \bbn\; \bbv^{(p)}+\bba\cdot \bbn\; \bbu^n}{2}\; d\gamma\bigg )
    \end{split}
    \end{equation*}
    which is solvable explicitly.
    \item In \cite{Mario} is shown that only 2 iterations of this algorithm to obtain a second approximation of the solution. This is easy to extend to the non linear and scalar case, see the above mentionned reference
\end{itemize}
The generalization of this method  and its   analysis are discussed in \cite{AbgrallDec,Sixtine1,Sixtine2}.
 The resulting scheme   has always    the form of a series of explicit Euler  steps, similarly to e.g. classical explicit Runge-Kutta discretizations.
 Other extensions and applications to time dependent problems   can be found in \cite{ABGRALL2019274,abgrall2021relaxation,ArR:14,rf14,micalizzi2022new,Mario2015,ArR:17,arpaia2020well}. 

\section{General theory}
\label{sec:gt}

All the previous examples have the following structure: starting from a mesh made of a collection of subdomains $\{K_l\}_{l=1, \ldots, n_e^{primal}}$ on which the solution of \eqref{hyper} is represented with degrees of freedom $\sigma$, a dual mesh is constructed. Let us denote the elements of this dual mesh by $\{C_k\}_{k=1, \ldots, n_e^{dual}}$.

The schemes we consider are defined via residuals, they can be understood as differences of flux values. Here, the notion of primal and dual mesh is not as it is usually done, it is mostly used as a terminology. The primal mesh is the mesh on which the residuals are naturally defined, and the dual mesh is the mesh on which we express the evolution  in time of  each  unknown as a sum of  fluctuations from surrounding elements, see below.

In the example of section \ref{sec:motivating:1D}, the primal mesh has elements $K_j=[x_j,x_{j+1}]$, and the dual mesh is with elements $[x_{j-1/2},x_{j+1/2}]$. In the example of section \ref{sec:FV:1}, the dual mesh is made of simplex, while the primal mesh is made of elements as depicted on figure  \ref{fig:fv}.  In the example of section \ref{sec:motivating:fem}, the primal and dual mesh are the same, made of simplex. In the example of section \ref{sec:motivating:FV:II}, the configuration is very similar to the one of section \ref{sec:FV:1}.

A general prototype can be written for all these cases. In particular, when solving
\begin{subequations}\label{eq:hyper}
\begin{equation}
\label{eq:hyper:1}
\dpar{\bbu}{t}+\text{ div }\bbf(\bbu)=0, \quad \bbx\in \R^d
\end{equation}
with
\begin{equation}
\label{eq:hyper:2}
\bbu(\bbx,0)=\bbu_0(\bbx).
\end{equation}
\end{subequations}
the scheme is described by a relation like:
\begin{subequations}\label{forme schemas}
\begin{equation}\label{rds_re}
\vert C_\sigma\vert \dfrac{\mathrm{d}\bbu_\sigma}{\mathrm{d}t}+\sum_{K, \sigma\in K}\Phi_\sigma^{K}(\bbu^h)
=0\end{equation}
and the residuals satisfy  conservation relations
\begin{equation}
\label{rds_re:conservation}
\Phi^K:=\sum_{\sigma\in K}\Phi_\sigma^K=\oint_{\partial K}\hbbf_\bbn(\bbu^h_{\vert K},\bbu^h_{\vert K'})\; d\gamma
\end{equation}
\end{subequations}
In \eqref{rds_re:conservation} $K'$ represent the elements of the primal mesh that  share a face of the dual mesh, see figure \ref{fig:general}. Last, $\vert C_\sigma\vert$ is an area. 
The precise definition of all these entities (dual mesh, residual, area)  depends on
the scheme itself. 
Here,  we only need to use the general structure of the discretization which is  given by  \eqref{rds_re:conservation}-\eqref{forme schemas} for all methods.
We have not taken into account boundary term, assuming that $\Omega=\R^d$, but they could be taken into account in a similar way. The quantity $\Phi^K$ is the total residual.
\begin{figure}[h]
\begin{center}
\begin{subfigure}{0.5\textwidth}
{\pgfkeys{/pgf/fpu/.try=false}%
 \ifdim\dimen1=0pt\ifdim\dimen3=0pt\dimen1=1000sp\dimen3\dimen1
  \else\dimen1\dimen3\fi\else\ifdim\dimen3=0pt\dimen3\dimen1\fi\fi
\begin{tikzpicture}[x=+\dimen1, y=+\dimen3]
{\ifx\XFigu\undefined\catcode`\@11
\def\temp{\alloc@1\dimen\dimendef\insc@unt}\temp\XFigu\catcode`\@12\fi}
\XFigu3946sp
\ifdim\XFigu<0pt\XFigu-\XFigu\fi
\catcode`\@11
\pgfutil@ifundefined{pgf@pattern@name@xfigp0}{
\pgfdeclarepatternformonly{xfigp0}
{\pgfqpoint{-1bp}{-1bp}}{\pgfqpoint{9bp}{5bp}}{\pgfqpoint{8bp}{4bp}}
{	\pgfsetdash{}{0pt}\pgfsetlinewidth{0.45bp}
	\pgfpathqmoveto{-1bp}{4.5bp}\pgfpathqlineto{9bp}{-0.5bp}
	\pgfusepathqstroke
}
}{}
\catcode`\@12
\definecolor{cyan2}{rgb}{0,0.69,0.69}
\clip(428,-11572) rectangle (12322,-813);
\tikzset{inner sep=+0pt, outer sep=+0pt}
\pgfsetlinewidth{+7.5\XFigu}
\pgfsetstrokecolor{black}
\pgfsetfillpattern{xfigp0}{black}
\draw[pattern,preaction={fill=cyan2}] (5925,-6000)--(5175,-2175)--(9075,-825)--(12000,-3225)--(11100,-6525)--(5925,-6000);
\pgfsetfillcolor{red}
\filldraw  (5925,-6000) circle [radius=+106];
\filldraw  (5625,-4575) circle [radius=+106];
\filldraw  (4950,-7050) circle [radius=+106];
\draw (11100,-6525)--(11700,-9750)--(5850,-10125)--(3525,-8550)--(5925,-6000);
\draw (5925,-6000)--(1200,-4275)--(3525,-8550);
\draw (1200,-4275)--(1425,-2175)--(5175,-2175);
\pgfsetlinewidth{+15\XFigu}
\pgfsetdash{{+90\XFigu}{+90\XFigu}}{++0pt}
\draw (4350,-3975)--(675,-2325);
\draw (3375,-6675)--(450,-8100);
\draw (5775,-9150)--(5175,-11550);
\draw (9900,-7350)--(12300,-7950);
\draw (8250,-2850)--(10650,-1350);
\draw (4350,-3975)--(4575,-1500);
\draw (8250,-2850)--(5925,-1425);
\draw (4350,-3975)--(8250,-2850)--(9900,-7350)--(5775,-9075)--(3375,-6675)--(4350,-3975);
\pgfsetfillcolor{black}
\pgftext[base,left,at=\pgfqpointxy{10200}{-3675}] {\fontsize{12}{14.4}\usefont{T1}{ptm}{m}{n}$\beta$}
\pgftext[base,left,at=\pgfqpointxy{5550}{-6600}] {\fontsize{12}{14.4}\usefont{T1}{ptm}{m}{n}$\sigma_1$}
\pgftext[base,left,at=\pgfqpointxy{8175}{-6750}] {\fontsize{12}{14.4}\usefont{T1}{ptm}{m}{n}$\sigma_2$}
\pgftext[base,left,at=\pgfqpointxy{3225}{-4500}] {\fontsize{12}{14.4}\usefont{T1}{ptm}{m}{n}$\alpha$}
\pgfsetlinewidth{+7.5\XFigu}
\pgfsetdash{}{+0pt}
\pgfsetfillcolor{red}
\filldraw  (8100,-6225) circle [radius=+106];
\pgfsetfillcolor{black}
\pgftext[base,left,at=\pgfqpointxy{5850}{-7425}] {\fontsize{12}{14.4}\usefont{T1}{ptm}{m}{n}$\beta'$}
\end{tikzpicture}}%
\end{subfigure}
\end{center}
\caption{\label{fig:general} A refaire Dotted line: primal mesh with element $\alpha$, Plain lines: dual mesh with elements $\beta$. The element $\beta'$ shares with $\beta$ the face containing $\sigma_1$ and $\sigma_2$.}
\end{figure}
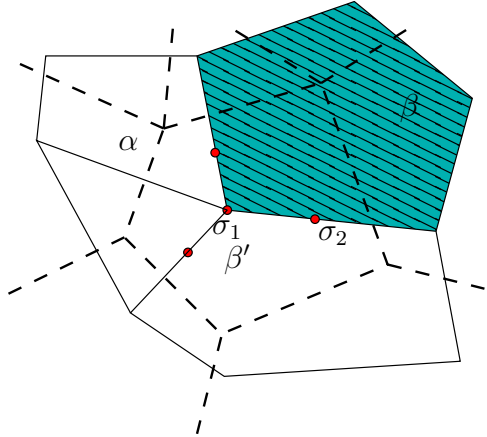

The purpose of this section is to provide a Lax-Wendroff like theorem, and to show that any scheme of the form \eqref{rds_re}-\eqref{rds_re:conservation} can be rewritten as a finite volume scheme.
\subsection{Lax Wendroff like theorem}
Here we state a result that explains why the conservation relations \eqref{rds_re:conservation} will guaranty that in the limit of mesh refinement, and provided the scheme is stable, that the limit solution is a weak one. This will be revisited in section \ref{sec:rd:flux}  by using different arguments showing that RD scheme can indeed be seen as finite volume schemes.

The notations are those defined in section \ref{notations_fem}.
Using the conservation relation \eqref{rds_re:conservation}, we obtain:
\begin{lemma} for any $\bbv^h\in V^h$,
$$\bbv_h=\sum_{\sigma} \bbv_\sigma \varphi_\sigma,$$ the following relation:
\begin{equation}\label{rds_re:algebre2}
\begin{split}
0&=  -\oint_\Omega \nabla \bbv_h  \bbf(\bbu^h) \; d\bbx  +\sum\limits_{e\in \mathcal{E}_h} \int_e[\bbv^h]\hbbf_\bbn(\bbu^h,\bbu^{h,-})\; d\gamma\\
&\qquad +\sum\limits_{K\subset \Omega}\frac{1}{\#K}\bigg ( \sum\limits_{\sigma,\sigma'\in K} (\bbv_\sigma-\bbv_{\sigma'})\bigg ( \Phi_\sigma^K(\bbu^h)-\Phi_\sigma^{K, Gal}(\bbu^h) \bigg ) 
\bigg ),
\end{split}
\end{equation}
where
$$\Phi_\sigma^{K, Gal}(\bbu^h)=-\oint_K\nabla\varphi_\sigma \bbf(\bbu^h) \; d\bbx +\oint_{\partial K} \varphi_\sigma \hbbf_\bbn(\bbu^h,\bbu^{h,-}) \; d\gamma.$$
\end{lemma}
\begin{proof}
We start from \eqref{rds_re} which is multiplied by $\bbv_\sigma$, and these relations are added for each $\sigma$. We get ($\Omega=\R^d$)
$$
0=\sum\limits_{\sigma\in \mathcal{S}}\bbv_\sigma\bigg ( \sum\limits_{K\subset \Omega, \sigma\in K} \Phi_\sigma^K(\bbu^h)
\bigg ).
$$
Permuting the sums on $\sigma$ and $K$, then on $\sigma$ and $\Gamma$, we get:
$$
0=\sum\limits_{K\subset \Omega} \bigg ( \sum\limits_{\sigma\in K} \bbv_\sigma\Phi_\sigma^K(\bbu^h)\bigg ) 
 .
$$
We have, introducing $\Phi_\sigma^{K, Gal}$  as above
and $\#K$ the number of degrees of freedom in $K$,
\begin{equation*}
\begin{split}
\sum\limits_{\sigma\in K} \bbv_\sigma\Phi_\sigma^K(\bbu^h)&=\sum\limits_{\sigma\in K} \bbv_\sigma \Phi_\sigma^{K, Gal}(\bbu^h) +\sum\limits_{\sigma\in K} \bbv_\sigma\bigg ( \Phi_\sigma^K(\bbu^h)-\Phi_\sigma^{K, Gal}(\bbu^h) \bigg )\\
&= -\oint_K \nabla \bbv_h  \bbf(\bbu^h) \; d\bbx +\oint_{\partial K} \bbv^h \hbbf_\bbn(\bbu^h, \bbu^{h,-})\; d\gamma\\&\qquad\qquad+ \sum\limits_{\sigma\in K} \bbv_\sigma\bigg ( \Phi_\sigma^K(\bbu^h)-\Phi_\sigma^{K, Gal}(\bbu^h) \bigg )\\
&=-\oint_K \nabla \bbv_h  \bbf(\bbu^h) \; d\bbx +\oint_{\partial K} \bbv^h \hbbf_\bbn(\bbu^h, \bbu^{h,-})\; d\gamma\\&\qquad \qquad+ \frac{1}{\#K}\sum\limits_{\sigma,\sigma'\in K} (\bbv_\sigma-\bbv_{\sigma'})\bigg ( \Phi_\sigma^K(\bbu^h)-\Phi_\sigma^{K, Gal}(\bbu^h) \bigg )
\end{split}
\end{equation*}
because $$\sum\limits_{\sigma\in K} \big ( \Phi_\sigma^K(\bbu^h)-\Phi_\sigma^{K, Gal}(\bbu^h) \big )=0.$$

We finally get:
\begin{equation*}
\begin{split}
0&= \sum\limits_{K\subset \Omega} \bigg ( -\oint_K \nabla \bbv_h  \bbf(\bbu^h) \; d\bbx +\oint_{\partial K} \bbv^h \hbbf_\bbn(\bbu^h, \bbu^{h,-})\; d\gamma\bigg ) \\
& \qquad  +\sum\limits_{K\subset \Omega}\frac{1}{\#K}\bigg ( \sum\limits_{\sigma,\sigma'\in K} (\bbv_\sigma-\bbv_{\sigma'})\bigg ( \Phi_\sigma^K(\bbu^h)-\Phi_\sigma^{K, Gal}(\bbu^h) \bigg ) \bigg )
\end{split}
\end{equation*}
i.e. after having defined $[\bbv^h]= \bbv^h-\bbv^{h,-}$ and chosen one orientation of the internal edges $e\in \mathcal{E}_h$,  we get \eqref{rds_re:algebre2}.
\end{proof}

The relation \eqref{rds_re:algebre2} is instrumental in proving the following results.
The first one is proved in  \cite{AbgrallRoe}, and is a generalisation of the classical Lax-Wendroff theorem.
\begin{theorem}\label{rds_re:th:LW}
Assume the family of meshes $\mathcal{T}=(\mathcal{T}_h)$
is shape regular. We assume that the residuals $\{\Phi_\sigma^{\mathcal{K}}\}_{\sigma\in \KK}$,
 for $\mathcal{K}$ an element or a boundary element of $\mathcal{T}_h$, satisfy: \begin{itemize}
\item There exists a constant $C$ which depends only on the family of meshes $\mathcal{T}_h$ and $M\in \R^+$ such that
 for any $\bbu^h\in V^h$ with $||\bbu^h||_{\infty}\leq M$, then
\begin{equation}
\label{rds_re:H1}\big|\Phi^\KK_\sigma({\bbu^h}_{|\KK})\big |\leq C h^{d-1}\;\sum_{\sigma, \sigma'\in \SS_K}|\bbu_\sigma^h-\bbu_{\sigma'}^h|,
\end{equation}
where $\SS_K$ is the list of degrees of freedom that appears in the expression of $\Phi^\KK_\sigma({\bbu^h}_{|\KK})$. 
\item There exists $C\in \N$, independent of $\mathcal{T}_h$, so that $\#\SS_K$ the cardinal of $\SS_K$, is bounded by $C$,
\begin{equation}
\label{rds_re:H2}\#\SS_K\leq C,
\end{equation}
\item The conservation relations \eqref{rds_re:conservation} where the quadrature formula satisfy: for any $w\in L^\infty(\Omega \times [0,T]$ and any $K$ regular enough,
$$\bigg \vert\int_K w(\bbx)\; d\bbx-\oint_K w(\bbx)\; d\bbx\bigg \vert =\Vert w\vert_\infty O(h^{d+1})$$
where $C$ depends only on $w$ via its $L^\infty$, norm
\end{itemize}
Then if there exists a constant $C_{max}$ such that the solutions of the scheme 
\begin{equation}\label{rds_re:sec:rd:leschema:formule}
\bbu^{n+1}_\sigma= \bbu^{n}_\sigma - \dfrac{\Delta t}{{|C_\sigma|}} \sum_{K, \sigma \in K} \Phi_\sigma^K(\bbu^{n}),
\end{equation}
where $C_\sigma$ is the dual control volume associated { with the dof $\sigma$},
 satisfy $||\bbu^h||_{\infty}\leq C_{max}$ and a function $\bbv\in L^2(\Omega)$ such that $(\bbu^h)_{h}$ or at least a sub-sequence converges to $\bbv$ in $L^2(\Omega)$, then $\bbv$ is a weak solution of \eqref{eq:hyper}.
\end{theorem}

The proof will be given in the next section \ref{rds_re:sec:th:LW}.
Another consequence of \eqref{rds_re:algebre2} is the following result on entropy inequalities:
\begin{proposition}\label{rds_re:th:entropy}
Let $(\eta,\mathbf{g})$ be a  entropy-flux couple for \eqref{eq:hyper} and $\hat{\mathbf{g}}_\bbn$ be a numerical entropy flux consistent with $\mathbf{g} \bbn$. Assume that the residuals satisfy:
for any element $K$,
\begin{subequations}\label{rds_re:entropy}
\begin{equation}\label{rds_re:entropy:1}
\sum_{\sigma \in K}\langle\nabla_\bbu \eta(\bbu_\sigma), \Phi_\sigma^K\rangle \geq \oint_{\partial K} \hat{\mathbf{g}}_\bbn(\bbu^h,\bbu^{h,-}) \; d\gamma
\end{equation}
\end{subequations}
Then, under the assumptions of theorem \ref{rds_re:th:LW}, the limit weak solution also satisfies the following entropy inequality: for any $\varphi\in C^1(\overline{\Omega})$, $\varphi\geq 0$, 
$$\int_0^T\int_\Omega\dpar{\varphi}{t}\eta(\bbu)\; d\bbx-\int_0^T\int_\Omega \nabla \varphi \mathbf{g}(\bbu) \; d\bbx
 -\int_\Omega\varphi(\bbx,0)\eta(\bbu_0(\bbx))\; d\bbx\leq 0.$$
\end{proposition}
The proof is similar to that of theorem \ref{rds_re:th:LW}. 
\subsection{Sketch of the proof of theorem \ref{rds_re:th:LW}.}\label{rds_re:sec:th:LW}

We only give the results and refer to \cite{AbgrallRoe} for the proof itself.

\subsubsection{Technical assumptions}\label{sec:rd:sec:lax}
Let us first give some conditions on the mesh,  the residuals
$\Phi_\sigma^{K}$, and define some notations and
functional spaces. {We first do the proof for globally continuous elements, and we will explain how to (slightly) modify it for discontinuous elements at the end of this section, se remark \ref{rds_re:remark}.}

\begin{assumption}
\label{sec:rd:H0}
The mesh $\TT_h$ is  regular. By regular we mean that all elements
 are roughly the same size, more 
precisely that there exist  constants 
 $C_1$ and  $C_2$ such that for any element
$$ K, \; \; C_1 \leq \sup_{K \in \mathcal{K}_h} \dfrac{h^d}{|K|} \leq C_2.$$
\end{assumption}
We say that two elements are neighbors if they have a common edge.
Let $\TT_h$ be a triangulation satisfying assumption \ref{sec:rd:H0},
 and ${\cal C}_h$ be 
a set of dual volumes
associated with the degrees of freedom $\sigma$. Then we consider the 
following subspaces,
\begin{equation}\label{LxW:spaces}
\begin{array}{l}
V_h^k=\{ v_h \in \big (C^0(\R^2)\big )^p;  \forall ~K \in
\TT_h{v_{h}}_{|K} \text{ polynomial of degree k } \}\\
~~\\
X^h=\{v_h; {v_h}_{|C} \text{ constant}\in \R^p, \forall ~C \in \CC_h\}.
\end{array}
\end{equation}
The basis functions of $V_h^k$ are $\{\varphi_\sigma\}_{\sigma}$. Since $\P^k$ is finite dimensional, there exist two constants, $C_1>0$ and $C_2>$ depending only on $k$ such that for any $v\in \P^k$, 
\begin{equation}
\label{sec:rd:assumption basis functions}
C_{1}|K|\sum_{\sigma\in K}|v_\sigma| \leq \int_K\big |v\big |\; d\bbx\leq C_{2}|K|\sum_{\sigma\in K} |v_\sigma|.
\end{equation}

We denote  by $\pi_h^1v$ the piecewise linear interpolation of a continuous
 function.
Let  $L_h:  V_h^k \rightarrow X^h$ be the mass lumping operator,
$L_h(v)=\sum_\sigma v(x_\sigma) \chi_{\sigma}$ where $\chi_{\sigma}$ is the characteristic 
function of the cell $C_\sigma$.
Last, $\pi_h^k$ is the $\P^k$ interpolation defined on $V_h^k$. 
\bigskip

The proof is strongly inspired by \cite{Kroner:96}.
 For the
sake of simplicity, we assume $d=2$ and $p=1$, the generalisation is immediate.

As in \cite{Kroner:96}, we start by 
\begin{lemma}
\label{sec:rd:lemme1}
Let  $T>0$ and  $N$ the integer part of  $\tfrac{T}{\Delta t}$.
We consider $\QQ \subset \R^2$, a bounded  domain.
Let 
$(\bbu^h)_{h}$ be a sequence such that  $\bbu^h(\; . \; ,t_n) \in X_h$ for any 
$n\leq N$. We assume 
there exists a constant  $C$ independent of  $h$ and  $\bbu\in  
L^{2}_{loc}(\QQ\times [0,T])$ such that 
$$
\sup_h \sup_{\bbx,t}\Vert \bbu^h(\bbx,t)\Vert\leq  C ,\; \; \; \lim_h
\Vert{\bbu^h-\bbu}\Vert_{L^2(\QQ\times [0,T])}=0.
$$
We define $$\overline{\bbu^h}_K=\dfrac{1}{|K|}\int_K \bbu^h\; d\bbx$$
the average value of $\bbu^h$ in $K$.
Then 
\begin{enumerate}
\item \begin{equation}\label{sec:rd:BV:1}\lim\limits_{h\rightarrow 0} \bigg (\sum\limits_{n=0}^N \Delta t \sum\limits_{K\subset \QQ} 
|K|\sum\limits_{\sigma
\in K} \Vert\bbu^h_\sigma-\overline{\bbu^h}_K)\Vert\bigg )=0.\end{equation}
\item and 
\begin{equation}\label{sec:rd:BV:2}\lim\limits_h \bigg [h \; \Vert{\nabla
{\bbu}^h}\Vert_{L^2(\QQ\times [0,T])}\bigg ]=0.\end{equation}
\end{enumerate}
\end{lemma}

Then we have the following very classical lemma, see e.g. \cite{GodlewskiRaviartTome1}
\begin{lemma}
\label{th:LW}
Let $\varphi \in C_0^{k+1}(\R^2 \times \R^+)$. With the assumptions of 
Theorem
\ref{rds_re:th:LW}, one has
$$ \sum\limits_n\Delta t \sum\limits_\sigma |C_\sigma| \big (u^{n+1}_{\sigma}-u^n_{\sigma} \big ) \varphi(\sigma,t_n)+
\int_{\R^2\times \R^+} u \dpar{\varphi}{t} \; d\bbx dt +\int_{\R^2} u_0(x)
\varphi(x,0)\rightarrow 0$$
when $h \rightarrow 0$.
\end{lemma}

\medskip
Again we introduce $\phi_\sigma$ the basis function at $\sigma$ and $\Phi_\sigma^{K,Gal}$ the Galerkin residuals
$$
\Phi_\sigma^{K,Gal}=-\oint_{K} \nabla \varphi_\sigma \bbf^h(\bbu^h)+\oint_{\partial K} \varphi_\sigma \hbbf_\bbn^h(\bbu^h,\bbu^{h,-}) \; d\gamma.
$$
For ease of notations, we have dropped the superscript $n$ that indicates the time step.
We remind that
$$\sum_{\sigma \in K} \Phi_\sigma^{K,Gal}=\oint_{\partial K}\hbbf_\bbn^h(\bbu^h, \bbu^{h,-}) \; d\gamma=\sum_{\sigma \in K} \Phi_\sigma^K.$$
We have
\begin{lemma}
\label{sec:rd:lemme3}
Let $\bbv \in C^{k+1}_0(\R^2\times \R^+)$, $\bbf\in
(C^{1}(\R^m))^{d}$, $d=2$. Assume that
 $\hbbf_\bbn$ is a consistant Lipschitz continuous numerical flux,  and that  $\bbu_h$ satisfies those of Theorem \ref{th:LW}.
 Then
$$
\Delta t   \sum\limits_{n,K}\sum\limits_{\sigma \in K'_K \subset K} \bbv_\sigma^n
\Phi_{\sigma}^{K,Gal} (\bbu_h^n)
+
\int_{\R^2\times \R^+} \nabla \bbv(\bbx,t) \bbf(\bbu(\bbx,t))  \; d\bbx dt \rightarrow 0$$
when $h \rightarrow 0$.
\end{lemma}
This time we give the proof because it explains the role of the quadrature formula.

\begin{proof}
Let $\bbv\in C_0^1(\R^2 \times [0,+\infty[)$. Let $\Omega$ and $T$
such that 
$\text{supp}(\bbv) \subset \Omega \times [0,T]$. Consider $K \in
\TT_h$.
Now $\pi_h^k \bbv$ is the interpolation of $\bbv$ of degree $k$. We have:
\begin{equation*}
\begin{split}
\Delta t \sum_{n,K}  \sum_{\sigma \in K} &\bbv_\sigma^n
\Phi_{\sigma}^{K, Gal} (\bbu_h^n) =-\sum_{n,K} \int_{t^n}^{t^{n+1}} \int_K\nabla\pi_h^k \bbv(\bbx,t_n)  \bbf^h (\bbu_h^n) \; d\bbx dt\\
&-\underbrace{\sum_{n,K} \int_{t^n}^{t^{n+1}} \bigg (\oint_K\nabla\pi_h^k \bbv(\bbx,t_n)  \bbf^h (\bbu_h^n) \; d\bbx dt-\int_K\nabla\pi_h^k \bbv(\bbx,t_n)  \bbf^h (\bbu_h^n) \; d\bbx dt\bigg )}_{(I)}\\
&
\qquad +
\underbrace{\sum_{n,K} \int_{t^n}^{t^{n+1}}\oint_{\partial K}\pi_h^k \bbv(\bbx,t_n) \hbbf_\bbn(\bbu^h,\bbu^{h,-})\; d\gamma}_{(II)}
\end{split}
\end{equation*}
since  
\begin{equation*}
\begin{split}
 \sum_{\sigma\in K}  \bbv_\sigma^n \Phi_{\sigma}^{K,Gal} (\bbu_h^n)&=
-\sum\limits_K \oint_K \nabla(\pi_h^k \bbv)(x,t_n) \; . \; \bbf^h(\bbu_h^n) \; d\bbx +  \oint_{\partial K}
\pi_h^k \bbv(x,t^n) \hbbf_\bbn(\bbu_h^n,\bbu^{h,-}) d\gamma.
\end{split}\end{equation*}

Moreover,  $\bbv$ is $C^{k+1}$ and the triangulation is regular:
 $\nabla (\pi_h^k \bbv)$ is uniformly bounded by a constant $C$ 
independent of $h$.
Thus we have
\begin{equation*}
\begin{split}
\bigg |\sum\limits_{n,K} \int_{K\times [t_n,t_{n+1}]} \nabla(\pi_h^k&
\bbv)(\bbx,t_n) \;
. \; \bbf^h(\bbu_h^n) \; d\bbx dt- \sum\limits_{n,K} \oint_{K\times [t_n,t_{n+1}]} \nabla \bbv \; . \; \bbf(\bbu) \; d\bbx dt \bigg |\\
&\leq  C  \sum \limits_{n,K}\oint_{K\times[t^n,t^{n+1}]} \Vert{\bbf^h(\bbu_h)-\bbf(\bbu)}\Vert \; d\bbx  dt \\
&\quad + \sum\limits_{n,K}\oint_{K\times[t^n,t^{n+1}]} \Vert{\nabla \pi_h^k\bbv-\nabla \bbv} \Vert\; \Vert{\bbf(\bbu_h)}\Vert\;\; d\bbx dt
 \end{split}\end{equation*}
The first sum is less than
$\Vert{\bbf^h(\bbu_h)-\bbf(\bbu)}\Vert_{L^1(\text{supp }\bbv \times [0,T])}\leq C_\bbf ||\bbu_h-\bbu||_{L^1(\text{supp }\bbv \times [0,T])}$
and tends to  $0$ because  $\Vert{u_h}\Vert_{\infty}$ is bounded independently of
 $h$, 
 $\bbf$ is  $C^1$ and $\bbu_h \rightarrow \bbu$ in  $L^2_{loc}$.

Similarly, since  $\bbu_h$ is bounded and  $\bbf$ is continuous, $\bbf(\bbu_h)$ is bounded
uniformly in $h$ by a constant $C$.
The second term of the right hand side sum is bounded by the $L^1$ norm of 
 $\nabla \pi_h^k\bbv-\nabla \bbv$ that
tends to $0$ since the triangulation is uniform.

Since the quadrature formula sastifies for any $w\in L^\infty(\Omega\times [0,T])$
$$\int_K w\, d\bbx-\oint_K w\; d\bbx=\vert K\vert \times o(1),$$
the term $(I)$ tends to $0$.

We discuss $(II)$. Since
\begin{equation}\label{rds_re:jump} \begin{split}\sum_{n,K} \int_{t^n}^{t^{n+1}}&\oint_{\partial K} \pi_h^k \bbv(\bbx,t_n)  \hbbf_\bbn(\bbu^h,\bbu^{h,-})\; d\gamma\\&
\qquad =
\sum_{n,e\in \EE} \int_{t^n}^{t^{n+1}}\oint_e\bigg [\pi_h^k \bbv(\bbx,t_n)\bigg]  \hbbf_\bbn(\bbu^h,\bbu^{h,-})\; d\gamma, \end{split}\end{equation}
where $[w]$ is the jump of $w$ across $e$, 
we get
\begin{equation*}\begin{split}\bigg \vert \sum_{n,K} \int_{t^n}^{t^{n+1}}\oint_{\partial K} \pi_h^k \bbv(\bbx,t_n) \hbbf_\bbn(\bbu^h,\bbu^{h,-})\; d\gamma\bigg \vert &\leq 
C \sum_{n,e\in \EE} \int_{t^n}^{t^{n+1}} \oint_{\partial K} \bigg \vert \bigg [  \pi_h^k \bbv(\bbx,t_n)\bigg ]\bigg\vert  \; d\gamma\\&\rightarrow 0,\end{split}\end{equation*}
see the remark \ref{rds_re:remark} below.
This ends the proof.
\end{proof}
\begin{rem}
\label{rds_re:remark}
In the case of a continuous representation, $\hbbf_\bbn(\bbu^h,\bbu^{h,-})=\bbf(u^h) \bbn$ by consistency and $\big [  \pi_h^k \bbv(\bbx,t_n)\big ]=0$ because $\pi_h^k\bbv$ is continuous. If one has a discontinuous approximation as for DG, then $\big [  \pi_h^k \bbv(\bbx,t_n)\big ]=O(h^{k+1})$ so that the sum of contributions on the edges scales like $O(h^{d-1}h^{k+1}h^{-d})=O(h^k)$ because there is $O(h^{-d})$ edges in a regular mesh.
This is the only difference between the continuous and the discontinuous case.
\end{rem}

The theorem \ref{rds_re:th:LW} is obtained by a simple combination of each of the above lemmas.

\subsection{Flux formulation}
\label{sec:rd:flux}
In this section we show that the scheme \eqref{eq:hyper} also admits a flux formulation, with an explicit form of the flux: the method is also locally conservative. Local conservation is of course  well known for the Finite Volume and discontinuous Galerkin  approximations. It is much less understood for the  continuous finite elements methods, despite the papers \cite{Hughes1,BurQS:10}.

Inspired by the RD formulation of finite volume schemes given in section \ref{sec:FV:1}, we show that any scheme of the form \eqref{forme schemas} can be equivalently rephrased as a finite volume scheme, we explicitly provide the flux formula as well as the control  volumes. In order to illustrate this result, we give several examples: the general RD scheme with $\P^1$ and $\P^2$ approximation on simplex, the case of a $\P^1$ RD scheme using a particular form of the residuals so that one can better see the connection with more standard formulations, and finally  an example with a discontinuous Galerkin formulation using $\P^1$  approximation.


Interpreting a scheme  as  a finite volume schemes amounts to defining control volumes and flux functions. 
We first have to adapt the notion of consistency.  The property that stands for the consistency is that if all the states  are identical in an element, then each of the residuals vanishes. Hence, we recall the definition of a multidimensional flux, see definition \ref{sec:rd:MD:consistency-mr}:
 A multidimensional flux 
$$\hbbf_\bbn:=\hbbf_\bbn(\bbu_1, \ldots , \bbu_N)$$
is consistent if, when $\bbu_1= \bbu_2= \ldots = \bbu_N=\bbu$ then
$$\hbbf_\bbn(\bbu, \ldots , \bbu)=\bbf(\bbu) \bbn.
$$

We proceed first with the general case and show the connection with elementary fact about graphs, and then provide several examples. The results of this section apply to any  finite element of finite volume 
method but also to discontinuous Galerkin methods. {An application in the Lagrangian formalism described in section \ref{sec:Lagrangian} 
will also be discussed.} There is no need for exact evaluation of integral formula (surface or boundary), so that these results apply to schemes as they are implemented.

\subsubsection{General case}
Let $K$ be an element of the primal mesh, we denote by $\mathcal{S}$ the set of degrees of freedom which are contained in $K$. By construction, some are on the boundary of $K$. We assume that we have constructed on $K$ a graph, denoted by $\mathcal{T}_K$, 
whose nodes are the elements of $\mathcal{S}$. This graph,
%
is assumed to be simply connected and oriented. We also assume that the vertices of the graph are $\mathcal{S}$.

The problem is to find quantities $\hbbf_{\sigma,\sigma'}$ for any edge $[\sigma,\sigma']$ of   $\mathcal{T}_K$ such that:
\begin{subequations}\label{sec:rd:GC:1}
\begin{equation}
\label{sec:rd:GC:1.1}
\Phi_\sigma=\sum_{\text{ edges }[\sigma,\sigma']} \hbbf_{\sigma,\sigma'}+\hbbf_\sigma^{b}
\end{equation}
and we impose
\begin{equation}
\hbbf_{\sigma,\sigma'}=-\hbbf_{\sigma',\sigma}.
\label{sec:rd:GC:1.2}
\end{equation}
The boundary flux  $\hbbf_\sigma^{b}$ is the 'part' of $\oint_{\partial K} \hbbf_\bbn(\bbu^h,\bbu^{h,-}) \; d\gamma$ associated to $\sigma$. {The control volumes will be defined by their normals so that we get consistency.} Concrete examples will be given later, for each of the case we have considered.

Note that \eqref{sec:rd:GC:1.2} implies that the boundary flux must satisfy the following relation
\begin{equation}
\label{sec:rd:GC:conservation}
\sum\limits_{\sigma\in K}\Phi_\sigma=\sum\limits_{\sigma\in K}\hbbf_\sigma^b.
\end{equation}
\end{subequations}

We say that an edge $[\sigma,\sigma']$ is direct if it is postively oriented.
Any edge $[\sigma,\sigma']$ is either direct or, if not, $[\sigma',\sigma]$ is direct. 
Because of \eqref{sec:rd:GC:1.2}, we only need to know $\hbbf_{\sigma,\sigma'}$ for direct edges. Thus we introduce the notation $\hbbf_{\{\sigma,\sigma'\}}$ for  the flux  assigned to  the direct edge whose extremities are $\sigma$ and $\sigma'$. 
For any $\sigma\in \mathcal{S}$,
we can rewrite \eqref{sec:rd:GC:1.1} as
\begin{equation}
\label{sec:rd:GC:1.1bis}
\sum_{\sigma'\in \mathcal{S}} \varepsilon_{\sigma,\sigma'} \hbbf_{\{\sigma,\sigma'\}}=\Psi_\sigma:=\Phi_\sigma-\hbbf_\sigma^b,
\end{equation}
with
$$
\varepsilon_{\sigma,\sigma'}=\left \{
\begin{array}{ll}
0& \text{ if }\sigma \text{ and }\sigma' \text{ are not on the same edge of }\mathcal{T}_K\\
1& \text{ if } [\sigma,\sigma']\text{ is an edge and } \sigma \rightarrow \sigma' \text{ is direct,}\\
-1&  \text{ if } [\sigma,\sigma']\text{ is an edge and } \sigma' \rightarrow \sigma \text{ is direct.}
\end{array}
\right .
$$
$\mathcal{E}^+$ represents the set of direct edges.
Of course we have
\begin{equation}\label{eq:equilibre flux}
    \sum_{\sigma\in K} \Psi_\sigma=\sum_{\sigma\in K,\sigma'\in K}\hbbf_{\sigma,\sigma'}=0
\end{equation}
which is the generalisation of \eqref{eq:point-fv4} of section \ref{sec:motivating:FV:II}.

Hence the problem is to find  a vector $\hbbf=(\hbbf_{\{\sigma,\sigma'\}})_{\{\sigma,\sigma'\} \text{ direct edges}}$ such that
$$\bbA\hbbf=\Psi,$$
where $\Psi=(\Psi_\sigma)_{\sigma\in \mathcal{S}}$ and $A_{\sigma \sigma'}=\varepsilon_{\sigma,\sigma'}$.

We have  the following lemma which shows the existence of a solution.
\begin{lemma}\label{sec:rd:lemma:flux}
For any couple $\{\Phi_\sigma\}_{\sigma\in \mathcal{S}}$ and $\{\hbbf_\sigma^{b}\}_{\sigma\in \mathcal{S}}$ satisfying the condition  \eqref{sec:rd:GC:conservation}, there exists numerical flux functions $\hbbf_{\sigma,\sigma'}$ that satisfy \eqref{sec:rd:GC:1}. Recalling that the  matrix of the Laplacian of the graph is $\bbL=\bbA\bbA^{\mathtt{T}}$, we have
\begin{enumerate}
\item The rank of $\bbL$ is $|\mathcal{S}|-1$ and its image is $\big (\text{span}\{\mathbf{1}\})^\bot$. We still denote the inverse of $\bbL$ on $\big (\text{span}\{\mathbf{1}\} )^\bot$ by $\bbL^{-1}$,
\item 
With the previous notations, a solution is 
\begin{equation}
\label{sec:rd:eq:lemma}\big (\hbbf_{\{\sigma,\sigma'\}}\big )_{\{\sigma,\sigma'\} \text{ direct edges}}=\bbA^{\mathtt{T}}\bbL^{-1} \big (\Psi_\sigma\big )_{\sigma\in \mathcal{S}}.\end{equation}
\end{enumerate}
\end{lemma}
\begin{proof}
We first have $\mathbf{1}^{\mathtt{T}}\,\bbA=0$: $\text{ Im } \bbA\subset \big (\text{span }\{1\}\big )^\bot (\subset \R^{|\mathcal{S}|})$. Let us show that we have equality. 
In order to show this, we notice that  the matrix $\bbA$ is nothing more that the incidence matrix of the oriented graph $\mathcal{G}$ defined by the triangulation $\mathcal{T}_K$. 
It is known \cite{graph} that its null space of $\bbL$  is equal to the number of connected  components of the graph, i.e. here $\dim \ker \bbL=1$. Since 
$$ \bbL\,\mathbf{1}=0,$$ we see that $\ker \bbL=\text{span }\{\mathbf{1}\}$, so that $\text{ Im }\bbL= \big (\text{span }\{\mathbf{1}\}\big )^{\mathtt{T}}$ because $\bbL$ is symmetric. We can define the inverse of $\bbL$ on $\text{Im }\bbL$,  denoted by $\bbL^{-1}$. 

Let $x\in \big (\text{span }\{\mathbf{1}\}\big )^\bot=\text{ Im }\bbL$. There exists $y\in \R^{|S|}$ such that $x=\bbL y=\bbA(\bbA^{\mathtt{T}}y)$: this shows that $x\in \text{ Im }\bbA$ and thus
$\text{ Im }\bbA=\big (\text{span }\{\mathbf{1}\}\big )^\bot=\big (\text{Im }\bbL\big )^\bot$.  From this we deduce that $\text{rank }\bbA=|\mathcal{S}|-1$ because $\text{ Im } \bbA\subset \R^{|\mathcal{S}|}$.

Let $\Psi\in \R^{|\mathcal{S}|}$ be such that $\mathbf{1}^{\mathtt{T}}\Psi =0$. 
We know there exists a unique $z\in \big (\text{span }\{\mathbf{1}\}\big )^\bot$ such that $\bbL z=\Psi$, i.e.
$$\bbA(\bbA^{\mathtt{T}}z)=\Psi.$$
This shows that a solution is given by \eqref{sec:rd:eq:lemma}.
\end{proof}

This set of flux are consistent and we can estimate the normals $\bbn_{\sigma,\sigma'}$.
In the case of a constant state, we have $\Phi_\sigma=0$ for all $\sigma\in K$. Let us assume that
\begin{equation}
\label{sec:rd:GC:consistency}
\hbbf_\sigma^b=\bbf(\bbu^h) \mathbf{N}_\sigma
\end{equation}
with $\sum\limits_{\sigma\in K} \mathbf{N}_\sigma=0$: this is the case for all the examples we consider. The flux $\bbf(\bbu^h)$ has components on the canonical basis of $\R^d$:
$\bbf(\bbu^h)=\big (f_1(\bbu^h), \ldots , f_d(\bbu^h)\big )$, so that
$$\hbbf_\sigma^b=\sum\limits_{i=1}^d f_i(\bbu^h)\mathbf{N}^i_\sigma.$$
Applying this to $\big (\hbbf_{\sigma_1}^b, \ldots, \hbbf_{\sigma_{\#K}}^b\big )$, we see that the $j$-th component of $\bbn_{\sigma,\sigma'}$ for $[\sigma,\sigma']$ direct, must satisfy:
$$\text{ for any }\sigma\in K, \; \mathbf{N}^j_\sigma=\sum\limits_{[\sigma,\sigma']\text{ edge }}\varepsilon_{\sigma,\sigma'}\bbn_{\sigma,\sigma'}^j$$
i.e.
$$\big ( \mathbf{N}^j_{\sigma_1}, \ldots , \mathbf{N}^j_{\sigma_{\#K}}\big )^{\mathtt{T}}= \bbA \; \big ( \bbn_{\sigma,\sigma'}^j\big )_{[\sigma,\sigma']\in \mathcal{E}^+}.$$
We can solve the system and the solution, with some abuse of language, is
\begin{equation}
\label{sec:rd:GC:normals}
\big ( \bbn_{\sigma,\sigma'}\big )_{[\sigma,\sigma']\in \mathcal{E}^+}=\bbA^{\mathtt{T}}\bbL^{-1} \big ( \mathbf{N}_{\sigma_1}, \ldots , \mathbf{N}_{\sigma_{\#K}}\big )^{\mathtt{T}}
\end{equation}
This also defines the control volumes since we know their normals. We can state:
\begin{proposition}
If the residuals $(\Phi_\sigma)_{\sigma\in K}$ and the boundary fluxes $(\hbbf_\sigma^b)_{\sigma\in K}$ satisfy \eqref{sec:rd:GC:conservation}, and if the boundary fluxes satisfy the consistency relations \eqref{sec:rd:GC:consistency}, then we can find  a set of consistent flux $(\hbbf_{\sigma,\sigma'})_{[\sigma,\sigma']} $ satisfying \eqref{sec:rd:GC:1}. They are given by \eqref{sec:rd:eq:lemma}. In addition, for a constant state,
$$\hbbf_{\sigma,\sigma'}(\bbu^h)=\bbf(\bbu^h)\bbn_{\sigma,\sigma'}$$ for the normals defined by \eqref{sec:rd:GC:normals}.
\end{proposition}

\bigskip

We can state a couple of general remarks:
\begin{rem}
\begin{enumerate} 
\item The flux $\hbbf_{\sigma,\sigma'}$ depend on the $\Psi_\sigma$ and not directly on the $\hbbf_\sigma^b$. We can design the fluxes independently of the boundary flux, and their consistency directly comes from the consistency of the boundary fluxes.
\item 
The residuals depends on more than 2 arguments. For stabilized finite element methods, or the non linear stable residual distribution
 schemes, see e.g.  \cite{Hughes1,struijs,abgrallLarat}, the residuals depend on all the states on  $K$. Thus
the formula \eqref{sec:rd:eq:lemma} shows that the flux depends on more than two states in contrast to the  1D case. In the finite volume case however, the support of the flux function is generally larger than the three states of $K$, think for example of an ENO/WENO method, or a simpler MUSCL one.
\item The numerical fluxes from formula \eqref{sec:rd:eq:lemma} are influenced by the form of the total residual \eqref{rds_re:conservation}. In particular, if the residuals are defined from "standard" flux (like HLL, or Roe's), there is no reason why the flux computed here will coincide. This will be the case if the interior triangulation $\mathcal{T}_K$ is the same as that has been used to define the residual, and if the boundary flux are chosen accordingly.
\item The formula \eqref{sec:rd:eq:lemma} makes no assumption on the approximation space $V^h$: 
the flux formulation is 
valid for continuous and discontinuous approximations. The structure of the approximation space appears only in the total residual.
\end{enumerate}

\end{rem}

\subsubsection{Some particular cases: fully explicit formulas}\label{sec:rd:particular}
Let $K$ be a fixed triangle. We are given a set of residues $\{\Phi_\sigma^K\}_{\sigma\in K}$, our aim here is to define a
 flux function. 
 
 We start by the continuous finite element case 
for $\P^1$ and $\P^2$ interpolant. The the element $K$ is any simplex of the mesh. We consider the Lagrange degrees of freedom, and the graph that connect them, see figures \ref{sec:rd:fig:P2} for the $\P^2$ case, the graph for the $\P^1$ case being the 3 vertices of $K$.

\paragraph{The general  $\P^1$ case.}

The boundary flux are defined as
$$\hbbf_\sigma^b=-\bbf(\bbu_\sigma) \frac{\bbn_\sigma}{2}, $$
where $\bbn_{\bsigma}$ is the inward scaled normal opposite to the vertex $\bsigma$. This comes from left figure of figure \ref{fig:fv}, where, by analogy with the finite volume case, we want that the residual $\Phi_{\bsigma}^K$ be  the sum of the normal flux to the sub-cell $\bm{c}$ equal to $1IGK$, $2JGI$ or $3KGJ$. Hence, the boundary flux for $\bsigma$ will be the sum of $\bbf(\bbu_{\bsigma})$ against the outward normal to $\bm{c}$ contained in the boundary of $K$.

The adjacent matrix is 
$$\bbA=\left (\begin{array}{rrr} 1&0&-1\\
-1&1&0\\
0&-1&1
\end{array}\right ).
$$
A straightforward calculation shows that the matrix $L=A^{\mathtt{T}}A$ has eigenvalues $0$ and $3$ with multiplicity 2 with eigenvectors
$$\bbR=\begin{pmatrix}
\frac{1}{\sqrt{3}} & \frac{1}{\sqrt{2}} & \frac{1}{\sqrt{6}}\\
\frac{1}{\sqrt{3}} & \frac{-1}{\sqrt{2}} & \frac{1}{\sqrt{6}}\\
\frac{1}{\sqrt{3}} & 0                         & \frac{-2}{\sqrt{6}}
\end{pmatrix}
$$
To solve $A\hbbf=\Psi$, we decompose $\Psi$ on the eigen-basis:
$$\Psi=\alpha_2 R_2+\alpha_3R_3,$$ where explicitly
$$\begin{array}{l}
\alpha_2=\frac{1}{\sqrt{2}} \big (\Psi_1-2\Psi_2+\Psi_3\big )\\
\\
\alpha_3=\sqrt{\frac{3}{2}}\big (\Psi_1-\Psi_3\big )
\end{array}
$$
so that 
$$\hbbf=\frac{1}{3}\begin{pmatrix}\Psi_1-\Psi_3 \\ \Psi_2-\Psi_3 \\ \Psi_3-\Psi_2 \end{pmatrix}.$$

In order to describe the control volumes, we first have to make precise the normals $\bbn_\sigma$ in that case. It is easy to see that in all the cases described above, we have 
$$\mathbf{N}_\sigma=-\frac{\bbn_\sigma}{2}.$$ Then a short calculation shows that
$$\begin{pmatrix}
\bbn_{12} \\ \bbn_{23} \\ \bbn_{31} \end{pmatrix}=
\frac{1}{6}\begin{pmatrix} \bbn_1-\bbn_2 \\ \bbn_2 -\bbn_3 \\ \bbn_3 -\bbn_1 \end{pmatrix}.
$$
Using elementary geometry of the triangle, we see that these  are the normals of the elements of the dual mesh. For example, the normal $\bbn_{12}$ is the normal of 
 $PG$, see figure \ref{fig:fv}.
 
 Relying more on the geometrical interpretation (once we know the control volumes), we also can recover the same formula by elementary calculations.

\paragraph{The general example of the $\P^2$ approximation. }
Using a similar method, we get (see figure \ref{sec:rd:fig:P2} for some notations):
$$
\begin{array}{lcl}
\hbbf_{14}&=&\dfrac{1}{12}\big (\Psi_1-\Psi_4\big )+\dfrac{1}{36}\big ( \Psi_6-\Psi_5\big )+\dfrac {7}{36}\big (\Psi_1- \Psi_2\big )+\dfrac {5 }{36}\big (\Psi_3-\Psi_1\big )\\
&\\
\hbbf_{16}&=&\dfrac{1}{12}\big ( \Psi_4-\Psi_1\big )+\dfrac {5}{36}\big ( \Psi_5-\Psi_1)
+\dfrac {7}{36}\big ( \Psi_6-\Psi_1\big ) +\dfrac{1}{36}\big ( \Psi_3- \Psi_2\big )\\
&\\
\hbbf_{46}&=&\dfrac{2}{9}\big (\Psi_2-\Psi_6\big )+\dfrac{1}{9}\big (  \Psi_3- \Psi_5\big )\\
&\\
\hbbf_{54}&=&\dfrac{2}{9}\big (  \Psi_5-\Psi_2\big )+\dfrac{1}{9}\big ( \Psi_5-\Psi_1\big )\\
\end{array}
$$
$$
\begin{array}{lcl}
\hbbf_{42}&=&\dfrac {7}{36}\big (\Psi_2-\Psi_3\big ) +\dfrac{5}{36}\big (\Psi_1-\Psi_3\big )+\dfrac{1}{12}\big(\Psi_6-\Psi_3\big )+\dfrac{1}{36}\big (\Psi_5-\Psi_4\big ) \\

&\\
\hbbf_{25}&=&\dfrac{1}{36}\big (\Psi_2-\Psi_1\big )+\dfrac{5}{36}\big (\Psi_3-\Psi_5\big ) +\dfrac{7}{36}\big ( \Psi_3-\Psi_5\big )+\dfrac{1}{12}\big (\Psi_3-\Psi_6\big )  \\

&\\
\hbbf_{53}&=&\dfrac{1}{36}\big (\Psi_1-\Psi_6\big )+\dfrac{5}{36}\big (\Psi_3-\Psi_5\big )+\dfrac{7}{36}\big (\Psi_4-\Psi_5\big )+\dfrac{1}{12}\big (\Psi_2-\Psi_5\big )
\\
&\\
\hbbf_{63}&=& \dfrac{1}{36}\big (\Psi_4-\Psi_3\big )+\dfrac{5}{36}\big (\Psi_5-\Psi_1\big )+\dfrac{7}{36}\big (\Psi_5-\Psi_6\big )+\dfrac{1}{12}\big (\Psi_5-\Psi_2\big )\\
&\\

\hbbf_{65}&=&\dfrac{1}{9}\big (\Psi_1- \Psi_3\big )+\dfrac{2}{9}\big ( \Psi_6- \Psi_4\big )\end {array}
$$
Then we choose the boundary flux:
$$\hbbf_\sigma^b=\int_{\partial K}\varphi_\sigma\bbn\; d\gamma$$ and get:
$$
\begin{array}{lll}
\mathbf{N}_l=-\dfrac{\bbn_l}{6} & \text{if }l=1,2,3\\ &&\\
\mathbf{N}_4=\dfrac{\bbn_3}{3}& \mathbf{N}_5=\dfrac{\bbn_1}{3}& \mathbf{N}_6=\dfrac{\bbn_2}{3}
\end{array}
$$
The normals are given by:
$$
\begin{array}{lcl}
\bbn_{14}&=&\dfrac{1}{12}\big (\bbn_1-\bbn_4\big )+\dfrac{1}{36}\big ( \bbn_6-\bbn_5\big )+\dfrac {7}{36}\big (\bbn_1- \bbn_2\big )+\dfrac {5 }{36}\big (\bbn_3-\bbn_1\big )\\
&\\
\bbn_{16}&=&\dfrac{1}{12}\big ( \bbn_4-\bbn_1\big )+\dfrac {5}{36}\big ( \bbn_5-\bbn_1)
+\dfrac {7}{36}\big ( \bbn_6-\bbn_1\big ) +\dfrac{1}{36}\big ( \bbn_3- \bbn_2\big )\\
&\\

\bbn_{46}&=&\dfrac{2}{9}\big (\bbn_2-\bbn_6\big )+\dfrac{1}{9}\big (  \bbn_3- \bbn_5\big )\\
&\\
\bbn_{54}&=&\dfrac{2}{9}\big (  \bbn_5-\bbn_2\big )+\dfrac{1}{9}\big ( \bbn_5-\bbn_1\big )
\end{array}
$$
$$
\begin{array}{lcl}
\bbn_{42}&=&\dfrac {7}{36}\big (\bbn_2-\bbn_3\big ) +\dfrac{5}{36}\big (\bbn_1-\bbn_3\big )+\dfrac{1}{12}\big(\bbn_6-\bbn_3\big )+\dfrac{1}{36}\big (\bbn_5-\bbn_4\big ) \\

&\\
\bbn_{25}&=&\dfrac{1}{36}\big (\bbn_2-\bbn_1\big )+\dfrac{5}{36}\big (\bbn_3-\bbn_5\big ) +\dfrac{7}{36}\big ( \bbn_3-\bbn_5\big )+\dfrac{1}{12}\big (\bbn_3-\bbn_6\big )  \\

&\\
\bbn_{53}&=&\dfrac{1}{36}\big (\bbn_1-\bbn_6\big )+\dfrac{5}{36}\big (\bbn_3-\bbn_5\big )+\dfrac{7}{36}\big (\bbn_4-\bbn_5\big )+\dfrac{1}{12}\big (\bbn_2-\bbn_5\big )
\\
&\\
\bbn_{63}&=& \dfrac{1}{36}\big (\bbn_4-\bbn_3\big )+\dfrac{5}{36}\big (\bbn_5-\bbn_1\big )+\dfrac{7}{36}\big (\bbn_5-\bbn_6\big )+\dfrac{1}{12}\big (\bbn_5-\bbn_2\big )\\
&\\

\bbn_{65}&=&\dfrac{1}{9}\big (\bbn_1- \bbn_3\big )+\dfrac{2}{9}\big ( \bbn_6- \bbn_4\big )\end {array}
$$

\bigskip

There is not uniqueness, and it is possible to construct different solutions to the problem, simply because it is possible to construct different graphs. In what follows, we show another possible construction. 
We consider the set-up defined by Figure \ref{sec:rd:fig:P2}.
\begin{figure}[h!]
\begin{center}
\includegraphics[width=0.45\textwidth]{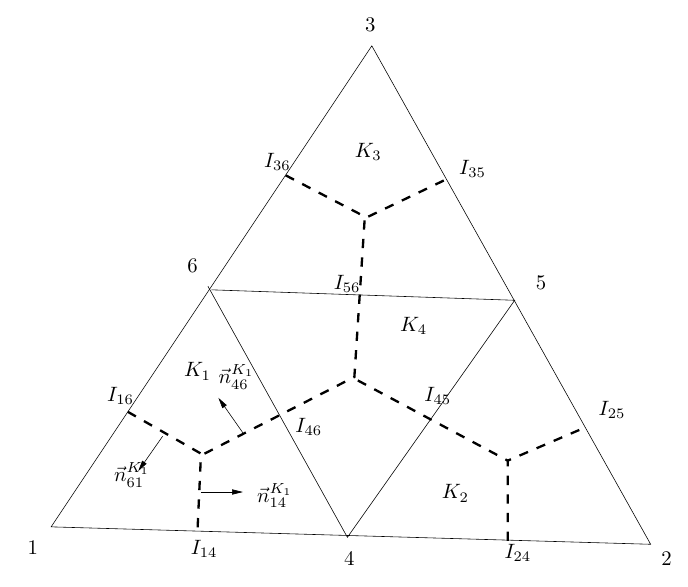}
\end{center}
\caption{\label{sec:rd:fig:P2} Geometrical elements for the $\P^2$ case. $I_{ij}$ is the mid-point between the vertices $i$ and $j$. The intersections of the dotted lines are the centroids of the sub-elements.}
\end{figure}
The triangle is split first into 4 sub-triangles $K_1$, $K_2$, $K_3$ and $K_4$. From this sub-triangulation,
 we can construct a dual mesh as in the $\P^1$ case and we have represented the  6 sub-zones
 that are the intersection of the dual control volumes and the triangle $K$.  Our notations are as follow: given any sub-triangle $K_\xi$,
 if $\gamma_{ij}$ is intersection between two adjacent control volumes (associated to $\sigma_i$ and $\sigma_j$ vertices of $K_\xi$), the normal
 to $\gamma_{ij}$ in the direction $\sigma_i$ to $\sigma_j$ is denoted by $\bbn_{ij}^{\xi}$. Similarly the flux across $\gamma_{ij}$ is denoted 
$\hbbf_{ij}^\xi$.

Then we need to define boundary fluxes.  If $\sigma$ belongs to $K_l$, we denote the boundary flux as $\hbbf_\sigma^{K_l}$. A rather  natural condition is that
$$\begin{array}{ll}
\hbbf_l^{K_l}=\hbbf_l^b & l=1,2,3\\
\hbbf_4^{K_l}=\frac{1}{3}\hbbf_4^b & l=1,2,4\\
\hbbf_5^{K_l}=\frac{1}{3}\hbbf_5^b & l=2,3,4\\
\hbbf_6^{K_l}=\frac{1}{3}\hbbf_6^b & l=1,3,4.\\
\end{array}
$$
We recover the conservation relation. Other choices are possible since this one is arbitrary: the only true condition is that the sum of the boundary flux is equal to the sum  of the $\hbbf_j^b$ for $j=1, \ldots, 6$: this is the conservation relation.

Then we set:
\begin{equation}
\label{sec:rd:rd:fv:P2}
\begin{array}{lll} 
\Phi_1&= -\hbbf_{\bbn_{61}}^1+\hbbf_{\bbn_{14}}^1&+\hbbf_1^b\\
\Phi_2&=-\hbbf_{\bbn_{42}}^2+\hbbf_{\bbn_{25}}^2&+\hbbf_2^b\\
\Phi_3&=-\hbbf_{\bbn_{53}}^3+\hbbf_{\bbn_{36}}^3&+\hbbf_3^b\\
\Phi_4&=-\hbbf_{\bbn_{14}}^1+\big (\hbbf_{\bbn_{46}}^1-\hbbf_{\bbn_{64}}^4\big )+\big (\hbbf_{\bbn_{45}}^4-\hbbf_{\bbn_{54}}^2\big )+\hbbf_{\bbn_{42}}^2&+\hbbf_4^b\\
\Phi_5&= -\hbbf_{\bbn_{25}}^2+\big ( \hbbf_{\bbn_{54}}^2-\hbbf_{\bbn_{45}}^4\big ) + \big ( \hbbf_{\bbn_{56}}^4-\hbbf_{\bbn_{65}}^3\big ) + \hbbf_{\bbn_{53}}^3&+\hbbf_5^b\\
\Phi_6&=-\hbbf_{\bbn_{36}}^3+\big (\hbbf_{\bbn_{65}}^3-\hbbf_{\bbn_{56}}^4\big )+\big (\hbbf_{\bbn_{64}}^4-\hbbf_{\bbn_{46}}^1\big )+\hbbf_{\bbn_{61}}^1&+\hbbf_6^b
\end{array}
\end{equation}
We can group the terms in \eqref{sec:rd:rd:fv:P2}  by sub-triangles, namely:
\begin{equation}
\label{sec:rd:rd:fv:P2:2}
\begin{array}{lclclcl}
\Phi_1&=& \big (-\hbbf_{\bbn_{61}}^1+\hbbf_{\bbn_{14}}^1+\hbbf_1^b \big )& \\
\Phi_2&=&\big (-\hbbf_{\bbn_{42}}^2+\hbbf_{\bbn_{25}}^2+\hbbf_2^b \big )&\\
\Phi_3&=&\big (-\hbbf_{\bbn_{53}}^3+\hbbf_{\bbn_{36}}^3+\hbbf_3^b\big )&\\
\Phi_4&=&\big (-\hbbf_{\bbn_{14}}^1+\hbbf_{\bbn_{46}}^1+\hbbf_4^{K_1}\big ) &+&\big ( -\hbbf_{\bbn_{64}}^4+\hbbf_{\bbn_{45}}^4+\hbbf_1^{K_4} \big )\\
&&&+&\big (-\hbbf_{\bbn_{54}}^2+\hbbf_{\bbn_{42}}^2+\hbbf_4^{K_2} \big )\\
\Phi_5&=&\big (-\hbbf_{\bbn_{25}}^2+\hbbf_{\bbn_{54}}^2+\hbbf_5^{K_2} \big)&
+&\big(-\hbbf_{\bbn_{45}}^4+\hbbf_{\bbn_{56}}^4+\hbbf_5^{K_4}\big )\\
&&& +&\big (-\hbbf_{\bbn_{65}}^3+\hbbf_{\bbn_{53}}^3+\hbbf_5^{K_3}\big )\\
\Phi_6&=&\big (-\hbbf_{\bbn_{36}}^3+\hbbf_{\bbn_{65}}^3+\hbbf_6^{K_3} \big )& 
+&\big (-\hbbf_{\bbn_{56}}^4+\hbbf_{\bbn_{64}}^4+\hbbf_6^{K_4} \big )\\
&&&+&\big ( -\hbbf_{\bbn_{46}}^1+\hbbf_{\bbn_{61}}^1+\hbbf_6^{K_1}\big ).
\end{array}
\end{equation}

Then we define the sub-residuals per sub elements:
\begin{equation}
\label{sec:rd:rd:fv:P2:3}
\begin{split}
\Phi_1^1=-\hbbf_{\bbn_{61}}^1+\hbbf_{\bbn_{14}}^1+\hbbf_1^{b_{\phantom{1}}}  &,\qquad\Phi_4^2=-\hbbf_{\bbn_{54}}^2+\hbbf_{\bbn_{42}}^2+\hbbf_4^{K_2}\\
\Phi_4^1=-\hbbf_{\bbn_{14}}^1+\hbbf_{\bbn_{46}}^1+\hbbf_4^{K_1} &,\qquad\Phi_2^2=-\hbbf_{\bbn_{42}}^2+\hbbf_{\bbn_{25}}^2+\hbbf_2^{K_2}\\
\Phi_6^1=-\hbbf_{\bbn_{46}}^1+\hbbf_{\bbn_{61}}^1+\hbbf_6^{K_1}  &,\qquad\Phi_5^2=-\hbbf_{\bbn_{25}}^2+\hbbf_{\bbn_{54}}^2+\hbbf_5^{K_2}\\
&\\
\Phi_5^3=-\hbbf_{\bbn_{65}}^3+\hbbf_{\bbn_{53}}^3+\hbbf_5^{K_3}  &,\qquad\Phi_4^4=-\hbbf_{\bbn_{64}}^4+\hbbf_{\bbn_{45}}^4+\hbbf_4^{K_4}\\
\Phi_3^3=-\hbbf_{\bbn_{36}}^3+\hbbf_{\bbn_{65}}^3+\hbbf_3^{K_3}  &,\qquad\Phi_5^4=-\hbbf_{\bbn_{45}}^4+\hbbf_{\bbn_{56}}^4+\hbbf_5^{K_4}\\
\Phi_6^3=-\hbbf_{\bbn_{36}}^3+\hbbf_{\bbn_{65}}^3+\hbbf_6^{K_3}  &,\qquad\Phi_6^4=-\hbbf_{\bbn_{56}}^4+\hbbf_{\bbn_{64}}^4+\hbbf_6^{K_4},
\end{split}
\end{equation}
so we are back to the $\P^1$ case: in each sub-triangle, we can define flux that will depend on the 6 states of the element via the boundary flux.
This is legitimate because in the $\P^1$ case,  we have not used the fact that the interpolation is linear, we have only used the fact that we have 3 vertices. Clearly the fluxes are consistent in the sense of definition \ref{sec:rd:MD:consistency-mr}.

The same argument can be clearly extended to higher degree element, as well as to non triangular element: what is needed is to subdivide the element into 
sub-triangles. 

The two solutions we have presented for the $\P^2$ case are different: the control volumes are different, since they have more sides in the second case than in the first one.

\paragraph{Discontinuous Galerkin schemes.} $ $\\
We show two examples, that of a triangular element and $\P^1$ approximation, and that of  a quadrangle and $\P^2$ approximation.

In the $\P^1$ case, the flux between two DOFs $\sigma$ and $\sigma'$ is given by
$$\hbbf_{\sigma,\sigma'}(\bbu^h,\bbu^{h,-})=\oint_{\partial K}(\varphi_\sigma-\varphi_{\sigma'})\hbbf_\bbn(\bbu^h,\bbu^{h,-})\; d\gamma-\oint_K \nabla\big (\varphi_\sigma-\varphi_{\sigma'}\big ) \bbf(\bbu^h)\; d\bbx.$$
Again, from simple geometry, $$\nabla\big (\varphi_\sigma-\varphi_{\sigma'}\big )=-\frac{\bbn_{\sigma\sigma'}}{|K|},$$ so that
\begin{equation*}
\hbbf_{\sigma,\sigma'}(\bbu^h,\bbu^{h,-})=\oint_{\partial K}(\varphi_\sigma-\varphi_{\sigma'})\hbbf_\bbn(\bbu^h,\bbu^{h,-})\; d\gamma+\frac{\oint_K  \bbf(\bbu^h)\; d\bbx}{|K|}\bbn_{\sigma\sigma'}.
\end{equation*}
Note that $\oint_{\partial K}(\varphi_\sigma-\varphi_{\sigma'})\; d\gamma=0$ if we take the same quadrature formula on each edge, as it is usually done. Hence, denoting by $\bar \bbu$ the cell average of $\bbu^h$ on $K$, we can rewrite the flux as 
\begin{equation}
\label{sec:rd:flux:DG}
\hbbf_{\sigma,\sigma'}(\bbu^h,\bbu^{h,-})=\frac{\oint_K  \bbf(\bbu^h)\; d\bbx}{|K|}\bbn_{\sigma\sigma'}+ \oint_{\partial K}(\varphi_\sigma-\varphi_{\sigma'})\big ( \hbbf_\bbn(\bbu^h,\bbu^{h,-})-\bbf(\bar\bbu)\bbn\big )\; d\gamma
\end{equation}
so that the second term can be interpreted as a dissipation. The control volume is depicted in figure  \ref{sec:rd:control:DG}.
\begin{figure}[ht]
\begin{center}
\includegraphics[width=0.5\textwidth]{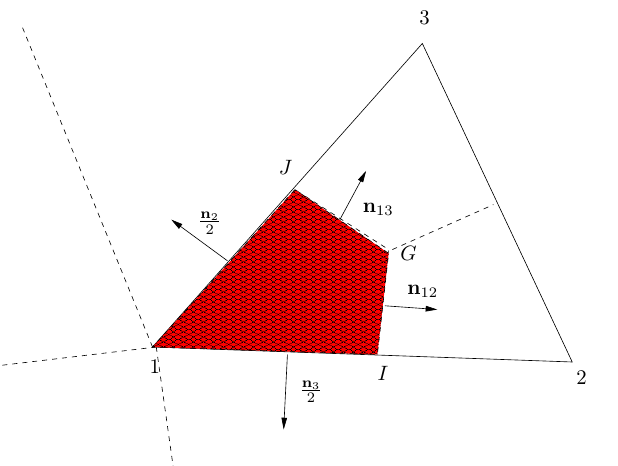}
\end{center}
\caption{\label{sec:rd:control:DG} Representation of the control volume associated to DOF $1$.}
\end{figure}
Referring to figure \ref{sec:rd:control:DG} for the DOF $\# 1$, the flux on the faces $I1J$ is $\oint_{\partial K}\varphi_1 \hbbf_\bbn(\bbu^h,\bbu^{h,-})\; d\gamma$. In order to respect some geometrical assignment, the flux on $1I$ is set to
$$\oint_{1I}\varphi_1 \hbbf_\bbn(\bbu^h,\bbu^{h,-})\; d\gamma$$
and on $1J$,
$$\oint_{1J}\varphi_1 \hbbf_\bbn(\bbu^h,\bbu^{h,-})\; d\gamma.$$

 \bigskip
 Let us give the example of a third order dG scheme on a quad, see figure \ref{dG3}

\begin{figure}[ht]
\begin{center}
\includegraphics[width=0.45\textwidth]{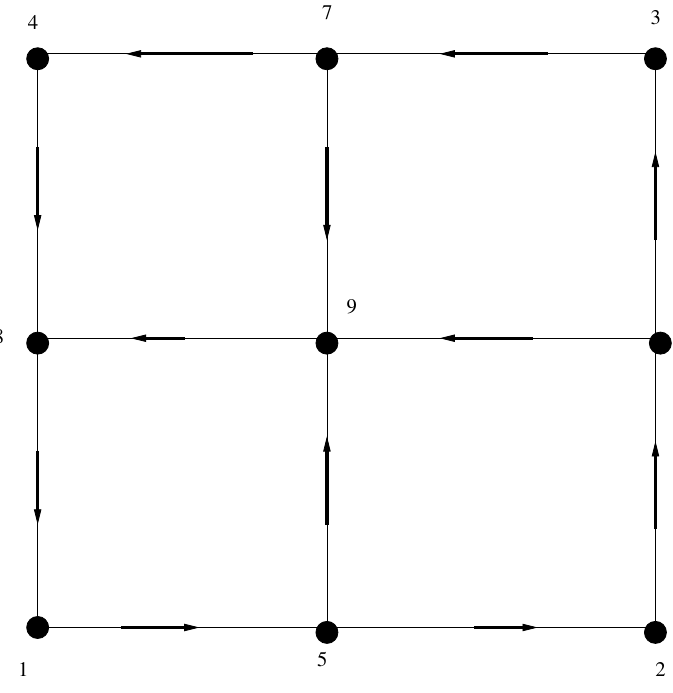}
\end{center}
\caption{\label{dG3} Incidence graph and numbering of the degrees of freedom.}
\end{figure}
Again we set 
$\hbbf_i^{b}=\int_{\partial K} \varphi_i\hat{f}_{\mathbf{n}} d\gamma$.
and 
set $\Psi_\sigma=\Phi_\sigma-f_\sigma^{B}$, so that
\begin{equation*}
\begin{split}
\Psi_1=\hbbf_{15}-\hbbf_{81}, & \quad \Psi_2=-\hbbf_{52}-\hbbf_{26}, \quad \Psi_3=-\hbbf_{63}+\hbbf_{37}\\
\Psi_4=-\hbbf_{74}+\hbbf_{48}&\quad
\Psi_5=-\hbbf_{15}+\hbbf_{52}+\hbbf_{59}, \quad
\Psi_6=-\hbbf_{26}+\hbbf_{63}+\hbbf_{69}\\
\Psi_7=-\hbbf_{37}+\hbbf_{74}+\hbbf_{79}&\quad
\Psi_8=-\hbbf_{48}-\hbbf_{98}+\hbbf_{81}, \quad
\Psi_9=-\hbbf_{59}-\hbbf_{69}+\hbbf_{79}+\hbbf_{98}
\end{split}
\end{equation*}
The incidence matrice is 
 $$
\bbA
=
\left ( \begin{array}{rrrrrrrrrrrr}
1&0 &0&0&0&0&0&-1&0&0&0&0\\
0&-1&1&0&0&0&0&0&0 &0&0&0\\
0& 0&0&-1&1&0&0&0&0 &0&0&0\\
0&0&0&0&0&-1&1&0&0 &0&0&0\\
-1&1&0&0&0&0&0&0&0 &0&0&1\\
0&0&-1&1&0&0&0&0&0 &0&1&0\\
0&0&0&0&-1&1&0&0&1 &0&0&0\\
0&0&0&0&0&0&-1&1&0 &-1&0&0\\
0&0&0&0&0&0&0&0&-1&1&-1&-1\\
\end{array}\right )
$$
The unknowns flux and  known $\Psi$ are
$$f=\big ( \hbbf_{15}, \hbbf_{52}, \hbbf_{26} , \hbbf_{63}, \hbbf_{37}, \hbbf_{74} , \hbbf_{48} ,\hbbf_{81} ,\hbbf_{79} , \hbbf_{98} , \hbbf_{69} ,\hbbf_{59}\big )^{\mathtt{T}}, \qquad 
\Psi=\big (\Psi_{1},\Psi_2, \ldots ,\Psi_9 \big )^{\mathtt{T}}, \Psi_\sigma=\Phi_\sigma-f_\sigma^{B}$$
so that $$\Psi=\bbA\;f$$
The Graph Laplacian is
$$\bbL=\bbA \bbA^{\mathtt{T}}
=\left (\begin{array}{ccccccccc} 
2&0&0&0&-1&0&0&-1&0\\
0&2&0&0&-1&-1&0&0&0\\
0&0&2&0&0&-1&-1&0&0\\
0&0&0&2&0&0&-1&-1&0\\
-1&-1&0&0&3&0&0&0&-1\\
0&-1&-1&0&0&3&0&0&-1\\
-1&0&0&-1&0&0&0&3&-1\\
0&0&0&0&-1&-1&-1&-1&4
\end{array}
\right )
$$
We know that 
$$\ker \bbL=\text{ span }\{\mathbf{x}_0=(1,1,1,1,1,\ldots , 1)^{\mathtt{T}}\}\text{ so that }\R^n=\ker \bbL\oplus \big (\ker \bbL\big )^\bot.$$
For any $\alpha>0$ we observe that
$$\bbL= \big ( \bbL+(\alpha-1)\dfrac{\mathbf{x}_0\otimes\mathbf{x}_0}{n}\big ) +\alpha \dfrac{\mathbf{x}_0\otimes\mathbf{x}_0}{n}$$
with $$\bbL+(\alpha-1)\dfrac{\mathbf{x}_0\otimes\mathbf{x}_0}{n}$$ invertible.
From this we get {$$\bbL^{-1}= \big ( \bbL+(\alpha-1)\dfrac{\mathbf{x}_0\otimes\mathbf{x}_0}{n}\big )^{-1} +\dfrac{1}{\alpha }\dfrac{\mathbf{x}_0\otimes\mathbf{x}_0}{n}$$}

We get 
$$\bbL^{-1}=\left( \begin {array}{ccccccccc} {\frac {13}{24}}&\frac{-1}{12}&-{\frac {5}{
24}}&\frac{-1}{12}&{\frac {7}{72}}&-{\frac {11}{72}}&-{\frac {11}{72}}&{\frac 
{7}{72}}&\frac{-1}{18}\\ \noalign{\medskip}\frac{-1}{12}&{\frac {13}{24}}&\frac{-1}{12}&-{
\frac {5}{24}}&{\frac {7}{72}}&{\frac {7}{72}}&-{\frac {11}{72}}&-{
\frac {11}{72}}&\frac{-1}{18}\\ \noalign{\medskip}-{\frac {5}{24}}&\frac{-1}{12}&{
\frac {13}{24}}&\frac{-1}{12}&-{\frac {11}{72}}&{\frac {7}{72}}&{\frac {7}{72}
}&-{\frac {11}{72}}&\frac{-1}{18}\\ \noalign{\medskip}\frac{-1}{12}&-{\frac {5}{24}}&\frac{-1}{12}&{\frac {13}{24}}&-{\frac {11}{72}}&-{\frac {11}{72}}&{\frac {7}{
72}}&{\frac {7}{72}}&\frac{-1}{18}\\ \noalign{\medskip}{\frac {7}{72}}&{\frac 
{7}{72}}&-{\frac {11}{72}}&-{\frac {11}{72}}&{\frac {13}{36}}&\frac{-1}{18}&-{
\frac {5}{36}}&\frac{-1}{18}&0\\ \noalign{\medskip}-{\frac {11}{72}}&{\frac {7
}{72}}&{\frac {7}{72}}&-{\frac {11}{72}}&\frac{-1}{18}&{\frac {13}{36}}&\frac{-1}{18}&
-{\frac {5}{36}}&0\\ \noalign{\medskip}-{\frac {11}{72}}&-{\frac {11}{
72}}&{\frac {7}{72}}&{\frac {7}{72}}&-{\frac {5}{36}}&\frac{-1}{18}&{\frac {13
}{36}}&\frac{-1}{18}&0\\ \noalign{\medskip}{\frac {7}{72}}&-{\frac {11}{72}}&-
{\frac {11}{72}}&{\frac {7}{72}}&\frac{-1}{18}&-{\frac {5}{36}}&\frac{-1}{18}&{\frac {
13}{36}}&0\\ \noalign{\medskip}\frac{-1}{18}&\frac{-1}{18}&\frac{-1}{18}&\frac{-1}{18}&0&0&0&0&\frac{2}{9}
\end {array} \right) 
$$
and 
$$f=\bbA\bbL^{-1}\Psi$$
The normals are again obtained by consistency.

\section{A cell-centered subface-based unconventional Finite Volume discretization of Lagrangian hydrodynamics}
\label{sec:Lagrangian}
We consider now the discretization of the Lagrangian gas dynamics equations \eqref{eq:Lcv}. These equations are  characterized by their moving-grid framework. onsequently, the evolving geometry must be treated with particular care.
We present a subface-based, cell-centered finite volume method. To ensure consistency with the grid motion, the volume equation also known as the Geometrical Conservation Law (GCL)-is discretized in a compatible manner. This approach naturally leads to a node-based approximation of the volume flux, 
then also extended to construct the approximations of the momentum and total energy fluxes. The 
resulting method is   a
generalization to moving meshes  the multidimensional finite volume approach of section \ref{sec:motivating:FV:II}, thus naturally fitting the general framework of this paper. 
Such a node-based treatment of the physical flux is less common than the classical face-based flux formulations, which are conservative by construction. As already said, the proposed scheme  provides a nontrivial example of a finite volume method that is not inherently expressed in terms of face-based fluxes, but corner multidimensional fluxes.
The flux formulation of Section~\ref{sec:gt} allows
to recast the scheme in terms of subface-based fluxes.

For clarity of exposition, we describe the discretization in two dimensions, noting that its extension to three dimensions is straightforward and introduces no additional conceptual difficulties.

\subsection{Moving grid Finite Volume discretization}
\label{ssec:FVdiscretization}
\subsubsection{Moving grid notation}
\label{sssec:Moving}
The two-dimensional Euclidean space is equipped with the direct orthonormal basis $(\mathbf{e}_x,\mathbf{e}_y)$ which is supplemented by $\mathbf{e}_z=\mathbf{e}_x \times \mathbf{e}_y$. The computational domain at time $t$, $\mathcal{D}(t)$, is paved by means of a collection of nonoverlapping conformal polygonal cells denoted $\omega_c(t)$ and thus $\mathcal{D}(t)=\cup_c \omega_c(t)$. Here, we shall utilize the notations introduced in Section~\ref{sec:motivating:FV:II}. For the sake of completeness, let us recall that the polygonal cell $\omega_c(t)$ is characterized by the set of its vertices $\mathcal{P}(c)$. The generic vertex also named point is denoted using the label $p$ and its vector position is $\mathbf{x}_p(t)$. In the counterclockwise ordered list of points of cell $\omega_c(t)$, the vertex $p^{+}$ is the next vertex with respect to $p$ and $p^{-}$ is the previous one, refer to Fig.~\ref{fig:polygrida-mr}.  Gathering the subcells surrounding the generic vertex $p$ we define the dual cell, refer to Fig.~\ref{fig:polygridb-mr}
\begin{equation}
  \label{eq:dualcell}
  \omega_p=\bigcup_{c \in \mathcal{C}(p)} \omega_{pc},
\end{equation}
where $\mathcal{C}(p)$ is the set of cells sharing vertex $p$. The volume of the polygonal cell might be expressed in terms of the position vectors of its vertices
\begin{equation}
  \label{eq:volume}
  |\omega_c(t)|=\sum_{p \in \mathcal{P}(c)} \frac{1}{2} (\mathbf{x}_{p} \times \mathbf{x}_{p^{+}})\cdot \mathbf{e}_z.
\end{equation}
The above formula stems from the decomposition of the polygonal cell into triangular subcells, refer for instance to \cite{Whalen1996}. 

Finally, let us recall that each face $f \in \mathcal{F}(c)$ of the polygonal cell might be split in two subfaces which are the line segments connecting the midpoint and the vertices, refer to Fig.~\ref{fig:polygrida-mr}. For a point $p \in \mathcal{P}(c)$ we consider the set of subfaces attached to $p$ and to $c$, $\mathcal{SF}(pc)$, that is the set of subfaces of cell $c$ impinging at point $p$. In the present two-dimensional case, the subfaces are nothing but the ``half-faces'' introduced previously and thus $\mathcal{SF}(pc)=[p^{-\frac{1}{2}},p] \cup [p,p^{+\frac{1}{2}}]$. For $f \in \mathcal{SF}(pc)$ we denote by $l_{pcf}$ the measure of the face and $\mathbf{n}_{pcf}$ its unit outward normal. We point out that the set of subfaces of a given cell is a partition of the set of its faces $\mathcal{F}(c)=\bigcup_{p \in \mathcal{P}(c)} \mathcal{SF}(pc)$. In two dimensions, the notion of a subface may appear unnecessary. Nevertheless, it proves useful in what follows, as it allows us to formulate the Finite Volume discretization in a general framework that naturally extends to three dimensions. The interested reader might refer to \cite{Delmas2025} for a detailed description of the construction of subfaces on general polyhedral grids.
\subsubsection{The corner normal vector and the Geometrical Conservation Law}
\label{sssec:corner}
By virtue of \eqref{eq:volume}, it is clear that the cell volume is a function of its vertices coordinate, {\it i.e.}, $|\omega_c(t)|:=|\omega_c(t)(\mathbf{x}_1,\dots,\mathbf{x}_p,\dots,\mathbf{x}_{|\mathcal{P}(c)|})|$. Then, we compute the time rate of change of the polygonal cell volume simply applying chain rule of composed derivatives
$$\frac{\mathrm{d}}{\mathrm{d}t} |\omega_c(t)|=\sum_{p \in \mathcal{P}(c)} \frac{\partial |\omega_c|}{\partial \mathbf{x}_p} \cdot \frac{\mathrm{d} \mathbf{x}_{p}}{\mathrm{d} t}.$$
This formula puts forward the crucial role plays by the partial derivative of the volume with respect to the vertex position vector which is known as the corner normal vector, cf. \cite{Caramana1998} and the references therein
\begin{equation}
  \label{eq:defcornervec}
  l_{pc}\mathbf{n}_{pc}=\frac{\partial |\omega_c|}{\partial \mathbf{x}_p},
\end{equation}
where $\mathbf{n}_{pc}^{2}=1$. Substituting \eqref{eq:volume} into the foregoing definition leads to the expression of the corner normal of a generic polygonal cell
\begin{equation}
  \label{eq:fromcornervec}
  l_{pc}\mathbf{n}_{pc}=\frac{1}{2} [\mathbf{x}_{p^{+}}(t)-\mathbf{x}_{p^{-}}(t) ] \times \mathbf{e}_z=l_{pc}^{-}\mathbf{n}_{pc}^{-}+l_{pc}^{+}\mathbf{n}_{pc}^{+}.
\end{equation}
This shows that the corner normal is the summation of the outward normals to the ``half-faces'' impinging at point $p$, refer to Fig.~\ref{fig:polygrida-mr}. Finally, introducing the corner normal allows us to rewrite the time rate of change of the polygonal cell volume under the form
\begin{equation}
  \label{eq:GCLsemidiscrete}
  \frac{\mathrm{d}}{\mathrm{d}t} |\omega_c(t)|-\sum_{p \in \mathcal{P}(c)} l_{pc}\mathbf{n}_{pc} \mathbf{v}_p=0.
\end{equation}
Here, we have also made use of the point $p$ semi-discrete trajectory equation
\begin{equation}
  \label{eq:trajeqsd}
  \frac{\mathrm{d}\mathbf{x}_p}{\mathrm{d}t}=\mathbf{v}_{p},\;\mathbf{x}_p(0)=\mathbf{X}_p,
\end{equation}
where $\mathbf{v}_{p}$ is the point $p$ velocity needed to move the mesh. By proceeding like this we have obtained a semi-discrete version of the volume equation which is fully compatible with the semi-discrete trajectory equation of the grid nodes.

Noticing that the corner normal vector $l_{pc}\mathbf{n}_{pc}$ is the sum of the normal vectors to the subfaces belonging to $\mathcal{SF}(pc)$,
equation \eqref{eq:GCLsemidiscrete}
 is rewritten under the equivalent form
  \begin{equation}
  \label{eq:GCLsemidiscretebis}
  \frac{\mathrm{d}}{\mathrm{d}t} |\omega_c(t)|-\sum_{p \in \mathcal{P}(c)}\sum_{f \in \mathcal{SF}(pc)} l_{pcf}\mathbf{n}_{pcf} \mathbf{v}_p=0.
\end{equation}
This makes appear quite naturally a node-centered volume flux expressed in terms of the node velocity and the subfaces normals.
  
Finally, we observe that the corner normal fulfills the fundamental geometrical identity
\begin{equation}
  \label{eq:geomiden}
  \sum_{p \in \mathcal{P}(c)} l_{pc}\mathbf{n}_{pc}=\mathbf{0}.
\end{equation}
This result, by virtue of \eqref{eq:fromcornervec}, is a direct consequence of the fact that the sum of the outward normals to a closed contour is equal to zero.

The corner vector is also the cornerstone to define discrete mathematical operators such as the discrete divergence operator following the approach introduced in \cite{Shashkov1996} in the more general framework of mimetic finite difference method. For any vector field $\mathbf{q}:=\mathbf{q}(\mathbf{x}),\;\text{for}\; \mathbf{x} \in \omega_c$ we define the cell-centered discrete divergence 
\begin{equation}
  \label{eq:discretediv}
  \text{DIV}_c(\mathbf{q})=\frac{1}{|\omega_c|}\sum_{p \in \mathcal{P}(c)} l_{pc}\mathbf{n}_{pc} \cdot \mathbf{q}_{p},
\end{equation}
where $\mathbf{q}_p=\mathbf{q}(\mathbf{x}_{p})$. Identity \eqref{eq:geomiden} implies that the divergence of constant vectors is equal to zero. The mathematical properties of the discrete operators and their application to construct numerical methods for solving different classes of partial differential equations are described in the review paper \cite{Lipnikov2014}. Finally, utilizing the discrete gradient operator the semi-discrete GCL rewrites
\begin{equation}
  \label{eq:GCLdiv}
  \frac{1}{|\omega_c|}\frac{\mathrm{d}}{\mathrm{d}t} |\omega_c(t)|=\text{DIV}_c(\mathbf{v}).
\end{equation}
This highlights the mechanical interpretation of the discrete divergence operator as the normalized discrete time rate of change of the cell volume. 
\subsubsection{Formal FV discretization of the physical conservation laws}
\label{sssec:ffvdisc}
The discrete form of the mass conservation equation \eqref{eq:mass} for the cell $\omega_c$ takes the trivial form
\begin{equation}
  \label{eq:mmesd}
  \int_{\omega_c(t)} \rho(\mathbf{x},t)\,\mathrm{d}v=\int_{\omega_c(0)} \rho^{0}(\mathbf{X})\,\mathrm{d}V=m_c,
\end{equation}
where $\rho^{0}$ is the initial mass density and $m_c$ the constant cell mass, namely $\frac{\mathrm{d} m_{c}}{\mathrm{d} t}=0$. Bearing this in mind, we define the cell-centered mass-averaged quantities
\begin{equation}
  \label{eq:massaver}
  \mathbf{u}_c=\frac{1}{m_c} \int_{\omega_c} \rho \mathbf{u}\,\mathrm{d}v,
\end{equation}
and thus $\tau_c$ satisfies $m_c \tau_c=|\omega_c|$.

Employing the previous notation, the Lagrangian system of conservation laws \eqref{eq:Lcv} written over $\omega_c(t)$ takes the following form
$$m_c \frac{\mathrm{d} \mathbf{u}_c}{\mathrm{d} t}+\int_{\partial \omega_c(t)} \mathbf{f}\mathbf{n}\,\mathrm{d}s=\mathbf{0}.$$
Since $m_c \tau_c=|\omega_c(t)|$, it is worth noticing that the first equation of the previous system is nothing but the Geometrical Conservation Law that we have already discretized, see \eqref{eq:GCLsemidiscretebis}. Taking inspiration from the GCL discretization we adopt a node-based discretization of the integral flux in the previous equation 
$$\int_{\partial \omega_c(t)} \mathbf{f}(\mathbf{u})\mathbf{n}\,\mathrm{d}s=\sum_{p \in \mathcal{P}(c)} \sum_{f \in \mathcal{SF}(pc)} l_{pcf} \mathbf{f}_{pcf},$$
where $\mathbf{f}_{pcf}$ is a consistent numerical subface flux in the direction $\mathbf{n}_{pcf}$, refer to Fig.~\ref{fig:polygonal-cell}, that is
$$l_{pcf} \mathbf{f}_{pcf} \approx \int_{f\in \mathcal{SF}(pc)} \mathbf{f}  \mathbf{n}\,\mathrm{d}s.$$
This provides us the subface-based Finite Volume discretization of the Lagrangian system of conservation laws \eqref{eq:Lcv}
\begin{equation}
  \label{eq:Lcvsemidisc}
  m_c \frac{\mathrm{d} \mathbf{u}_c}{\mathrm{d} t}+\sum_{p \in \mathcal{P}(c)} \sum_{f \in \mathcal{SF}(pc)} l_{pcf} \mathbf{f}_{pcf}=\mathbf{0}.
\end{equation}
\begin{figure}
\centering
\includegraphics[width=0.5\textwidth]{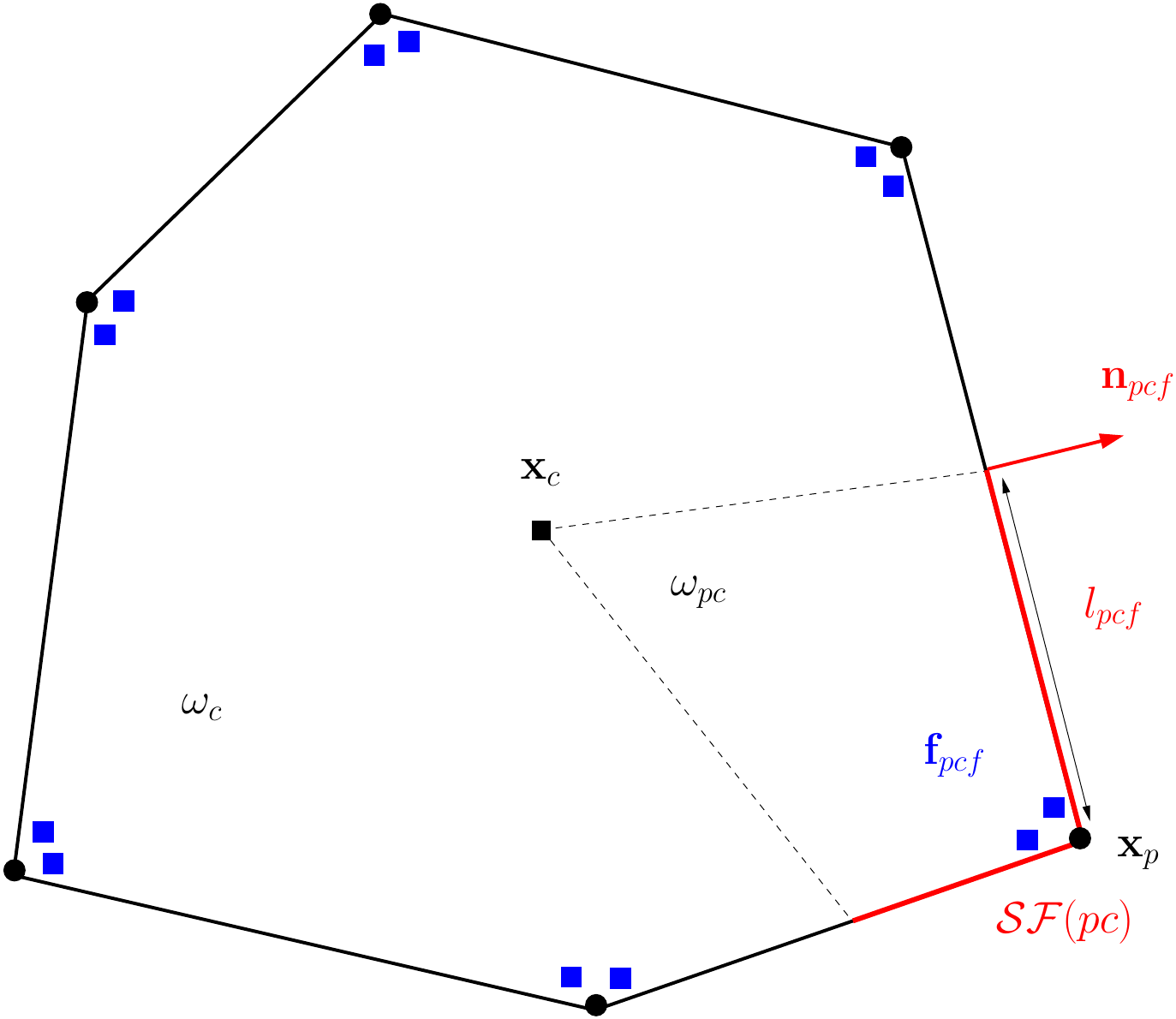}
\caption{Polygonal cell: representation of the subface-based discretization; The two subfaces related to the corner $pc$ are colored in red; The subface fluxes are displayed by means of the blue squares located at the cell corners.}
\label{fig:polygonal-cell}
\end{figure}
The subface flux in the direction $\mathbf{n}_{pcf}$ reads
$$
\mathbf{f}_{pcf}=
\begin{pmatrix}
  -\mathbf{v}_{p}\cdot \mathbf{n}_{pcf} \\
  p_{pcf} \mathbf{n}_{pcf} \\
  (p\mathbf{v}\cdot \mathbf{n})_{pcf}
\end{pmatrix}
.
$$
We note that the first component arises directly from the GCL equation. In contrast, $p_{pcf}\mathbf{n}_{pcf}$ and $(p\mathbf{v}\cdot \mathbf{n})_{pcf}$ are formal notations representing, respectively, the momentum flux and the total energy flux at the subface $f \in \mathcal{SF}(pc)$. It is reasonable to assume that $(p\mathbf{v}\cdot \mathbf{n})_{pcf}=p_{pcf}\mathbf{n}_{pcf} \cdot \mathbf{v}_{p}$ and under this assumption the subface flux becomes
\begin{equation}
  \label{eq:compft}
  \mathbf{f}_{pcf}=
  \begin{pmatrix}
  -\mathbf{v}_{p}\cdot \mathbf{n}_{pcf} \\
  p_{pcf} \mathbf{n}_{pcf} \\
  p_{pcf} \mathbf{n}_{pcf} \cdot \mathbf{v}_{p}
  \end{pmatrix}
 .
\end{equation}
Gathering the foregoing results, we arrive at the semi-discrete subface-based Finite Volume discretization of the Lagrangian gas dynamics
\begin{equation}
  \label{eq:semidiscsyst}
  m_c \frac{\mathrm{d} }{\mathrm{d} t} \begin{pmatrix}
    \tau_c \\
    \mathbf{v}_c \\
    e_c
  \end{pmatrix}
  +\sum_{p \in \mathcal{P}(c)} \sum_{f \in \mathcal{SF}(c)} l_{pcf} \begin{pmatrix}
    -\mathbf{v}_{p} \cdot \mathbf{n}_{pcf} \\
    p_{pcf} \mathbf{n}_{pcf} \\
    p_{pcf} \mathbf{n}_{pcf} \cdot \mathbf{v}_{p}
  \end{pmatrix}
  =\mathbf{0}.
\end{equation}
It remains to determine the nodal velocity, $\mathbf{v}_p$, and the subface momentum flux, $p_{pcf}$. This will be achieved through the introduction of a specific Riemann solver which shall allow us to compute a consistent numerical approximation of the subface flux, $\mathbf{f}_{pcf}$, in the direction $\mathbf{n}_{pcf}$. 

\subsection{Approximate Riemann solver for the numerical flux approximation}
\label{ssec:ars}

Let $\mathbf{n}$ be a unit normal vector, we are interested in studying the approximate solution to the Riemann problem defined in the $\mathbf{n}$ direction
$$
(\mathcal{RP})
\begin{dcases}
\frac{\partial \mathbf{u}}{\partial t}+\frac{\partial \mathbf{f}_{\mathbf{n}}(\mathbf{u})}{\partial m}=\mathbf{0},\\
\mathbf{u}(m,0)=
    \begin{dcases}
      \mathbf{u}_{l} \quad \text{if} \quad m <0, \\
      \mathbf{u}_{r} \quad \text{if} \quad m \geq 0.
    \end{dcases}
\end{dcases}
$$
The Lagrangian mass coordinate, $m$, is defined by by $\mathrm{d}m=\rho\,\mathrm{d}x_{\mathbf{n}}$ where $x_{\mathbf{n}}=\mathbf{x}\cdot \mathbf{n}$. The vector of conservative variables, $\mathbf{u}=\mathbf{u}(m,t)$, and the flux vector, $\mathbf{f}_{\mathbf{n}}$, written under Lagrangian representation are
$$\mathbf{u}=\begin{pmatrix}
\tau \\
v_{\mathbf{n}} \\
\mathbf{v}_{t} \\
e
\end{pmatrix}
,\;\text{and}\;
\mathbf{f}_{\mathbf{n}}=
\begin{pmatrix}
  -v_{\mathbf{n}}\\
  p \\
  \mathbf{0} \\
  pv_{\mathbf{n}}
\end{pmatrix}
.
$$
Here, $v_{\mathbf{n}}=\mathbf{v} \cdot \mathbf{n}$ is the normal component of the velocity, whereas $\mathbf{v}_{t}=\mathbf{v}-v_{\mathbf{n}} \mathbf{n}$ is its tangentiel component. Following \cite{GalliceHDR2002,Gallice2022}, we introduce the ``simple'' approximate Riemann solver
$$
\mathbf{r}(\mathbf{u}_l,\mathbf{u}_r,\xi)=
\begin{cases}
  \mathbf{u}_l\; &\text{if} \;\xi < -\lambda_l, \\
  \mathbf{u}_l^{\star} \; & \text{if} \;-\lambda_l \leq \xi < 0, \\
  \mathbf{u}_r^{\star} \; & \text{if} \; 0 \leq \xi < \lambda_r, \\
  \mathbf{u}_r \; & \text{if} \; \lambda_r \leq \xi,
\end{cases}
$$
where $\xi=\frac{m}{t}$. This Riemann solver consists of the 4 constant states $\mathbf{u}_l,\mathbf{u}_l^{\star},\mathbf{u}_r^{\star},\mathbf{u}_r$ separated by the 3 waves characterized by the 3 wave speeds $-\lambda_l,0,\lambda_r$ and we assume that $\lambda_l$ and $\lambda_r$ are positive, refer to Fig.~\ref{fig:diagmt}. We shall determine the intermediate states $\mathbf{u}_l^{\star}$, $\mathbf{u}_r^{\star}$ and the wave speeds $\lambda_l$, $\lambda_r$ in terms of the left and right states to achieve the characterization of this Riemann solver.
\begin{figure}
\centering
\includegraphics[width=0.5\textwidth]{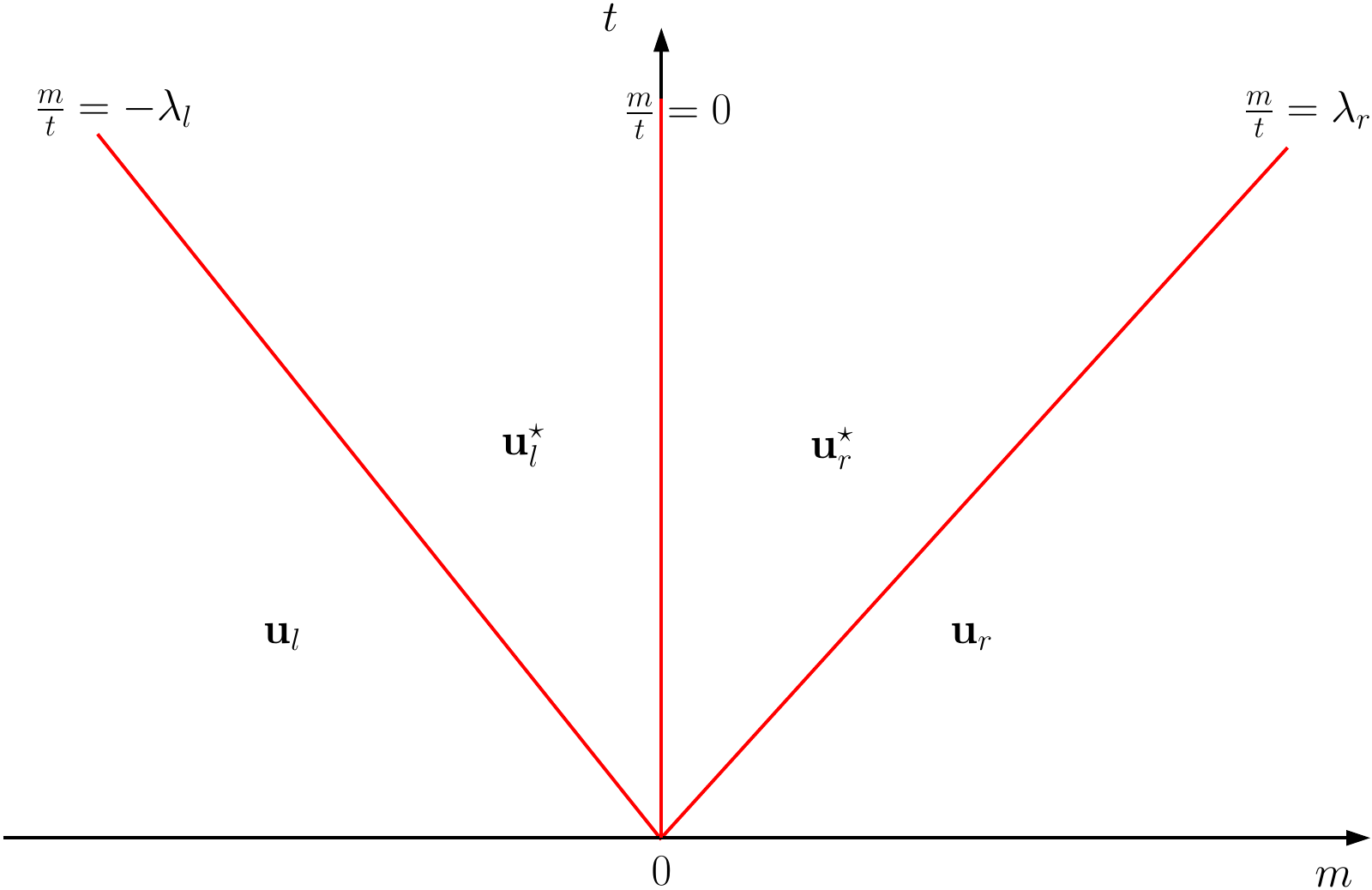}
\caption{Diagram $m-t$ representation of the simple approximate Riemann solver.}
\label{fig:diagmt}
\end{figure}

The intermediate state vector of conservative variables and the intermediate flux vector write
$$\mathbf{u}_{s}^{\star}=
\begin{pmatrix}
  \tau_s^{\star} \\
  v_{\mathbf{n},s}^{\star} \\
  \mathbf{v}_{t,s}^{\star} \\
  e_{s}^{\star}
\end{pmatrix}
,\;\text{and}\;
\mathbf{f}_{\mathbf{n},s}^{\star}=
\begin{pmatrix}
  -v_{\mathbf{n},s}^{\star} \\
  p_{s}^{\star} \\
  \mathbf{0} \\
  (pv_{\mathbf{n}})_{s}^{\star}
\end{pmatrix}
,\;\text{where}\;s=l,r.
$$
We note that $\mathbf{f}_{\mathbf{n},s}^{\star} \neq \mathbf{f}_{\mathbf{n}}(\mathbf{u}_{s}^{\star})$. The unknowns shall be determined writing the conservation relations across the waves of the approximate Riemann solver. Before going any further, we reduce the number of unknowns assuming that $(pv_{\mathbf{n}})_{s}^{\star}=p_{s}^{\star}v_{\mathbf{n},s}^{\star}$ and $v_{\mathbf{n},l}^{\star} = v_{\mathbf{n},r}^{\star} = v_{\mathbf{n}}^{\star}$ which finally leads to $2\text{d}+5$ scalar unknowns. Writing down the conservation relations respectively across the left- and the right-sided wave yields
\begin{subequations}
\label{eq:consrela}
\begin{align}
  &\lambda_{l} (\mathbf{u}_{l}^{\star}-\mathbf{u}_{l})+\mathbf{f}_{\mathbf{n},l}^{\star}-\mathbf{f}_{\mathbf{n},l}=\mathbf{0},\label{eq:consrelaa}\\
  -&\lambda_{r} (\mathbf{u}_{r}-\mathbf{u}_{r}^{\star})+\mathbf{f}_{\mathbf{n},r}-\mathbf{f}_{\mathbf{n},r}^{\star}=\mathbf{0},\label{eq:consrelab}
\end{align}
\end{subequations}
where $\mathbf{f}_{\mathbf{n},s}=\mathbf{f}_{\mathbf{n}}(\mathbf{u}_{s})$ for $s=l,r$. These two vectorial equations represent $2\text{d}+4$ scalar equations which write
\begin{equation*}
  (\mathcal{S}_l)
  \begin{cases}
    \lambda_l(\tau_l^{\star}-\tau_l)-(v_{\mathbf{n}}^{\star}-v_{\mathbf{n},l})=0,\\
    \lambda_l(v_{\mathbf{n}}^{\star}-v_{\mathbf{n},l})+ p_l^{\star}- p_l =0,\\
    \lambda_l(\mathbf{v}_{t,l}^{\star}-\mathbf{v}_{t,l})=\mathbf{0},\\
    \lambda_l(e_l^{\star}-e_l)+ p^{\star}_lv_{\mathbf{n}}^{\star} - p_lv_{\mathbf{n},l} =0,
  \end{cases}
  \;
  (\mathcal{S}_r)
  \begin{cases}
    \lambda_r(\tau_r^{\star}-\tau_r)+v_{\mathbf{n}}^{\star}-v_{\mathbf{n},r}=0,\\
    \lambda_r(v_{\mathbf{n}}^{\star}-v_{\mathbf{n},r})-(p^{\star}_r- p_r )=0,\\
    \lambda_r(\mathbf{v}_{t,r}^{\star}-\mathbf{v}_{t,r})=\mathbf{0},\\
    \lambda_r(e_r^{\star}-e_r)-(p^{\star}_rv_{\mathbf{n}}^{\star} - p_rv_{\mathbf{n},r}) =0.
\end{cases}
\end{equation*}
Here, $e_s^{\star}=\varepsilon_s^{\star}+\frac{1}{2} (v_{\mathbf{n}}^{\star})^{2}+\frac{1}{2}(\mathbf{v}_{t,s}^{\star})^{2}$ for $s=l,r$. The third equation of $(\mathcal{S}_l)$ and $(\mathcal{S}_r)$ shows that the tangential component of the velocity is conserved across the left- and the right-sided wave. The above systems involve $2\text{d}+5$ unknowns and $2\text{d}+4$ equations. Consequently, we may select $v_{\mathbf{n}}^{\star}$ as a free parameter, in terms of which the solutions can be expressed. This extra degree of freedom shall be useful to determine the nodal velocity required in the Finite Volume discretization to displace the computational grid, refer to Section~\ref{sssec:corner}.

The positivity of the intermediate states, {\it i.e.}, $\tau_s^{\star}>0$ and $\varepsilon_s^{\star} >0$ for $s=l,r$ is ensured by adapting the wave speeds $\lambda_l$ and $\lambda_r$ as follows
\begin{equation}
  \label{eq:pcev1}
\lambda_l \geq \max \left(\frac{p_l}{\sqrt{2\varepsilon_l}}, -\dfrac{v_{\mathbf{n}}^{\star}-v_{\mathbf{n},l}}{\tau_l}\right),\;\; \text{and}\;\;\lambda_r \geq \max \left(\frac{p_r}{\sqrt{2\varepsilon_r}}, \dfrac{v_{\mathbf{n}}^{\star}-v_{\mathbf{n},r}}{\tau_r}\right).
\end{equation}
The proof of this result might be found in \cite{Gallice2022}. Moreover, noticing that for a convex equation of state, {\it i.e.}, $\tau \mapsto p(\tau,\eta)$ strictly convex, there holds $\frac{a^{2}}{\tau^{2}} \geq \frac{p^{2}}{2 \varepsilon}$ \cite{Menikoff1989}, the 
positivity condition turns into
\begin{equation}
  \label{eq:pcev}
\lambda_l \geq \max \left(\frac{a_l}{\tau_l}, -\dfrac{v_{\mathbf{n}}^{\star}-v_{\mathbf{n},l}}{\tau_l}\right),\;\; \text{and}\;\;\lambda_r \geq \max \left(\frac{a_r}{\tau_r}, \dfrac{v_{\mathbf{n}}^{\star}-v_{\mathbf{n},r}}{\tau_r}\right).
\end{equation}
We note that these positivity conditions depend weakly on the equation of state via the isentropic sound speed.

By summing the left-sided conservation relation \eqref{eq:consrelaa} with the right-sided one \eqref{eq:consrelab}, we obtain
\begin{equation}
  \label{eq:fluxdiff}
  \mathbf{f}_{\mathbf{n},r}^{\star}-\mathbf{f}_{\mathbf{n},l}^{\star}= \lambda_{l} (\mathbf{u}_{l}^{\star}-\mathbf{u}_{l})-\lambda_{r} (\mathbf{u}_{r}-\mathbf{u}_{r}^{\star})+\mathbf{f}_{\mathbf{n},r}-\mathbf{f}_{\mathbf{n},l}.
  \end{equation}
On the other hand, the difference between $\mathbf{f}_{\mathbf{n},r}^{\star}$ and $\mathbf{f}_{\mathbf{n},l}^{\star}$ computed from their components, is given by
\begin{equation}
  \label{eq:fluxdiff2}
  \mathbf{f}_{\mathbf{n},r}^{\star}-\mathbf{f}_{\mathbf{n},l}^{\star}=(p_{r}^{\star}-p_{l}^{\star})
  \begin{pmatrix}
    0 \\
    1 \\
    \mathbf{0} \\
    v_{\mathbf{n}}^{\star}
  \end{pmatrix}
.
\end{equation}
Since, {\it a priori} $p_{r}^{\star}-p_{l}^{\star}\neq 0$ the left and right intermediate fluxes are distinct which means that the left-hand side of \eqref{eq:fluxdiff} does not vanish and thus in general the HLL consistency condition \cite{Harten1981,Harten1981} with the integral form of the conservation law is not fulfilled. Now, summing the second equations of $(\mathcal{S}_l)$ and $(\mathcal{S}_r)$ we obtain the expression of $p_{r}^{\star}-p_{l}^{\star}$ in terms of the parameter $v_{\mathbf{n}}^{\star}$
\begin{equation}
  \label{eq:difpressure}
  p_{r}^{\star}-p_{l}^{\star}=(\lambda_l+\lambda_r) \left (v_{\mathbf{n}}^{\star}-\overline{v}_{\mathbf{n},lr}\right),
\end{equation}
where  $\overline{v}_{\mathbf{n},lr}$ is the face-based normal velocity defined by
\begin{equation}
  \label{eq:facebasedvel}
  \overline{v}_{\mathbf{n},lr}=\frac{\lambda_l v_{\mathbf{n},l}+\lambda_r v_{\mathbf{n},r}}{\lambda_l+\lambda_r}-\frac{p_r-p_l}{\lambda_l+\lambda_r},
\end{equation}
which is nothing but the one-dimensional Godunov-like interface normal velocity, refer to \cite{Chan2021}. Gathering the foregoing results we express the difference between the left and the right intermediate fluxes as follows
\begin{equation}
  \label{eq:fluxdiffultim}
  \mathbf{f}_{\mathbf{n},r}^{\star}-\mathbf{f}_{\mathbf{n},l}^{\star}=(\lambda_l+\lambda_r) \left (v_{\mathbf{n}}^{\star}-\overline{v}_{\mathbf{n},lr}\right)
  \begin{pmatrix}
    0 \\
    1 \\
    \mathbf{0} \\
    v_{\mathbf{n}}^{\star}
  \end{pmatrix}
.
\end{equation}
At this stage, we have to cope with the following alternative regarding the parameter $v_{\mathbf{n}}^{\star}$
\begin{itemize}
\item Either $v_{\mathbf{n}}^{\star}=\overline{v}_{\mathbf{n},lr}$, then $p_l^{\star}= p_r^{\star}$ and the approximate Riemann solver is consistent with the integral form of the underlying conservation law since the HLL consistency condition is satisfied. This Riemann solver induces a unique interface flux, {\it i.e.}, $\mathbf{f}_{\mathbf{n},l}^{\star}=\mathbf{f}_{\mathbf{n},r}^{\star}$ and a Godunov-type face-based Finite Volume scheme. 
\item Or $v_{\mathbf{n}}^{\star}\neq \overline{v}_{\mathbf{n},lr}$, then $p_l^{\star} \neq p_r^{\star}$ and the HLL consistency condition is not fulfilled. This Riemann solver does not induce a conservative Godunov-type Finite Volume in the classical sense since the interface flux is not uniquely defined, {\it i.e.}, $\mathbf{f}_{\mathbf{n},l}^{\star}\neq \mathbf{f}_{\mathbf{n},r}^{\star}$.
\end{itemize}
In what follows, we shall investigate how to determine the parameter $v_{\mathbf{n}}^{\star}$ and to restore the conservation property when the cell interface numerical flux is not uniquely defined anymore. 
\begin{rem}[Left and right intermediate fluxes in terms of their arithmetic average]
  \label{rem:lrsfaa}
Halving the summation of \eqref{eq:consrelaa} and \eqref{eq:consrelab} provides the arithmetic average of the left and right intermediate fluxes
  \begin{align*}
    \langle \mathbf{f}_{\mathbf{n}} \rangle_{lr}^{\star}=&\frac{1}{2} \left (\mathbf{f}_{\mathbf{n},l}^{\star}+\mathbf{f}_{\mathbf{n},r}^{\star}\right)\\
    =&\frac{1}{2} \left (\mathbf{f}_{\mathbf{n},l}+\mathbf{f}_{\mathbf{n},r} \right)-\frac{\lambda_l}{2}(\mathbf{u}_{l}^{\star}-\mathbf{u}_{l})-\frac{\lambda_r}{2}(\mathbf{u}_{r}-\mathbf{u}_{r}^{\star}),
  \end{align*}
  which formally coincides with the classical one-dimensional interface flux, but differs in that the left and right intermediate fluxes are not equal. On the other hand, the difference between the left and right intermediate fluxes being given by \eqref{eq:fluxdiffultim} we arrive at
\begin{subequations}
\label{eq:exprflux}
\begin{align}
  &\mathbf{f}_{\mathbf{n},l}^{\star}=\langle \mathbf{f}_{\mathbf{n}} \rangle_{lr}^{\star}-\frac{1}{2}(\lambda_l+\lambda_r) (v_{\mathbf{n}}^{\star}-\overline{v}_{\mathbf{n},lr}) \begin{pmatrix}
    0 \\
    1 \\
    \mathbf{0}\\
    v_{\mathbf{n}}^{\star}
  \end{pmatrix}
  \label{eq:exprfluxl}\\
  &\mathbf{f}_{\mathbf{n},r}^{\star}=\langle \mathbf{f}_{\mathbf{n}} \rangle_{lr}^{\star}+\frac{1}{2}(\lambda_l+\lambda_r) (v_{\mathbf{n}}^{\star}-\overline{v}_{\mathbf{n},lr}) \begin{pmatrix}
    0 \\
    1 \\
    \mathbf{0}\\
    v_{\mathbf{n}}^{\star}
  \end{pmatrix}
  \label{eq:exprfluxr}
\end{align}
\end{subequations}
\end{rem}
\begin{rem}[Expression of the left- and right-sided fluxes in terms of the Riemann solver]
  \label{rem:fluxrs}
  Integrating the conservation law of $(\mathcal{RP})$ respectively over $[0,\Delta t] \times [-\Delta m_l,0]$ and $[0,\Delta t] \times [0,\Delta m_r]$, see Fig.~\ref{fig:diagmt} and replacing $\mathbf{u}(m,t)$ by the approximate Riemann solver, $\mathbf{r}(\mathbf{u}_l,\mathbf{u}_r,\xi)$, we express the left- and the right-sided fluxes in terms of the approximate Riemann solver
\begin{subequations}
\label{eq:exprfluxrs}
\begin{align}
  \mathbf{f}_{\mathbf{n}}^{-}=&\mathbf{f}_{\mathbf{n},l}-\int_{-\infty}^{0} \left [\mathbf{r}(\mathbf{u}_l,\mathbf{u}_r,\xi)-\mathbf{u}_l\right] \,\mathrm{d}\xi, \label{eq:exprfluxrsl}\\
  \mathbf{f}_{\mathbf{n}}^{+}=&\mathbf{f}_{\mathbf{n},r}+\int_{0}^{+\infty} \left [\mathbf{r}(\mathbf{u}_l,\mathbf{u}_r,\xi)-\mathbf{u}_r\right]\,\mathrm{d}\xi, \label{eq:exprfluxrsr}
\end{align}
\end{subequations}
where $\xi=\frac{m}{t}$ is the self-similar variable. We note that these definitions of the left- and righ-sided fluxes might be found in \cite{Harten1983,Bouchut2004}. Substituting the expression of the approximate Riemann solver in the above formulas yields
\begin{align*}
 \mathbf{f}_{\mathbf{n}}^{-}=&\mathbf{f}_{\mathbf{n},l}-\lambda_l(\mathbf{u}_{l}^{\star}-\mathbf{u}_{l}),\\
 \mathbf{f}_{\mathbf{n}}^{+}=&\mathbf{f}_{\mathbf{n},r}-\lambda_r(\mathbf{u}_{r}-\mathbf{u}_{r}^{\star}).
\end{align*}
Then, by virtue of the conservation conditions \eqref{eq:consrelaa} and \eqref{eq:consrelab} written across the left- and the right-sided wave we arrive at
\begin{equation}
  \label{eq:iflff}
  \mathbf{f}_{\mathbf{n}}^{-}=\mathbf{f}_{\mathbf{n},l}^{\star},\;\text{and}\; \mathbf{f}_{\mathbf{n}}^{+}=\mathbf{f}_{\mathbf{n},r}^{\star},
\end{equation}
which shows that the left- and right-sided fluxes are identical to the corresponding left and right intermediate fluxes. This result stems from the fact that, in the Lagrangian representation, the contact wave coincides with the interface between the left and right states. 
\end{rem}
\subsection{An unconventional multidimensional Godunov-type FV discretization}
\label{ssec:unconv}
\subsubsection{Subface flux approximation and domain preserving property}
\label{sssec:sfa}
Let us write the explicit time discretization of the subface-based Finite Volume formulation introduced in Section~\ref{sssec:ffvdisc}
\begin{equation}
  \label{eq:Lcvdisc}
  m_c (\mathbf{u}_c^{n+1}-\mathbf{u}_c^{n})+\Delta t\sum_{p \in \mathcal{P}(c)} \sum_{f \in \mathcal{SF}(pc)} l_{pcf} \mathbf{f}_{pcf}=\mathbf{0}.
\end{equation}
Here, $\mathbf{u}_c^{n}$ denote the value of the mass averaged variable $\mathbf{u}$ at time $t^{n}$ and $\Delta t =t^{n+1}-t^{n}$ is the time step. The subface flux, $\mathbf{f}_{pcf}$, is defined by means of the approximate Riemann studied in the previous section. To this end, we introduce the Riemann problem constructed in the direction of the unit outward normal $\mathbf{n}_{pcf}$ to the subface $f$ located at the interface between cells $\omega_c$ and $\omega_d$, refer to Fig.~\ref{fig:interface-cd}
$$
(\mathcal{RP}_{pcf})
\begin{dcases}
\frac{\partial \mathbf{u}}{\partial t}+\frac{\partial \mathbf{f}_{\mathbf{n}_{pcf}}(\mathbf{u})}{\partial m_{pcf}}=\mathbf{0},\\
\mathbf{u}(m_{pcf},0)=
    \begin{dcases}
      \mathbf{u}_{l} \quad \text{if} \quad m_{pcf} <0, \\
      \mathbf{u}_{r} \quad \text{if} \quad m_{pcf} \geq 0.
    \end{dcases}
\end{dcases}
$$
The Lagrangian mass coordinate $m_{pcf}$ is defined by $\mathrm{d}m_{pcf}=\rho\,\mathrm{d}x_{\mathbf{n}_{pcf}}$, where $x_{\mathbf{n}_{pcf}}=\mathbf{x}\cdot \mathbf{n}_{pcf}$. 
\begin{figure}
\centering
\includegraphics[width=0.5\textwidth]{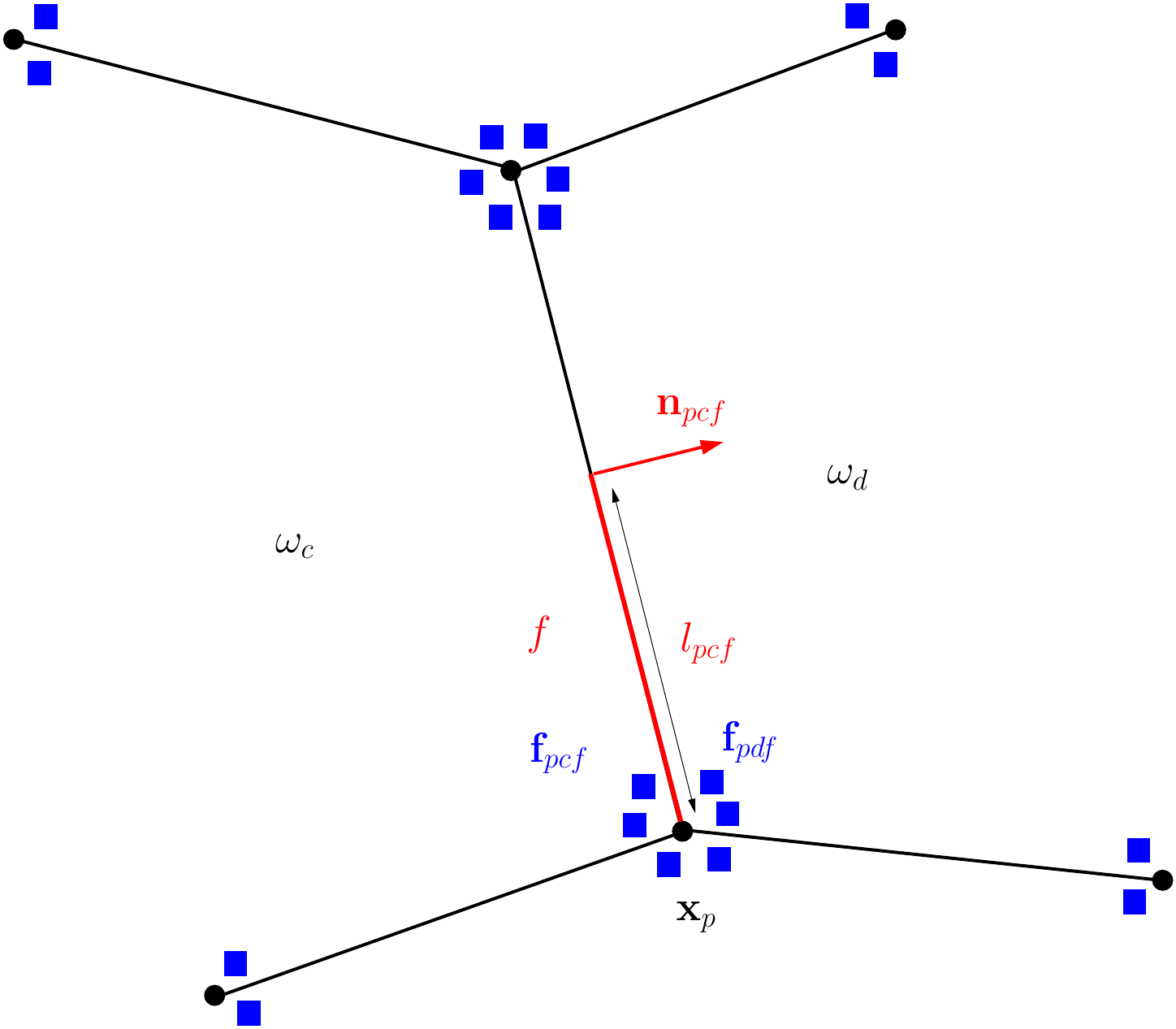}
\caption{Interface between cells $\omega_c$ and $\omega_d$: the subface fluxes are displayed by means of blue squares.}
\label{fig:interface-cd}
\end{figure}
Next, inspired by Section~\ref{ssec:ars} we consider the approximate Riemann solver, see Fig.~\ref{fig:diagmt-cellc}
$$\mathbf{r}(\mathbf{u}_c^{n},\mathbf{u}_d^{n},v_{\mathbf{n}_{pcf}}^{\star},\xi_{pcf}).$$ 
where $\xi_{pcf}=\frac{m_{pcf}}{t}$ is the self-similar variable and $\mathbf{u}_d^{n}$ denotes the mass-averaged state of $u$ within the right-hand neighbor of cell $\omega_c$, see Fig.~\ref{fig:interface-cd}. We explicitly include $v_{\mathbf{n}_{pcf}}^{\star}$ as an argument to emphasize the parametrization of the Riemann solver with respect to this quantity. With this in mind, and after a straightforward adaptation of the notation, we define the subface flux in terms of the left-sided flux, see Remark~\ref{rem:fluxrs}, as follows
\begin{equation}
  \label{eq:subffdef}
  \mathbf{f}_{pcf}=\mathbf{f}_{\mathbf{n}_{pcf}}^{-}=\mathbf{f}_{\mathbf{n}_{pcf}}(\mathbf{u}_{c}^{n})-\int_{-\infty}^{0} \left [\mathbf{r}(\mathbf{u}_c^{n},\mathbf{u}_d^{n},v_{\mathbf{n}_{pcf}}^{\star},\xi_{pcf})-\mathbf{u}_c^{n}\right] \,\mathrm{d}\xi,
\end{equation}
where $\mathbf{f}_{\mathbf{n}_{pcf}}(\mathbf{u}_{c}^{n})=\mathbf{f}(\mathbf{u}_{c}^{n})  \mathbf{n}_{pcf}$.
\begin{figure}
\centering
\includegraphics[width=0.5\textwidth]{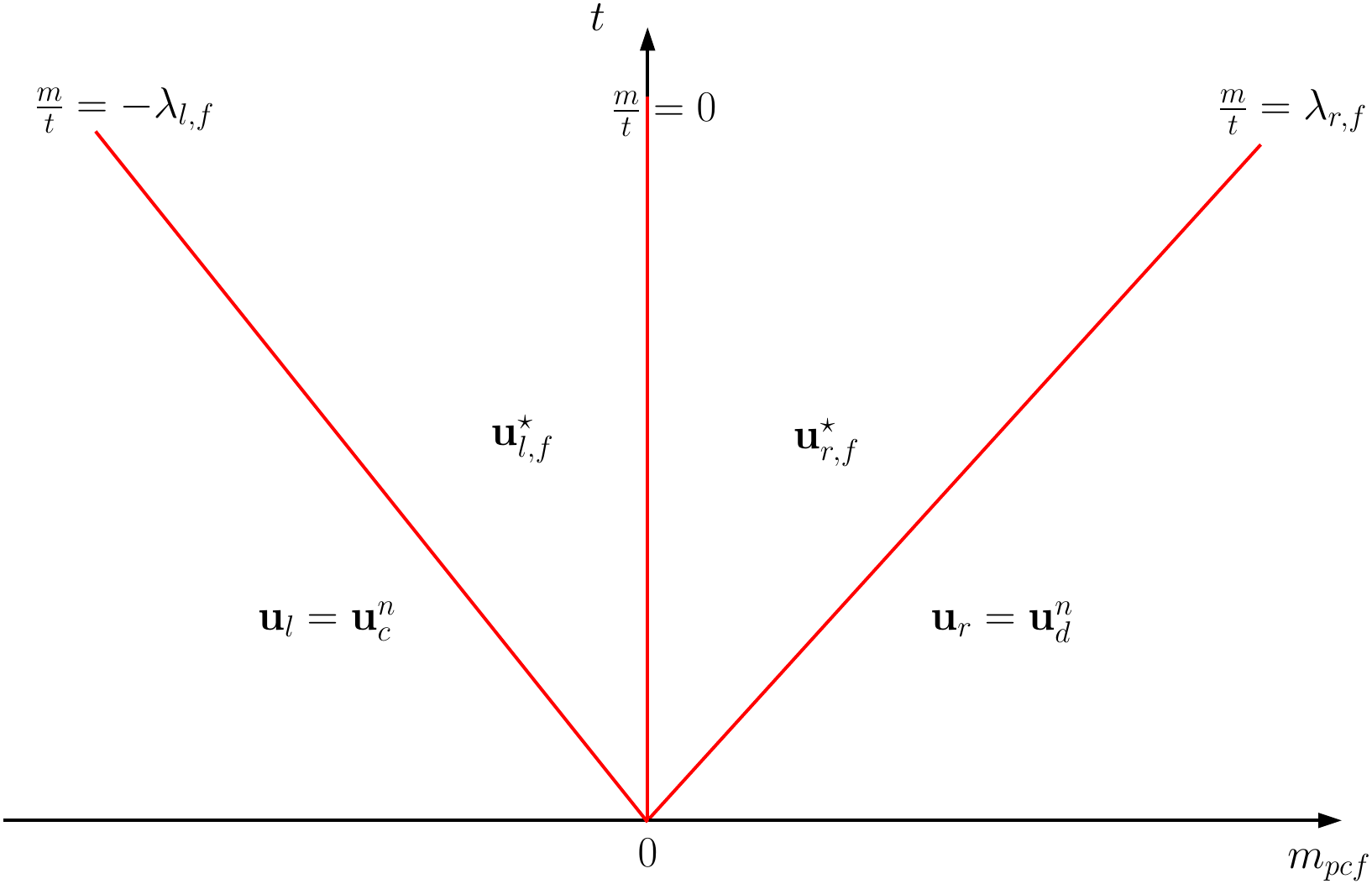}
\caption{Diagram $m-t$ representation of the simple approximate Riemann solver $\mathbf{r}(\mathbf{u}_c^{n},\mathbf{u}_d^{n},v_{\mathbf{n}_{pcf}}^{\star},\xi_{pcf})$ at the subface $f$ in the direction of the unit normal $\mathbf{n}_{pcf}$.}
\label{fig:diagmt-cellc}
\end{figure}
Considering the subface $f$ shared by corners $pc$ and $pd$, {\it i.e.}, $f \in \mathcal{SF}(pc) \cap \mathcal{SF}(pd)$, we observe that, in general, the corresponding subface fluxes are distinct, $\mathbf{f}_{pcf} \neq \mathbf{f}_{pdf}$, and therefore the resulting subface based Finite Volume cannot be conservative in the classical sense.
By virtue of Remark~\ref{rem:fluxrs} and with obvious notation adaptation the subface flux writes
\begin{equation}
  \label{eq:subffdef1}
  \mathbf{f}_{pcf}=\mathbf{f}_{\mathbf{n}_{pcf},l}-\lambda_{l,f}(\mathbf{u}_{l,f}^{\star}-\mathbf{u}_{c}^{n}),
  \end{equation}
  where $\mathbf{f}_{\mathbf{n}_{pcf},l}=\mathbf{f}(\mathbf{u}_{c}^{n})  \mathbf{n}_{pcf}$. Substituting the above subface flux expression into the Finite Volume discretization \eqref {eq:Lcvdisc} we arrive at
$$
m_c (\mathbf{u}_c^{n+1}-\mathbf{u}_c^{n})-\Delta t\sum_{p \in \mathcal{P}(c)} \sum_{f \in \mathcal{SF}(pc)} l_{pcf} \lambda_{l,f} (\mathbf{u}_{l,f}^{\star}-\mathbf{u}_{c}^{n})=\mathbf{0},
$$
thanks to the geometrical identity \eqref{eq:geomiden}. Collecting the terms in factor of the intermediate states $\mathbf{u}_{l,f}^{\star}$ and $\mathbf{u}_{c}^{n}$ allows us to express $\mathbf{u}_{c}^{n+1}$ as a linear combination of the intermediate states 
\begin{equation}
  \label{eq:linearcomb}
  \mathbf{u}_c^{n+1}=\left (1-\frac{\Delta t}{m_c}\sum_{p \in \mathcal{P}(c)} \sum_{f \in \mathcal{SF}(pc)} l_{pcf} \lambda_{l,f} \right) \mathbf{u}_c^{n}+\frac{\Delta t}{m_c} \sum_{p \in \mathcal{P}(c)} \sum_{f \in \mathcal{SF}(pc)} l_{pcf} \lambda_{l,f} \mathbf{u}_{l,f}^{\star}.
\end{equation}
Since the mass flux paramater $\lambda_{l,f}$ of the approximate Riemann solver is non-negative, it follows directly that $\mathbf{u}_c^{n+1}$ is a convex combination of $\mathbf{u}_c^{n}$ and the intermediate states $\mathbf{u}_{l,f}^{\star}$ for $f \in \mathcal{SF}(pc),\; \forall p \in \mathcal{P}(c)$, provided that the time step satisfies the constraint
\begin{equation}
  \label{eq:CFLlike1}
  \Delta t \leq \frac{m_c}{\sum_{p \in \mathcal{P}(c)} \sum_{f \in \mathcal{SF}(pc)} l_{pcf} \lambda_{l,f}}.
\end{equation}
Introducing
$$\Delta t_c=\frac{m_c}{\sum_{p \in \mathcal{P}(c)} \sum_{f \in \mathcal{SF}(pc)} l_{pcf} \lambda_{l,f}},$$
allows us to write the global CFL-like time step condition
\begin{equation}
  \label{eq:CFLlike2}
  \Delta t \leq \min_{c} \Delta t_c.
\end{equation}
The convex combination property, under the given time step condition, is very important as it ensures that our Finite Volume scheme is domain preserving. More precisely, if $\mathbf{u}_{c}^{n} \in \mathcal{R}$, then both
$$\mathbf{r}(\mathbf{u}_c^{n},\mathbf{u}_d^{n},v_{\mathbf{n}_{pcf}}^{\star},\xi_{pcf}) \in \mathcal{R}\;\text{and}\;\mathbf{u}_{c}^{n+1},$$
provided that the positivity condition \eqref{eq:pcev} on the mass flux parameters $\lambda_{l,f}$, $\lambda_{r,f}$, together with the time step condition \eqref{eq:CFLlike2}, are satisfied. In other words, this guarantees that our Finite Volume scheme preserves positivity for both the specific volume and the specific internal energy provided that conditions \eqref{eq:pcev} and \eqref{eq:CFLlike2} are fulfilled. 
\subsubsection{From the fluctuations consistency condition to the node-based conservation condition}
\label{sssec:nbc}
By virtue of the geometrical identity \eqref{eq:geomiden} we are able to rewrite the subface-based Finite Volume under the form
\begin{equation}
  \label{eq:flucform}
  m_c (\mathbf{u}_c^{n+1}-\mathbf{u}_c^{n})+\Delta t\sum_{p \in \mathcal{P}(c)} \sum_{f \in \mathcal{SF}(pc)} l_{pcf} \left [\mathbf{f}_{pcf}-\mathbf{f}(\mathbf{u}_c^{n})  \mathbf{n}_{pcf} \right]=\mathbf{0},
\end{equation}
which makes appear the fluctuation within cell $\omega_c$ attached to node $p$
\begin{equation}
  \label{eq:fluctuationpc}
  \mathbf{\Phi}_{c}^{p}=\sum_{f \in \mathcal{SF}(pc)} l_{pcf} \left [\mathbf{f}_{pcf}-\mathbf{f}(\mathbf{u}_c^{n}) \mathbf{n}_{pcf} \right].
\end{equation}
Inspired by the theory of the Residual Distribution schemes \cite{Abgrall2018}, refer 
to section \ref{sec:gt} and relations \eqref{forme schemas} and in particular \eqref{rds_re:conservation}, we impose the consistency condition, which requires that the sum of the fluctuations around node $p$ equals the integral of the flux across the boundary of the dual cell $\omega_p$
\begin{equation}
  \label{eq:concon}
  \sum_{c\in \mathcal{C}(p)} \mathbf{\Phi}_{c}^{p} =\int_{\partial \omega_p} \mathbf{f}(\mathbf{u}) \mathbf{n}\,\mathrm{d}s.
\end{equation}
The approximation of the integral at the right-hand side is computed decomposing it over the subcells surrounding node $p$, see Fig.~\ref{fig:polygridb-mr}
$$
\int_{\partial \omega_p} \mathbf{f}(\mathbf{u}) \mathbf{n}\,\mathrm{d}s=\sum_{c \in \mathcal{C}(p)} \int_{\partial \omega_{pc} \cap (\omega_c \setminus \partial \omega_c)}\mathbf{f}(\mathbf{u}) \mathbf{n}\,\mathrm{d}s.$$
Since the contour of the subcell $\omega_{pc}$ is closed we have $\int_{\partial \omega_{pc}} \mathbf{n}\,\mathrm{d}s=\mathbf{0}$ and then
$$\int_{\partial \omega_{pc} \cap (\omega_c \setminus \partial \omega_c)}\mathbf{n}\,\mathrm{d}s=-\sum_{f \in \mathcal{SF}(pc)} l_{pcf} \mathbf{n}_{pcf}.$$
Finally, the right-hand side of \eqref{eq:concon} might be rewritten
\begin{equation}
  \label{eq:rhsconcon}
  \int_{\partial \omega_p} \mathbf{f}(\mathbf{u}) \mathbf{n}\,\mathrm{d}s=-\sum_{f \in \mathcal{SF}(pc)} l_{pcf} \mathbf{f}(\mathbf{u}_c^{n}) \mathbf{n}_{pcf}.
\end{equation}
Substituting \eqref{eq:rhsconcon} and the fluctuation expression \eqref{eq:fluctuationpc} in the consistency condition \eqref{eq:concon} turn into
\begin{equation}
  \label{eq:nbc}
  \sum_{c\in \mathcal{C}(p)} \sum_{f \in \mathcal{SF}(pc)} l_{pcf} \mathbf{f}_{pcf}=\mathbf{0}.
\end{equation}
This is nothing but the node-based conservation condition introduced in the design of cell-centered finite volume for Lagrangian hydrodynamics \cite{LMR2016} and which provides the so called nodal solver to compute the node velocity. This condition, originated from the Lagrangian setting,  underpins also
  the construction of multidimensional finite volume schemes recalled in Section~\ref{sec:motivating:FV:II}. 
It is remarkable to observe that this node-based conservation condition, which might be obtained as a sufficient global conservation condition \cite{LMR2016} is strictly equivalent to the consistency condition \eqref{rds_re:conservation} introduced classically in the Residual Distribution framework. This is a nice example of the duality between conservation in terms  of flux,
and conservation in terms of residuals, which 
is the  core theme of this paper.
\subsubsection{Nodal solver}
\label{sssec:nodalsolver}
The node-based conservation condition \eqref{eq:nbc} states that the sum, over all cells $c$ sharing the point $p$, of the fluxes associated with the subfaces incident to $p$ is equal to zero. The subface fluxes are illustrated as blue patches on both sides of each subface emanating from $p$; see Figure~\ref{fig:geom2Dpointa} for a two-dimensional representation of the subface fluxes associated with a given node of a generic grid.
\begin{figure}  
\begin{center}          
\begin{subfigure}{0.45\textwidth}
\centering
\includegraphics[width=\textwidth]{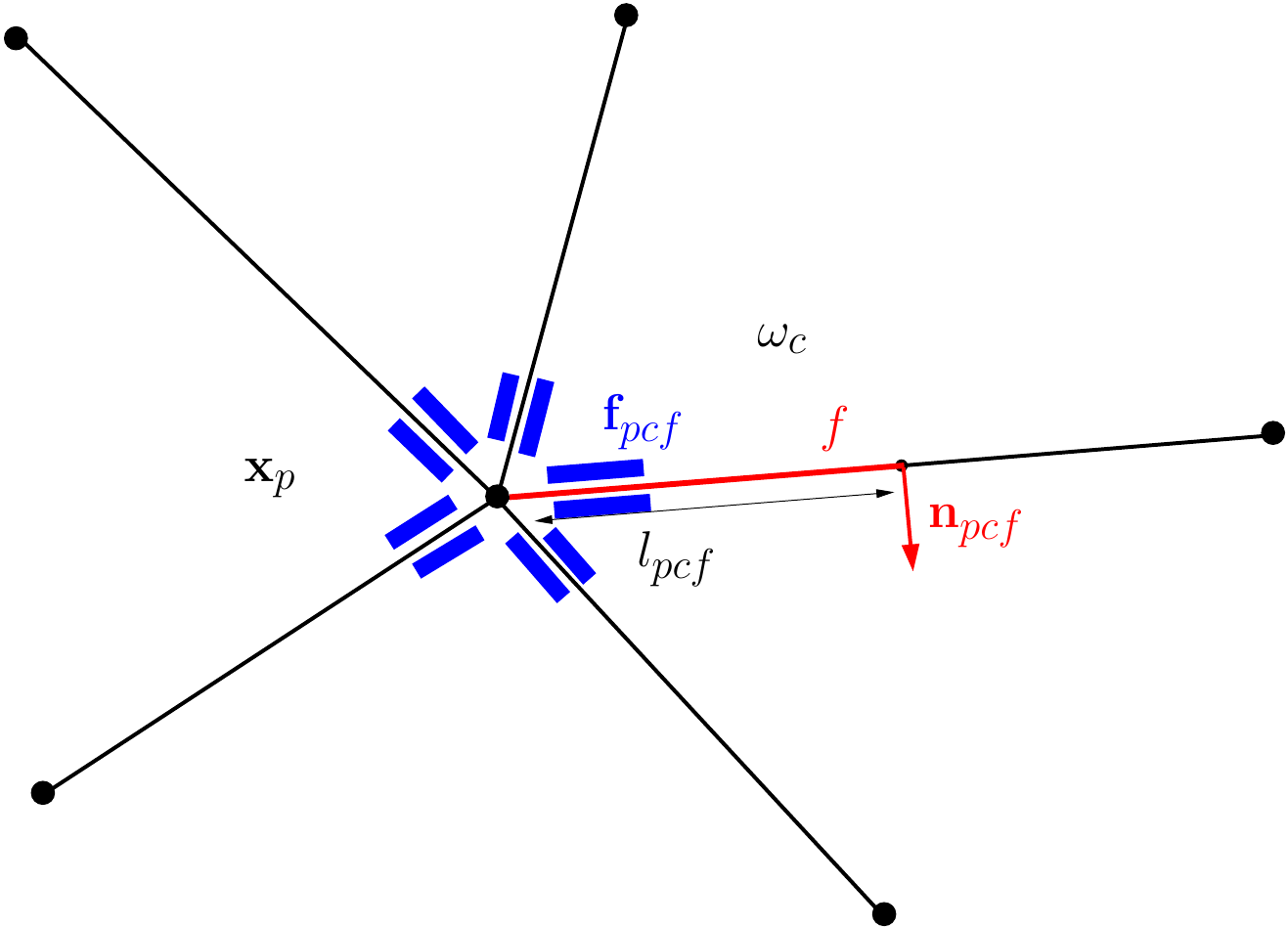}     
\caption{Representation of the subface flux.}
\label{fig:geom2Dpointa}
\end{subfigure}
\begin{subfigure}{0.45\textwidth}
\centering
\includegraphics[width=\textwidth]{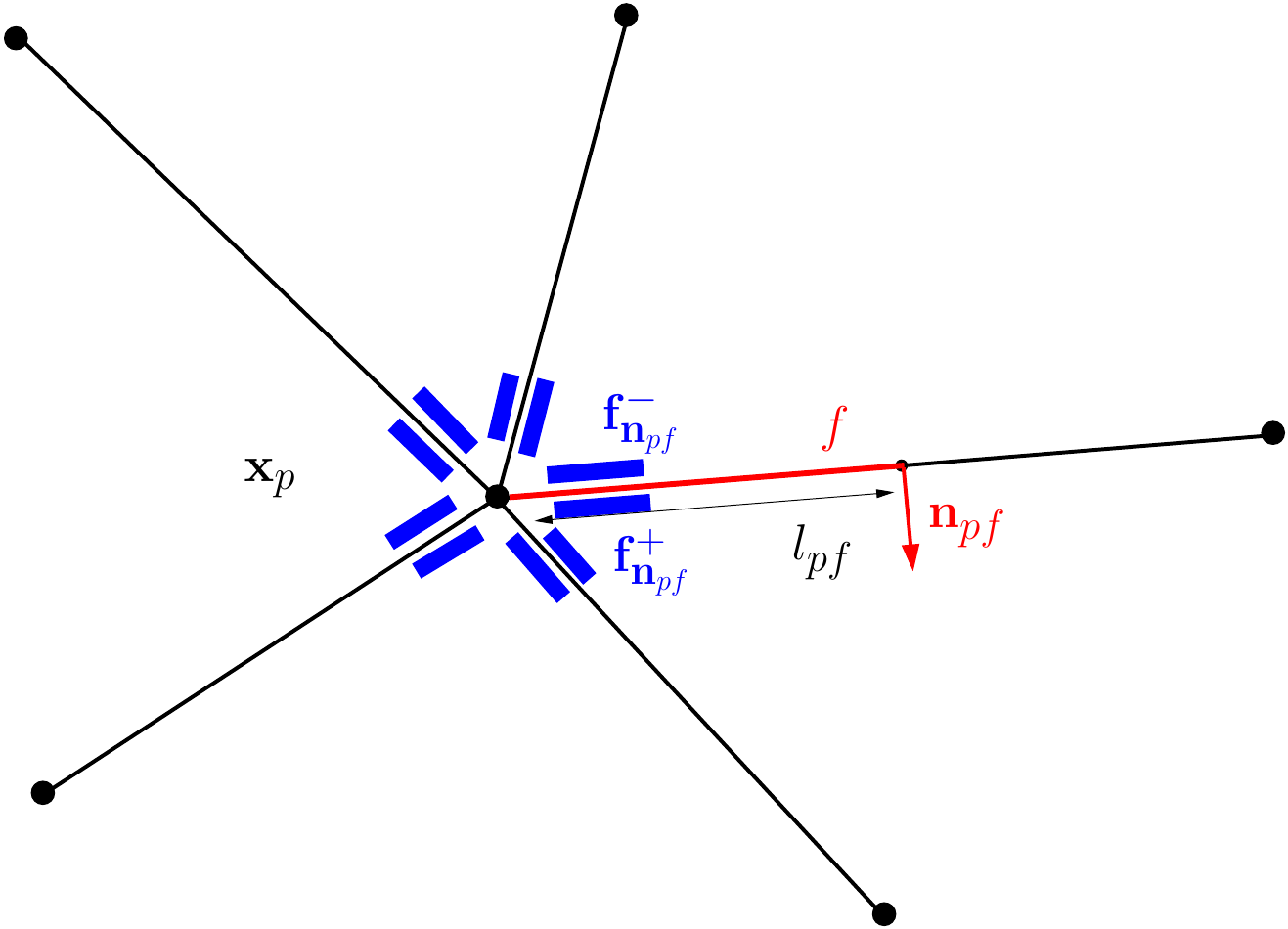}
\caption{Representation of the left- and right-sided fluxes.}
\label{fig:geom2Dpointb}
\end{subfigure}
\caption{Grid fragments in the vicinity of point $p$.}
\label{fig:geom2Dpoint} 
\end{center}
\end{figure}
Note that the sum, over all cells $c$ sharing the point $p$, of the fluxes associated with the subfaces incident to $p$ is exactly equal to the sum of the left- and right-sided fluxes associated with those same subfaces, see Figure~\ref{fig:geom2Dpointb}. This amounts to rewrite condition \eqref{eq:nbc} into
\begin{equation}
  \label{eq:condcons2}
  \sum_{f \in \mathcal{SF}(p)} l_{pf} (\mathbf{f}_{\mathbf{n}_{pf}}^{+}-\mathbf{f}_{\mathbf{n}_{pf}}^{-})=\mathbf{0}.
\end{equation}
Here, $\mathcal{SF}(p)$ denotes the set of subfaces impinging at point $p$. For $f \in \mathcal{SF}(p)$, $l_{pf}$ is the length of subface $f$ and $\mathbf{n}_{pf}$ is its unit normal. In the foregoing equation, $\mathbf{f}_{\mathbf{n}_{pf}}^{-}$ (resp. $\mathbf{f}_{\mathbf{n}_{pf}}^{+}$) denotes the left-sided (resp. right-sided) flux attached to the subface $f$, refer to Figure~\ref{fig:geom2Dpointb}. Thanks to \eqref{eq:iflff} and \eqref{eq:fluxdiffultim} the difference between the left and the right-sided fluxes is expressed in terms of the approximate Riemann solver as
$$
\mathbf{f}_{\mathbf{n}_{pf}}^{+}-\mathbf{f}_{\mathbf{n}_{pf}}^{-}=(\lambda_{l,pf}+\lambda_{r,pf}) \left (v_{\mathbf{n}_{pf}}^{\star}-\overline{v}_{\mathbf{n}_{pf},lr} \right)\begin{pmatrix}
0\\
1\\
\mathbf{0}\\
v_{\mathbf{n}_{pf}}^{\star}
\end{pmatrix}
,
$$
where $\overline{v}_{\mathbf{n}_{pf},lr}$ is normal velocity of face $f$ determined by \eqref{eq:facebasedvel} and $v_{\mathbf{n}_{pf}}^{\star}$ is the unknown normal velocity parameter attached to our approximate Riemann solver. This allows to write the node-based conservation condition \eqref{eq:nbc} component-wise as follows
\begin{subequations}
\label{eq:condconsun}
\begin{align}
&\sum_{f \in \mathcal{SF}(p)} l_{pf} (\lambda_{l,pf}+\lambda_{r,pf}) \left (v_{\mathbf{n}_{pf}}^{\star}-\overline{v}_{\mathbf{n}_{pf},lr} \right)\mathbf{n}_{pf}=\mathbf{0},\label{eq:condconsuna}\\
&\sum_{f \in \mathcal{SF}(p)} l_{pf} (\lambda_{l,pf}+\lambda_{r,pf}) \left (v_{\mathbf{n}_{pf}}^{\star}-\overline{v}_{\mathbf{n}_{pf},lr} \right)v_{\mathbf{n}_{pf}}^{\star}=0.\label{eq:condconsunb}
\end{align}
\end{subequations}
We have $|\mathcal{SF}(p)|$ scalar unknowns, the normal velocity $\overline{v}_{\mathbf{n}_{pf},lr}$ of each subface impinging at $p$, for only $\text{d+1}$ scalar equations. For instance, in the case of a Cartesian grid, we have $|\mathcal{SF}(p)|=4$ and $d+1=3$. Therefore, to close this system of equations, we make the assumption that the parameter $v_{\mathbf{n}_{pf}}^{\star}$ is the projection of the unknown nodal vector $\mathbf{v}_p$ onto the unit normal $\mathbf{n}_{pf}$, that is
\begin{equation}
\label{eq:nodalvelocity}
v_{\mathbf{n}_{pf}}^{\star}=\mathbf{v}_p \cdot \mathbf{n}_{pf},\;\forall\,f\in \mathcal{SF}(p).
\end{equation}
This fundamental assumption drastically reduces the number of unknowns to the vectorial unknown $\mathbf{v}_{p}$, which can be interpreted as an approximation of the nodal velocity. With this assumption the conservation condition \eqref{eq:condconsunb} is equivalent to the conservation condition \eqref{eq:condconsuna}. Finally, the node-based conservation condition \eqref{eq:condconsuna} boils down to the system
\begin{equation}
\label{eq:nodalsolv}
\sum_{f \in \mathcal{SF}(p)} l_{pf} (\lambda_{l,pf}+\lambda_{r,pf})( \mathbf{n}_{pf}\otimes  \mathbf{n}_{pf}) \mathbf{v}_{p}=\sum_{f \in \mathcal{SF}(p)} l_{pf} (\lambda_{l,pf}+\lambda_{l,pf}) \overline{v}_{\mathbf{n}_{pf}} \mathbf{n}_{pf}.
\end{equation}
This system always admits a unique solution that provides the expression of the nodal velocity $\mathbf{v}_p$ which allows to determine the subface fluxes. It is thus called the nodal solver, refer to \cite{Maire2009}.

\subsubsection{Flux characterization}
\label{sssec:fluxcharac}
We observe that the unconventional finite volume discretization introduced above relies on non-unique subface-based fluxes, which leads to the loss of conservation in the classical face-based sense. To overcome this difficulty, we introduced the node-based condition \eqref{eq:nbc} and showed its equivalence with the Residual Distribution consistency condition. It is important to note that the node-based condition can be interpreted as a sufficient conservation condition, derived from enforcing a global balance over the entire computational domain \cite{MaireHDR2011}. Building on the methodology proposed in \cite{Abgrall2018} and described in section \ref{sec:rd:flux}, we demonstrate here the existence of ``half-face'' fluxes on each subface incident to a node, provided that the node-based condition \eqref{eq:nbc} is satisfied. This construction enables us to obtain an explicit expression for each half-face flux, thereby establishing that our unconventional finite volume formulation possesses a local conservation property. In other words, this shows that the node-based condition is not only sufficient, but also necessary.
\begin{figure}
\centering
\includegraphics[width=0.75\textwidth]{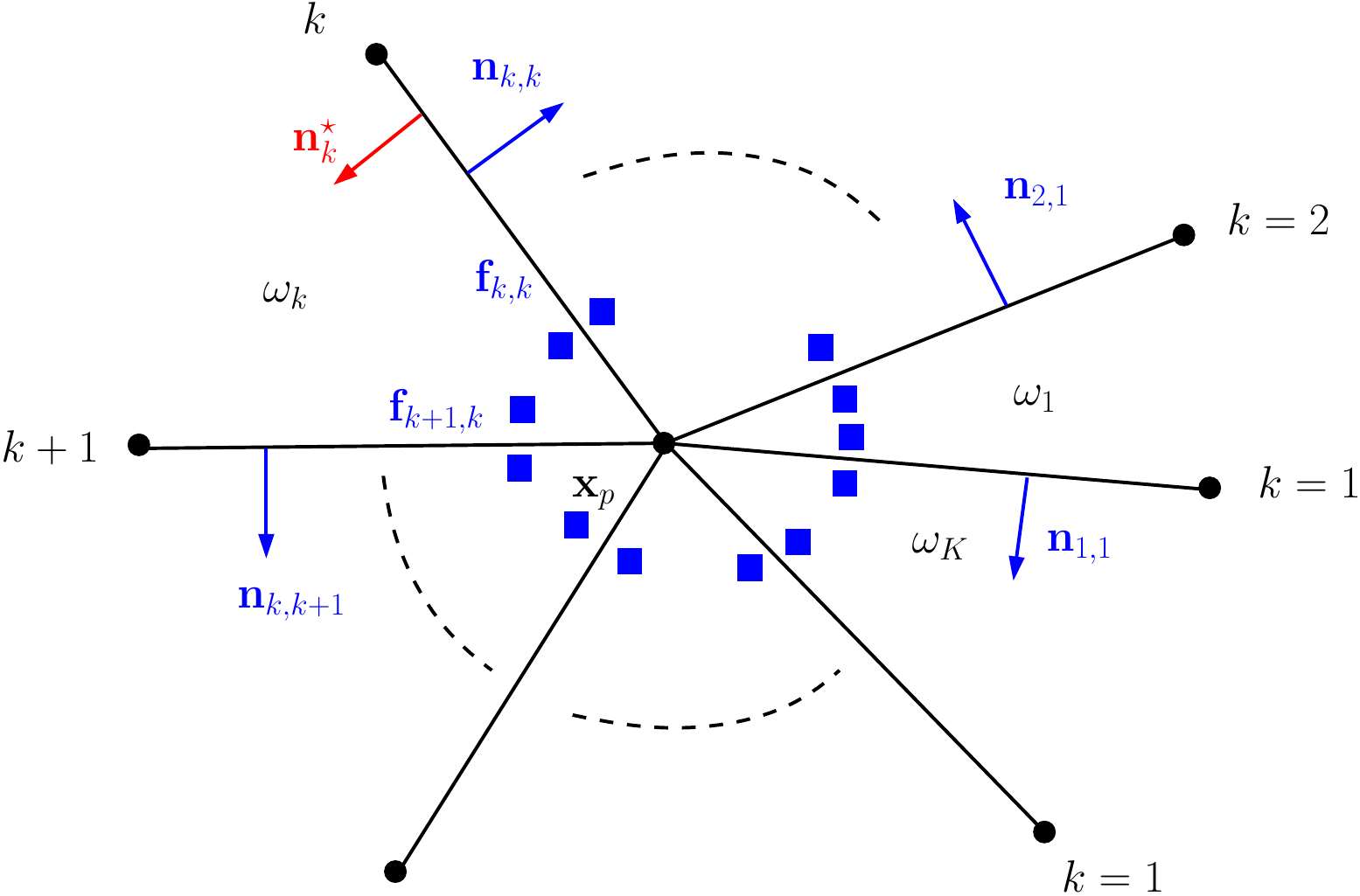}
\caption{Grid fragment in the vicinity of a node: Representation of the cells surrouding the target node; The blue squares corresponds to the subface fluxes; Cells and faces are numbered in the counterclockwise order using cyclic notation.}
\label{fig:cellaroundnode}
\end{figure}

Let us consider the node $p$ surrounded by $K$ cells, that is $\vert\mathcal{C}(p)\vert=K$, where $\mathcal{C}(p)$ is the set of cells surrounding $p$. Before going any further, we introduce some local notations which are displayed in Fig.~\ref{fig:cellaroundnode} to facilitate the presentation. First of all, the generic cell surrounding $p$ is denoted by means of label $k$, where $k=1,\cdots,K$. The faces impinging at $p$ are also numbered using label $k$. More precisely, the set of subfaces belonging to corner $pk$, $\mathcal{SF}(pk)$, is composed of the faces $k$ and $k+1$. Here, $l_{k}$ for $k=1,K$, is the measure of the ``half-face'' $k$. We adopt the convention of a cyclic numbering and thus $l_{K+1} \equiv l_{1}$. The unit normals to faces $k$ and $k+1$ pointing outward cell $k$ are respectively $\mathbf{n}_{k,k}$ and $\mathbf{n}_{k+1,k}$, thus the subface fluxes in the directions of these unit normals are respectively $\mathbf{f}_{k,k}$ and $\mathbf{f}_{k+1,k}$. Here, the first subscript denotes the face, while the second refers to the cell. Using these notations, the node-based condition can be rewritten under the form
\begin{equation}
  \label{eq:nbcnew}
  \sum_{k=1}^{K} l_{k} \mathbf{f}_{k,k}+l_{k+1} \mathbf{f}_{k+1,k}=\mathbf{0}.
\end{equation}
As previously noted, each ``half-face'' $k$ is associated with two subface fluxes, $\mathbf{f}_{k,k-1}$ and $\mathbf{f}_{k,k}$.

Let us investigate the question of the existence of a unique local subface flux associated with each subface. For each ``half-face'' $k$ impinging at $p$, we introduce a unique subface flux $\mathbf{f}_k^{\star}$ defined according to the unit direction $\mathbf{n}_{k}^{\star}$ which points from cell $k-1$ to cell $k$, see Fig.~\ref{fig:cellaroundnode}. If this unique flux exists then it should satisfy the following conservation condition
\begin{equation}
  \label{eq:consconvloc}
  l_{k}\mathbf{f}_{k,k}+l_{k+1}\mathbf{f}_{k+1,k} =-\mathbf{f}_{k}^{\star} +\mathbf{f}_{k+1}^{\star},\;\forall\,k=1\cdots K.
\end{equation}
Here, we point out that the length is included in the definition of the flux $\mathbf{f}_{k}^{\star}$. Finally, the linear system satisfies by the fluxes $\mathbf{f}_{k}^{\star}$ writes
\begin{align*}
  -\mathbf{f}_{1}^{\star} +\mathbf{f}_{2}^{\star}=&l_{1}\mathbf{f}_{1,1}+l_{2}\mathbf{f}_{2,1}, \\
  -\mathbf{f}_{2}^{\star} +\mathbf{f}_{3}^{\star}=&l_{2}\mathbf{f}_{2,2}+l_{3}\mathbf{f}_{3,2}, \\
  \vdots \\
  -\mathbf{f}_{k}^{\star} +\mathbf{f}_{k+1}^{\star}=&l_{k}\mathbf{f}_{k,k}+l_{k+1}\mathbf{f}_{k+1,k}, \\
  \vdots \\
  -\mathbf{f}_{K-1}^{\star} +\mathbf{f}_{K}^{\star}=&l_{K-1}\mathbf{f}_{K-1,K-1}+l_{K}\mathbf{f}_{K,K-1}, \\
  \vdots \\
  -\mathbf{f}_{K}^{\star} +\mathbf{f}_{1}^{\star}=&l_{K}\mathbf{f}_{K,K}+l_{1}\mathbf{f}_{1,K}.
\end{align*}
Summing these equations, we recover the node-based conservation condition \eqref{eq:nbcnew}. We rewrite this linear system under the compact form
\begin{equation}
  \label{eq:lisys}
  \mathbf{A}  \mathbf{f}^{\star}=\mathbf{S},
\end{equation}
where the vector of unknown fluxes and the right-hand side are respectively given by
$$\mathbf{f}^{\star}=(\mathbf{f}_{1}^{\star},\cdots,\mathbf{f}_{K}^{\star})^{\mathtt{T}},\;\text{and}\;\mathbf{S}=(l_{1}\mathbf{f}_{1,1}+l_{2}\mathbf{f}_{2,1},\cdots,l_{K}\mathbf{f}_{K,K}+l_{1}\mathbf{f}_{1,K})^{\mathtt{T}}.$$
The above linear system is characterized by the $K \times K$ matrix
$$\mathbf{A}=
\begin{pmatrix}
  -1 &  1 & 0 & \cdots & \cdots & 0 \\
   0 & -1 & 1 & 0 & \cdots & 0 \\
   \vdots & \ddots & \ddots & \ddots & \ddots & \vdots \\
   \vdots & \ddots & \ddots & \ddots & \ddots & 0 \\
   0 & \cdots & \cdots & 0 & -1 & 1 \\
   1 & 0 & \cdots & \cdots & 0 & -1
\end{pmatrix}
.$$
Computing $\mathbf{L}=\mathbf{A}\mathbf{A}^{\mathtt{T}}$, we get
$$\mathbf{L}=
\begin{pmatrix}
  2 &  -1 & 0 & \cdots & \cdots & -1 \\
   -1 & 2 & -1 & 0 & \cdots & 0 \\
   \vdots & \ddots & \ddots & \ddots & \ddots & \vdots \\
   \vdots & \ddots & \ddots & \ddots & \ddots & 0 \\
   0 & \cdots & \cdots & -1 & 2 & -1 \\
   -1 & 0 & \cdots & \cdots & -1 & 2
\end{pmatrix}
,$$
which is nothing but the matrix of discrete Laplacian operator written over a periodic computational domain. First, it is clear that $\text{Ker} (\mathbf{L})=\text{Span}\{\mathbf{1}_{K}\}$ where $\mathbf{1}_{K}=(1,\cdots,1)^{t}$ and $\text{Im} (\mathbf{L})=\left (\text{Ker} (\mathbf{L})\right)^{\perp}$. Noting that the node-based condition \eqref{eq:nbcnew} can be written as $\mathbf{S}\cdot \mathbf{1}_{K}=0$ directly implies that $\mathbf{S} \in \text{Im} (\mathbf{L})$. Therefor, there exists a unique $\tilde{\mathbf{f}} \in \text{Im}(L)$ such that $\mathbf{L}  \tilde{\mathbf{f}}=\mathbf{S}$. This shows that the solution of \eqref{eq:lisys} might be written under the form $\mathbf{f}^{\star}=\mathbf{A}^{\mathtt{T}} \tilde{\mathbf{f}}$ and thus
\begin{equation}
  \label{eq:solut}
  \mathbf{f}^{\star}=\mathbf{A}^{\mathtt{T}}\mathbf{L}^{-1}  \mathbf{S}.
\end{equation}
It is worth noticing that the solvability of \eqref{eq:lisys} is equivalent to the satisfaction of the node-based condition \eqref{eq:nbcnew}.

We finish this paragraph providing some results useful to obtain a practical expression of the inverse of $\mathbf{L}$. First of all, since $\mathbf{L}$ is a circulant matrix\footnote{A circulant matrix is such that each row is a circular shift of the first row.} its eigenvectors $\mathbf{e}_l$, for $l=0\cdots K-1$, are expressed in terms of the Kth roots of unity
$$\mathbf{e}_l=\left (\alpha_{K}^{0},\alpha_{K}^{l},\alpha_{K}^{2l},\cdots,\alpha_{K}^{l(K-1)}\right)^{\mathtt{T}},$$
where $\alpha_{K}^{l}=\exp(\frac{2\pi li}{K})$ for $l \in [0,K-1]$ and $i^{2}=-1$. This comes from the fact that $\alpha_{K}^{l+K}=\alpha_{K}^{l}$ and one can also easily check that $\{\mathbf{e}_{l} \}_{l=0\cdots K-1}$ is an orthogonal basis for the Hermitian scalar product. The corresponding eigenvalues are given by
$$\lambda_l=2-\alpha_{K}^{l}-\alpha_{K}^{l(K-1)}=2 -2\cos\left (\frac{2\pi l}{K} \right),\;\text{for}\;l=0\cdots K-1.$$
The first eigenvalue is $\lambda_{0}=0$ and the others are non negative. Finally, the inverse of $\mathbf{L}$ defined over $(\text{span} (\mathbf{1}_{K}))^{\perp}$ writes explicitly as
\begin{equation}
  \label{eq:invL}
  \mathbf{L}^{-1}=\sum_{l=1}^{K-1} \frac{1}{\lambda_l} \left (\mathbf{e}_l \otimes \mathbf{e}_l \right).
\end{equation}
This demonstrates that $\mathbf{f}^{\star}$ can be computed explicitly in terms of all the subface fluxes associated with the subfaces incident at node $p$. More precisely, it implies that the unique flux $\mathbf{f}_{k}^{\star}$ attached to the ``half-face'' $k$, depends on all the states surrounding the node-hence the designation of multipoint flux approximation \cite{Gallice2022}.
\subsection{Further readings}
The node-based flux approximation described above leads to a cell-centered finite volume method that is only first-order accurate in space and time for the discretization of the Lagrangian gas dynamics equations. A second-order space--time extension was developed in \cite{Maire2009} using the Generalized Riemann Problem (GRP) methodology introduced by Ben-Artzi and Falcovitz in their monograph \cite{Ben-Artzi2003}. This yields an efficient one-step algorithm in which the equation of state is invoked only once per time step. Building upon this, a third-order space--time extension was proposed in \cite{Vilar2012,Vilar2014} through an original Discontinuous Galerkin method formulated over the initial configuration, i.e. within the total Lagrangian representation of the gas dynamics equations. In this framework, particular attention must be given to the discretization of the deformation gradient tensor, which maps the initial configuration to the current one and must satisfy an involution constraint rigorously equivalent to the Piola compatibility condition. This guarantees consistency between the geometry at t=0t=0 and the geometry at later times.

The unconventional cell-centered finite volume discretization has also been successfully extended to more complex physical models, including nonlinear solid mechanics, both for hypoelastic constitutive laws \cite{Maire2013} and hyperelastic ones \cite{Kluth2010,Boscheri2022}. More recently, its application to Lagrangian ideal magnetohydrodynamics (MHD) has led to a cell-centered finite volume method \cite{Boscheri2023} that rigorously preserves the divergence-free property of the magnetic field at the discrete level.

Finally, we mention the development of a staggered finite volume discretization of Lagrangian hydrodynamics \cite{Maire2011}, in which the momentum equation is discretized on dual cells and the momentum flux is approximated using Riemann solvers associated with the "half-faces" of the dual cells adjacent to cell centers. This results in cell-centered momentum flux approximations that are not uniquely defined, leading to a loss of conservation. Conservation is once again recovered by enforcing a cell-center-based conservation condition, which coincides with the consistency condition of Residual Distribution schemes. This procedure yields both the cell-centered velocity and the cell-centered flux approximation, and it also naturally provides an expression for a multidimensional artificial viscosity.

The first attempt to extend the unconventional Lagrangian finite volume formulation to gas dynamics expressed in an Eulerian framework was presented in \cite{Shen2014_I}. In this work, the authors introduced what they termed a 2D HLLC nodal solver, a multidimensional extension of the classical HLLC Riemann solver. The extension is achieved by prescribing a contact velocity, defined as the projection of the nodal velocity onto the normal of each face incident to the node. This formulation leads to non-unique fluxes and consequently the loss of face-based conservation. As in the Lagrangian case, conservation is restored through a node-based conservation condition, which ultimately reduces to the Lagrangian counterpart. However, within this approach, the ordering of wave speeds in the Riemann solver is not guaranteed, and neither entropy stability nor positivity preservation can be ensured. Despite these limitations, the reported numerical experiments show that this method is remarkably robust: it is largely insensitive to well-known numerical pathologies such as odd--even decoupling and the carbuncle phenomenon \cite{Quirk1994}, while preserving contact discontinuities.

These promising features motivated us to develop a general formalism for systematically extending the unconventional Lagrangian FV formulation to hyperbolic systems of conservation laws written in Eulerian form, provided the system includes a continuity equation. In \cite{Gallice2022}, we introduced a subface flux-based finite volume method for discretizing multidimensional hyperbolic systems on general unstructured grids. The numerical flux approximation is constructed from the notion of a simple Eulerian Riemann solver, first introduced by Gallice in \cite{Gallice2003}. This solver is obtained systematically from its Lagrangian counterpart via the Lagrange-to-Euler mapping, a procedure that transfers desirable properties such as positivity preservation and entropy stability. In this framework, conservation and entropy stability are no longer expressed in the traditional face-based sense but instead arise from a node-based vectorial conservation equation and a scalar entropy inequality, respectively. The resulting multidimensional FV scheme is governed by an explicit time step restriction that guarantees both positivity and entropy stability.

Applied to gas dynamics, this formulation yields an original multidimensional FV scheme that is simultaneously conservative and entropy-stable, with numerical fluxes computed through a nodal solver analogous to that used in Lagrangian hydrodynamics. Its robustness and accuracy have been validated on a range of test problems, where the scheme exhibits strong resistance to the numerical instabilities that often affect classical face-based contact-preserving FV formulations. A three-dimensional extension of the method to general hybrid unstructured grids has recently been presented in \cite{Delmas2025}.

As in the Lagrangian case, these unconventional FV schemes derive conservation from a node-based condition, which ensures local conservation. In addition, the Eulerian node-based condition is formally identical to its Lagrangian counterpart. This follows directly from the fact that Eulerian fluxes are constructed from Lagrangian fluxes through the Lagrange-to-Euler mapping, establishing a rigorous equivalence between the two formulations. This also shows that the Eulerian node-based conservation coincides with the consistency condition attached to the Residuals Distribution scheme. 


\section{Other applications of the RD formalism}
\subsection{Construction of schemes on staggered meshes}\label{sec:stagerred}
In this section, we make a brief summary of the papers \cite{AbgrallTokareva1,AbgrallTokareva2,AbgrallStaggered}. Since the method is essentially similar, and maybe more general in \cite{AbgrallStaggered}, we summarize it.

The system is that of the Euler equations, where the variables are the density $\rho$, the velocity $\bbv$ and the internal energy $e$, i.e.
\begin{equation}
\label{euler:nc}
\begin{split}
\dpar{\rho}{t}&+\text{div }(\rho \bbv)=0\\
\dpar{\bbv}{t}&+\big ( \bbv \cdot\nabla)\bbv+\frac{\nabla p}{\rho}=0\\
\dpar{\rho\varepsilon}{t}&+\big ( \bbv \cdot \nabla)(\rho \varepsilon)+\rho h \text{ div }\bbv.
\end{split}
\end{equation}
The equation of state is of the form $p=p(\rho,e)$ and the numerical examples will be obtained with the perfect equation of state.
This system is not suitable for the solution of problems with discontinuity, but it has the advantage to be written with the variable that are useful to the engineer. There has been a  vast literature on the discretization of the problem \eqref{euler:nc}, and here we show how the techniques described in this paper may help to have systematic manner to approximate it in such a way that the correct weak solutions are recovered at mesh convergence.

The procedure has two steps. First we write a stable scheme for \eqref{euler:nc} that can be written in the form \eqref{forme schemas}. This scheme will produce wrong solutions at mesh convergence. Then we show how to correct it  in order to recover the correct weak solutions. One of the possible problems of the method is that we do not know a priori if the new scheme will still be stable, but all the examples we have run have been run in a stable manner.

We proceed by giving one example. The computational domain, i.e. $\R^d$ to simplify,  is covered by simplex that are denoted generically by $K$. The approximation space for the thermodynamics, i.e. the density and the internal energy is $${\TT}_j=\bigoplus_K \P^r(K)$$ and the approximation for the velocity is $${\KK}_h=\bigg (\bigoplus_K\P^{r+1}(K)\bigg )^d\cap \bigg (C^0(\R^d)\bigg )^d.$$ Hence the thermodynamical parameters can be discontinuous across the faces of the elements $K$. The thermodynamical degrees of freedom are denoted by $\bsigma_{{\TT}}$ and the kinematic degrees of freedom by $\bsigma_{{\KK}}$. They are a priori distinct. To fix ideas, we  will assume that the thermodynamical degrees of freedom are the Lagrange points of $\P^r(K)$, and we denote them by $\sigma_{{\TT}}$. For the kinematic degrees of freedom, we chose here to expand, in each element $K$, the velocity as a sum of B\'ezier polynomial. In that we follow \cite{AbgrallStaggered} where a non linear RD scheme is used to update in time the velocity, the reason for choosing B\'ezier polynomial is that for all $K$, the integral of the basis functions  attached to $K$ is strictly positive. This allows to use a deferred correction technique to update the velocity in time, see \cite{Mario,AbgrallDec,AbgrallStaggered} for more details. Other spatial approximation could be used, such as in \cite{Sixtine1,Sixtine2}, also leading to a mass matrix free approximation.

The semi-discrete scheme for the density is the standard dG scheme,
\begin{subequations}
    \label{eq:staggered}
    \begin{equation}\label{eq:staggered:rho}
        \int_K\varphi_{\sigma_{\TT}}\dfrac{d\rho^h}{dt}\; d\bbx{-\int_K\nabla\varphi_{\sigma_{\TT}} \cdot\big (\rho^h\bbv^h\big )\; d\bbx+\int_{\partial K} \varphi_{\sigma_{\TT}}\hbbf_\bbn^\rho (\bbu_K, \bbu^-)\; d\gamma}=0
    \end{equation}
    The velocity equation is approximated by
    \begin{equation}
        \label{eq:staggered:v}
        \int_{\R^d}\varphi{\sigma_{\KK}}\dfrac{d\bbv^h}{dt}\; d\bbx+\sum_{K}\bigg ({\int_K\varphi_{{\KK}} (\bbv^h\cdot\nabla)\bbv^h\; d\bbx+\int_{K}\varphi_{{\KK}}\dfrac{\nabla p^h}{\rho^h}\; d\bbx}\bigg )=0
    \end{equation}
    and 
    \begin{equation}
        \label{eq:staggered:e}
        \int_K\varphi_{\sigma_{\TT}}\dfrac{d e^h}{dt}\; d\bbx
        +{\int_K\varphi_{{\TT}}\bigg ( \big ( \bbv^h\cdot \nabla\big )\bbv^h+(e^h+p^h)\text{div } \bbv^h\bigg )\; d\bbx}=0
    \end{equation}
\end{subequations}

The relations \eqref{eq:staggered:rho} and \eqref{eq:staggered:e} lead to small and invertible mass matrix like for dG. The numerical flux in \eqref{eq:staggered:rho} is a standard numerical flux. The integrals in \eqref{eq:staggered:v} and \eqref{eq:staggered:e} are evaluated using quadrature formula, as well as the surface integral of \eqref{eq:staggered:rho}. The quadrature formula are chosen so that the formal accuracy is respected. We skip the important details of getting an oscillation free scheme, 
for which we refer to \cite{AbgrallTokareva1,AbgrallTokareva2,AbgrallStaggered}.

From \eqref{eq:staggered}, we see that each of the relations are put in the formalism of relation \eqref{forme schemas}. The conservation relation \eqref{rds_re:conservation} holds true only for the mass, and not for the velocity and the internal energy. 

In this paper, we only consider the case of a scheme that is first order in time, for the sake of simplicity. The extension to higher order is done in the above mentioned references, via a defect correction algorithm because of the velocity equation.

The only important fact is about the structure of the \textit{discrete} temporal evolution equations:  the discrete evolution equations of the density and the internal energy will have the following structure: for all $\sigma_{{\TT}}\in K$,
\begin{subequations}\label{eq:discrete}
\begin{equation}\label{discrete rho}
\vert C_{\sigma_{{\TT}}}\vert\;\big ( \rho_{\sigma_{\TT}}^{n+1}-\rho_{\sigma_{\TT}}^{n}\big )+\Phi_{\sigma_{{\TT}}}^{\rho,K}=0
\end{equation}
and 
\begin{equation}\label{discrete e}
\vert C_{\sigma_{{\TT}}}\vert\;\big ( e_{\sigma_{\TT}}^{n+1}-e_{\sigma_{\TT}}^{n}\big )+\Phi_{\sigma_{{\TT}}}^{e,K}=0
\end{equation}
while the velocity update is done for each $\sigma_{{\KK}}$ by
\begin{equation}\label{discrete v}
\vert C_{\sigma_{{\KK}}}\vert\;
\big ( \bbv_{\sigma_{\KK}}^{n+1}-\bbv_{\sigma_{\KK}}^{n}\big )
+\sum_{K, \sigma_{{\KK}}\in K} \Phi_{\sigma_{{\TT}}}^{\bbv,K}=0.
\end{equation}
\end{subequations}
The residuals $\Phi$  are an approximation of the time-space  integrals of \eqref{eq:staggered}. 

For example,
$$\Phi_{\sigma_{{\TT}}}^{\rho,K}\approx\int_{t_n}^{t_{n+1}}\bigg [\int_K\varphi_{\sigma_{\TT}}\dfrac{d\rho^h}{dt}\; d\bbx-\int_K\nabla\varphi_{\sigma_{\TT}} \cdot \big (\rho^h\bbv^h\big )\; d\bbx+\int_{\partial K} \varphi_{\sigma_{\TT}}\hbbf_\bbn^\rho (\bbu_K, \bbu^-)\; d\gamma
\bigg ] d\xi.
$$

To get a locally conservative scheme, the idea is to modify the residuals in order to force the relations \eqref{rds_re:conservation}. For the density, there is nothing to do. For the velocity, we proceed as follows. 

If we had a Lax-Wendroff theorem to prove, we would have to study, for any $K$, integrals of the form
$$\int_{K} \psi^h(t_n, \bbx)\big (\rho^{n+1}\bbv^{n+1}-\rho^{n}\bbv^{n}\big ),$$
where $\psi^h$ is an approximation of  a smooth test function $\psi$ with compact support. We do not have to worry about accuracy, we must choose $\psi^h$ so that in the limit of mesh refinement, if the solution is bounded and converges in $L^2$ (for example) to some momentum $\rho\bbv$, the sum of these integrals converges to
$$\int_{\R^d}\dpar{\psi}{t}\rho\bbv\; d\bbx.$$
Hence a good choice is to take $\psi^h(s,\bbx)$ in $K\times [t_n,t_{n+1}]$ as its average in $K$ at time $t_n$. We are left with integrals of the type
$$\int_K \big (\rho^{n+1}\bbv^{n+1}-\rho^{n}\bbv^{n}\big )\; d\bbx.$$
Because of the structure of the algorithm \eqref{eq:discrete}, we have to study terms of the form 
$$\int_K \big (\rho^{n+1}\bbv^{n+1}-\rho^{n}\bbv^{n}\big )\; d\bbx$$ and we notice that
$$\rho^{n+1}\bbv^{n+1}-\rho^{n}\bbv^{n}=
\rho^{n+1} \big ( \bbv^{n+1}-\bbv^{n}\big ) +
\bbv^{n}\big ( \rho^{n+1}-\rho^{n}\big ).$$

Introducing $\Delta \bbv_{\sigma_{\KK}}=:\bbv_{\sigma_{\KK}}^{{n+1}}-\bbv_{\sigma_{\KK}}^{{n}}$ and
$\Delta \rho_{\sigma_{{\TT}}}:=\rho_{\sigma_{{\TT}}}^{{n+1}}-\rho_{\sigma_{{\TT}}}^{{n}}$, we can write
\begin{equation*}
\begin{split}\int_K\big (\rho^{n+1}\bbv^{n+1}-\rho^{n}\bbv^{n}\big )\; d\bbx&= \sum_{\sigma_{{\KK}}\in K} \Delta \bbv_{\sigma_{\KK}}\int_K \rho^{{n+1}}\varphi_{\sigma_{\KK}}\; d\bbx\\
&\qquad + \sum_{\sigma_{{\TT}}\in K}\Delta \rho_{\sigma_{{\TT}}}\int_K \bbv^{{n}}\varphi_{\sigma_{\TT}}\; d\bbx.
\end{split}
\end{equation*}

Then we get
\begin{equation}
\label{momentum:2}
\begin{split}
\int_{K}& \big ( \rho^{{n+1}}\bbv^{{n+1}}-\rho^{{n}}\bbv^{{n}}\big ) \; d\bbx\\
& =\sum_{\sigma_{{\KK}}\in K} \Delta \bbv_{\sigma_{\KK}}\int_K \rho^{{n+1}}\varphi_{\sigma_{\KK}}\; d\bbx+ \sum_{\sigma_{{\TT}}\in K}\Delta \rho_{\sigma_{{\TT}}}\int_K \bbv^{{n}}\varphi_{\sigma_{\TT}}\; d\bbx
\\
&\quad=\sum_{\sigma_{{\KK}}\in K} \omega^{\rho,n+1,K}_{{\sigma_{{\KK}}}} 
|C_{\sigma_{{\KK}}}|\Delta \bbv_{\sigma_{\KK}}
 + \sum_{\sigma_{{\TT}}\in K}\omega^{\bbv,n,K}_{\sigma_{\TT}}
|C_{\sigma_{\TT}}|\Delta \rho_{\sigma_{{\TT}}},
\end{split}
\end{equation}
where we have set for simplicity
$$\omega^{\rho,n+1,K}_{{\sigma_{{\KK}}}}:=\dfrac{\int_K \rho^{{n+1}}\varphi_{\sigma_{\KK}}\; d\bbx}{|C_{\sigma_{{\KK}}}|}\; \text{ , }\;
\omega^{\bbv,n,K}_{\sigma_{\TT}}:=\dfrac{\int_K \bbv^{{n}}\varphi_{\sigma_{\TT}}\; d\bbx}{|C_{\sigma_{\TT}}|}.$$

For the velocity, we have:
$$|C_{\sigma_{{\KK}}}|\big ( \bbv_{\sigma_{\KK}}^{{n+1}}-\bbv^{{n}}_{\sigma_{\KK}}\big ) +\sum_{K, \sigma_{\KK}\in K}\Phi^\bbv_{\sigma_{\KK},K}=0.$$

For the density, we have
$$\vert C_{\sigma_{\TT}}\vert \big ( \rho_{\sigma_{{\TT}}}^{{n+1}}-\rho_{\sigma_{{\TT}}}^{{n}}\big ) +
\sum_{K, \sigma_{{\TT}}\in K} \Phi_{\sigma_{{\TT}},K}^{\rho}=0$$
and we note that the sum reduces to one term, hence
$$\vert C_{\sigma_{\TT}}\vert \big ( \rho_{\sigma_{{\TT}}}^{{n+1}}-\rho_{\sigma_{{\TT}}}^{{n}}\big ) +
\Phi^\rho_{\sigma_{{\TT}},K}=0.$$
Using these relations in \eqref{momentum:2}, we get
\begin{equation}
\label{momentum:3}
\begin{split}
\int_{\R^d}&\psi(\bbx,t) \; \big ( \rho^{{n+1}}\bbv^{{n+1}}-\rho^{{n}}\bbv^{{n}}\big ) \; d\bbx+\Delta t_n\underbrace{
\sum_{\sigma_{{\KK}}} \psi_{\sigma_{{\KK}}}^n \omega_{\sigma_{{\KK}}}^{\rho, n+1}
\bigg [ \sum_{K, \sigma_{{\KK}}\in K}\Phi^\bbv_{\sigma_{{\KK}},K}\bigg ]}_{I}\\&
+
\Delta t_n\sum_K\bigg \{\sum_{\sigma_{{\KK}}\in K} \big (\psi_K^n-\psi_{\sigma_{{\KK}}}\big )\omega_{\sigma_{{\KK}}}^{\rho, n+1,K}\Big [ \sum_{K', \sigma_{{\KK}}\in K'\cap K} \Phi^\bbv_{\sigma_b,K}\Big ]\bigg \}
\\
&\qquad+
\Delta t_n\sum_K \psi^n_K \Big [\sum_{\sigma_{{\TT}}\in K}\omega^{\bbv,n,K}_{\sigma_{\TT}}\Phi^\rho_{\sigma_{{\TT}},K}\Big ]
\\
&\qquad\qquad=0
\end{split}
\end{equation}
The term $I$ can be rewritten as 
\begin{equation*}
\begin{split}
\sum_{\sigma_{{\KK}}} &
\psi_{\sigma_{{\KK}}}^n \omega_{\sigma_{{\KK}}}^{\rho, n+1}
\bigg [ \sum_{K, \sigma_{{\KK}}\in K}\Phi^\bbv_{\sigma_{{\KK}},K}\bigg ]=
\sum_{K} \psi^n_K \sum_{\sigma_{{\KK}}\in K} \omega_{\sigma_{{\KK}}}^{\rho, n+1}\Phi^\bbv_{\sigma_{{\KK}},K}\\
&\qquad 
+
\sum_K \bigg [ \sum_{\sigma_{{\KK}}\in K} \big ( \psi^n_K-\psi^n_{\sigma_{{\KK}}}\big )\omega_{\sigma_{{\KK}}}^{\rho, n+1}\Phi^\bbv_{\sigma_{{\KK}},K}
\bigg ]
\end{split}
\end{equation*}
Hence, gathering all together, we get
\begin{equation*}
\begin{split}
\int_{\R^d}&\psi(\bbx,t) \; \big ( \rho^{{n+1}}\bbv^{{n+1}}-\rho^{{n}}\bbv^{{n}}\big ) \; d\bbx\\
 &+
\sum_{K} \psi^n_K\Bigg [  \sum_{\sigma_{{\KK}}\in K} \omega_{\sigma_{{\KK}}}^{\rho, n+1}\Phi^\bbv_{\sigma_{{\KK}},K}
\sum_{\sigma_{{\TT}}\in K}\omega^{\bbv,n,K}_{\sigma_{\TT}}\Phi^\rho_{\sigma_{{\TT},K}}\Bigg ]\\
&\qquad \qquad =0
\end{split}
\end{equation*}

Thus, we obtain the master equation:
\begin{subequations}
\begin{equation}
\label{master:1}
\begin{split}
 \int_{\R^d}\psi(\bbx,t) \; &\big ( \rho^{{n+1}}\bbv^{{n+1}}-\rho^{{n}}\bbv^{{n}}\big ) \; d\bbx 
 \\&+\sum_K \psi_K^n \Bigg [ 
\sum_{\sigma_{\KK}\in K} \omega^{\rho,n+1}_{\sigma_{{\KK}}} \Phi_{\sigma_{\KK},K}^\bbv 
+\sum_{\sigma_{{\TT}}\in K}\omega^{\bbv,n,K}_{\sigma_{{\KK}}}
\Phi^\rho_{\sigma_{{\TT}},K}
 \Big ]\\
 \\
&\qquad+ 
\sum_{K}\Big ( F_K^{\mathbf{m}}+ D_K^{\mathbf{m}}
\sum_{\sigma_{\KK}\in K}D^K_{\sigma_{\KK}}
\Big )=0
\end{split}
\end{equation}
with
 \begin{equation}
\label{master:2}
\begin{split}
F_K^{\mathbf{m}}&=\sum_{\sigma_{{\KK}}\in K} \big ( \psi^n_K-\psi^n_{\sigma_{{\KK}}}\big )\omega_{\sigma_{{\KK}}}^{\rho, n+1}\Phi^\bbv_{\sigma_{{\KK}},K}\\
D_{K}^{\mathbf{m}}&=\sum_{\sigma_{{\KK}}\in K} \big (\psi_K^n-\psi_{\sigma_{{\KK}}}\big )\omega_{\sigma_{{\KK}}}^{\rho, n+1,K}\Big [ \sum_{K', \sigma_{{\KK}}\in K'\cap K} \Phi^\bbv_{\sigma_b,K}\Big ]\\
 \omega_{\sigma_{{\KK}}}^{\rho,n+1,K}&=\dfrac{ \int_K \rho^{{n+1}}\varphi_{\sigma_{\KK}}\; d\bbx}{\vert C_{\sigma_{\KK}}\vert}, \qquad
\omega_{\sigma_{\KK}}^{\rho,n+1}=
\sum\limits_{K, \sigma_{{\KK}}\in K}\omega_{\sigma_{\KK}}^{\rho,n+1,K}
\\
\omega^{\bbv,n,K}_{\sigma_{\TT}}&=\dfrac{\int_K \bbv^{{n}}\varphi_{\sigma_{\TT}}\; d\bbx}{\vert C_{\sigma_{\TT}}\vert}. 
\end{split}
\end{equation}  
\end{subequations}

The same kind of relations can be obtained for the internal energy, and it is possible to show the following result
\begin{proposition}\label{LxW}
Assume that the mesh $\mathcal{T}_h$ is shape regular, we denote by $h$ the maximum diameter of the element of the mesh.
For any $K$, the residuals $\Phi_{\sigma_\mathcal{E},K}^\rho$, $\Phi_{\sigma_\mathcal{E},K}^e$, $\Phi_{\sigma_\mathcal{V},K}^\bbv$ are Lipschitz continuous functions of their arguments, with Lipschitz constant of the form $C\cdot h$, where $C$ only depends on $\alpha$ and the maximum norm of the solution.

Assume that we have a family of meshes $\mathcal{F}=\{\mathcal{T}_{h_k}\}$ with $\lim\limits_{k\rightarrow +\infty} h_k=0$. We denote by $(\bbu_{h_k})_{k\geq 0}$ the sequence of functions fulfilling:
$$\text{ if }t\in [t_n, t_{n+1}[,\; \bbu(\bbx, t)=(\rho(\bbx, t_n), \bbv(\bbx,t_n), e(\bbx,t_n))^{\mathtt{T}}$$
with, if $K$ is the element that exists almost everywhere such that $\bbx\in K$,
$$\rho(\bbx,t_n)=\sum_{\sigma_\mathcal{E}\in K} \rho_{\sigma_\mathcal{T}}^n\varphi_{\sigma_\mathcal{T}}(\bbx), \quad e(\bbx,t_n)= \sum_{\sigma_\mathcal{T}\in K} e_{\sigma_\mathcal{T}}^n\varphi_{\sigma_\mathcal{T}}(\bbx),$$
and $$\bbv(\bbx,t_n)=\sum_{\sigma_\mathcal{K}}\bbv_{\sigma_\mathcal{K}}^n\varphi_{\sigma_\mathcal{K}}(\bbx).$$ Here $\{(\rho^n_{\sigma_\mathcal{T}}), (\bbv^n_{\sigma_\mathcal{K}}),(e^n_{\sigma_\mathcal{T}})\}_{n\geq 0, \sigma_\mathcal{T}, \sigma_\mathcal{K}}$ are defined by the scheme \eqref{eq:discrete}. 

We assume that the density, velocity and internal energy are uniformly bounded and that a subsequence converges in $L^2$ towards $(\rho, \bbv, e)$, where $\rho, e\in L^2(\R^d\times [0,T])$ and $\bbv\in \big (L^2(\R^d\times [0,T]))^d$ . 

 We also assume  that the residuals satisfy
 \begin{equation}\label{momentum}
\sum_{\sigma_\mathcal{K}\in K} \omega^{\rho,n+1}_{\sigma_\mathcal{K}} \Phi_{\sigma_\mathcal{K},K}^\bbv 
+\sum_{\sigma_{\mathcal{T}}\in K}\omega^{\bbv,n,K}_{\sigma_\mathcal{T}}
\Phi^\rho_{\sigma_{\mathcal{T}},K}=\int_{\partial K} \bbf^{\mathbf{m}}(U^{{n}})\cdot \bbn\; d\gamma
\end{equation}
and 
\begin{equation}\label{energy}
\sum_{\sigma_\mathcal{T}\in K}\Phi_{\sigma_\mathcal{T},K}^e+\sum_{\sigma_\mathcal{K}\in K} \theta_{\sigma_\mathcal{K}}^{{\mathbf{m}}}\cdot \Phi_{\sigma_\mathcal{K},K}^\bbv+\frac{1}{2}\sum_{\sigma_{\mathcal{T}}}\theta_{\sigma_\mathcal{T}}^{q^2, K}\Phi_{\sigma_{\mathcal{T}}}=\int_{\partial K}\bbf^E(\bbu^{{n}})\cdot \bbn\; d\gamma,
\end{equation}
where we have set
\begin{equation}\label{coeff:corr}
\begin{split}
\omega_{\sigma_\mathcal{K}}^{\rho,n+1}=\dfrac{\sum\limits_{K, \sigma_\mathcal{K}\in K} \int_K \rho^{{n+1}}\varphi_{\sigma_{\mathcal{K}}}\; d\bbx}{\vert C_{\sigma_\mathcal{K}}\vert}, &\qquad \omega^{\bbv,n,K}_{\sigma_\mathcal{K}}=\dfrac{\int_K \bbv^{{n}}\varphi_{\sigma_\mathcal{T}}\; d\bbx}{\vert C_{\sigma_\mathcal{T}}\vert},\\
\widetilde{{\mathbf{m}}}=\frac{\rho^{{n+1}}\bbv^{{n+1}}+\rho^{{n}}\bbv^{{n}}}{2}, &\qquad \widetilde{q^2}=\bbv^{{n+1}}\cdot \bbv^{{n}},
\\
\theta_{\sigma_\mathcal{K}}^{{\mathbf{m}}}=\dfrac{\sum\limits_{K, \sigma_\mathcal{K}\in K}\int_K\widetilde{{\mathbf{m}}}\varphi_{\sigma_\mathcal{V}}\; dx }{\vert  C_{\sigma_\mathcal{V}}\vert},&\qquad
 \theta_{\sigma_\mathcal{T}}^{q^2, K}=\dfrac{\int_K\widetilde{q^2}\varphi_{\sigma_\mathcal{T}}\; dx}{\vert  C_{\sigma_\mathcal{T}}\vert}
\end{split}
\end{equation}
with the assumption that there exists $C$ independent of $n$, such that $\Delta t\leq C h$. { In \eqref{momentum} (resp. \eqref{energy}), $\bbf^{\mathbf{m}}(\bbu^{{n}})\cdot \bbn$ (resp. $\bbf^E(\bbu^{{n}})\cdot \bbn$) is the momentum component of the normal flux (resp. its energy component).}

Then $V=(\rho, \rho \bbv, e+\tfrac{1}{2}\rho\bbv^2)$ is a weak solution of the problem.
\end{proposition}

\paragraph{How to achieve   discrete conservation.}
Since there is no ambiguity, we drop the dependency of the residuals with respect to the element.

  We show how to slightly modify the original scheme so that the new one will satisfy \eqref{rds_re:conservation},  \eqref{momentum} and \eqref{energy}, and hence if the scheme converges, we have convergence towards a weak solution.
 To achieve this, following \cite{AbgrallTokareva1,paola,Abgrall2018}, we introduce the correction terms in the residuals. This needs to be done only for the velocity and the internal energy.
 
 Knowing at time $t_n$ the solution $(\rho^n, \bbv^n, e^n)$ we obtain with the forward Euler step $(\rho^{n+1}, \bbv^{n+1}, e^{n+1})$. For this, we first compute $\rho^{{n+1}}$ and then perform the update for the velocity and the energy:

\begin{itemize}
 \item \textbf{Momentum.}\par
We introduce a  correction $r^\bbv_{\sigma,K}$ so that 
\begin{equation}
\label{correction:u}
\Psi_{\sigma_{V}}^\bbv=\Phi^{\bbv}_{\sigma_{V}}(\bbu^{{n}})+r^\bbv_{\sigma_\mathcal{V}}
\end{equation}
is such that
\eqref{momentum} holds true for the new set of residuals, 
i.e.
\begin{equation*}
\sum_{\sigma_\mathcal{V}\in K} \omega^{\rho,n+1}_{\sigma_\mathcal{V}} r_{\sigma_\mathcal{V}}^\bbv=\int_{\partial K}\bbf^{\mathbf{m}}(\bbu^{{n}}) \bbn\; d\gamma-\bigg \{\sum_{\sigma_\mathcal{V}\in K} \omega^{\rho,n+1}_{\sigma_\mathcal{V}} \Phi_{\sigma_\mathcal{V},K}^\bbv 
+\sum_{\sigma_{\mathcal{E}}\in K}\omega^{\bbv,n,K}_{\sigma_\mathcal{E}}
\Phi^\rho_{\sigma_{\mathcal{E}},K}\bigg\}.
\end{equation*}
There is no reason to have a different value of $r_{\sigma_\mathcal{V}}^\bbv$ unless for possible special needs, so we set $r_{\sigma_\mathcal{V}}^\bbv=r^\bbv$, and since a priori\footnote{Here we use the fact that the B\'ezier polynomials are $>0$ on K, and assume that the density at $t_n$ is also positive.}
$$\sum_{\sigma_\mathcal{V}\in K} \omega^{\rho,n+1}_{\sigma_\mathcal{V}}>0$$
 we get a unique value of $r^\bbv$
 defined by
 \begin{equation}
\label{1}
\begin{split}
\bigg (\sum_{\sigma_\mathcal{V}\in K} \omega^{\rho,n+1}_{\sigma_\mathcal{V}}\bigg ) r^\bbv=\int_{\partial K} \bbf^{\mathbf{m}}(\bbu^{{n}})& \bbn\; d\gamma-\bigg \{\sum_{\sigma_\mathcal{V}\in K} \omega^{\rho,n+1}_{\sigma_\mathcal{V}} \Phi_{\sigma_\mathcal{V},K}^\bbv \\
&\quad +\sum_{\sigma_{\mathcal{E}}\in K}\omega^{\bbv,n,K}_{\sigma_\mathcal{E}}
\Phi^\rho_{\sigma_{\mathcal{E}},K}\bigg\}.
\end{split}
\end{equation}

Once this is known, we can update the velocity and compute $\bbv_{\sigma_{\mathcal{V}}}^{{n+1}}$.

 \item \textbf{Energy.}\par 
Now we know $\rho^{{n}}$, $\rho^{{n+1}}$, $\bbv^{{n}}$, $\bbv^{{n+1}}$ and $e^{{n}}$, and have the \emph{updated} residuals for the velocity (there is no change for the density). Again we introduce a correction on the energy, $r_{\sigma_\mathcal{E}}^e$, and for the residual
$$\Psi_{\sigma_\mathcal{E}}^e=\Phi_{\sigma_\mathcal{E}}^e(\bbu^{{n}})+r_{\sigma_\mathcal{E}}^e$$ to satisfy \eqref{energy}, we simply need:
\begin{equation}\label{2}\begin{split}
\sum_{\sigma_\mathcal{E}\in K}r_{\sigma_\mathcal{E}}^e=\int_{\partial K} \bbf^E(\bbu^{{n}})& \bbn\; d\gamma-\bigg \{\sum_{\sigma_\mathcal{E}\in K}\Phi_{\sigma_\mathcal{E},K}^e+\sum_{\sigma_\mathcal{V}\in K} \theta_{\sigma_\mathcal{V}}^{{\mathbf{m}}}\cdot \Phi_{\sigma_\mathcal{V},K}^\bbv\\
&\qquad \qquad \qquad \qquad +\frac{1}{2}\sum_{\sigma_{\mathcal{E}}}\theta_{\sigma_\mathcal{V}}^{q^2, K}\Phi_{\sigma_{\mathcal{E}}}\bigg \}.
\end{split}
\end{equation}
 Since there is no reason to favour one degree of freedom with respect to the other ones, we take $r^e_{\sigma_{\mathcal{E}}}=r^e$, and again we can explicitly solve the equation and obtain the energy at the new time instance.

{
\item Modifications for other schemes}
We have presented this conservation recovery method using a class of schemes that might seem a bit narrow. In this section, we want to explain that it is not the case. This can apply to more general schemes, as soon as the update of any variable $w$ (density, velocity, energy), described by degrees of freedom $\sigma$ (point values, averages, moments), can be written as:
$$
\sum_{K, \sigma\in K} \Phi_\sigma^K.$$

\end{itemize}

\subsection{Application of the flux formulation to mesh refinement}
In section \ref{sec:rd:flux}, we have seen that any dG scheme can be reinterpreted as a a finite volume scheme, more precisely that the residuals defined by the dG scheme can be rewritten as a boundary contribution and a sum of internal flux. This idea has been exploited, first in one dimension in \cite{Vilar1}, and then in several dimensions in \cite{VilarAbgrall} for scalar and systems. In in \cite{Vilar2}, this decomposition has been used to enforce domain preservation as  well as better control of the entropy.

Considering a polygonal cell $K$, the idea is first to define a  family of sub-cells $\omega_p$ that cover $K$ and are non overlapping. If the polynomial approximation in $K$ is $k$, we need $N_k=\frac{(k+1)(k+2)}{2}$ such sub-cells to have a chance that the mapping
$$\begin{array}{ccl}
\P^r(K)&\rightarrow &\R^{N_r}\\
u&\mapsto &(\II_p(u))_{1\leq p\leq N_k},
\end{array}$$ where
$$\II_p(u)=\dfrac{\int_{\omega_p} u(\bbx)\; d\bbx}{\vert \omega_p\vert}$$is one to one.
It will be one-to one under a Vandermonde like condition which can be easily achieved. The Figure \ref{fig:vilar:1} shows some examples.
\begin{figure}[!ht]
\begin{center}
\begin{subfigure}{0.45\textwidth}
\centering
\includegraphics[width=\textwidth]{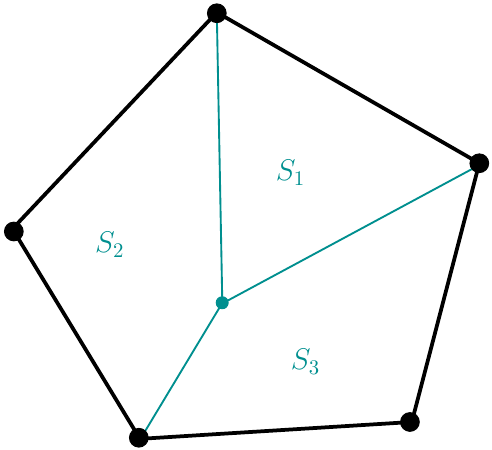}
\end{subfigure}
\begin{subfigure}{0.45\textwidth}
\includegraphics[width=\textwidth]{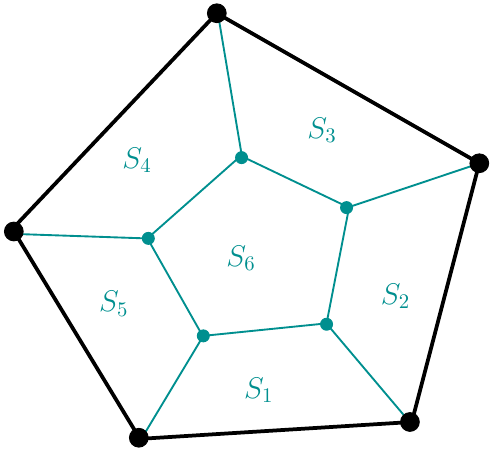}
\end{subfigure}
\begin{subfigure}{0.45\textwidth}
\centering
\includegraphics[width=\textwidth]{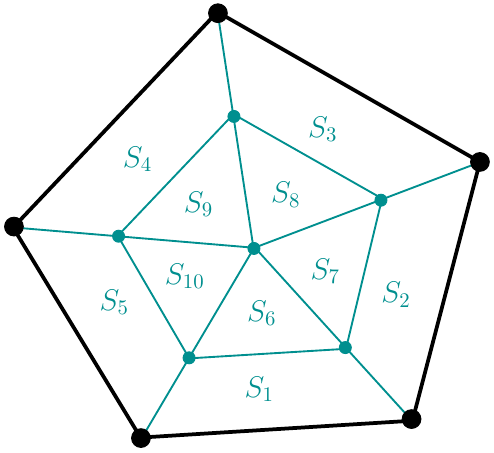}
\end{subfigure}
    \caption{Examples of subdivision for a polygonal cell and DG schemes from $\P^1$ to $\P^3$. This figure is taken from \protect\cite{VilarAbgrall}.}
  \label{fig:vilar:1}
  \end{center}
\end{figure}
The polynomials of $\P^r(K)$ are described in a basis (Lagrange, modal, or other), 
$$u=\sum_{p=1}^{N_k} u_p\varphi_p, \qquad \varphi_p\in \P^r(K)$$
 and we have 
 $$\II_p(u)=\sum_{p'=1}^{N_r} u_p\II_{p}(\varphi_{p'}).$$
 Let us denote by $\bbP$ the matrix
 $$P=\begin{pmatrix}
 \II_{p}(\varphi_{p'})\end{pmatrix}_{1\leq p,p'\leq N_r}$$
 In the original basis, the semi-discrete dG is written as 
 $$\dfrac{d\bbU}{dt}=\bbM^{-1}\Phi$$ where $\bbU$ is the vector of the $u_p$, $\bbM$ is the mass matrix and $\Phi$ the vector of the dG residuals,
 $$\Phi_p=-\int_K \bbf(\bbu)\nabla\varphi_p\; d\bbx+\int_{\partial K}\varphi_p\hbbf_\bbn \; d\gamma.$$
Once this is done,  we obtain an equation on the average of $u$ on each of the $\omega_p$:
\begin{equation}
    \label{eq:vilar:1}
\bbD \dfrac{d\bar \bbu}{dt}=\bbD\bbP\bbM^{-1} \Phi:=\begin{pmatrix}\tilde{\Phi}_1\\
\vdots \\
\tilde{\Phi}_l\\ \vdots \\
\tilde{\Phi}_{N_r}
\end{pmatrix},
\end{equation}
where $\bar \bbu$ is the  vector in $\R^{N_r}$ which components are the $\II_p(u)$ and $\bbD$ is the diagonal matrix which terms are the area $\vert \omega_p\vert$.
We note that
$$\sum_{p=1}^{N_k}\vert \omega_p\vert \II_p(u)=\int_K u\, d\bbx,$$
denoting by $\mathbf{1}$ the diagonal matrix which elements are egal to $1$, we have
$$\mathbf{1}^{\mathtt{T}}\bbD\bar \bbu=\int_K u\; d\bbx.$$
Then,
\begin{lemma}
Whatever the basis $\{\varphi_p\}_{1\leq p\leq N_r}$ of $\P^r$, we have
\begin{equation}
    \label{eq:vilar:2}
\mathbf{1}^{\mathtt{T}}\bbD\bbP\bbM^{-1}\Phi=\int_{\partial K}\hbbf_\bbn\; d\gamma
\end{equation}
\end{lemma}
\begin{proof}
The proof is done in several steps. We first evaluate $\bby^{\mathtt{T}}=\mathbf{1}^{\mathtt{T}}\bbD\bbP$, and then we compute 
$\bby^{\mathtt{T}}\Phi$.
\begin{itemize}
    \item Evaluation of $\bby^{\mathtt{T}}=\mathbf{1}^{\mathtt{T}}\bbD\bbP$, i.e. $\bby=\bbP^{\mathtt{T}}\bbD^{\mathtt{T}}\mathbf{1}$. Since $\bbD$ is diagonal, 
    $\bbD^{\mathtt{T}}=\bbD$ and we see that
    $$\bbD\mathbf{1}=\begin{pmatrix}\vert\omega_1\vert \\\vdots \\ \vert \omega_{N_r}\vert \end{pmatrix}
$$
so that the $l$-th term of the vector $\bbP^{\mathtt{T}}\bbD^{\mathtt{T}}\mathbf{1}$ writes:
$$\sum_{p=1}^{N_r}\vert \omega_p\vert \II_p(\varphi_l)=\sum_{p=1}^{N_r}\int_{\omega_p}\varphi_l\; d\bbx\int_K\varphi_l \; d\bbx$$
This shows that $\bby=\big (\int_K\varphi_1\; d\bbx, \ldots, \int_K\varphi_{N_r}\big )$.
\item Then we compute $\bby^{\mathtt{T}}\bbM^{-1}$, i.e. $\bbM^{-\mathtt{T}}\bby$, and since $\bbM$ is symmetric, we evaluate $\bbM^{-1}\bby=\bbz=(z_1, \ldots, z_{N_r})$.
We will have for all $1\leq p\leq N_r$,
$$y_p=\sum_{k=1}^{N_r} \bbM_{kp}z_k=\int_K\varphi_p\big ( \sum_{k=1}^{N_r} z_k\varphi_k\big ) \; d\bbx$$ and since $$\sum_{k=1}^{N_r} z_k\varphi_k\in \P^r,$$ and $\{\varphi_p\}$ a basis of $\P^r$, we must have
\begin{equation}\label{eq:vilar:3}1=\sum_{k=1}^{N_r} z_k \varphi_k.\end{equation}
\end{itemize}
We can now evaluate $ \mathbf{1}^{\mathtt{T}}\bbD\bbP\bbM^{-1}\Phi=\bby^{\mathtt{T}}\bbM^{-1}\Phi=\bbz^{\mathtt{T}}\Phi^{\mathtt{T}}$ since
\begin{equation*}
    \begin{split}
\bbz^{\mathtt{T}}\Phi^{\mathtt{T}}&=\sum_{k=1}^{N_k} z_k\bigg ( -\int_K  \bbf(\bbu)\nabla\varphi_k\; d\bbx+\int_{\partial K}\varphi_k\hbbf_{\bbn}\; d\gamma\bigg )\\&\quad =
-\int_K\nabla\bigg (\sum_{k=1}^{N_k} z_k \varphi_k\bigg ) \bbf(\bbu)\; d\bbx+
\int_{\partial K} \bigg (\sum_{k=1}^{N_k} z_k \varphi_k\bigg )\hbbf_{\bbn}\; d\gamma\\&\qquad=
\int_{\partial K}\hbbf_{\bbn}\; d\gamma
    \end{split}
\end{equation*}
thanks to \eqref{eq:vilar:3}.
\end{proof}

\bigskip

Hence, we get into the usual framework. Defining the boundary flux by
$$\hbbf^b_{p}=\int_{\partial\omega_p\cap\partial K}\hbbf_\bbn\; d\gamma,$$ and noticing that the boundary flux is $0$ if $\omega_p$ does not intersect $\partial K$, we define
$$\tilde{\Psi}_p=\tilde{\Phi}_p-\hbbf^b_p.$$ The second step is to construct from the sub-cells a graph, i.e. a list of connections between sub-cells. Two cells will be connected if they share a common interface. This defines a graph that is simply connected. Then we can compute flux, denoted by $\FF_{pp'}$, across the face that is the intersection of the sub-cells $\omega_p$ and $\omega_{p'}$ (is the intersection of these two sub-cells is non empty, i.e. if they corresponds to an edge of the graph).

In \cite{VilarAbgrall}, this was used as follows: we had a mechanism to detect "bad" cells, i.e. cells for which the average $\II_p(u)$ is outside of the invariant domain. Then using  a classical bound preserving numerical flux, for example the local Lax-Friedrichs one, we simply replace each of  flux $\FF_{pp'}$ by this Local Lax-Friedrichs flux, and reconstruct a modified dG residual by adding the flux. 

The criteria to flag "bad" cells was a bit heuristic and have been improved in \cite{Vilar2} by blending the flux $\FF_{pp'}$ and the local Lax-Friedrichs flux accross the face between the cells $\omega_p$ and $\omega_{p'}$ in such a way that $\II_p(u)$ will stay in the invariant domain. See \cite{VilarAbgrall} and \cite{Vilar2} for more details.

\subsection{Application to entropy and some differential constraints.}
\subsubsection{The one-dimensional case}
Here, we reused the notations introduced in Section~\ref{sec:motivating:1D}. We are interested in studying the discretization of the system \eqref{eq:RP} equipped with an entropy, {\it i.e.} a concave function $\eta$ of $\mathbf{u}$, and with an entropy flux $g$ such that the weak solution satisfies in addition the following entropy inequality (in the distribution sense)
\begin{equation}
\label{eq:RP:entropie}
\frac{\partial \eta}{\partial t}+\frac{\partial g (\mathbf{u})}{\partial x} \geq 0.
\end{equation}
We follow the same lines as in Section~\ref{sec:motivating:1D}, that is starting from the notion of entropic Godunov-type Finite Volume method we put in evidence the fluctuations associated with the entropy flux and the consistency condition satisfied by these fluctuations to ensure entropy stability. The numerical solution provided by the Godunov-type scheme \eqref{eq:God} by means of the Riemann solver $\mathbf{r}(\mathbf{u}_l,\mathbf{u}_r,\tfrac{x}{t})$
$$\mathbf{u}_i^{n+1}=\frac{1}{2} \left (\mathbf{u}_{i-\frac{1}{2}}^{+}+\mathbf{u}_{i+\frac{1}{2}}^{-} \right),$$
where $\mathbf{u}_{i-\frac{1}{2}}^{+}$ and $\mathbf{u}_{i-\frac{1}{2}}^{+}$ are given by \eqref{eq:FVS}, satisfies
\begin{equation}
\label{eq:Gode:entro}
\eta \left (\mathbf{u}_i^{n+1}\right):=\eta_i^{n+1}\geq \frac{1}{2}\left [ \eta(\mathbf{u}_{i-1/2}^+)+\eta(\mathbf{u}_{i+\frac{1}{2}}^-)\right ].
\end{equation}
By virtue of entropy concavity and the Jensen inequality \footnote{For a concave function $\eta$ there holds $$\eta \left [\frac{1}{b-a} \int_{a}^{b} f(x)\,\mathrm{d}x \right ] \geq \frac{1}{b-a} \int_{a}^{b} \eta \left [f(x) \right ]\,\mathrm{d}x.$$} we arrive at
\begin{align*}
\eta \left (\mathbf{u}_{i-\frac{1}{2}}^{+}\right ) \geq & \eta_{i-\frac{1}{2}}^{+}:=\frac{2}{\Delta x_i} \int_{x_{i-\frac{1}{2}}}^{x_{i}} \eta \left [\mathbf{r}(\mathbf{u}_{i-1}^{n}, \mathbf{u}_{i}^{n},\frac{x-x_{i-\frac{1}{2}}}{\Delta t}) \right ]\,\mathrm{d}x,\\
\eta \left (\mathbf{u}_{i+\frac{1}{2}}^{-}\right ) \geq &\eta_{i+\frac{1}{2}}^{-}:=\frac{2}{\Delta x_i} \int_{x_{i}}^{x_{i+\frac{1}{2}}} \left [\mathbf{r}(\mathbf{u}_{i}^{n}, \mathbf{u}_{i+1}^{n},\frac{x-x_{i+\frac{1}{2}}}{\Delta t})\right ]\,\mathrm{d}x.
\end{align*}
Then, \eqref{eq:Gode:entro} becomes
$$\Delta x_i \eta_i^{n+1}\geq \frac{\Delta x_i}{2} \eta_{i-\frac{1}{2}}^{+} + \frac{\Delta x_i}{2} \eta_{i+\frac{1}{2}}^{-}.$$
Introducing $\eta_{i}^{n}$ in the foregoing inequality turns it into
\begin{equation}
\label{eq:Godm:entro}
\Delta x_i\left (\eta_i^{n+1}-\eta_i^{n}\right ) +\frac{\Delta x_i}{2}\left [ \left ( \eta_i^{n}-\eta_{i+1/2}^{-})\right)-\left (\eta_{i-\frac{1}{2}}^{+}-\eta_i^{n} )\right )\right ] \geq 0.
\end{equation}
This allows to identify the entropy fluxes attached to the cell interfaces $x_{i-\frac{1}{2}}$ and $x_{i+\frac{1}{2}}$
\begin{align*}
g_{i-\frac{1}{2}}^{+}=&g(\mathbf{u}_i^n) -\frac{\Delta x_i}{2 \Delta t} \left (\eta_{i}^{n}-\eta_{i-\frac{1}{2}}^{+} \right) \\
g_{i+\frac{1}{2}}^{-}=&g(\mathbf{u}_i^n) -\frac{\Delta x_i}{2 \Delta t} \left (\eta_{i+\frac{1}{2}}^{-}-\eta_{i}^{n} \right).
\end{align*}
The entropy inequality \eqref{eq:Godm:entro} rewrites under the flux form
\begin{equation}
\label{eq:Godm:entroflux}
\Delta x_i\left (\eta_i^{n+1}-\eta_i^{n}\right ) +\Delta t \left (g_{i+\frac{1}{2}}^{-}-g_{i-\frac{1}{2}}^{+} \right ) \geq 0.
\end{equation}
Let us recall that approximate Riemann solver $\mathbf{r}(\mathbf{u}_l,\mathbf{u}_r,\frac{x}{t})$ is consistent with the integral form of the entropy inequality if and only if
\begin{equation}
\label{eq:HLLentrop}
g_{i+\frac{1}{2}}^{+} -g_{i+\frac{1}{2}}^{-} \geq 0.
\end{equation}
This is the famous HLL consistency condition, refer to \cite{Harten1983}.

Introducing the entropy flux evaluated at the cell averaged value into \eqref{eq:Godm:entroflux} leads us to rewrite it under the form
$$\Delta x_i\left (\eta_i^{n+1}-\eta_i^{n}\right ) +\Delta t \left (g_{i+\frac{1}{2}}^{-}-g(\mathbf{u}_i^n)+g(\mathbf{u}_i^n) -g_{i-\frac{1}{2}}^{+} \right ) \geq 0,$$
which allows us to identify the fluctuations of entropy flux attached to the cell interfaces $x_{i-\frac{1}{2}}$ and $x_{i+\frac{1}{2}}$
\begin{equation}
\label{eq:fluc:entro}
\Phi_{g,i}^{i-\frac{1}{2}}=g(\mathbf{u}_i^n) -g_{i-\frac{1}{2}}^{+}\;\text{and}\; \Phi_{g,i}^{i+\frac{1}{2}}=g_{i+\frac{1}{2}}^{-}-g(\mathbf{u}_i^n).
\end{equation}
Observing that
$$g_{i+\frac{1}{2}}^{+}-g_{i+\frac{1}{2}}^{-}=g(\mathbf{u}_{i+1}^n)-\Phi_{g,i+1}^{i+\frac{1}{2}}-\Phi_{g,i}^{i+\frac{1}{2}}-g(\mathbf{u}_i^n),$$
we claim that the HLL consistency condition with the integral form of the entropy inequality is equivalent to the following consistency condition imposed on the entropy fluctuations at node $x_{i+\frac{1}{2}}$.
\begin{equation}
\label{eq:conscondrd:entropy}
\Phi_{g,i}^{i+\frac{1}{2}}+\Phi_{g,i+1}^{i+\frac{1}{2}} \leq g_{i+1}^n-g_i^n,
\end{equation}
where $g_i^n:=g(\mathbf{u}_{i}^n)$ and $g_{i+1}^n:=g(\mathbf{u}_{i+1}^n)$. Let us point out the right-hand side of the above inequality is nothing but the integral of $\frac{\partial g}{\partial x}$ over the dual cell $[x_i,x_{i+1}]$, where $x_i$ is the midpoint of the cell $[x_{i-\frac{1}{2}},x_{i+\frac{1}{2}}]$. 

It is worth observing that condition \eqref{eq:conscondrd:entropy} is rigorously equivalent to the classical entropy stability condition introduced by Tadmor in \cite{TadmorVieux,TadmorActa}. To show this, we start  from the Finite  Volume scheme semi-discretized in time
\begin{equation}
\label{eq:FVgene}
\Delta x_i \frac{\mathrm{d} \mathbf{u}}{\mathrm{d} t} + \hat{\mathbf{f}}_{i-\frac{1}{2}}-\hat{\mathbf{f}}_{i+\frac{1}{2}}=\mathbf{0},
\end{equation}
where $\hat{\mathbf{f}}=\hat{\mathbf{f}} (\mathbf{u}_l,\mathbf{u}_r)$ is a consistent numerical flux defined at the cell interface $x_{i+\frac{1}{2}}$, {\it i.e.}, $\hat{\mathbf{f}} (\mathbf{u},\mathbf{u})=\mathbf{f}(\mathbf{u})$. Dot-multiplying the above equation by the entropy variable $\mathbf{w}_i:=\frac{\partial \eta }{\partial \mathbf{u}}(\mathbf{u}_i)$ yields
\begin{equation}
\label{eq:entropfirst}
\Delta x_i  \mathbf{w}_i \cdot \frac{\mathrm{d} \mathbf{u}}{\mathrm{d} t}+ \mathbf{w}_i\cdot \left (\hat{\mathbf{f}}_{i+\frac{1}{2}}-\hat{\mathbf{f}}_{i-\frac{1}{2}}\right) =0.
\end{equation}
We define the arithmetic average and the jump of $(\mathbf{w}_i,\mathbf{w}_{i+1})$ at the interface $x_{i+\frac{1}{2}}$
$$\langle \mathbf{w} \rangle_{i+\frac{1}{2}}=\frac{1}{2} \left (\mathbf{w}_i+\mathbf{w}_{i+1} \right),\;\text{and}\;\llbracket \mathbf{w} \rrbracket_{i+\frac{1}{2}}=\mathbf{w}_{i+1}-\mathbf{w}_{i}.$$
Noticing that
$$\mathbf{w}_i=\langle \mathbf{w} \rangle_{i-\frac{1}{2}}+\frac{1}{2}\llbracket \mathbf{w} \rrbracket_{i-\frac{1}{2}}=\langle \mathbf{w} \rangle_{i+\frac{1}{2}}-\frac{1}{2}\llbracket \mathbf{w} \rrbracket_{i+\frac{1}{2}},$$
we obtain
\begin{align*}
\mathbf{w}_i\cdot \hat{\mathbf{f}}_{i-\frac{1}{2}}=& \langle \mathbf{w} \rangle_{i-\frac{1}{2}} \cdot \hat{\mathbf{f}}_{i-\frac{1}{2}}+\frac{1}{2}\llbracket \mathbf{w} \rrbracket_{i-\frac{1}{2}}\cdot \hat{\mathbf{f}}_{i-\frac{1}{2}} \\
\mathbf{w}_i\cdot \hat{\mathbf{f}}_{i+\frac{1}{2}}=& \langle \mathbf{w} \rangle_{i+\frac{1}{2}} \cdot \hat{\mathbf{f}}_{i+\frac{1}{2}}-\frac{1}{2}\llbracket \mathbf{w} \rrbracket_{i+\frac{1}{2}}\cdot \hat{\mathbf{f}}_{i+\frac{1}{2}}.
\end{align*}
Following Tadmor, we introduce the  potential $\psi= \mathbf{w}\cdot \mathbf{f}-g$, which statisfies $\frac{\partial \psi}{\partial \mathbf{w}}=\mathbf{f}$. Then, it is quite natural to define the numerical entropy flux in terms of the potential $\psi$ as follows
$$\hat{g}_{i+\frac{1}{2}}:=\langle \mathbf{w} \rangle_{i+\frac{1}{2}} \cdot  \mathbf{f}_{i+\frac{1}{2}}-\langle \psi \rangle_{i+\frac{1}{2}},$$
and we can easily check that the proposed numerical entropy flux is consistent with the physical entropy flux. Gathering the previous results we get
\begin{align*}
\mathbf{w}_i\cdot \mathbf{f}_{i-\frac{1}{2}}=&\hat{g}_{i-\frac{1}{2}}+\langle \psi \rangle_{i-\frac{1}{2}}+\frac{1}{2} \llbracket \mathbf{w} \rrbracket_{i-\frac{1}{2}}\cdot \hat{\mathbf{f}}_{i-\frac{1}{2}},\\
\mathbf{w}_i\cdot \mathbf{f}_{i+\frac{1}{2}}=&\hat{g}_{i+\frac{1}{2}}+\langle \psi \rangle_{i+\frac{1}{2}}-\frac{1}{2} \llbracket \mathbf{w} \rrbracket_{i+\frac{1}{2}}\cdot \hat{\mathbf{f}}_{i+\frac{1}{2}}.
\end{align*}
Finally, substituting the foregoing expressions into \eqref{eq:entropfirst} leads to
\begin{equation}
\label{eq:entroprod}
\Delta x_i \frac{\mathrm{d} \eta_i}{\mathrm{d} t}+\hat{g}_{i+\frac{1}{2}}-\hat{g}_{i-\frac{1}{2}}=\frac{1}{2} \left (\llbracket \mathbf{w} \rrbracket_{i-\frac{1}{2}}\cdot \hat{\mathbf{f}}_{i-\frac{1}{2}}-\llbracket \psi \rrbracket_{i-\frac{1}{2}}+ \llbracket \mathbf{w} \rrbracket_{i+\frac{1}{2}}\cdot \hat{\mathbf{f}}_{i+\frac{1}{2}}-\llbracket \psi \rrbracket_{i+\frac{1}{2}}\right ).
\end{equation}
Observing the right-hand side of the above entropy equation and following Tadmor \cite{TadmorVieux,TadmorActa} we define the notion of entropy stable scheme as follows
\begin{definition}[Entropy-stable scheme]
  \label{def:ES}
The Finite Volume scheme \eqref{eq:FVgene} is entropy-stable, for a concave entropy, if the numerical flux at each cell interface satisfies
\begin{equation}
\label{eq:Tadmor}
\hat{\mathbf{f}}_{i+\frac{1}{2}} \cdot \llbracket \mathbf{w} \rrbracket_{i+\frac{1}{2}}-\llbracket \psi \rrbracket_{i+\frac{1}{2}} \geq 0.
\end{equation}
It is entropy-conservative when we have an equality.
\end{definition}
The right-hand side of the entropy equation \eqref{eq:entroprod} represents the entropy production within cell $[x_{i-\frac{1}{2}},x_{i+\frac{1}{2}}]$.
To establish the link between the Tadmor's condition \eqref{eq:Tadmor} and the consistency condition expressed in terms of the entropy flux fluctuations \eqref{eq:conscondrd:entropy} let us rewrite the Finite Volume scheme in terms of the flux fluctuations
We also have
$$\Delta x_i \frac{\mathrm{d} \mathbf{u}_i}{\mathrm{d}t} +\mathbf{\Phi}_{i}^{i+\frac{1}{2}}+\mathbf{\Phi}_{i}^{i-\frac{1}{2}}=\mathbf{0},$$
where the flux fluctuations at node $x_{i+\frac{1}{2}}$ steming from cells $i$ and $i+1$ are
$$\mathbf{\Phi}_{i}^{i+\frac{1}{2}}=\hat{\mathbf{f}}_{i+\frac{1}{2}}-\mathbf{f}(\mathbf{u}_i),\;\text{and}\; \mathbf{\Phi}_{i+1}^{i+\frac{1}{2}}=\mathbf{f}(\mathbf{u}_{i+1})-\hat{\mathbf{f}}_{i+\frac{1}{2}}.$$
To determine the flux fluctuations contributions to the entropy production we compute
\begin{align*}
  \mathbf{w}_i\cdot \mathbf{\Phi}_{i}^{i+\frac{1}{2}}+ \mathbf{w}_{i+1}\cdot  \mathbf{\Phi}_{i+1}^{i+\frac{1}{2}}=&\mathbf{w}_i\cdot \left [\hat{\mathbf{f}}_{i+\frac{1}{2}}-\mathbf{f}(\mathbf{u}_i) \right]+\mathbf{w}_{i+1} \cdot \left [ \mathbf{f}(\mathbf{u}_{i+1})-\hat{\mathbf{f}}_{i+\frac{1}{2}} \right ],\\
  =& -\hat{\mathbf{f}}_{i+\frac{1}{2}}\cdot \llbracket \mathbf{w} \rrbracket_{i+\frac{1}{2}}+\llbracket \mathbf{f} \cdot \mathbf{w} \rrbracket_{i+\frac{1}{2}},\\
  =& -\hat{\mathbf{f}}_{i+\frac{1}{2}}\cdot \llbracket \mathbf{w} \rrbracket_{i+\frac{1}{2}}+\llbracket \psi \rrbracket_{i+\frac{1}{2}} +g_{i+1}-g_{i}.
  \end{align*}
Here, we have made use of the thermodynamic potential definition $\psi=\mathbf{w}\cdot \mathbf{f}-g$. Finally, we claim that 
\begin{equation}
\label{eq:entropfluctuat}
\mathbf{w}_i\cdot  \mathbf{\Phi}_{i}^{i+\frac{1}{2}}+ \mathbf{w}_{i+1}\cdot  \mathbf{\Phi}_{i+1}^{i+\frac{1}{2}}\leq g_{i+1}-g_i,
\end{equation}
is equivalent to Tadmor's condition 
$$\hat{\mathbf{f}}_{i+\frac{1}{2}}\cdot \llbracket \mathbf{w} \rrbracket_{i+\frac{1}{2}}-\llbracket \psi \rrbracket_{i+\frac{1}{2}}\geq 0.$$
\subsubsection{The multidimensional extension}
\label{eq:entropmultiD}
Here, we describe extension of \eqref{eq:entropfluctuat} to the multidimensional case and we reuse the notations of Section~\ref{sec:motivating:FV:II}. Before proceeding any further, let us recall the Finite Volume scheme resulting from the discretization over the polygonal cell $\omega_c$ of the system of conservation laws \eqref{hyper} equipped with the entropy inequality \eqref{entropyinequality}
\begin{equation}
  \label{eq:FVsemidis}
  |\omega_c| \frac{\mathrm{d} \mathbf{u}_c}{\mathrm{d} t} +\sum_{p \in \mathcal{P}(c)} \sum_{f \in \mathcal{SF}(pc)} l_{pcf} \mathbf{f}_{pcf}=\mathbf{0},
\end{equation}
where $\mathbf{f}_{pcf}$ is some consistent numerical flux defined in normal direction $\mathbf{n}_{pcf}$ of the subface $f$ belonging to the corner $pc$. By virtue of the geometrical identity
$$\sum_{p \in \mathcal{P}(c)} \sum_{f \in \mathcal{SF}(pc)} l_{pcf} \mathbf{n}_{pcf}=\mathbf{0},$$
we rewrite the Finite Volume discretization under the fluctuation form
$$|\omega_c| \frac{\mathrm{d} \mathbf{u}_c}{\mathrm{d} t} +\sum_{p \in \mathcal{P}(c)} \mathbf{\Phi}_{c}^{p}=\mathbf{0},$$
where the fluctation attached to corner $pc$ reads
$$\mathbf{\Phi}_{c}^{p}=\sum_{f \in \mathcal{SF}(pc)} l_{pcf} \left [\mathbf{f}_{pcf}-\mathbf{f}(\mathbf{u}_c)  \mathbf{n}_{pcf} \right].$$
The natural multidimensional extension of the entropy-stable condition \eqref{eq:entropfluctuat} writes
\begin{equation}
\label{eq:entropyCond}
\sum_{c\in \mathcal{C}(p)} \mathbf{w}_c\cdot \mathbf{\Phi}_{c}^{p} \leq \int_{\partial\omega_p}\mathbf{g} (\mathbf{u})\cdot \mathbf{n}\,\mathrm{d}\gamma.
\end{equation}
Here, $(\eta,\mathbf{g})$ is entropy-entropy flux pair which satisfies the entropy inequality
$$\frac{\mathrm{d} \eta}{\mathrm{d} t}+\nabla \cdot (\mathbf{g}(\mathbf{u})) \geq 0.$$
Indeed, in the context of Finite Volume methods characterized by ``multidimensional'' numerical fluxes, {\it cf.} Section~\ref{sec:motivating:FV:II} and Section~\ref{sec:rd:flux},  one  shows immediately that condition \eqref{eq:entropyCond} is equivalent to (see Figure.~\ref{fig:polygrid-mr} and Section~\ref{sec:motivating:FV:II} for the notation)
\begin{equation}
\label{eq:entropyCond-flux}
\sum_{c\in\mathcal{C}(p)} \sum_{f \mathcal{SF}(pc)} l_{pcf} \left [\mathbf{f}_{pcf} -\mathbf{f}(\mathbf{u}_c)  \mathbf{n}_{pcf} \right ]\cdot \mathbf{w}_c \leq -\sum_{c\in\mathcal{C}(p)} \sum_{f \mathcal{SF}(pc)} l_{pcf} \mathbf{g}(\mathbf{u}_c) \cdot \mathbf{n}_{pcf},
\end{equation}
which is the natural multidimensional generalization of Tadmor's shuffle condition, see Definition~\ref{def:ES}. It is also worth pointing out that a similar node-based entropy condition has been derived in \cite{Gallice2022} for a multidimensional Finite Volume discretization of the Euler equations.
As shown in \cite{Abgrall2018}, it is always possible to modify of the locally conservative scheme \eqref{eq:FVsemidis} so that it remains locally conservative while enforcing its entropy stability, that is it satisfies \eqref{eq:entropyCond}. To this end, we consider the fluctuation form of \eqref{eq:FVsemidis} assuming that the residuals $\left \{\mathbf{\Phi}_{c}^{p}\right \}_{c \in \mathcal{C}(p)}$ satisfy the conservation condition
$$\sum_{c \in \mathcal{C}(p)} \mathbf{\Phi}_{c}^{p} =\int_{\partial \omega_p} \mathbf{f}(\mathbf{u}) \mathbf{n}\,\mathrm{d}\gamma.$$
Following \cite{Abgrall2018} where this construction was provided in the finite element context described in section \ref{sec:motivating}, let us introduce the following modification of the residuals $\{\mathbf{\Phi}_c^p\}_{c\in \mathcal{C}(p)}$
$$\widetilde{\mathbf{\Phi}}_c^p=\mathbf{\Phi}_c^p+\alpha_p \mathbf{A}_0 (\overline{\mathbf{u}}_p ) \left (\mathbf{w}_c-\overline{\mathbf{w}}_p\right ).$$
Here, $\mathbf{A}_{0}$ is the inverse of the entropy Hessian with respect to the conservative variable, {\it i.e.},
$$\mathbf{A}_0=\left (\frac{\partial \mathbf{u}}{\partial \mathbf{w}} \right)^{-1}.$$
By virtue of the strict concavity of $\eta$, this matrix is negative definite. In addition, $\overline{\mathbf{u}}_p$ and $\overline{\mathbf{w}}_p$ are some average values of $\mathbf{u}_c$ and $\mathbf{w}_c$ at the node $p$, for instance we can define them as their arithmetic averages
$$\overline{\mathbf{u}}_p=\frac{1}{|\mathcal{C}(p)|}\sum_{c\in \mathcal{C}(p)}\mathbf{u}_c,\;\text{and}\; \overline{\mathbf{w}}_p=\frac{1}{|\mathcal{C}(p)|} \sum_{c\in \mathcal{C}(p)} \mathbf{w}_c.$$
By construction
$$\sum_{c \in \mathcal{C}(p)} \mathbf{w}_c -\overline{\mathbf{w}}_p=\mathbf{0},$$
thus the $\{\widetilde{\mathbf{\Phi}_c^p}\}_{c \in \mathcal{C}(p)}$ still satisfy the conservation condition since
$$\sum_{c \in \mathcal{C}(p)} \widetilde{\mathbf{\Phi}}_{c}^{p}=\sum_{c \in \mathcal{C}(p)} \mathbf{\Phi}_{c}^{p}.$$
We are now in position to determine $\alpha_p$ such that and we can enforce the entropy stability condition \eqref{eq:entropyCond}. First, let us compute
\begin{align*}
  \sum_{c \in \mathcal{C}(p)} \widetilde{\mathbf{\Phi}}_{c}^{p}\cdot \mathbf{w}_c &-\int_{\partial \omega_p} \mathbf{g}(\mathbf{u})\cdot \mathbf{n}\,\mathrm{d}\gamma= \sum_{c \in \mathcal{C}(p)} \mathbf{\Phi}_{c}^{p}\cdot \mathbf{w}_c +\alpha_p \mathbf{A}_{0}(\overline{\mathbf{u}}_{p}) (\mathbf{w}_{c}-\overline{\mathbf{w}}_{p}) \cdot \mathbf{w}_c \\
 & \qquad\qquad\qquad -\int_{\partial \omega_p} \mathbf{g}(\mathbf{u})\cdot \mathbf{n}\,\mathrm{d}\gamma\\
  =&\sum_{c \in \mathcal{C}(p)} \alpha_p \mathbf{A}_{0}(\overline{\mathbf{u}}_{p}) (\mathbf{w}_{c}-\overline{\mathbf{w}}_{p}) \cdot (\mathbf{w}_{c}-\overline{\mathbf{w}}_{p}) +\mathbf{\Phi}_{c}^{p}\cdot \mathbf{w}_c
  \\
  &\qquad\qquad\qquad -\int_{\partial \omega_p} \mathbf{g}(\mathbf{u})\cdot \mathbf{n}\,\mathrm{d}\gamma.
\end{align*}
Setting
$$\mathcal{D}_p=\sum_{c \in \mathcal{C}(p)} \mathbf{A}_{0}(\overline{\mathbf{u}}_{p}) (\mathbf{w}_{c}-\overline{\mathbf{w}}_{p}) \cdot (\mathbf{w}_{c}-\overline{\mathbf{w}}_{p}) <0,$$
and
$$\mathcal{E}_p=\int_{\partial \omega_p} \mathbf{g}(\mathbf{u})\cdot \mathbf{n}\,\mathrm{d}\gamma-\sum_{c \in \mathcal{C}(p)}\mathbf{\Phi}_{c}^{p}\cdot \mathbf{w}_c,$$
leads us to claim that the entropy stability condition holds true if and only if $\alpha_p$ satisfies $\alpha_p \mathcal{D}_p -\mathcal{E}_p \leq 0$, that is
\begin{equation}
  \label{eq:condalpha}
  \alpha_p \geq \frac{\mathcal{E}_p}{\mathcal{D}_p},
\end{equation}
since by virtue of the \emph{strict} concavity of entropy, $\mathbf{A}_0(\overline{\mathbf{u}}_p )$ is a positive negative matrix, and $\mathcal{D}_p < 0$ if not all the $\mathbf{w}_c$, {\it i.e.}, the $\mathbf{u}_c$ are equal.

It is important to see under which condition(s), the ratio $\frac{\mathcal{E}_p}{\mathcal{D}_p}$ is bounded, to get a bounded value of $\alpha_p$ knowing that the value of $\alpha_p$ might have an influence on the CFL number. We show this is correct, provided we are away from vaccum, and that the flux $\bbf_{pcf}$ are Lipschitz continuous.

\begin{proof}
In the following, we shall write that the flux $\mathbf{f}$ and the entropy flux $\mathbf{g}$ depend on $\mathbf{w}$ and not $\mathbf{u}$ to simplify the writing. This is not an issue since the mapping $\mathbf{w} \mapsto \mathbf{u}$ is one-to-one by virtue of the strict concavity of the entropy Hessian. 

Using the notations of figure \ref{fig:polygrid-mr}, and \eqref{eq:nbc} as well as \eqref{eq:geomiden}, we see that 
\begin{align*}
\mathcal{E}_p=&\sum_{c \in \mathcal{C}(p)} \sum_{f \in \mathcal{SF}(pc)} l_{pcf} \left \{ -\mathbf{g}(\mathbf{w}_c)\cdot \mathbf{n}_{pcf} -\left [\mathbf{f}_{pcf} -\mathbf{f}(\mathbf{w}_c)  \mathbf{n}_{pcf} \right ]\cdot \mathbf{w}_c \right \}, \\
=&\sum_{c \in \mathcal{C}(p)} \sum_{f \in \mathcal{SF}(pc)} l_{pcf} \left \{ \left [\mathbf{f}(\mathbf{w}_c)^{\mathtt{T}}  \mathbf{w}_c-\mathbf{g}(\mathbf{w}_c) \right ]\cdot \mathbf{n}_{pcf} -\mathbf{f}_{pcf}\cdot \mathbf{w}_c \right \} \\
&=\sum_{c \in \mathcal{C}(p)} \sum_{f \in \mathcal{SF}(pc)} l_{pcf} \left [ \mathbf{\psi}(\mathbf{w}_c)\cdot \mathbf{n}_{pcf} -\mathbf{f}_{pcf}\cdot \mathbf{w}_c \right ].
\end{align*}
Here, $\mathbf{\psi}$ is the potential function given by
$$\mathbf{\psi}(\mathbf{w}) \cdot \mathbf{n}=\mathbf{f}(\mathbf{w})  \mathbf{n} \cdot \mathbf{w}-\mathbf{g}(\mathbf{w})^\mathtt{T} \mathbf{n},\;\forall \mathbf{n} \in \mathbb{R}^{d}.$$
Finally, $\mathcal{E}_p$ might be rewritten
$$\mathcal{E}_p=\sum_{c \in \mathcal{C}(p)} \sum_{f \in \mathcal{SF}(pc)} l_{pcf} \left \{ \left [\mathbf{\psi}(\mathbf{w}_c)-\mathbf{\psi}(\overline{\mathbf{w}}_p)\right] \cdot \mathbf{n}_{pcf} -\mathbf{f}_{pcf}\cdot (\mathbf{w}_c-\overline{\mathbf{w}}_p) \right \},$$
since by virtue of the conservation condition
$$\sum_{c \in \mathcal{C}(p)} \sum_{f \in \mathcal{SF}(pc)} \mathbf{f}_{pcf}=\mathbf{0}.$$
Knowing that the potential satisfies $\nabla_{\mathbf{w}}\psi=\mathbf{f}$, we see that, using Taylor formula with integral rest, 
\begin{equation*}
\begin{split}\big (\psi(\mathbf{w}_c)&-\psi(\overline{\mathbf{w}}_p)\big )\cdot \mathbf{n}_{pcf}-\big ( \mathbf{w}_c-\overline{\mathbf{w}}_p\big )\cdot \mathbf{f}_{pcf}\\
&=
\big ( \mathbf{f}(\mathbf{w}_c)\cdot \mathbf{n}_{pcf}\big ) \cdot ( \mathbf{w}_c-\overline{\mathbf{w}}_p )+\bigg (\int_0^1 \dfrac{\partial^2\psi}{\partial \bbw^2}(s\bbw_c+(1-s)\bar \bbw_p)(1-s)\; ds\bigg )\big (\bbw_c-\bar\bbw_p)^2\\
&\qquad - \big ( \mathbf{f}(\mathbf{w}_c)\cdot \mathbf{n}_{pcf}\big ) \cdot ( \mathbf{w}_c-\overline{\mathbf{w}}_p )
-
\big ( \mathbf{w}_c-\overline{\mathbf{w}}_p\big )\cdot \big ( \mathbf{f}_{pcf}-\mathbf{f}(\mathbf{w}_c)\cdot \mathbf{n}_{pcf}\big )\\
&=\bigg (\int_0^1\nabla_{\mathbf{w}}\mathbf{f} (s\bbw_c+(1-s)\bar\bbw_p) \; (1-s) ds\bigg ) (\bbw_c-\bar\bbw_p)^2\\&
\qquad \qquad\qquad \qquad -\big ( \mathbf{w}_c-\overline{\mathbf{w}}_p\big )\cdot \big ( \mathbf{f}_{pcf}-\mathbf{f}(\mathbf{w}_c)\cdot \mathbf{n}_{pcf}\big )
\end{split}
\end{equation*}
Since $\mathbf{f}_{pcf}$ is assumed to be Lipschitz continuous, we see that
$$\Vert \EE_p\Vert \leq C \; \bigg (\max\limits_{c\in \CC(p), f\in \SS\PP(pc)} l_{pcf}\bigg ) \; \sum\limits_{c\in \CC(p)} \Vert \mathbf{w}_c-\overline{\mathbf{w}}_p \Vert^2,$$
where $C$ is a bound on $\nabla_\bbw \bbf$ and assuming the existence of a constant $C'$ such that $$\Vert \DD_p\Vert \geq C' \sum\limits_{c\in \CC(p)} \Vert \mathbf{w}_c-\overline{\mathbf{w}}_p \Vert^2,$$
we get the boundedness condition.

The constant $C$ and $C'$ depends on the $L^\infty$ norm of $\mathbf{w}$: hence we need to assume that $\mathbf{u}$ stay bounded.
For the Euler equations, the last condition holds true if the solution is away from vaccum.
\end{proof}
\begin{rem}[Adaptation with the other methods sketched in section \ref{sec:motivating}]
$ $
\begin{enumerate}
  \item  In \cite{Abgrall2018}, a similar estimate was given for the dG scheme, where the entropy flux is not defined exactly as in Tadmor, mostly for ease of computation via the Taylor formula, and for some continuous finite element like methods.

\item In \cite{AbgrallOffnerRanocha}, this technique has been extended to satisfy several differential constraints. One of the application was to get an entropy satisfying scheme that is also kinetic preserving, in the dG framework. 
\end{enumerate}

\end{rem}
In \cite{AbgrallDumbserBusto}, this technique has been used in a completely different purpose. 
The model of 
\cite{PeshRom2014,HPRmodel} can be used for fluid and solid if one adapt correctly  the parameters $\chi_1$ and $\chi_2$ below. It writes:
\begin{subequations}\label{eq:dum}
\begin{equation}\label{eq:dum:1}
\dpar{\rho}{t}+\text{ div }(\rho \bbv)-\text{ div }(\varepsilon \nabla\rho)=0
\end{equation}
\begin{equation}
\label{eq:dum:2}
\begin{split}
\dpar{(\rho \bbv)}{t}&+\text{ div }\big (\rho\bbv\otimes\bbv+p\Id\big )+\text{div }\big (\mathbf{\sigma}+\mathbf{\omega}\big )-\text{ div }\big (\varepsilon\nabla(\rho \bbv)\big )=0
\end{split}
\end{equation}
\begin{equation}
\label{eq:dum:3}
\begin{split}
\dpar{(\rho \eta)}{t}&+\text{div }\big (\rho\eta\bbv+\mathbf{\beta}\big )-\text{ div }\big (\varepsilon \nabla\big(\rho\eta\big )=\Pi+\dfrac{\mathbf \alpha:\mathbf \alpha}{\chi_1T}+\dfrac{\bbeta^2}{\chi_2}\geq 0
\end{split}
\end{equation}
\begin{equation}
\label{eq:dum:4}
\begin{split}
    \dpar{\bbA}{t}&+\text{ div }\big ( \bbA \bbv\big )+\bbv\text{curl }\bbA-\text{ div }\big ( \varepsilon\nabla\bbA\big )=-\dfrac{\mathbf\alpha}{\chi_1}
    \end{split}
\end{equation}
\begin{equation}
\label{eq:dum:5}
\begin{split}
\dpar{\bbJ}{t}&+\text{ div }\big ( \bbJ\bbv+\bbT\big )-\bbv\text{ curl }\bbJ-\text{div }\big (\varepsilon\nabla\bbJ\big )=-\dfrac{\mathbf\beta}{\chi_2}
\end{split}
\end{equation}
\begin{equation}
\label{eq:dum:6}
\begin{split}
\dpar{E}{t}&+\text{ div }\big ( (\EE_1+\EE_2+\EE_3+\EE4)\bbv\big )+\text{ div }\bigg (\big (p\Id+\mathbf\omega+\mathbf\sigma\big )\bbv\big )-\text{div }\big ( \varepsilon\nabla E\big )=0
\end{split}
\end{equation}
\end{subequations}
with  $\bbA\in M_d(\mathbb{R})$ and $\bbJ\in \mathbb{R}^d$ in addition to the usual notations. The total energy is $  E =\mathcal{E}_1  +\mathcal{E}_2 {\, + \, \mathcal{E}_3 + \mathcal{E}_4} $ with $\mathcal{E}_1$ being the internal energy, $\EE_2$ the kinetic energy, and $\EE_3$, $\EE_4$ energy densities linked to $\bbA$ and $\bbJ$ respectively. Last, we have $\mathbf\alpha=\dpar{\EE}{\bbA}\in M_d(\mathbb{R})$ and $\mathbf\beta=\dpar{\EE}{\bbJ}\in \mathbb{R}^d$.  More details can be found in \cite{PeshRom2014}.  The system above belongs to the class of overdetermined hyperbolic systems because it can be shown that \eqref{eq:dum:5} is a consequence of the other relations in \eqref{eq:dum}. More precisely, one can write a relation of the form
\begin{equation}
    \label{eq:dum:7}
\eqref{eq:dum:6}=r\times \eqref{eq:dum:1}+\bbv \eqref{eq:dum:2}+T\times \eqref{eq:dum:3}+\mathbf\alpha:\eqref{eq:dum:4}+\mathbf\beta\cdot , \eqref{eq:dum:5}
\end{equation}
where $\theta$ is the temperature (hence one needs a complete equation of state) and $r=\dpar{\EE}{\rho}$. One can see \cite{AbgrallDumbserBusto} for more details.

Using the RD formalism, it was possible to design a class of schemes that discretise \eqref{eq:dum:1}-\eqref{eq:dum:5} which guaranty local conservation, including for the total energy $E$. This means that one, by tuning appropriately a numerical dissipation  from the same principles as what we have described above for the entropy, it becomes possible to get a relation that discretely mimic \eqref{eq:dum:7} so that we have local conservation of the total entropy. This was done using finite volume and  discontinuous Galerkin approaches. One may consult \cite{AbgrallDumbserBusto} for more details.


\section{Conclusions, perspective, and one outlier.}

This paper has presented a  truly multidimensional setting allowing to embed conservation constraints
when discretizing hyperbolic systems. We have focused on homogeneous systems,
for which conservation  is  equivalent to the preservation in
 time of appropriate integrals on the mesh of the densities of several physical 
quantities  (mass, momentum, energy, entropy  etc).

Several aspects that fit naturally in the framework discussed here  been left out, or remain under investigation. For example, we have not discussed how the presence of source terms  modifies  the notion
of conservation. It is well known that even in one dimension the notion of trivial (constant) solution
is not suited. One must thus account for the existence of non-trivial
states which satisfy the steady limit of the PDE, and which may require an enhanced approximation.
This notion is often referred to as well-balanced.   In 
\cite{abgrall2022hyperbolic} we have provided an in depth discussion of how the present 
framework can be easily extended to this case, using the notion of so-called global fluxes
which allows to recast the non-homogenous system in pseudo-conservative form.
An interesting particular case are  steady equilibria involving the existence of a (however complex) set of stationary invariants $V$ depending 
on the physical unknowns, and on some external data. For these steady solutions one has $v=v_0$ constant. 
For these problems, the  source terms may be simply included in the definition of the residual $\Phi$, and the design of well-balanced schemes boils down to  two elements: the choice of the appropriate quadrature; the design of a numerical dissipation compatible with the 
discrete equilibrium.  Some examples are discussed in \cite{abgrall2022hyperbolic} and references therein.
 Examples in the framework of multidimensional finite volume solvers with corner fluxes can be found in \cite{PHRaph2}.

A more ambitious objective is the preservation of general steady states, which do not admit any set of stationary  invariants. There are many interesting such states (as e.g. isolated vortices, as well as many other
potential flow solutions). In this case the problem is much more complex as it requires embedding the preservation of a (possibly non-homogenous) differential constraint. Depending on the PDE system, this constraint
can be of the type div = d$_0$ or curl=c$_0$.  The preservation of these constraints
is indeed possible in the framework discussed here. In the homogeneous case, promising  work on the curl involution for  linear  homogeneous problems using finite volumes with corner fluxes 
has been presented in \cite{Barsukow_nodeconservative2025}.  On Cartesian meshes,
a general methodology to embed  the preservation of  discrete analogs of the  div = d$_0$ constraints 
are discussed in  \cite{brt25,bcrt25}. In both cases, a full residual $\Phi$ which is a cell integral
of the flux divergence plus the source is shown to play a crucial role. 
The ideas discussed here are also linked to earlier work e.g. by Mishra and Tadmor
 or Torrilhon \cite{mishra2011constraint,jeltsch2006curl}.
We believe the framework discussed here is very well suited to further pursued these objectives on arbitrary meshes.  

Another very interesting aspect is the use of the ideas of this paper to generalize to arbitrary high
order  the finite volume setting using  corner fluxes. The main issue is somewhat related to quadrature.
The question is how to  increase the accuracy without including edge Riemann problems based on
one-dimensional fluxes which may  destroy some of the nice properties obtained with point numerical fluxes.
This is perhaps related to another aspect which is how to include in the present framework
methods which use approximations based on some moments of the solutions or other non-Lagrange 
approach (modal expansions, Taylor expansions etc.).

To end this paper, let us make the following comments. 
It seems that all the schemes we are aware of can be cast into the framework we have described here, except maybe one class. This exception is the so-called "active flux" method. This methodology uses meshes that can be triangular, Cartesian, or made of polygons, see \cite{PampaPoly,Calhoun2023-ms,Barsukow}. Using a globally continuous approximation, these methods combine moments (which are logically stored in the elements) and point values (which are localized at the boundary of elements). It was first described by P.L. Roe and his students, see e.g.  \cite{AF1,AF2,AF3,EymannRoe2013,Maeng,He} for an early contribution. This was later formalized in \cite{abgrall_AF}, where in particular a Lax-Wendroff like theorem is shown. High order extensions were first considered in \cite{BarsukowAbgrall}. One particular feature of these methods is that the conservation constraint are only applied to the first moment, i.e. the average. The point values can be updated without conservation constraint. This has been used for example in \cite{BPPAMPA_1D} to stabilize non linearly the method. In \cite{PampaPoly}, this is combined with a Virtual finite element framework, in order to handle general polygons.  In \cite{arXiv:2508.17147} has been identified a link between the discontinous Galerkin approach and this methodology.
These methods seem to have surprising properties, see e.g. \cite{Barsukow} for low Mach flows, or \cite{arXiv:2508.17147} for built in bound preserving properties.

There is still a lot to do!

\section*{Acknowledgements.}
RA has been partially funded by SNSF grant 200020\_204917 "Structure preserving and fast methods for hyperbolic systems of conservation laws".

This work would not have been possible without the help of many colleagues,  former students and postdocs:  L. Arpaia,   W. Barsukow, W. Boscheri, S. Busto, P. Bacigaluppi, M. Ciallella,   H. Deconinck, V. Delmas, M. Dumbser, A.G. Filippini, E. Gaburro, G. Gallice, A. del Grosso,  D. Kuzmin, A. Larat, R. Loub\`ere,  Y. Liu, M. Mezine, L. Micalizzi, S. Michel, P. \"Offner,  B. Rebourcet, D. de Santis, M. Shaskhov, S. Tokareva,  D. Torlo, F. Vilar,  K. Wu.

\bibliographystyle{plain}

\bibliography{./biblio_am}

\begin{thebibliography}{100}

\bibitem{abgrall_eno}
R.~Abgrall.
\newblock On essentially non-oscillatory schemes on unstructured meshes:
  {Analysis} and implementation.
\newblock {\em J. Comput. Phys.}, 114(1):45--58, 1994.

\bibitem{Abgrall99}
R.~Abgrall.
\newblock {Toward the ultimate conservative scheme: following the quest.}
\newblock {\em J. Comput. Phys.}, 167(2):277--315, 2001.

\bibitem{energie}
R.~{Abgrall}.
\newblock {Essentially non-oscillatory residual distribution schemes for
  hyperbolic problems.}
\newblock {\em J. Comput. Phys.}, 214(2):773--808, 2006.

\bibitem{abgrall:dgrds}
R.~Abgrall.
\newblock A residual method using discontinuous elements for the computation of
  possibly non smooth flows.
\newblock {\em Adv. Appl. Math. Mech}, 2010.

\bibitem{AbgrallDec}
R.~{Abgrall}.
\newblock {High order schemes for hyperbolic problems using globally continuous
  approximation and avoiding mass matrices.}
\newblock {\em J. Sci. Comput.}, 73(2-3):461--494, 2017.

\bibitem{Abgrall2018}
R.~Abgrall.
\newblock A general framework to construct schemes satisfying additional
  conservation relations. application to entropy conservative and entropy
  dissipative schemes.
\newblock {\em J. Comput. Phys.}, 372:640--666, 2018.

\bibitem{abgrall_AF}
R.~Abgrall.
\newblock A combination of residual distribution and the active flux
  formulations or a new class of schemes that can combine several writings of
  the same hyperbolic problem: application to the 1d {Euler} equations.
\newblock {\em Commun. Appl. Math. Comput.}, 5(1):370--402, 2023.

\bibitem{AbgrallStaggered}
R.~Abgrall.
\newblock Staggered schemes for compressible flow: a general construction.
\newblock {\em {SIAM} J. Sci. Comput.}, 46(1):a399--a428, 2024.

\bibitem{ABGRALL2019274}
R.~Abgrall, P.~Bacigaluppi, and S.Tokareva.
\newblock High-order residual distribution scheme for the time-dependent euler
  equations of fluid dynamics.
\newblock {\em Comput. Math. Appl.}, 78(2):274--297, 2019.

\bibitem{paola}
R.~Abgrall, P.~Bacigaluppi, and S.~Tokareva.
\newblock A high-order nonconservative approach for hyperbolic equations in
  fluid dynamics.
\newblock {\em Comput. Fluids}, 169:10--22, 2018.

\bibitem{BarsukowAbgrall}
R.~Abgrall and W.~Barsukow.
\newblock Extensions of active flux to arbitrary order of accuracy.
\newblock {\em ESAIM, Math. Model. Numer. Anal.}, 57(2):991--1027, 2023.

\bibitem{AbgrallDumbserBusto}
R.~Abgrall, S.~Busto, and M.~Dumbser.
\newblock A simple and general framework for the construction of
  thermodynamically compatible schemes for computational fluid and solid
  mechanics.
\newblock {\em Appl. Math. Comput.}, 440:40, 2023.
\newblock Id/No 127629.

\bibitem{abgralldeSantisSISC}
R.~Abgrall and D.~de~Santis.
\newblock High-order preserving residual distribution schemes for
  advection-diffusion scalar problems on arbitrary grids.
\newblock {\em {SIAM} J. Sci. Comput.}, 36(3):A955--A983, 2014.

\bibitem{BPPAMPA_1D}
R.~Abgrall, M.~Jiao, Y.~Liu, and K.~Wu.
\newblock Bound preserving {P}oint-{A}verage-{M}oment
  {P}olynomi{A}l-interpreted ({PAMPA}) scheme: one-dimensional case.
\newblock {\em Commun. Comput. Phys.}, 2025.
\newblock In press, also arXiv:2412.03423.

\bibitem{Karni}
R.~Abgrall and S.~Karni.
\newblock A comment on the computation of non-conservative products.
\newblock {\em J. Comput. Phys.}, 229(8):2759--2763, 2010.

\bibitem{abgrallLarat}
R.~Abgrall, A.~Larat, and M.~Ricchiuto.
\newblock {Construction of very high order residual distribution schemes for
  steady inviscid flow problems on hybrid unstructured meshes.}
\newblock {\em J. Comput. Phys.}, 230(11):4103--4136, 2011.

\bibitem{ABGRALL20091314}
R.~Abgrall, A.~Larat, M.~Ricchiuto, and C.~Tavé.
\newblock A simple construction of very high order non-oscillatory compact
  schemes on unstructured meshes.
\newblock {\em Comput. Fluids}, 38(7):1314--1323, 2009.
\newblock Special Issue Dedicated to Professor Alain Lerat on the Occasion of
  his 60th Birthday.

\bibitem{AbgrallTokareva2}
R.~Abgrall, K.~Lipnikov, N.~Morgan, and S.~Tokareva.
\newblock Multidimensional staggered grid residual distribution scheme for
  {Lagrangian} hydrodynamics.
\newblock {\em {SIAM} J. Sci. Comput.}, 42(1):a343--a370, 2020.

\bibitem{abgrall2021relaxation}
R.~Abgrall, E.~Le M\'el\'edo, P.~\"Offner, and D.~Torlo.
\newblock Relaxation {Deferred} {Correction} {Methods} and their {Applications}
  to {Residual} {Distribution} {Schemes}.
\newblock {\em The {SMAI} J. Comput. Math.}, 8:125--160, 2022.

\bibitem{arXiv:2508.17147}
R.~Abgrall, P.~{\"O}ffner, and Y.~Liu.
\newblock Some new properties of the {PamPa} scheme.
\newblock Preprint, {arXiv}:2508.17147, 2025.

\bibitem{AbgrallOffnerRanocha}
R.~Abgrall, P.~{\"O}ffner, and H.~Ranocha.
\newblock Reinterpretation and extension of entropy correction terms for
  residual distribution and discontinuous {Galerkin} schemes: application to
  structure preserving discretization.
\newblock {\em J. Comput. Phys.}, 453:24, 2022.
\newblock Id/No 110955.

\bibitem{RD-ency2}
R.~Abgrall and M.~Ricchiuto.
\newblock High-order methods for cfd.
\newblock In {\em Encyclopedia of Computational Mechanics Second Edition},
  pages 1--54. John Wiley \& Sons, Ltd, 2017.

\bibitem{AbgrallRoe}
R.~Abgrall and P.~L. Roe.
\newblock {High-order fluctuation schemes on triangular meshes.}
\newblock {\em J. Sci. Comput.}, 19(1-3):3--36, 2003.

\bibitem{abgrall:shu}
R.~Abgrall and C.-W. Shu.
\newblock {Development of residual distribution schemes for discontinuous
  Galerkin methods}.
\newblock {\em Commun. Comput. Phys.}, 5:376--390, 2009.

\bibitem{AbgrallTokareva1}
R.~Abgrall and S.~Tokareva.
\newblock Staggered grid residual distribution scheme for {Lagrangian}
  hydrodynamics.
\newblock {\em {SIAM} J. Sci. Comput.}, 39(5):a2317--a2344, 2017.

\bibitem{abgrall2022hyperbolic}
R{\'e}mi Abgrall and Mario Ricchiuto.
\newblock Hyperbolic balance laws: residual distribution, local and global
  fluxes.
\newblock In {\em Numerical fluid dynamics. Methods and computations}, pages
  177--222. Singapore: Springer, 2022.

\bibitem{PampaPoly}
Rémi Abgrall, Yongle Liu, and Walter Boscheri.
\newblock Bound preserving {P}oint-{A}verage-{M}oment
  {P}olynomi{A}l-interpreted ({PAMPA}) on polygonal meshes.
\newblock {\em submitted}, 2024.
\newblock Preprint, {arXiv}:2502.1006.

\bibitem{dervieux3}
F.~Angrand, V.~Boulard, A.~Dervieux, J.~P{\'e}riaux, and G.~Vijayasundaram.
\newblock Triangular finite element methods for the {Euler} equations.
\newblock Computing methods in applied sciences and engineering {VI}, {Proc}.
  6th {Int}. {Symp}., {Versailles} 1983, 535-563, 1984.

\bibitem{dervieux2}
F.~Angrand and A.~Dervieux.
\newblock Some explicit triangular finite element schemes for the {Euler}
  equations.
\newblock {\em Int. J. Numer. Meth. Fluid}, 4:749--764, 1984.

\bibitem{ArR:17}
L.~Arpaia and M.~Ricchiuto.
\newblock r-adaptation for shallow water flows: conservation, well
  balancedness, efficiency.
\newblock {\em Comput. Fluids}, 160:175--203, 2018.

\bibitem{arpaia2020well}
L.~Arpaia and M.~Ricchiuto.
\newblock {Well balanced residual distribution for the ALE spherical shallow
  water equations on moving adaptive meshes}.
\newblock {\em J. Comput. Phys.}, 405:109173, 2020.

\bibitem{ArR:14}
L.~Arpaia, M.~Ricchiuto, and R.~Abgrall.
\newblock An {ALE} formulation for explicit {R}unge-{K}utta residual
  distribution.
\newblock {\em J. Sci. Comput.}, 190(34):1467--1482, 2014.

\bibitem{bcrt25}
W.~Barsukow, M.~Ciallella, M.~Ricchiuto, and D.~Torlo.
\newblock Genuinely multi-dimensional stationarity preserving global flux
  finite volume formulation for nonlinear hyperbolic pdes.
\newblock Preprint, {arXiv}:2506.21700, 2025.

\bibitem{Barsukow_nodeconservative2025}
W.~Barsukow, R.~Loub\`ere, and P.-H. Maire.
\newblock A node-conservative vorticity preserving finite volume method for
  linear acoustics on unstructured grids.
\newblock {\em Math. of. Comput.}, 94:2299--2343, 2025.

\bibitem{brt25}
W.~Barsukow, M.~Ricchiuto, and D.~Torlo.
\newblock Structure preserving nodal continuous finite elements via global flux
  quadrature.
\newblock {\em Numer. Methods Partial Differ. Equ.}, 41(1):e23167, 2025.

\bibitem{Barsukow}
Wasilij Barsukow, Jonathan Hohm, Christian Klingenberg, and Philip~L. Roe.
\newblock The active flux scheme on {Cartesian} grids and its low {Mach} number
  limit.
\newblock {\em J. Sci. Comput.}, 81(1):594--622, 2019.

\bibitem{Ben-Artzi2003}
M.~Ben-Artzi and J.~Falcovitz.
\newblock {\em {Generalized Riemann problems in Computational Fluids
  Dynamics}}.
\newblock Cambridge University press, 2003.

\bibitem{Boscheri2023}
W.~Boscheri, R.~Loub\`{e}re, and P.-H. Maire.
\newblock {An Unconventional Divergence Preserving Finite-Volume Discretization
  of Lagrangian Ideal MHD}.
\newblock {\em Commun. Appl. Math. Comput.}, 6:1665--1719, 2023.

\bibitem{Boscheri2022}
W.~Boscheri, R.~Loubère, and P.-H. Maire.
\newblock {A 3D cell-centered ADER MOOD Finite Volume method for solving
  updated Lagrangian hyperelasticity on unstructured grids}.
\newblock {\em J. Comput. Phys.}, 449, 2022.

\bibitem{Bouchut2004}
F.~Bouchut.
\newblock {\em Nonlinear Stability of Finite Volume Methods for Hyperbolic
  Conservation Laws}.
\newblock Birkhauser, first edition, 2004.

\bibitem{burman}
E.~Burman and P.~Hansbo.
\newblock Edge stabilization for {G}alerkin approximation of
  convection-diffusion-reaction problems.
\newblock {\em Comp. Meth. Appl. Mech. Engrg.}, 193:1437--1453, 2004.

\bibitem{BurQS:10}
E.~Burman, A.~Quarteroni, and B.~Stamm.
\newblock Interior penalty continuous and discontinuous finite element
  approximations of hyperbolic equations.
\newblock {\em J. Sci. Comput.}, 43(3):293--312, 2010.

\bibitem{Calhoun2023-ms}
D.~Calhoun, E.~Chudzik, and C.~Helzel.
\newblock The {C}artesian grid {A}ctive {F}lux method with adaptive mesh
  refinement.
\newblock {\em J. Sci. Comput.}, 94:54, 2023.

\bibitem{Caramana1998}
E.J. Caramana, D.E. Burton, M.J. Shashkov, and P.P. Whalen.
\newblock The construction of compatible hydrodynamics algorithms utilizing
  conservation of total energy.
\newblock {\em J. Comput. Phys.}, 146:227--262, 1998.

\bibitem{Chan2021}
A.~Chan, G.~Gallice, R.~Loub\`{e}re, and {P.-H.} Maire.
\newblock {Positivity preserving and entropy consistent approximate Riemann
  solvers dedicated to the high-order MOOD-based Finite Volume discretization
  of Lagrangian and Eulerian gas dynamics}.
\newblock {\em Comput. Fluids}, 229:27, 2021.
\newblock Id/No 105056.

\bibitem{graph}
F.~R.K. Chung.
\newblock {\em Spectral Graph Theory}, volume~92 of {\em CBMS Regional
  Conference Series in Mathematics}.
\newblock American Mathematical Society, 1997.

\bibitem{ciarlet}
P.~Ciarlet.
\newblock {\em The finite element method for elliptic problems}.
\newblock North-Holland, Amsterdam, 1978.

\bibitem{crd02}
\'A. Cs\'{\i}k, M.~Ricchiuto, and H.~Deconinck.
\newblock {A Conservative Formulation of the Multidimensional Upwind Residual
  Distribution Schemes for General Nonlinear Conservation Laws}.
\newblock {\em J. Comput. Phys.}, 179(2):286--312, 2002.

\bibitem{RD-ency}
H.~Deconinck and M.~Ricchiuto.
\newblock {Residual Distribution Schemes: Foundations and Analysis}.
\newblock In {\em Encyclopedia of Computational Mechanics Second Edition},
  pages 1--53. John Wiley \& Sons, Ltd, 2017.

\bibitem{Deconinck1993}
H.~Deconinck, P.L. Roe, and R.~Struijs.
\newblock {A multidimensional generalization of Roe's flux difference splitter
  for the Euler equations}.
\newblock {\em Comput. Fluids}, 22:215--222, 1993.

\bibitem{PHRaph2}
A.~{Del Grosso}, M.J. Castro, A.~Chan, G.~Gallice, R.~Loub\`ere, and P.-H.
  Maire.
\newblock {A well-balanced, positive, entropy--stable, and
  multi--dimensional--aware finite volume scheme for 2D shallow--water
  equations with unstructured grids}.
\newblock {\em J. Comput. Phys.}, 503:112829, 2024.

\bibitem{Delmas2025}
V.~Delmas, R.~Loub{\`e}re, and P.-H. Maire.
\newblock A node conservative cell-centered finite volume method for solving
  multidimensional {Euler} equations over general unstructured grids.
\newblock {\em J. Comput. Phys.}, 539:42, 2025.
\newblock Id/No 114246.

\bibitem{Despres2017}
B.~Despr\'{e}s.
\newblock {\em Numerical Methods for Eulerian and Lagrangian Conservation
  Laws}.
\newblock Birkhh\"{a}user, 2017.

\bibitem{Despres2005}
B.~Despr\'{e}s and C.~Mazeran.
\newblock {{L}agrangian gas dynamics in two dimensions and {L}agrangian
  systems}.
\newblock {\em Arch. Ration. Mech. Anal.}, 178:327--372, 2005.

\bibitem{ADERCWENO}
M.~Dumbser, W.~Boscheri, M.~Semplice, and G.~Russo.
\newblock Central weighted eno schemes for hyperbolic conservation laws on
  fixed and moving unstructured meshes.
\newblock {\em {SIAM} J. Sci. Comput.}, 39(6):A2564--A2591, 2017.

\bibitem{HPRmodel}
M.~Dumbser, I.~Peshkov, E.~Romenski, and O.~Zanotti.
\newblock {High order ADER schemes for a unified first order hyperbolic
  formulation of continuum mechanics: Viscous heat--conducting fluids and
  elastic solids}.
\newblock {\em J. Comput. Phys.}, 314:824--862, 2016.

\bibitem{ErnGuermond}
A.~Ern and J.-L. Guermond.
\newblock {\em Theory and practice of finite elements}, volume 159 of {\em
  Applied Mathematical Sciences}.
\newblock Springer Verlag, 2004.

\bibitem{GuermondErn}
A.~Ern and J.-L. Guermond.
\newblock Discontinuous {Galerkin} methods for {Friedrichs}' systems. {I}:
  {General} theory.
\newblock {\em SIAM J. Numer. Anal.}, 44(2):753--778, 2006.

\bibitem{AF3}
T.A. Eyman.
\newblock {\em Active flux}.
\newblock PhD thesis, University of Michigan, 2013.

\bibitem{AF1}
T.A. Eyman and P.L. Roe.
\newblock Active flux.
\newblock 49th AIAA Aerospace Science Meeting, 2011.

\bibitem{AF2}
T.A. Eyman and P.L. Roe.
\newblock Active flux for systems.
\newblock 20 th AIAA Computationa Fluid Dynamics Conference, 2011.

\bibitem{EymannRoe2013}
T.~A. Eymann and P.~L. Roe.
\newblock Multidimensional {A}ctive {F}lux schemes.
\newblock In American~Institute of~Aeronautics and Astronautics, editors, {\em
  21st AIAA Computational Fluid Dynamics Conference}, 2013.

\bibitem{Feireisl}
E.~Feireisl, M.~Luk{\'a}{\v{c}}ov{\'a}-Medvi{\v{d}}ov{\'a}, H.~Mizerov{\'a},
  and B.~She.
\newblock {\em Numerical analysis of compressible fluid flows}, volume~20 of
  {\em MS\&A, Model. Simul. Appl.}
\newblock Cham: Springer, 2021.

\bibitem{dervieux1}
F.~Fezoui, B.~Stoufflet, J.~Periaux, and A.~Dervieux.
\newblock Implicit high-order upwind finite-element schemes for the {Euler}
  equations.
\newblock Innovative numerical methods in engineering, {Proc}. 4th {Int}.
  {Symp}., {Atlanta}/{Ga}. 1986, 323-329 (1986)., 1986.

\bibitem{stoufflet}
L.~Fezoui and B.~Stoufflet.
\newblock A class of implicit upwind schemes for {Euler} simulations with
  unstructured meshes.
\newblock {\em J. Comput. Phys.}, 84(1):174--206, 1989.

\bibitem{Adessio1988}
M.~Cline F.L.~Adessio and J.K. Dukowicz.
\newblock {A General Topology, Godunov Method}.
\newblock {\em Comput. Phys. Commun.}, 48, 1988.

\bibitem{Oliver_Friedrich}
O.~Friedrich.
\newblock Weighted essentially non-oscillatory schemes for the interpolation of
  mean values on unstructured grids.
\newblock {\em J. Comput. Phys.}, 144(1):194--212, 1998.

\bibitem{grd25}
E.~Gaburro, M.~Ricchiuto, and M.~Dumbser.
\newblock {On general and complete multidimensional Riemann solvers for
  nonlinear systems of hyperbolic conservation laws}.
\newblock Preprint, {arXiv}:2506.00207, 2025.

\bibitem{GalliceHDR2002}
G.~Gallice.
\newblock {\em {Approximation num\'{e}rique de Syst\`{e}mes Hyperboliques
  Non-lin\'{e}aires Conservatifs ou Non-conservatifs}}.
\newblock Habilitation \`{a} diriger des recherches, Bordeaux University, 2002.
\newblock Available at
  \url{https://hal-cea.archives-ouvertes.fr/tel-01320526/file/Hdr.pdf}.

\bibitem{Gallice2003}
G.~Gallice.
\newblock {Positive and Entropy Stable Godunov-Type Schemes for Gas Dynamics
  and MHD Equations in Lagrangian or Eulerian Coordinates}.
\newblock {\em Numer. Math.}, 94(4):673--713, 2003.

\bibitem{Gallice2022}
G.~Gallice, A.~Chan, R.~Loub\`{e}re, and P.-H. Maire.
\newblock {Entropy stable and positivity preserving Godunov-type schemes for
  multidimensional hyperbolic systems on unstructured grid}.
\newblock {\em J. Comput. Phys.}, 468:33, 2022.
\newblock Id/No 111493.

\bibitem{Raviart1996}
E.~Godlewski and P.-A. Raviart.
\newblock {\em Numerical Approximation of Hyperbolic Systems of Conservation
  Laws}.
\newblock Number 118 in Applied Mathematical Sciences. Springer, 1996.

\bibitem{GodlewskiRaviartTome1}
E.~Godlewski and P.A. Raviart.
\newblock {\em Hyperbolic systems of conservation laws}.
\newblock Math\'ematiques et Applications. Ellipses, Paris, 1991.

\bibitem{Godunov}
S.~K. Godunov.
\newblock An interesting class of quasilinear systems.
\newblock {\em Sov. Math.}, 2:947--949, 1961.

\bibitem{GodKsenia}
S.K. Godunov.
\newblock Interesting class of quasilinear systems.
\newblock {\em J. Comput. Phys.}, 520:3, 2025.
\newblock Id/No 113521.

\bibitem{Godunov1979}
S.K. Godunov, A.~Zabrodine, M.~Ivanov, A.~Kraiko, and G.~Prokopov.
\newblock {\em R\'{e}solution num\'{e}rique des probl\`{e}mes
  multidimensionnels de la dynamique des gaz}.
\newblock Mir, 1979.

\bibitem{guermond2018second}
J.-L. Guermond, M.~Nazarov, B.~Popov, and I.~Tomas.
\newblock {Second-order invariant domain preserving approximation of the Euler
  equations using convex limiting}.
\newblock {\em {SIAM} J. Sci. Comput.}, 40(5):A3211--A3239, 2018.

\bibitem{IDP0}
J.-L. Guermond and B.~Popov.
\newblock Invariant domains and first-order continuous finite element
  approximation for hyperbolic systems.
\newblock {\em SIAM J. Numer. Anal.}, 54(4):2466--2489, 2016.

\bibitem{guermond2019invariant}
J.-L. Guermond, B.~Popov, and I.~Tomas.
\newblock Invariant domain preserving discretization-independent schemes and
  convex limiting for hyperbolic systems.
\newblock {\em Comp. Meth. Appl. Mech. Engrg.}, 347:143--175, 2019.

\bibitem{Gurtin2009}
M.E. Gurtin, E.~Fried, and L.~Anand.
\newblock {\em {The Mechanics and Thermodynamics of Continua}}.
\newblock Cambridge University Press, 2009.

\bibitem{Harten1981}
A.~Harten and {P.D.} Lax.
\newblock {A Random Choice Finite Difference Scheme for Hyperbolic Conservation
  Laws}.
\newblock {\em SIAM J. Numer. Anal.}, 18(2):289--315, 1981.

\bibitem{Harten1983}
A.~Harten, {P.D.} Lax, and B.~{van Leer}.
\newblock {On upstream Differencing and Godunov-Type schemes for Hyperbolic
  Conservation Laws}.
\newblock {\em SIAM Review}, 25(1):35--61, 1983.

\bibitem{He}
Fanchen He.
\newblock {\em Towards a New-generation Numerical Scheme for the Com- pressible
  Navier-Stokes Equations with the Active Flux Method}.
\newblock PhD thesis, {Applied and Interdisciplinary Mathematics, University of
  Michigan}, 2021.
\newblock https://deepblue.lib.umich.edu/handle/2027.42/169687.

\bibitem{Hughes1}
T.J.R. Hughes, L.P. Franca, and M.~Mallet.
\newblock A new finite element formulation for {CFD}: {I}. symmetric forms of
  the compressible {E}uler and {N}avier-{S}tokes equations and the second law
  of thermodynamics.
\newblock {\em Comp. Meth. Appl. Mech. Engrg.}, 54:223--234, 1986.

\bibitem{jeltsch2006curl}
R.~Jeltsch and M.~Torrilhon.
\newblock On curl-preserving finite volume discretizations for shallow water
  equations.
\newblock {\em BIT Numerical Mathematics}, 46:35--53, 2006.

\bibitem{Kluth2010}
G.~Kluth and B.~Despr\'{e}s.
\newblock {Discretization of hyperelasticity on unstructured mesh with a
  cell-centered Lagrangian scheme}.
\newblock {\em J. Comput. Phys.}, 229:9092--9118, 2010.

\bibitem{Kroner}
D.~Kr{\"o}ner.
\newblock {\em Numerical schemes for conservation laws}.
\newblock Chichester: Wiley; Stuttgart: Teubner, 1997.

\bibitem{Kroner:96}
D.~Kr\"oner, M.~Rokyta, and M.~Wierse.
\newblock A {Lax-Wendroff} type theorem for upwind finite volume schemes in
  $2$-d.
\newblock {\em East-West J. Numer. Math.}, 4(4):279--292, 1996.

\bibitem{Szekelyhidi}
S.~G. Krupa and L.~Sz{\'e}kelyhidi jun.
\newblock Contact discontinuities for 2-{D} isentropic {Euler} are unique in
  1-{D} but wildly non-unique otherwise.
\newblock {\em Commun. Math. Phys.}, 406(5):24, 2025.
\newblock Id/No 109.

\bibitem{KuzminFCT1}
D.~Kuzmin.
\newblock Explicit and implicit {FEM}-{FCT} algorithms with flux linearization.
\newblock {\em J. Comput. Phys.}, 228(7):2517--2534, 2009.

\bibitem{KuzminHennes}
D.~Kuzmin and H.~Hajduk.
\newblock {\em Property-preserving numerical schemes for conservation laws}.
\newblock Singapore: World Scientific, 2024.

\bibitem{KuzminFCT0}
D.~Kuzmin and M.~M{\"o}ller.
\newblock Algebraic flux correction. {II}: {Compressible} {Euler} equations.
\newblock In {\em Flux-corrected transport. Principles, algorithms, and
  applications. Papers based on the workshop 'High-resolution schemes for
  convection-dominated flows: 30 years of FCT'. With foreword by Jay P.
  Boris.}, pages 207--250. Berlin: Springer, 2005.

\bibitem{LxW}
P.D. Lax and B.~Wendroff.
\newblock Systems of conservation laws.
\newblock {\em Commun. Pure Appl. Math.}, 13:217--237, 1960.

\bibitem{Lipnikov2014}
K.~Lipnikov, G.~Manzini, and M.~Shashkov.
\newblock Mimetic finite difference method.
\newblock {\em J. Comput. Phys.}, 257:1163--1227, 2014.

\bibitem{LMR2016}
R.~Loub\`{e}re, P.-H. Maire, and B.~Rebourcet.
\newblock {\em {Handbook of Numerical Methods for Hyperbolic Problems: Basic
  and Fundamental Issues, edited by R. Abgrall and C.-W. Shu}}, chapter {13
  Staggered and colocated Finite Volume scheme for Lagrangian hydrodynamics},
  pages 319--352.
\newblock North Holland, 2016.

\bibitem{Maeng}
Jungyeoul Maeng.
\newblock {\em On the Advective Component of Active Flux Schemes for Nonlinear
  Hyperbolic Conservation Laws}.
\newblock PhD thesis, {Applied and Interdisciplinary Mathematics, University of
  Michigan}, 2017.
\newblock {https://deepblue.lib.umich.edu/handle/2027.42/138695}.

\bibitem{Maire2009}
{P.-H.} Maire.
\newblock {A high-order cell-centered {L}agrangian scheme for two-dimensional
  compressible fluid flows on unstructured meshes.}
\newblock {\em J. Comput. Phys.}, 228:2391--2425, 2009.

\bibitem{MaireHDR2011}
P.-H. Maire.
\newblock {\em {Contribution to the numerical modeling of Inertial Confinement
  Fusion}}.
\newblock Habilitation \`{a} diriger des recherches, Bordeaux University, 2011.
\newblock Available at \url{https://hal.science/tel-00589758v1}.

\bibitem{Maire2013}
P.-H. Maire, R.~Abgrall, J.~Breil, R.~Loubère, and B.~Rebourcet.
\newblock {A nominally second-order cell-centered Lagrangian scheme for
  simulating elastic-plastic flows on two-dimensional unstructured grids}.
\newblock {\em J. Comput. Phys.}, 235:626--665, 2013.

\bibitem{Maire2011}
P.-H. Maire, R.~Loub\`{e}re, and P.~V\'{a}chal.
\newblock {Staggered Lagrangian Discretization Based on Cell-Centered Riemann
  Solver and Associated Hydrodynamics Scheme}.
\newblock {\em Commun. Comput. Phys.}, 10(4):940--978, 2011.

\bibitem{dalmaso}
G.~Dal Maso, P.~G. LeFloch, and F.~Murat.
\newblock Definition and weak stability of nonconservative products.
\newblock {\em J. Math. Pures Appl.}, 74(6):483--548, 1995.

\bibitem{Menikoff1989}
R.~Menikoff and B.J. Plohr.
\newblock {The Riemann problem for fluid flow of real materials}.
\newblock {\em Reviews of Modern Physics}, 61(1), 1989.

\bibitem{merle1}
F.~Merle, P.~Rapha\"el, I.~Rodnianski, and J.~Szeftel.
\newblock On the implosion of a compressible fluid. {I}: {Smooth} self-similar
  inviscid profiles.
\newblock {\em Ann. Math. (2)}, 196(2):567--778, 2022.

\bibitem{merle2}
F.~Merle, P.~Rapha\"el, I.~Rodnianski, and J.~Szeftel.
\newblock On the implosion of a compressible fluid. {II}: {Singularity}
  formation.
\newblock {\em Ann. Math. (2)}, 196(2):779--889, 2022.

\bibitem{micalizzi2022new}
L.~Micalizzi and D.~Torlo.
\newblock {A new efficient explicit Deferred Correction framework: analysis and
  applications to hyperbolic PDEs and adaptivity}.
\newblock {\em Commun. Appl. Math. Comput.}, 2023.

\bibitem{Sixtine1}
S.~Michel, D.~Torlo, M.~Ricchiuto, and R.~Abgrall.
\newblock Spectral analysis of continuous {FEM} for hyperbolic {PDEs}:
  influence of approximation, stabilization, and time-stepping.
\newblock {\em J. Sci. Comput.}, 89(2):41, 2021.
\newblock Id/No 31.

\bibitem{Sixtine2}
S.~Michel, D.~Torlo, M.~Ricchiuto, and R.~Abgrall.
\newblock Spectral analysis of high order continuous {FEM} for hyperbolic
  {PDEs} on triangular meshes: influence of approximation, stabilization, and
  time-stepping.
\newblock {\em J. Sci. Comput.}, 94(3):48, 2023.
\newblock Id/No 49.

\bibitem{mishra2011constraint}
S.~Mishra and E.~Tadmor.
\newblock Constraint preserving schemes using potential-based fluxes {I}.
  multidimensional transport equations.
\newblock {\em Commun. Comput. Phys.}, 9(3):688--710, 2011.

\bibitem{Mock}
M.~S. Mock.
\newblock Systems of conservation laws of mixed type.
\newblock {\em J. Differ. Equations}, 37:70--88, 1980.

\bibitem{Morgan2021}
N.R. Morgan and B.J. Archer.
\newblock {On the Origins of Lagrangian Hydrodynamics Methods}.
\newblock {\em Nuclear Technology}, 207, 2021.

\bibitem{PeshRom2014}
I.~Peshkov and E.~Romenski.
\newblock A hyperbolic model for viscous {{N}ewtonian} flows.
\newblock {\em Continuum Mechanics and Thermodynamics}, 28:85--104, 2016.

\bibitem{Quirk1994}
J.J. Quirk.
\newblock {A contribution to the great Riemann solver debate}.
\newblock {\em Int. J. Numer. Meth. Fluid}, 18:555--574, 1994.

\bibitem{HDR-MR}
M.~Ricchiuto.
\newblock {\em {Contributions to the development of residual discretizations
  for hyperbolic conservation laws with application to shallow water flows}}.
\newblock Habilitation {\`a} diriger des recherches, {Universit{\'e} Sciences
  et Technologies - Bordeaux I}, December 2011.
\newblock PDF available at {\tt https://theses.hal.science/tel-00651688/}.

\bibitem{Mario2015}
M.~Ricchiuto.
\newblock An explicit residual based approach for shallow water flows.
\newblock {\em J. Comput. Phys.}, 280:306--304, 2015.

\bibitem{Mario}
M.~Ricchiuto and R.~Abgrall.
\newblock {Explicit Runge-Kutta residual distribution schemes for time
  dependent problems: second order case.}
\newblock {\em J. Comput. Phys.}, 229(16):5653--5691, 2010.

\bibitem{rcd05}
M.~Ricchiuto, \'A. Cs\'{\i}k, and H.~Deconinck.
\newblock {Residual distribution for general time dependent conservation laws}.
\newblock {\em J. Comput. Phys.}, 209(1):249--289, 2005.

\bibitem{rf14}
M.~Ricchiuto and A.G. Filippini.
\newblock {Upwind residual discretization of enhanced Boussinesq equations for
  wave propagation over complex bathymerties}.
\newblock {\em J. Comput. Phys.}, 271:306--341, 2014.

\bibitem{Roe:87}
P.~L. Roe.
\newblock Linear advection schemes on triangular meshes.
\newblock Technical Report CoA 8720, Cranfield Institute of Technology, 1987.

\bibitem{Roe:90}
P.~L. Roe.
\newblock {``Optimum'' upwind advection on a triangular mesh}.
\newblock Technical Report ICASE 90-75, ICASE, 1990.

\bibitem{Roe81}
P.L. Roe.
\newblock Approximate {Riemann} solvers, parameter vectors, and difference
  schemes.
\newblock {\em J. Comput. Phys.}, 43:357--372, 1981.

\bibitem{Roe92}
P.L. Roe and D.~Sidilkover.
\newblock Optimum positive linear schemes for advection in two and three
  dimensions.
\newblock {\em SIAM J. Numer. Anal.}, 29(6):1542--1568, 1992.

\bibitem{Shashkov1996}
M.J. Shashkov.
\newblock {\em {Conservative Finite-Difference Methods on General Grids}}.
\newblock CRC Press, 1996.

\bibitem{Shen2014_I}
Z.~Shen, W.~Yan, and G.~Yuan.
\newblock {A robust and contact resolving Riemann solver on unstructured mesh,
  Part I, Euler method}.
\newblock {\em J. Comput. Phys.}, 268:432--455, 2014.

\bibitem{shuConservation}
C.~Shi and C.-W. Shu.
\newblock On local conservation of numerical methods for conservation laws.
\newblock {\em Comput. Fluids}, 169:3--9, 2018.

\bibitem{struijs}
R.~Struijs, H.~Deconinck, and P.L. Roe.
\newblock Fluctuation splitting schemes for the 2{D} {E}uler equations.
\newblock VKI-LS 1991-01, 1991.
\newblock Computational Fluid Dynamics.

\bibitem{TadmorVieux}
E.~Tadmor.
\newblock The numerical viscosity of entropy stable schemes for systems of
  conservation laws, {I}.
\newblock {\em Math. of. Comput.}, 49:91--103, 1987.

\bibitem{TadmorActa}
E.~Tadmor.
\newblock Entropy stability theory for difference approximations of nonlinear
  conservation laws and related time-dependent problems.
\newblock {\em Acta Numerica}, 13:451--512, 2003.

\bibitem{vanderWeideDeconinck}
E.~van~der Weide and H.~Deconinck.
\newblock Matrix distribution schemes for the system of {Euler} equations.
\newblock In {\em Euler and Navier-Stokes solvers using multi-dimensional
  upwinds schemes and multigrid acceleration. Results of the BRITE/EURAM
  projects AERO-CT89-00003 and AER2-CT92-00040, 1989-1995}, pages 113--139.
  Wiesbaden: Vieweg, 1997.

\bibitem{Vilar2012}
F.~Vilar.
\newblock {\em A high-order Discontinuous Galerkin discretization for solving
  two-dimensional Lagrangian hydrodynamics}.
\newblock PhD thesis, Bordeaux University, 2012.
\newblock Available at \url{https://theses.hal.science/tel-00765575v1}.

\bibitem{Vilar1}
F.~Vilar.
\newblock \emph{A posteriori} correction of high-order discontinuous {Galerkin}
  scheme through subcell finite volume formulation and flux reconstruction.
\newblock {\em J. Comput. Phys.}, 387:245--279, 2019.

\bibitem{Vilar2}
F.~Vilar.
\newblock Local subcell monolithic {DG}/{FV} convex property preserving scheme
  on unstructured grids and entropy consideration.
\newblock {\em J. Comput. Phys.}, 521:32, 2025.
\newblock Id/No 113535.

\bibitem{VilarAbgrall}
F.~Vilar and R.~Abgrall.
\newblock A posteriori local subcell correction of high-order discontinuous
  {Galerkin} scheme for conservation laws on two-dimensional unstructured
  grids.
\newblock {\em {SIAM} J. Sci. Comput.}, 46(2):a851--a883, 2024.

\bibitem{Vilar2014}
F.~Vilar, P.-H. Maire, and R.~Abgrall.
\newblock {A discontinuous Galerkin discretization for solving the
  two-dimensional gas dynamics equations written under total Lagrangian
  formulation on general unstructured grids}.
\newblock {\em J. Comput. Phys.}, 276:188--234, 2014.

\bibitem{Neumann1950}
J.~von Neumann and R.D. Richtmyer.
\newblock A method for the calculation of hydrodynamics shocks.
\newblock {\em Journal of Applied Physics}, 21:232--237, 1950.

\bibitem{Whalen1996}
P.P. Whalen.
\newblock Algebraic limitations on two-dimensional hydrodynamics simulations.
\newblock {\em J. Comput. Phys.}, 124:46--54, 1996.

\bibitem{XU2023112297}
K.~Xu, Z.~Gao, Z.~Qian, C.~Jiang, and C.-H. Lee.
\newblock Numerical path preserving godunov schemes for hyperbolic systems.
\newblock {\em Journal of Computational Physics}, 490:112297, 2023.

\end{thebibliography}

\end{document}